\documentclass[12pt]{amsart}
\usepackage{amsmath}
\usepackage{amsthm}
\usepackage{amssymb}
\usepackage{amscd}
\usepackage{bm}
\usepackage[all]{xy}
\usepackage{mathrsfs}
\usepackage{hyperref}
\usepackage{upref}
\usepackage{graphicx}
\usepackage[bbgreekl]{mathbbol}
\usepackage {color}

 \makeatletter
  
  \@addtoreset{equation}{section}
 \makeatother 

\makeatletter
\renewcommand{\@secnumfont}{\bfseries}
\makeatother

\theoremstyle{plain} 
 \newtheorem{theorem}[equation]{Theorem}
 \newtheorem{proposition}[equation]{Proposition}
 \newtheorem{propdef}[equation]{Proposition and Definition}
 \newtheorem{lemma}[equation]{Lemma}
 
 \newtheorem{corollary}[equation]{Corollary} 

 \theoremstyle{definition}
 \newtheorem{definition}[equation]{Definition}
 \newtheorem{condition}[equation]{Condition}
 
\theoremstyle{remark}
 \newtheorem{remark}[equation]{Remark}
 \newtheorem{example}[equation]{Example}
 \newtheorem{condition-list}[equation]{}
 
\def \N {\mathbb{N}}
\def \Z {\mathbb{Z}}

\def \Spf {\text{\rm Spf}}
\def \Spec {\text{\rm Spec}}

\def \Tor {\text{\rm Tor}}
\def \PD {\text{\rm PD}}
\def \Hom {\text{\rm Hom}}

\def \HIG {\text{\rm HIG}}
\def \Ker {\text{\rm Ker}}

\def \id {\text{\rm id}} 
\def \fpqc {\text{\rm fpqc}}
\def \Zar {\text{\rm Zar}}
\def \Der {\text{\rm Der}}
\def \MIC {\text{\rm MIC}}
\def \Strat {\text{\rm Strat}}
\def \Crystal{\text{\rm CR}}

\def \qnilp {\text{\rm q-nilp}}
\def \ZAR {\text{\rm ZAR}}
\def \Mod {\text{\rm\bf Mod}}

\def \CechA {\text{\rm \v{C}A}}
\def \cs {\text{\rm cs}}
\def \Tot {\text{\rm Tot}}
\def \Cov {\text{\rm Cov}}

\def \FSch {\text{\rm FSch}}

\def \hCrystal {\widehat{\text{\rm CR}}}
\def \Ob {\text{\rm Ob}\,}

\def \gr {\text{\rm gr}}
\def \ad {\text{\rm ad}}

\def \Mor {\text{\rm Mor\,}}
\def \DAlg {\text{\rm\bf DAlg}}

\def \Omegab {\Omega^{\bullet}}

\def \bis {$^{\text{bis}}$\,}

\def \projl {\underset{\leftarrow}\ell}

\def \pq {[p]_q}
\def \prodl {\prod\limits}
\def \suml {\sum\limits}

\def \fX {\mathfrak{X}}
\def \fY {\mathfrak{Y}}

\def \fS {\mathfrak{S}}
\def \fU {\mathfrak{U}}
\def \fD {\mathfrak{D}}

\def \fV {\mathfrak{V}}

\def \fa {\mathfrak{a}}
\def \fh {\mathfrak{h}}
\def \fm {\mathfrak{m}}

\def \CO {\mathcal {O}}

\def \CH {\mathcal{H}}

\def \CF {\mathcal{F}}

\def \CI {\mathcal{I}}

\def \CM {\mathcal{M}}
\def \CN {\mathcal {N}}
\def \CG {\mathcal{G}}

\def \CL {\mathcal{L}}

\def \uCF {\underline{\mathcal{F}}}
\def \uCG {\underline{\mathcal{G}}}
\def \uCH {\underline{\mathcal{H}}}
\def \uCM {\underline{\mathcal{M}}}

\def \scrS {\mathscr{S}}
\def \scrD {\mathscr{D}}
\def \hg {\widehat{g}}

\def \hD {\widehat{D}}
\def \hA {\widehat{A}}
\def \hR {\widehat{R}}
\def \hJ {\widehat{J}}
\def \hM {\widehat{M}}
\def \hf {\widehat{f}}

\def \hpartial{\widehat{\partial}}
\def \halpha {\widehat{\alpha}}
\def \hbeta {\widehat{\beta}}
\def \hvarphi {\widehat{\varphi}}
\def \hoD {\widehat{\overline{D}}}
\def \hh {\widehat{h}}
\def \hotimes {\widehat{\otimes}}

\def \tA {\widetilde{A}}
\def \tR {\widetilde{R}}
\def \tD {\widetilde{D}}
\def \tS {\widetilde{S}}
\def \tM {\widetilde{M}}
\def \tN {\widetilde{N}}
\def \tI {\widetilde{I}}
\def \ty {\widetilde{y}}
\def \ta {\widetilde{a}}
\def \te {\widetilde{e}}
\def \ts {\widetilde{s}}

\def \tb {\widetilde{b}}

\def \ttau {\widetilde{\tau}}
\def \tlt {\widetilde{t}}
\def \tT {\widetilde{T}}
\def \tLambda {\widetilde{\Lambda}}
\def \tiota {\widetilde{\iota}}
\def \tchi {\widetilde{\chi}}
\def \tuA {\widetilde{\underline{A}}}
\def \tue {\widetilde{\underline{e}}}
\def \tualpha {\widetilde{\underline{\alpha}}}
\def \tuiota {\widetilde{\underline{\iota}}}
\def \tupartial {\widetilde{\underline{\partial}}}

\def \oJ {\overline{J}}
\def \oT {\overline{T}}
\def \oS {\overline{S}}
\def \oR {\overline{R}}
\def \oA {\overline{A}}
\def \oD {\overline{D}}
\def \oT {\overline{T}}
\def \oB {\overline{B}}

\def \oP {\overline{P}}

\def \oa {\overline{a}}
\def \ob {\overline{b}}
\def \ot {\overline{t}}

\def \otau {\overline{\tau}}

\def \outau {\overline{\underline{\tau}}}
\def \osigma {\overline{\sigma}}

\def \oalpha {\overline{\alpha}}
\def \obeta {\overline{\beta}}
\def \opartial{\overline{\partial}}

\def \ofD {\overline{\mathfrak{D}}}

\def \tsigma {\widetilde{\sigma}}

\def \tx {\widetilde{x}}

\def \tD {\widetilde{D}}
\def \uD {\underline{D}}
\def \uM {\underline{M}}

\def \uA {\underline{A}}
\def \ul {\underline{l}}
\def \uc {\underline{c}}
\def \ug {\underline{g}}
\def \un {\underline{n}}

\def \ut  {\underline{t}}
\def \us {\underline{s}}
\def \ue {\underline{e}}
\def \ui {\underline{i}}
\def \ualpha {\underline{\alpha}}
\def \ubeta {\underline{\beta}}
\def \utau {\underline{\tau}}
\def \uiota {\underline{\iota}}
\def \usigma {\underline{\sigma}}
\def \utheta {\underline{\theta}}
\def \ugamma {\underline{\gamma}}

\def \uone {\underline{1}}
\def \um {\underline{m}}
\def \uj {\underline{j}}

\def \upartial {\underline{\partial}}

\def \bmI{\bm{I}}

\def \pq {[p]_q}

\DeclareSymbolFontAlphabet{\mathbbm}{bbold}
\DeclareSymbolFontAlphabet{\mathbb}{AMSb}

\newcommand{\prism} {\scalebox{0.8}{$\mathbbm \Delta$}}

\begin{document}

\title{Prismatic crystals and  $q$-Higgs fields}

\author{Takeshi Tsuji}
\address{Graduate School of Mathematical Sciences, The University of Tokyo, 
3-8-1 Komaba, Meguro-ku, Tokyo, 153-8914, Japan}
\email{t-tsuji@ms.u-tokyo.ac.jp}

\begin{abstract}
Similarly to the theory of crystalline cohomology, we give a local description of a prismatic crystal and its cohomology in terms of a $q$-Higgs module and the associated $q$-Higgs  complex 
on the  bounded prismatic envelope of an embedding into a framed smooth algebra,
when the base bounded prism is defined over the $q$-crystalline prism.
We also discuss its behavior under the scalar extension
by the Frobenius lifting and tensor products, and a
global description of the cohomology of a prismatic crystal via Zariski cohomological descent.
\end{abstract}

\maketitle

\tableofcontents

\section*{Introduction}

After the innovative work \cite{BMS} and \cite{BMS2} by B.~Bhatt, M.~Morrow, and P.~Scholze in the
integral $p$-adic Hodge theory, where they developed new cohomology theories
(whose coefficient rings are Fontaine's period ring $A_{\inf}$ and that of 
Breuil-Kisin modules $\fS=W[[u]]$) 
relating various known integral $p$-adic cohomology
theories, Bhatt and Scholze \cite{BS} introduced the notion of prisms and prismatic cohomology,
which is quite influential and has been  supplying us new and often refined ways to understand 
known integral $p$-adic cohomology theories, 
and also theories of $p$-adic Galois representations and $p$-adic local systems. 
Prismatic cohomology allows any bounded prism as a base, and the new cohomology 
theories in \cite{BMS} and \cite{BMS2} mentioned above can be reconstructed by 
taking $A_{\inf}$ and $\fS$ as a base. 

A prismatic site is a variant of a crystalline site,  and a notion of crystals is defined naturally.
The goal of this paper is to give a local description of a prismatic crystal and
its cohomology in terms of a $q$-Higgs module and its $q$-Higgs complex on the 
bounded prismatic envelope of an embedding into a framed smooth algebra,
when the base bounded prism is defined over the $q$-crystalline prism $(\Z_p[[q-1]],\pq\Z_p[[q-1]])$,
$\pq=\frac{q^p-1}{q-1}$, $\delta(q)=0\;(\Leftrightarrow \varphi(q)=q^p)$ (\cite[Example 1.3 (4)]{BS}).
 We also discuss its behavior under the scalar extension
by the Frobenius lifting and tensor products, and a
global description of the cohomology of a prismatic crystal via Zariski cohomological descent.
We remark that a local description of prismatic cohomology with constant coefficients
in terms of $q$-de Rham complexes has been already given in \cite[\S16]{BS}  via comparison 
with $q$-crystalline cohomology.

These kinds of local descriptions of a prismatic crystal and its cohomology
have already been studied in many papers. For a crystal on a relative prismatic site,
we have the following results:
by 
A.~Chatzistamatiou \cite{Chat} in terms of modules over a ring of 
$q$-differential operators; 
by M.~Morrow and the author \cite{MT} over the base $A_{\inf}$ in terms of $q$-connections
and $q$-Higgs fields (see below);
by Y.~Tian \cite{YTian} for reduced crystals (i.e.~crystals annihilated by the 
structure ideal of the base prism) over an arbitrary base bounded prism in terms of classical Higgs fields;
by B.~Bhatt and J.~Lurie \cite{BLPrismatization} under the same setting as Y.~Tian via
the relative Hodge-Tate stack; 
by A.~Ogus \cite{OgusCrysPrism} when the base bounded prism is defined over the $p$-crystalline prism
$(\Z_p,p\Z_p)$, i.e.~the case $q=1$, in terms of $p$-connections; and 
by M.~Gros, B.~Le Stum, and A.~Quir\'os in one dimensional case \cite{GLSQ}
by using a $q$-calculus developed in their preceding work.
The first two deal only with local descriptions of a crystal.
In every work except Ogus', a local description is given only over a framed smooth lifting,
whereas a local description over the bounded prismatic envelope of an embedding
into a framed smooth algebra is necessary to give a global description 
of cohomology via Zariski cohomological descent. As an application of the
global description, we prove a comparison theorem between prismatic
cohomology and $A_{\inf}$-cohomology with coefficients in a
subsequent paper \cite{TsujiPrismAinfCoh}.

We state our results on a local description of a prismatic crystal
and its cohomology, which have the same form as the
theory of crystalline cohomology. 
 Let $(R,\pq R)$ be a bounded prism over
the $q$-crystalline prism $(\Z_p[[q-1]],\pq\Z_p[[q-1]])$, let 
$\fX=\Spf(A)$ be a $p$-adic affine formal scheme over $\Spf(R/\pq R)$,
and let $\Spf(A)\hookrightarrow\fY=\Spf(B)$ be a closed immersion
into a smooth $(p,\pq)$-adic affine formal scheme over $\Spf(R)$
endowed with coordinates $t_i\in B$ $(i=1,2,\ldots, d)$.
We equip $B$ with the unique $\delta$-algebra structure
over $R$ satisfying $\delta(t_i)=0$ $(i\in \N \cap [1,d])$ 
(Proposition \ref{prop:DeltaStrExtEtSm} (2))
and assume that the pair of the $\delta$-ring $B$ and
the kernel of the surjection $B\to A$ has the bounded
prismatic envelope $D$ (Definition \ref{def:bddPrsimEnv}) 
over $(R,\pq R)$. (See Propositions \ref{prop:PrismEnvRegSeq} 
and \ref{prop:bddPrismEnvSmSmEmb} for sufficient conditions
for this assumption.)
Under this assumption, we define the notion of
a $q$-Higgs field on a $D$-module with respect to the
coordinates $t_i$ $(i\in\N\cap [1,d])$ (Definition \ref{def:qHiggsMod} (2)), 
and prove the following theorem.

\begin{theorem}[Proposition \ref{prop:CrysStratEquiv}, 
Theorem \ref{thm:StratHiggsEquiv}] \label{thm:main1}
Let $J$ be an ideal of $R$ containing a power of $pR+\pq R$.
Then there exists a canonical equivalence between
the category of crystals annihilated by $J$ on the prismatic site 
$(\fX/(R,\pq R))_{\prism}$ 
and that of $D$-modules annihilated by $J$
with quasi-nilpotent $q$-Higgs field.
\end{theorem}

We equip $(\fX/(R,\pq R))_{\prism}$ with a kind of fpqc topology
(Definition \ref{def:PrismaticSite} (1)), and write 
$u_{\fX/(R,\pq R)}$ for the projection to the Zariski topos
$(\fX/(R,\pq R))_{\prism}^{\sim}\to \fX_{\Zar}^{\sim}$ \eqref{eq:PrismToposZarProj}.
Let $v_D$ denote the morphism of topos 
$\Spf(D)_{\Zar}=\Spf(D/\pq D)_{\Zar}\to \fX_{\Zar}^{\sim}$ 
defined by the canonical morphism $A\to D/\pq D$. 

\begin{theorem}[Theorem \ref{th:CrystalCohqHiggs}]\label{thm:main2}
Let $J$ be an ideal of $R$ containing a power of $pR+\pq R$,
let $\CF$ be a crystal annihilated by $J$ on $(\fX/(R,\pq R))_{\prism}$,
and let $M$ be the $D$-module with quasi-nilpotent $q$-Higgs field
corresponding to $\CF$ by the equivalence in Theorem \ref{thm:main1}. 
Then we have a canonical isomorphism functorial in $\CF$
$$Ru_{\fX/(R,\pq R)*}\CF\cong v_{D*}(\CM\otimes_{\CO_{\fD}}q\Omega^{\bullet}_{\fD/R}),$$
where $\CM\otimes_{\CO_{\fD}}q\Omega^{\bullet}_{\fD/R}$
denotes the $q$-Higgs complex  on $\fD=\Spf(D)$ associated to $M$
\eqref{eq:qHiggsShfCpx}.
\end{theorem}

We also prove a variant for a complete crystal of $\CO_{\fX/R}$-modules
(Definition \ref{def:PrismaticSite} (2))
on $(\fX/(R,\pq R))_{\prism}$ 
(Theorem \ref{thm:CrystalCohqHigProj}) 
by considering the inverse system consisting of the reduction modulo
$(p,\pq)^{n+1}$ of $\CF$ for $n\in \N$. 
We prove Theorem \ref{thm:main2} along the same lines as the proof for crystalline cohomology
given in \cite{BerthelotCrisCoh} and \cite{BO}, i.e., by taking a resolution
\eqref{eq:qdRResol} of $\CF$ by a kind of  linearization
of the $q$-Higgs complex of $M$ and showing that its derived projection
to $\fX_{\Zar}$ gives the complex $v_{D*}(\CM\otimes_{\CO_{\fD}}q\Omegab_{\fD/R})$
(\eqref{eq:qdRCpxLinZarProj}, Proposition \ref{prop:LinearizationLocCoh}).
The proof of the fact that the linearization gives a resolution, i.e., a $q$-Higgs analogue
of Poincar\'e lemma is reduced to the case $(p,q-1)M=0$ and then to
the usual Poincar\'e lemma. 
(See Lemma \ref{lem:CrystalPoincareLemma}, Proposition \ref{prop:qprismEnvPL} (4), 
and Theorem \ref{thm:PoincareLem}.) 

As it is already mentioned, we also give global descriptions: Theorems \ref{th:CrystalCohqHiggsSimpl} 
and \ref{thm:CrystalCohqHiggsSimplProj}. 
In order to show them by Zariski cohomological descent on $\fX$ by constructing a simplicial 
version \eqref{eq:CrysZarProjDolCpxHC} of Theorem \ref{thm:main2}, we need 
a map between $q$-Higgs complexes with respect to scalar extension under a non-smooth homomorphism
(inducing a map between the sets of coordinates), 
for which the definition of scalar extension of a $q$-Higgs field is a little involved,
the naive pull-back map between $q$-Higgs complexes 
is not compatible with differential maps, and it requires a suitable twist by 
endomorphisms associated to the $q$-Higgs field. The twisted pull-back map between
$q$-Higgs complexes thus constructed 
is not compatible with the product morphism of $q$-Higgs complexes 
for the tensor product of two crystals;
to make it compatible, we need to consider the $q$-Higgs complex associated to
the tensor product on the bounded prismatic envelope of the diagonal immersion
$\fX\to \fY\times_{\Spf(R)}\fY$ (Remark \ref{rmk:qdRModfProdFunct}).
These are discussed under a general setting 
in \S\ref{sec:ScalarExtConnection}. See \eqref{eq:ConnPBFormula1},
the construction of \eqref{eq:dRCpxFunct}, and Proposition \ref{prop:ProddRCpxFunct}.

To give an idea  how a $q$-Higgs field and a $q$-Higgs complex look like
and to clarify problems in studying the case of a general base (lying over $\Z_p[[q-1]]$)
and a general crystal, we explain how a $q$-Higgs module and its Higgs complex
are constructed from a locally finite free crystal on the prismatic site of $\Spf(A)$ over $(R,\pq R)$ in \cite{MT}
when $R=A_{\inf}$,  $B$ is a lifting of $A$, and $t_i$ are invertible.

Let $C$ be a perfectoid field containing all $p$-power roots of unity
(e.g.~the $p$-adic completion of an algebraic closure
of a complete discrete valuation ring of mixed characteristic $(0,p)$),
let $\CO$ be its ring of integers, and let $A_{\inf}$ be Fontaine's period ring
$W(\CO^{\flat})$
($\CO^{\flat}=\varprojlim_{\N,\text{Frob}} \CO/p\CO$), 
which is equipped with a surjective ring
homomorphism $\vartheta\colon A_{\inf}\to \CO$.  
We choose and fix a compatible system of primitive $p^n$th roots
$\zeta_{n}\in \CO$, $\zeta_{n+1}^p=\zeta_n$, put
$\zeta=(\zeta_n\mod p)_{n\in\N}\in \CO^{\flat}$, and regard
$A_{\inf}$ as a prism over $\Z_p[[q-1]]$ by $q\mapsto [\zeta]$.
Then we can identify $\CO$ with the quotient $A_{\inf}/\pq A_{\inf}$ by the map 
$\vartheta\varphi^{-1}\colon A_{\inf}\to \CO$ whose kernel is $\pq A_{\inf}$. 

Let $\Spf(A)$ be a smooth $p$-adic affine formal scheme over $\Spf(\CO)$, and assume that we are given 
its smooth lifting $\Spf(B)$ over $\Spf(A_{\inf})$ with invertible coordinates $t_1,\ldots, t_d$.
By the \'etaleness of the $A_{\inf}$-algebra homomorphism 
$A_{\inf}[T_1^{\pm 1},\ldots, T_d^{\pm 1}]\to B;T_i\mapsto t_i$ modulo $(p,\pq)^{n+1}$ for each $n\in \N$,
there exists a unique automorphism $\gamma_i$ of $B$ over 
$A_{\inf}$ for each $i\in \N\cap [1,d]$ such that
$\gamma_i(t_i)=q^pt_i$ and $\gamma_i(t_j)=t_j$ $(j\neq i)$, and that 
$\gamma_i\equiv \id$ modulo $q^p-1$. They commute with each other. Since $A_{\inf}$
and hence $B$ are $(q-1)$-torsion free,  the triviality of $\gamma_i$ modulo $q^p-1$
allows us to define $A_{\inf}$-linear endomorphisms  $\theta_i=t_i^{-1}(q-1)^{-1}(\gamma_i-1)$ 
of $B$, which satisfy $\theta_i(t_i^n)=[pn]_qt_i^{n-1}$, $\theta_i(t_j^n)=0$ $(n\in \Z, j\neq i)$, where
$[m]_q$ $(m\in \Z)$ denotes the $q$-analogue $\frac{q^m-1}{q-1}$ of $m$,
 and a twisted Leibnitz rule
$\theta_i(ab)=\theta_i(a)\gamma_i(b)+a\theta_i(b)$ $(a,b\in B)$.
Equipping $B$ with the $\delta$-algebra structure over $A_{\inf}$ 
characterized by $\delta(t_i)=0\;(\Leftrightarrow \varphi(t_i)=t_i^p)$
and compatible with $\gamma_i$, 
we obtain an object $(\Spf(B)\leftarrow \Spf(B/\pq B)\xrightarrow{\id}\Spf(A))$ 
of the prismatic site of $\Spf(A)$ over $(A_{\inf},\pq A_{\inf})$
endowed with automorphisms $\gamma_i$.  When we construct a crystal 
from a $q$-Higgs module, we also use similar automorphisms of the product
$(\Spf(D(r))\leftarrow \Spf(D(r)/\pq D(r))\to\Spf(A))$ $(r=1,2)$
of $r+1$ copies of 
$(\Spf(B)\leftarrow \Spf(B/\pq B)\xrightarrow{\id}\Spf(A))$ in the prismatic site
(which is given by the bounded prismatic envelope of the completed tensor product $A(r)$ of $r+1$ copies of
$B$ over $A_{\inf}$ with respect to the kernel of the map to $A$)
and the associated twisted derivations on $D(r)$ defined as above.

Now suppose that we are 
given a locally finite free prismatic crystal $\CF$ on $\Spf(A)$ over $(A_{\inf},\pq A_{\inf})$.
Then the evaluation of $\CF$ at the object 
$(\Spf(B)\leftarrow \Spf(B/\pq B)\xrightarrow{\id}\Spf(A))$ 
above gives a finite projective
$B$-module $M$ equipped with $\gamma_i$-semilinear endomorphisms
$\gamma_{M,i}$ $(i\in\N\cap[1,d])$ commuting with each other
and trivial modulo $q-1$. The last triviality follows from the fact that
$\gamma_i$ becomes trivial on the object 
$(\Spf(B/(q-1)B)\leftarrow \Spf((B/(q-1)B)/p)\to\Spf(A))$ defined by the quotient prism $(B/(q-1)B,pB/(q-1))$
of the prism $(B,\pq B)$. The $q$-Higgs field on $M$ is defined by 
the commutative family of endomorphisms $\theta_{M,i}=t_i^{-1}(q-1)^{-1}(\gamma_{M,i}-1)$
which satisfy $\theta_{M,i}(ax)=\gamma_i(a)\theta_{M,i}(x)+\theta_i(a)x$ $(a\in B, x\in M)$. 
Note that $M$ is $(q-1)$-torsion free as $M$ is a finite projective $B$-module. 
Letting 
$q\Omega_{B/A_{\inf}}$ be the free $A$-module with a formal basis $d_{q}t_i$ $(i\in \N\cap [1,d])$
and letting $q\Omega_{B/A_{\inf}}^r$ $(r\in\N)$
denote its $r$th exterior power, we define
a complex $(M\otimes_{B}q\Omega_{B/A_{\inf}}^{\bullet},\theta^{\bullet}_M)$, which
we call the associated $q$-Higgs complex following \cite[\S2.3]{MT}, by 
$\theta^r_M(x\otimes d_qt_{i_1}\wedge\cdots\wedge d_qt_{i_r})
=\sum_{1\leq i\leq d}\theta_{M,i}(x)\otimes d_qt_i\wedge d_qt_{i_1}\wedge\cdots\wedge d_qt_{i_r}$.

\begin{remark}
If we take the reduction modulo $q-1$, i.e. the scalar extension to the crystalline prism
$(\Z_p,p\Z_p)$, then $\theta_i$ becomes
$p\frac{\partial}{\partial t_i}$ with respect to the coordinates
$t_1,\ldots, t_d$, and therefore the family $\theta_{M,i}$ becomes 
a $p$-connection\footnote{If we take the reduction modulo $\pq$, the structure ideal of 
the prism $(A_{\inf},\pq A_{\inf})$, then $\theta_i$ becomes $0$ because
$\gamma_i\equiv \id_{B}$ modulo $q^p-1$. Therefore
the family $\theta_{M,i}$ becomes a Higgs field as observed 
in the work of Y.~Tian \cite{YTian}. 
}.
In other words, the family $\theta_{M,i}$ is a $q$-deformation of a $p$-connection, and
it might be better to call it a $\pq$-$q$-connection. However
we call it a $q$-Higgs field because it appears in a $q$-analogue of the $p$-adic Simpson
correspondence in \cite{MT}.
\end{remark}

The construction of $\theta_i$ on $B$ and its variant for $D(r)$ $(r=1,2)$ 
work only when the base bounded prism is $(q-1)$-torsion free, and
that of $\theta_{M,i}$ does not work even when the base is $(q-1)$-torsion free
if we do not assume that  $\CF$ is locally finite free. For the first problem, observing
that $\theta_i$ satisfies $\theta_i(ab)=\theta_i(a)b+a\theta_i(b)+t_i(q-1)\theta_i(a)\theta_i(b)$,
we introduce a notion of a twisted derivation as follows: For an algebra $S$ over a ring $R$
and an element $\alpha$ of $S$, we define an $\alpha$-derivation 
$\partial\colon S\to S$ to be an $R$-linear map satisfying 
$\partial(1)=0$ and $\partial(ab)=\partial(a)b+a\partial(b)+\alpha\partial(a)\partial(b)$
$(a,b\in S)$ (Definition \ref{def:alphaDerivation} (1)). 
Then an $\alpha$-derivation has an interpretation as an $R$-algebra section
$s\colon S\to E^{\alpha}(S):=S[T]/(T^2-\alpha T);a\mapsto a+\partial(a)T$ of 
the $S$-homomorphism $E^{\alpha}(S)\to S;T\mapsto 0$ (Proposition \ref{prop:AlphaDerivInterpret}).
When $R$ is a $\delta$-ring and $S$ is a $\delta$-algebra over $R$,
we further introduce a notion of compatibility of $\partial$ with a $\delta$-structure
(Definition \ref{def:alphaDerivDeltaComp}),
which is interpreted as the $\delta$-compatibility of $s$ for a suitable
$\delta$-structure on $E^{\alpha}(S)$ (Proposition \ref{prop:DeltaAlphaDerivInt}).
Using this simple interpretation of $\partial$ by $s$ (instead of $\gamma=\id_S+\alpha\partial$), 
we can prove unique 
extension properties of $\alpha$-derivations along \'etale maps 
(Proposition \ref{prop:TwistDerEtaleExt}, Corollary \ref{cor:DeltaCompDerivEtExt}), 
$\delta$-envelopes (Proposition \ref{eq:DeltaEnvTwDerivExt}),  
and bounded prismatic envelopes (Proposition \ref{prop:TwDerivPrismExt}). These 
and some other properties allow us to construct a
$t_i(q-1)$-derivation $\theta_i$ on $B$, and its variant on $D(r)$ and also on the bounded 
prismatic envelope of an embedding 
into a framed smooth algebra for any base bounded prism lying over $(\Z_p[[q-1]],\pq\Z_p[[q-1]])$
(Proposition \ref{prop:qHiggsDeriv} (2)).
For the second problem, similarly to the case of crystalline site,
$M$ is equipped with the stratification $\varepsilon\colon 
M\otimes_{D,p_1}D(1)\xrightarrow{\cong} M\otimes_{D,p_0}D(1)$,
and the $q$-Higgs field on $M$ is obtained by transporting via $\varepsilon^{-1}$
the trivial $q$-Higgs field on $M\otimes_{D,p_0}D(1)$ defined by that of $D(1)$ relative to 
$p_0\colon B\to D(1)$,  and showing that it stabilizes $M\subset M\otimes_{D,p_1}D(1)$. 
For the construction  of a crystal from a $q$-Higgs module, the arguments
in \cite[\S3.2]{MT} works with a slight modification replacing automorphisms $\gamma$'s
on $D(r)$ by twisted derivations (Proposition \ref{prop:HigToStrat}).

This paper is organized as follows. 
After reviewing basic facts on $\delta$-rings in \S\ref{sec:deltarings}, 
we study in \S\ref{sec:DivDeltaEnv} the structure of a certain kind of $\delta$-envelope used 
in the construction of bounded prismatic envelopes, and then in \S\ref{sec:DPDeltaRings} 
we discuss divided powers of
an element of a $\delta$-ring not necessary $p$-torsion free, which are used 
to give a divided power structure on the reduction mod $q-1$ of 
a bounded prismatic envelope over the $q$-crystalline prism. In \S\ref{sec:PrismPrismEnv}, 
we summarize the definition of bounded prisms and 
bounded prismatic envelopes, and their basic properties, giving 
references or proofs. In \S\ref{sec:TwistedDeriv}, we introduce the notion of twisted derivations
and study its fundamental properties by using the interpretation in terms of 
section maps explained above. In \S\ref{sec:TwistedDerivDelta}, we define and study compatibility
of a twisted derivation with a $\delta$-structure, which is a key in the
construction of $q$-Higgs derivations on bounded prismatic envelopes over
the $q$-crystalline prism. In \S\ref{sec:TwDeivDivDeltaEnv}, 
we construct twisted derivations associated to coordinates
on a certain type of $\delta$-envelope considered in \S\ref{sec:DivDeltaEnv} when the
base $\delta$-ring is defined over $\Z[q]$, $\delta(q)=0$, and show that they become usual
derivations on the $p$-adically completed PD-polynomial ring when $q=1$
under a certain condition; by using the latter fact, we prove an analogue of 
Poincar\'e lemma for these twisted derivations.
The sections \ref{sec:connection} and \ref{sec:ScalarExtConnection} 
are devoted to preliminaries on an integrable connection
over a commutative family of  twisted derivations, and its 
de Rham complex; we define them in \S\ref{sec:connection}  and 
discuss scalar extension of integrable connections and its compatibility with tensor products 
and de Rham complexes in \S\ref{sec:ScalarExtConnection}. 
In \S\ref{qprismenv}, we apply results in the preceding sections to 
the bounded prismatic envelope of
an embedding into a framed smooth algebra over a bounded prism
lying over the $q$-crystalline prism. We introduce the notion of 
$q$-Higgs derivations, $q$-Higgs fields, and $q$-Higgs complexes 
on the envelope,  summarize their behavior under tensor 
products and scalar extensions, and state Poincar\'e lemma
and relation to PD-polynomial rings in this setting for later use.
We prove Theorem \ref{thm:main1}
after reviewing the definition of prismatic sites and prismatic crystals in \S\ref{sec:CrystalStratqHiggs}; 
we also discuss its compatibility with tensor products and 
the scalar extension under the Frobenius lifting.
After some preliminaries on an analogue of linearization for prismatic sites
in \S\ref{sec:LinCAcpx}, we prove Theorem \ref{thm:main2} 
and its variant for a complete crystal
together with their compatibility with tensor products and the scalar extension
under the Frobenius lifting in \S\ref{sec:PrismCohqDolb}. 
Based on the functoriality of 
the isomorphism in Theorem \ref{thm:main2} with respect to $A\leftarrow B/R$, $t_i$
studied in \S\ref{sec:PrismCohqDolbFunct}, 
we prove a global variant of Theorem \ref{thm:main2} 
and its compatibility with tensor products and the Frobenius scalar extension
in the last section \S\ref{sec:PrismCohqDolbGlobal}.
\par
\smallskip
\noindent 
{\bf Acknowledgment} \par
I would like to thank Matthew Morrow for our collaboration \cite{MT} on coefficients
for $A_{\inf}$-cohomology, which we named relative Breuil-Kisin-Fargues modules;
the comparison between BKF modules and prismatic $F$-crystals discussed in \cite{MT}
motivated me to develop a general theory describing a prismatic crystal and 
its cohomology over an arbitrary base lying over the $q$-crystalline prism in terms of $q$-Higgs
fields. This work was financially supported by JSPS Grants-in-Aid for Scientific Research, 
Grant number 20H01793.

\section{$\delta$-rings}\label{sec:deltarings}
We recall basic terminologies and properties on $\delta$-rings
(\cite{JoyalDeltaRing}, \cite[\S2]{BS}).
We fix a rational prime number $p$. Let $R$ be a commutative
ring. A {\it $\delta$-structure}
on $R$ is a map $\delta\colon R\to R$ satisfying
$\delta(0)=\delta(1)=0$ and
the following equalities for all $x$, $y\in R$.
\begin{align}
\delta(x+y)&=\delta(x)+\delta(y)-\sum_{i=1}^{p-1}
p^{-1}\binom{p}{i}x^i y^{p-i},\label{eq:DeltaStrSum}\\
\delta(xy)&=\delta(x)y^p+x^p\delta(y)+p\delta(x)\delta(y).
\label{eq:DeltaStrProd}
\end{align}
For a $\delta$-structure on $R$, one can easily deduce the 
following equalities by induction.
\begin{equation}\label{eq:DeltaStrIntegers}
\delta(n\cdot 1_R)=p^{-1}(n-n^p)\cdot1_R\quad (n\in \Z)
\end{equation}
By a {\it lifting of Frobenius on $R$}, we mean 
an endomorphism $\varphi$ of $R$ satisfying $\varphi(x)-x^p\in pR$
for every $x\in R$. If $f$ is a $\delta$-structure on $R$,
then the map $\varphi\colon R\to R$ defined by 
$\varphi(x)=x^p+p\delta(x)$ is a lifting of Frobenius
on $R$. The equality \eqref{eq:DeltaStrProd} can be rewritten as
\begin{equation}\label{eq:DeltaStrProd2}
\delta(xy)=\varphi(x)\delta(y)+\delta(x)y^p.
\end{equation}
By \eqref{eq:DeltaStrIntegers}, we have 
\begin{equation}\label{eq:DeltaFrobIntegers}
\varphi(n\cdot 1_R)=n\cdot 1_R \quad (n\in \Z).
\end{equation}
By using \eqref{eq:DeltaStrIntegers}, \eqref{eq:DeltaStrSum},
\eqref{eq:DeltaStrProd}, and Lemma \ref{lem:DeltaPower} shown later,
we obtain 
\begin{equation}\label{eq:DeltaFrobComm}
\varphi\circ \delta=\delta\circ\varphi.
\end{equation}

If $R$ is $p$-torsion free, $\delta$ is uniquely
determined by $\varphi$, and this construction gives
a bijection between the set of $\delta$-structures on 
$R$ and that of liftings of Frobenius on $R$, i.e,
if $\varphi$ is a lifting of Frobenius on $R$, then
the map $\delta\colon R\to R$ defined by
$\varphi(x)=p^{-1}(\varphi(x)-x^p)$ $(x\in R)$ 
is a $\delta$-structure on $R$. 
The algebra $\Z$ admits a unique lifting of Frobenius $\id$,
which defines the $\delta$-structure: $\delta(n)=p^{-1}(n-n^p)$
$(n\in \Z)$. 

We call a ring equipped with 
a $\delta$-structure a {\it $\delta$-ring} . A {\it $\delta$-homomorphism}
$f\colon R\to S$ between $\delta$-rings is a
homomorphism  of underlying rings compatible
with the $\delta$-structures: $f\circ\delta(x)=\delta\circ f(x)$
for all $x\in R$. For a $\delta$-ring $R$, 
a {\it $\delta$-algebra over $R$}
(or a {\it $\delta$-$R$-algebra})
is a $\delta$-ring $S$ equipped with a $\delta$-homomorphism
$R\to S$. A {\it $\delta$-ideal} of a $\delta$-ring $R$ is
an ideal $I$ of the underlying ring stable under $\delta$:
$\delta(I)\subset I$. If $I$ is a $\delta$-ideal of a $\delta$-ring $R$,
then the quotient ring $R/I$ is equipped with a unique
$\delta$-algebra structure over $R$. Let $\scrS$ be a subset
of a $\delta$-ring $R$. Then we see that
the ideal of $R$ generated by $\delta^n(\scrS)$ $(n\in \N)$ is
stable under $\delta$, whence is the smallest 
$\delta$-ideal of $R$ containing $\scrS$. 
(Note that for any ideal $J$ of $R$, the map 
$J\to R/J; x\mapsto (\delta(x)\mod J)$
is a $\varphi$-semilinear homomorphism of $R$-modules by
\eqref{eq:DeltaStrSum} and \eqref{eq:DeltaStrProd2}.)
We write $(\scrS)_{\delta}$ for this ideal and call
it {\it the $\delta$-ideal generated by $\scrS$}. We write
$I_{\delta}$ for $(I)_{\delta}$ if $I$ is an ideal of $R$.
If the ideal $I$ is generated by a subset $\scrS$ of $I$,
then we have $I_{\delta}=(\scrS)_{\delta}$
because the latter contains $I$. This implies that
if $S$ is a $\delta$-algebra over a $\delta$-ring $R$,
and $I$ is a $\delta$-ideal of $R$, then
the ideal $IS$ of $S$ generated by $I$ is a 
$\delta$-ideal of the $\delta$-ring $S$. 

By a {\it $\delta$-subring} of a $\delta$-ring $R$,
we mean a subring $R'$ of $R$ stable under the
$\delta$-structure: $\delta(R')\subset R'$. 
Similarly, a {\it $\delta$-$R$-subalgebra} of a $\delta$-algebra $S$
over a $\delta$-ring $R$ is an $R$-subalgebra $S'$ of $S$
stable under the $\delta$-structure of $S$: $\delta(S')\subset S'$.
For a subset $\scrS$ of a $\delta$-algebra $S$ over a $\delta$-ring
$R$, we have 
\begin{equation}\label{eq:DeltaGenerator}
\delta(R[\delta^m(\scrS);0\leq m\leq n])
\subset R[\delta^m(\scrS);0\leq m\leq n+1]\end{equation}
for each non-negative integer $n$. 
Note that for any $R$-subalgebra $S'$ of $S$, the
subset $\{x\in S'\vert \delta(x)\in S'\}$ of $S'$ is an $R$-subalgebra
by \eqref{eq:DeltaStrSum} and \eqref{eq:DeltaStrProd}.
This implies that the $R$-subalgebra of $S$ 
generated by $\delta^n(\scrS)$  $(n\geq 0)$ is stable under 
$\delta$, whence is the smallest $\delta$-$R$-subalgebra of $S$ 
containing $\scrS$, which is denoted by $R[\scrS]_{\delta}$ and 
called {\it the $\delta$-subalgebra over 
$R$ generated by $\scrS$}. 

Let $R$ be a commutative ring, and let $W_2(R)$ be the ring
of Witt vectors of $R$ of length $2$ with respect to the prime $p$.
Then the map $\delta\colon R\to R$ is a $\delta$-structure
if and only if the map $(1,\delta)\colon R\to W_2(R); x\mapsto (x,\delta(x))$
is a ring homomorphism. This interpretation allows us to show that
for a polynomial ring $S=R[T_{\lambda} (\lambda\in \Lambda)]$
over a $\delta$-ring $R$ and $s_{\lambda}\in S$ $(\lambda\in \Lambda)$, there
exists a unique $\delta$-$R$-algebra structure on $S$ such that $\delta(T_{\lambda})=s_{\lambda}$.
For two $\delta$-rings $(S,\delta_S)$ and $(S',\delta_{S'})$, a homomorphism
of rings $f\colon S\to S'$ is a $\delta$-homomorphism if and only if the following 
diagram commutes.
\begin{equation}
\xymatrix@C=50pt{
S\ar[d]_f\ar[r]^(.45){(1,\delta_S)}&W_2(S)\ar[d]^{W_2(f)}\\
S'\ar[r]^(.45){(1,\delta_{S'})}& W_2(S')
}
\end{equation}
If $S$ and $S'$ are $\delta$-algebras over a $\delta$-ring $R$, and 
$S$ is generated by a subset $\scrS$ of $S$ as an $R$-algebra, 
then the above
observation implies that an $R$-homomorphism $f\colon S\to S'$ is a
$\delta$-homomorphism if and only if $\delta(f(s))=f(\delta(s))$ for every 
$s\in \scrS$. For a polynomial algebra $R[T_{\lambda} (\lambda\in \Lambda)]$
over a $\delta$-ring $R$, we define a $\delta$-algebra 
$R[T_{\lambda}(\lambda\in \Lambda)]_{\delta}$ over $R$ to be the
polynomial algebra $R[T_{\lambda}][T_{\lambda,n+1}(\lambda\in \Lambda,n\in \N)]$
over $R[T_{\lambda}(\lambda\in \Lambda)]$ equipped with the $\delta$-$R$-algebra
structure defined by $\delta(T_{\lambda,n})=T_{\lambda,n+1}$ 
$(\lambda\in \Lambda, n\in \N)$, where we put $T_{\lambda,0}=T_{\lambda}$
$(\lambda\in \Lambda)$. Then any $R$-homomorphism 
$f\colon R[T_{\lambda}(\lambda\in \Lambda)]\to A$ to a $\delta$-$R$-algebra
uniquely extends to a $\delta$-$R$-homomorphism $R[T_{\lambda}(\lambda\in\Lambda)]_{\delta}
\to A$, which sends $T_{\lambda,n}$ to $\delta^n(f(T_{\lambda}))$. 
Let $S$ be an algebra over a $\delta$-ring $R$, and suppose that the $R$-algebra $S$ is 
generated by a subset $\scrS\subset S$, which gives a presentation 
$R[T_s (s\in \scrS)]/I\xrightarrow{\cong} S; (T_s\mod I)\mapsto s$ of $S$ as
a quotient of a polynomial algebra. Let $S_{\delta}$ be the
$\delta$-$R$-algebra $R[T_s (s\in \scrS)]_{\delta}/(I)_{\delta}$, which 
is equipped with an $R$-homomorphism $S\to S_{\delta}$ induced by
the homomorphism $R[T_s (s\in \scrS)]\to R[T_s (s\in \scrS)]_{\delta}$.
Then any $R$-homomorphism $f\colon S\to A$ to a $\delta$-$R$-algebra uniquely 
factors through a $\delta$-$R$-homomorphism $S_{\delta}\to A$
as follows: The composition $R[T_s (s\in  \scrS)]\to S\xrightarrow{f} A$
extends uniquely to a $\delta$-$R$-homomorphism $R[T_s (s\in \scrS)]_{\delta}\to A$,
which sends $\delta^n(x)$ $(x\in I, n\in \N)$ to $\delta^n(0)=0$. This universality shows
that $S_{\delta}$ with $S\to S_{\delta}$ is independent of the choice of $\scrS$
up to unique isomorphisms. We call a $\delta$-$R$-algebra $S_{\delta}$ with 
an $R$-homomorphism $S\to S_{\delta}$ satisfying the above universality
the {\it $\delta$-envelope of $S$ over $R$}.

Let $R$ be a $\delta$-ring, and let $I$ be an ideal of $R$ containing
$p$. Then, we see $\delta(I^{n+1})\subset I^n$ $(n\in \N)$
by induction on $n$, noting that the map $I^n\to R/I^n; a\mapsto (\delta(a)\bmod I^n)$
is a $\varphi$-semilinear map of $R$-modules and that $\varphi(I)\subset I$
by the assumption $p\in I$.
This implies 
$\delta(a+I^{n+1})\subset \delta(a)+I^n$ for $a\in R$
and $n\in \N$. Hence the composition 
$R\xrightarrow{\delta} R\to R/I^n$ factors as
$R\to R/I^{n+1}\xrightarrow{\delta_n}R/I^n$.
By taking the inverse limit of $\delta_n$, we obtain a $\delta$-structure
on $\hR=\varprojlim_nR/I^n$, which we call the $I$-adic completion 
of the $\delta$-structure on $R$. The canonical homomorphism
$R\to \hR$ is a $\delta$-homomorphism. 

\begin{definition}\label{def:formallyflat}
Let $R$ be a ring,  let $I$ be an ideal of $R$,
and let $S$ be an $R$-algebra.\par
(1) We say that $S$ is 
{\it $I$-adically \'etale (resp.~smooth, resp.~flat, resp.~faithfully flat) over 
$R$} if $S/I^nS$ is \'etale (resp.~smooth, resp.~flat, resp.~faithfully flat) over
$R/I^n$ for every positive integer $n$.\par
(2) Suppose that $S$ is $I$-adically smooth over $R$.
We say that elements $t_1,\ldots, t_d$ of $S$ are
{\it $I$-adic coordinates of $S$ over $R$} if 
$d(t_i\bmod IS)$ is a basis of the $S/IS$-module
$\Omega_{(S/IS)/(R/I)}$.
(This implies that $d(t_i\bmod I^nS)$ form a basis of 
the $(S/I^nS)$-module $\Omega_{(S/I^nS)/(R/I^n)}$,
and the $R/I^n$-homomorphism 
$R/I^n[T_1,\ldots, T_d]\to S/I^nS;T_i\mapsto (t_i\bmod I^nS)$
is \'etale for every positive integer $n$.)
\end{definition}

\begin{remark}\label{rmk:AdicallyFlatCriterion}
We have  the following criterion for $I$-adically flat homomorphisms
easily deduced from the local criteria of flatness: If $I$ is generated
by a regular sequence $a_1,\ldots, a_n$ which is also $S$-regular, 
and $S/IS$ is flat over $R/IR$, then $S$ is $I$-adically flat over $R$.
Indeed, the regularity implies that we have 
isomorphisms $R/I[X_1,\ldots, X_n]\xrightarrow{\cong}\gr^IR$
and $S/IS[X_1,\ldots,X_n]\xrightarrow{\cong} \gr^{IS}S$
sending $X_i$ to the image of $a_i$ in $I/I^2$ and in $IS/I^2S$,
respectively. Hence we have an isomorphism 
$\gr^IR\otimes_{R/I}S/IS\xrightarrow{\cong}\gr^{IS}S$, which, 
together with the flatness of $R/I\to S/IS$, implies that 
$R/I^nR\to S/I^nS$ $(n>0)$ are flat
by the local criteria of flatness.
\end{remark}

\begin{proposition}[cf.~{\cite[Lemma 2.18]{BS}}] \label{prop:DeltaStrExtEtSm}
Let $R$ be a $\delta$-ring, let $I$ be an ideal of $R$ containing
$p$, and let $S$ be an $R$-algebra $I$-adically complete
and separated.\par
(1) If $S$ is $I$-adically \'etale over $R$, then
there exists a unique $\delta$-$R$-algebra structure on $S$.\par
(2) Assume that $S$ is $I$-adically smooth over $R$,
and let $t_1,\ldots, t_d\in S$ be $I$-adic coordinates of $S$
over $R$ (Definition \ref{def:formallyflat} (2)). Then, for any $s_1,\ldots, s_d\in S$, there exists
a unique $\delta$-$R$-algebra structure on $S$ such that
$\delta(t_i)=s_i$ for every $i\in \N\cap [1,d]$. 
\end{proposition}

\begin{proof}
For a positive integer $m$, put $R_m=R/I^m$ 
and $S_m=S/I^mS$, and let $f_m$ denote the
the structure homomorphism $R_m\to S_m$.
Let $n$ be a positive integer. By the assumption
$p\in I$, the composition of 
$w\colon R\to W_2(R);x\mapsto (x,\delta(x))$ and
$W_2(R)\to W_2(R_n)$ factors as 
$R\to R_{n+1}\xrightarrow{w_n}W_2(R_n)$,
and the $(n+1)$th power of the kernel of 
$\varepsilon_n\colon W_2(S_n)\to S_n;
(x_0,x_1)\mapsto x_0$ is $0$. We see that the 
composition 
$R_{n+1}\xrightarrow{w_n}W_2(R_n)
\xrightarrow{W_2(f_n)} W_2(S_n)\xrightarrow{\varepsilon_n}S_n$
coincides with the composition $R_{n+1}\xrightarrow{f_{n+1}}
S_{n+1}\to S_n$.\par
(1) Since $f_{n+1}\colon R_{n+1}\to S_{n+1}$ is \'etale,
there exists a unique homomorphism $v_n\colon S_{n+1}\to W_2(S_n)$
making the following diagram commute.
\begin{equation*}
\xymatrix@C=50pt{
R_{n+1} \ar[r]^(.47){w_n}\ar[d]_{f_{n+1}}&
W_2(R_n)\ar[r]^{W_2(f_n)}&
W_2(S_n)\ar[d]^{\varepsilon_n}\\
S_{n+1}\ar@{-->}[urr]_{v_n}\ar[rr]&&
S_n
}
\end{equation*}
By the uniqueness, $v_n$ $(n\geq 1)$ form an inverse system,
and we obtain the claim by taking the inverse limit over $n$.\par
(2) For a positive integer $m$,
put $S_m'=R_m[T_1,\ldots, T_d]$, and let $g_m$ be
the $R_m$-homomorphism $S_m'\to S_m;T_i\mapsto t_i$. 
The composition $R_{n+1}\xrightarrow{w_n} W_2(R_n)
\xrightarrow{W_2(f_n)}W_2(S_n)$ uniquely extends
to a homomorphism $v_n'\colon S_{n+1}' \to W_2(S_n)$
sending $T_i$ to  $(t_i,s_i)$.  Since $g_{n+1}$ is \'etale,
this further extends to a unique homomorphism 
$v_n\colon S_{n+1}\to W_2(S_n)$ whose composition 
with $\varepsilon_n\colon W_2(S_n)\to S_n$ is the
projection map. 
\begin{equation*}
\xymatrix@C=50pt{
S_{n+1}'\ar[r]^{v_n'} \ar[d]_{g_{n+1}}&
W_2(S_n)\ar[d]^{\varepsilon_n}\\
S_{n+1}\ar@{-->}[ru]^{v_n}\ar[r]& S_n
}
\end{equation*}
The homomorphisms $v_n'$ $(n\geq 1)$ obviously form
an inverse system, and the uniqueness of 
$v_n$ for each $n\geq 1$ implies that $v_n$ $(n\geq 1)$
also form an inverse system. Hence we obtain
the claim by taking the inverse limit over $n$.
\end{proof}

\begin{proposition}\label{prop:DeltaCompEtExt}
Let $R$ be a $\delta$-ring, let $I$ be an ideal of $R$ containing $p$,
and let $f\colon R\to R'$ be a morphism of $\delta$-rings whose
underlying ring homomorphism is $I$-adically \'etale
(Definition \ref{def:formallyflat} (1)).
Let $S$ be a $\delta$-ring, let $g'\colon R'\to S$ be a ring homomorphism,
and put $g=g'\circ f$. If $g$ is a $\delta$-homomorphism and 
$S$ is $g(I)$-adically separated, then $g'$ is a $\delta$-homomorphism.
\end{proposition}

\begin{proof}
We want to show that the diagram of rings
\begin{equation*}
\xymatrix@C=50pt{
R'\ar[r]^(.45){w_{R'}}\ar[d]_{g'}&
W_2(R')\ar[d]^{W_2(g')}\\
S\ar[r]^(.45){w_S}&W_2(S),
}\tag {$*$}
\end{equation*}
is commutative, where $w_S$ (resp.~$w_{R'}$) is the
ring homomorphism defined by 
$w_S(x)=(x,\delta(x))$ (resp.~$w_{R'}(y)=(y,\delta(y))$). 
This diagram commutes after composing with 
$W_2(S)\to S;(x_0,x_1)\mapsto x_0$ (resp.~$f\colon R\to R'$
since $f$ and $g=g'\circ f$ are $\delta$-homomorphisms.)
Put $R_m=R/I^m$, $R'_m=R'/f(I^m)R'$, 
$S_m=S/g(I^m)S$ and $f_m=(f\mod I^m)$ for
a positive integer $m$. Let $n$ be a positive integer.
By the assumption $p\in I$, the above diagram
induces a diagram
\begin{equation*}
\xymatrix{
R'_{n+1}\ar[r]\ar[d]&
W_2(R'_n)\ar[d]\\
S_{n+1}\ar[r]&W_2(S_n),
}
\end{equation*}
and it is commutative
after composing with $f_{n+1}\colon R_{n+1}\to R'_{n+1}$
and $W_2(S_n)\to S_n$. Since the $(n+1)$th power
of the kernel of the last map is $0$ by $p\in I$, 
and $f_{n+1}$ is \'etale, this diagram itself is commutative. 
By taking the inverse limit over $n$, we see that
the diagram $(*)$ is also commutative since
the map $W_2(S)\to \varprojlim_nW_2(S_n)$ is
injective by assumption.
\end{proof}

\section{Divided $\delta$-envelopes}\label{sec:DivDeltaEnv}

\begin{proposition}[cf.~{\cite[Proposition 3.13]{BS}}]\label{prop:DeltaEnvRegSeq}
Let $R$ be a $\delta$-ring, let $\xi$ be an element of $R$, put
$I=pR+\xi R$, and assume that $\delta(\xi)$ mod $I$ is invertible in $R/I$.
Let $A$ be a $\delta$-$R$-algebra, and let $T_1,\ldots, T_d$ be
elements of $A$.
Put $D_0=A[S_1,\ldots, S_d]/(\xi S_i-T_i, i\in \N\cap [1,d])$,
and let $D$ be the $\delta$-envelope of $D_0$ over $A$. 
Let $\tau_i$ denote the image of $S_i$ in $D$ for $i\in \N\cap [1,d]$,
and define $\tau_i^{\{m\}_{\delta}}$ for $m\in \N$ and $i\in \N\cap [1,d]$ by
\begin{equation}\label{eq:TauDeltaPower}
\tau_i^{\{m\}_{\delta}}=\prod_{n=0}^N\delta^n(\tau_i)^{a_n},
\quad m=\sum_{n=0}^Na_np^n,\; a_n\in \N\cap [0,p-1]. 
\end{equation}

(1) $D/ID$ is a free $A/(IA+\sum_{i=1}^dT_iA)$-module
with a basis $\prod_{i=1}^d\tau_i^{\{m_i\}_{\delta}}$
mod $ID$ $((m_i)\in \N^d)$. \par
(2) For $i\in \N\cap [1,d]$ and $n\in \N$, the 
$A$-subalgebra of $D/ID$ generated by the images of 
$\delta^l(\tau_i)$ $(l\in \N\cap [0,n])$ is a free $A/(IA+\sum_{i=1}^dT_iA)$-module
with a basis $\tau_i^{\{m\}_{\delta}} \mod ID$ $(m\in \N\cap [0,p^{n+1}-1])$.\par
(3) Assume the sequence $T_1,\ldots, T_d$ is 
$A/IA$-regular,  $\Tor_1^{R/I^{N+1}}(A/I^{N+1}A,I^m/I^{m+1})=0$
for all $N\in \N$ and $m\in \N\cap [0,N]$, and
$\Tor_1^{R/I}(A/(IA+\sum_{i=1}^rT_iA),I^m/I^{m+1})=0$
for all $r\in \N\cap [1,d]$ and $m\in \N$.
Then we have $\Tor_1^{R/I^{N+1}}(D/I^{N+1}D,I^m/I^{m+1})=0$
for all $N\in \N$ and $m\in \N\cap [0,N]$.
If we further assume that $R/I\to A/(IA+\sum_{i=1}^dT_iA)$ is flat, 
then $R/I^{N+1}\to D/I^{N+1}D$ is flat for every $N\in \N$. \par
(4)  Under the assumption in (3), we further 
assume that the homomorphism $R/I\to A/(IA+\sum_{i=1}^dT_iA)$
is an isomorphism. Then, for any ideal $J$ of $R$ containing 
some power of $I$, $D/JD$ is a free $R/J$-module with 
a basis $\prod_{i=1}^d\tau_i^{\{m_i\}_{\delta}}$ mod $JD$
$((m_i)\in \N^d)$. 
\end{proposition}

\begin{remark}\label{rmk:DeltaEnvRegSeq}
(1) Under the setting of Proposition \ref{prop:DeltaEnvRegSeq}, 
the assumptions
in the claim (3) are satisfied if the homomorphism 
$R\to A$ is $I$-adically flat (Definition \ref{def:formallyflat} (1)) and the image 
of the sequence $T_1,\ldots, T_d$ in $A/IA$ is 
transversally regular relative to $R/I$ (\cite[D\'efinition (19.2.1)]{EGAIV}),
i.e., forms a regular sequence with quotients
$(A/IA)/\sum_{i=1}^rT_i(A/IA)$ $(r\in \N\cap [0,d])$
flat over $R/I$.\par
(2) Under the setting of Proposition \ref{prop:DeltaEnvRegSeq}, 
suppose that
we are given a homomorphism of $\delta$-rings
$R\to R'$. Let $A'$ be the $\delta$-ring
$A\otimes_RR'$, let $\xi'$ be the image of
$\xi$ in $R'$, and let $T_i'$ be the
image of $T_i$ in $A'$ for $i\in \N\cap [1,d]$. 
Similarly to $D_0$ and $D$, we define $D_0'$ and $D'$
to be $A'[S_1',\ldots, S_d']/(\xi'S_i'-T_i',i\in\N\cap [1,d])$
and its $\delta$-envelope over $A'$. Then
the $\delta$-homomorphism $A\to A';a\mapsto a\otimes 1$
extends to a homomorphism $D_0\to D_0'$ sending
$S_i$ to $S_i'$, and then to a $\delta$-homomorphism 
$D\to D'$. By the construction of $\delta$-envelope,
$D$ is the quotient of $A[S_1,\ldots, S_d]_{\delta}$
by the ideal generated by $\delta^n(\xi S_i-T_i)$
$(n\in \N, i\in \N\cap [1,d])$, and $A[S_1,\ldots, S_d]_{\delta}$
is the polynomial algebra over $A$ in variables
$\delta^n(S_i)$ $(n\in \N,i\in \N\cap [1,d])$. 
The same applies to $D'$. Hence the $\delta$-homomorphism
$D\to D'$ induces an isomorphism of $\delta$-$R'$-algebras
$D\otimes_{R}R'\xrightarrow{\cong} D'$. In particular, if 
the $\delta$-$R$-algebra $D$ satisfies the conclusion 
of the claim (4), then it also holds for the 
$\delta$-$R'$-algebra $D'$.
\end{remark}

\begin{lemma}\label{lem:DeltaIteration}
Let $R$ be a $\delta$-ring equipped with an element $\xi$,
and let $A$ be a $\delta$-$R$-algebra equipped with an element $T$.
Put $B=A[S]_{\delta}$. For an integer $n\geq -1$, let $B_n$ be the
$A$-subalgebra of $B$ generated by $\delta^m(S)$
$(m\in \N, m\leq n)$. 
(We have $B_{-1}=A$ and $\delta(B_{n-1})\subset B_n$ $(n\in \N)$
by \eqref{eq:DeltaGenerator}.). 
Then, for a positive integer $n$, $\delta^n(\xi S-T)$ is written in the form 
$$\delta^n(\xi S-T)=\varphi^n(\xi)\delta^n(S)+\sum_{\nu=0}^pc_{\nu,n}\delta^{n-1}(S)^{\nu},$$
where $c_{p,n}=(\sum_{m=0}^{n-1}p^{m(p-1)})\varphi^{n-1}(\delta(\xi))\in R$
and $c_{\nu,n}\in B_{n-2}$ $(\nu\in \N\cap [0,p-1])$.
\end{lemma}

\begin{proof}
We prove the claim by induction on $n$. For $n=1$, we have 
\begin{align*}
\delta(\xi S-T)
&=\varphi(\xi)\delta(S)+\delta(\xi)S^p+\delta(-T)
-\sum_{\nu=1}^{p-1}\frac{1}{p}\binom{p}{\nu}
\xi^{\nu}(-T)^{p-\nu}S^\nu,
\end{align*}
which is of the desired form because $\delta(-T)$, $\xi$, and 
$T$ belong to $B_{-1}=A$. Let $n$ be
a positive integer, and suppose that the claim is true for
$\delta^n(\xi S-T)$ with $c_{p,n}\in R$ and $c_{\nu,n}\in B_{n-2}$
($\nu\in \N\cap[0,p-1]$). 
Put $b_n=\sum_{\nu=0}^{p}c_{\nu,n}\delta^{n-1}(S)^{\nu}\in 
B_{n-1}$. Then we have 
\begin{align*}
\delta^{n+1}(\xi S-T)=&\delta(\varphi^n(\xi)\delta^n(S)+b_n)\\
\quad=&\varphi^{n+1}(\xi)\delta^{n+1}(S)+\delta(\varphi^n(\xi))\delta^n(S)^p+\delta(b_n)
-\sum_{\nu=1}^{p-1}\frac{1}{p}\binom {p}{\nu}\varphi^n(\xi)^{\nu}b_n^{p-\nu}
\delta^n(S)^{\nu}.
\end{align*}
Since $\varphi^n(\xi)$ and $b_n$ belong to $B_{n-1}$ and $\delta(\varphi^n(\xi))
=\varphi^n(\delta(\xi))$ by \eqref{eq:DeltaFrobComm}, it remains to 
show that $\delta(b_n)$ is of the form 
$p^{p-1}\varphi(c_{p,n})\delta^n(S)^p+\sum_{\nu=0}^{p-1}d_{\nu,n}\delta^n(S)^{\nu}$
$(d_{\nu,n}\in B_{n-1})$. We have
$\delta(x+y)\equiv \delta(x)+\delta(y)$ mod $B_{n-1}$ 
for $x, y\in B_{n-1}$ by \eqref{eq:DeltaStrSum}, and
$\delta(xy)\equiv\varphi(x)\delta(y)$ mod $B_{n-1}$
for $x\in B_{n-2}$, $y\in B_{n-1}$ by \eqref{eq:DeltaStrProd2}. Hence we have 
$$\delta(b_n)\equiv \varphi(c_{p,n})\delta(\delta^{n-1}(S)^p)+\sum_{\nu=0}^{p-1}\varphi(c_{\nu,n})\delta(\delta^{n-1}(S)^{\nu})\mod B_{n-1}.$$
By Lemma \ref{lem:DeltaPower} below, we have $\delta(1)=0$ and 
$$\delta(\delta^{n-1}(S)^{\nu})=\sum_{j=1}^{\nu}
\binom{\nu}{j}p^{j-1}(\delta^{n-1}(S))^{p(\nu-j)}\delta^n(S)^j\quad (\nu\in \N\cap [1,p]).$$
This completes the proof because $\delta^{n-1}(S)\in B_{n-1}$,
$\varphi(c_{\nu,n})\in \varphi(B_{n-2})\subset B_{n-1}$,
and $\delta^n(S)^p$ appears only when $\nu=p$, for which the
coefficient of $\delta^n(S)^p$ is $p^{p-1}$. 
\end{proof}

\begin{lemma}\label{lem:DeltaPower}
For a $\delta$-ring $S$, $x\in S$, and a positive integer $n$,
we have 
$$\delta(x^n)=\sum_{j=1}^n\binom{n}{j}p^{j-1}x^{p(n-j)}\delta(x)^j.$$
\end{lemma}

\begin{proof} We prove the claim by induction on $n$. 
The claim obviously holds for $n=1$. Let $n$ be a positive integer
and assume that the claim holds for $\delta(x^n)$. Then we have
\begin{align*}
\delta(x^{n+1})&=\delta(x^n)x^p+x^{np}\delta(x)+p\delta(x^n)\delta(x)\\
&=\sum_{j=1}^n\binom{n}{j}p^{j-1}x^{p(n+1-j)}\delta(x)^j+x^{np}\delta(x)
+\sum_{j=1}^n\binom{n}{j}p^jx^{p(n-j)}\delta(x)^{j+1}.
\end{align*}
By changing the parameter $j$ in the last term by $j=k-1$
$(k\in \N\cap [2,n+1])$, we see that this is the sum
of $\binom{n+1}{j}p^{j-1}x^{p(n+1-j)}\delta(x)^j$ $(j\in \N\cap [1,n+1])$. 
\end{proof}

\begin{lemma}\label{lem:RegSeqLift2}
Let $S\to T$ be a homomorphism of rings, let $I$ be a nilpotent
ideal of $S$, and assume $\Tor_1^S(T,I^n/I^{n+1})=0$
for every $n\in \N$. Put $\oS=S/I$ and $\oT=T/IT$.\par
(1) Let $x$ be an element of $T$. If $x$ is $\oT$-regular,
and $\Tor_1^{\oS}(\oT/x\oT, I^n/I^{n+1})=0$ for
every $n\in \N$, then $x$ is $T$-regular and
$\Tor_1^S(T/xT,I^n/I^{n+1})=0$ for every $n\in \N$.\par
(2) Let $x_1,\ldots,x_d$ be elements of $T$. 
If the sequence $x_1,\ldots, x_d$ is $\oT$-regular,\linebreak
and $\Tor^{\oS}_1(\oT/\sum_{i=1}^rx_i\oT,I^n/I^{n+1})=0$
for every $r\in \N\cap [1,d]$ and $n\in\N$, then the sequence
$x_1,\ldots, x_d$ is $T$-regular, and 
$\Tor_1^S(T/\sum_{i=1}^dx_iT, I^n/I^{n+1})=0$ for
every $n\in \N$. 
\end{lemma}

\begin{proof}
(1) By taking $I^n/I^{n+1}\otimes_{\oS}$ of the
exact sequence $0\to \oT\xrightarrow{x}\oT \to \oT/x\oT\to 0$,
we see that $x$ is $I^n/I^{n+1}\otimes_{\oS}\oT$-regular.
By assumption on $T$, we have 
$\Tor_1^S(T,S/I^n)=0$ for all $n\in \N$. 
Hence the sequence 
$0\to I^n/I^{n+1}\otimes_ST
\to S/I^{n+1}\otimes_ST\to S/I^n\otimes_ST\to 0$
is exact, and we see that $x$ is $T/I^nT$-regular by
induction on $n$. By taking $\Tor_*^S(-,I^n/I^{n+1})$
of the exact sequence $0\to T\xrightarrow{x} T
\to T/xT\to 0$, we obtain $\Tor_1^S(T/xT,I^n/I^{n+1})=0$.\par
(2) We prove the claim by induction on $d$. 
The case $d=1$ is the claim (1). Let $d$ be a positive 
integer $\geq 2$, and assume that $x_1,\ldots, x_{d-1}$
is $T$-regular and $\Tor_1^S(T/\sum_{i=1}^{d-1}x_iT,I^n/I^{n+1})=0$.
Then, by applying (1) to $T'=T/\sum_{i=1}^{d-1}x_iT$
and the image $x'_d$ of $x_d$ in $T'$, we see that 
$x'_d$ is $T'$-regular,
and $\Tor_1^{S}(T'/x'_dT',I^n/I^{n+1})=0$. 
\end{proof}

\begin{proof}[Proof of Proposition \ref{prop:DeltaEnvRegSeq}]
Put $\Lambda=\N\cap [1,d]$  and $B=A[S_i (i\in \Lambda)]_{\delta}$.
The unique $\delta$-$A$-homomorphism $B\to D$ sending
$S_i$ to $\tau_i$ induces an isomorphism of $\delta$-$A$-algebras
$$B/(\xi S_i-T_i, i\in \Lambda)_{\delta}\xrightarrow{\cong} D.$$
Let $P_n^{(i)}$ be the element $\delta^n(\xi S_i-T_i)\in B$ for
$i\in \Lambda$ and $n\in \N$.
For an integer $n\geq -1$, let $B_n$ be the $A$-subalgebra 
of $B$ generated by $\delta^m(S_i)$ $(i\in \Lambda, m\in \N, m\leq n)$.
We have $B_{-1}=A$ and $\delta(B_{n-1})\subset B_n$ for $n\in \N$ by 
\eqref{eq:DeltaGenerator}.
The latter implies $P_n^{(i)}\in B_n$. 
Put $S_{i,n}=\delta^n(S_i)\in B$ for $i\in \Lambda$ and $n\in \N$. 
Then $B_n$ is the polynomial ring in $(n+1)d$ variables
$S_{i,m}$ $(i\in \Lambda, m\in \N, m\leq n)$.
For each integer $n\geq -1$, we define $D_n$ to be the
quotient of $B_n$ by the ideal generated by 
$P_m^{(i)}$ $(i\in \Lambda,m\in \N, m\leq n)$. 
We have $D_{-1}=B_{-1}=A$, and 
\begin{equation}\label{eq:DeltaEnv-A}
D_n=D_{n-1}[S_{i,n} (i\in\Lambda)]/(P_n^{(i)}, i\in \Lambda)
\qquad (n\in \N).
\end{equation}
Since the ideal $(\xi S_i-T_i,i\in \Lambda)_{\delta}$ of $B$
is generated by $P_n^{(i)}$ $(i\in \Lambda,n\in \N)$, the inclusion 
maps $B_n\hookrightarrow B_m\hookrightarrow B$
$(m>n\geq -1)$ induce $A$-homomorphisms
$D_n\to D_m\to D$, and then an isomorphism of $A$-algebras
\begin{equation}\label{eq:DeltaEnv-B}
\varinjlim_nD_n\xrightarrow{\cong}D.
\end{equation}

Let $\oR$, $\oA$, $\oB$, $\oB_n$, $\oD$, and $\oD_n$ denote
the reduction mod $I$ of $R$, $A$, $B$, $B_n$, $D$, and $D_n$,
respectively. We write $\oS_{i,n}$ and $\oP_n^{(i)}$ for the images
of $S_{i,n}$ and $P_n^{(i)}$ in $\oB$, and $\oT_i$ for the image of 
$T_i$ in $\oA$. Put $c=(\delta(\xi)$ mod $IR$)$\in \oR$,
which is invertible by assumption. Then, by 
Lemma \ref{lem:DeltaIteration} and $\varphi^n(\xi)\equiv \xi^{p^n}\equiv 0
\mod IR$, we have 
\begin{align}
\oP_0^{(i)}&=-\oT_i,\label{eq:DeltaEnv-Cp}\\
\oP_{n+1}^{(i)}&=c^{p^n}\oS_{i,n}^p+\sum_{\nu=0}^{p-1}c_{\nu,n}^{(i)}
\oS_{i,n}^{\nu}\quad (c_{\nu,n}\in \oB_{n-1})\label{eq:DeltaEnv-C}
\end{align}
for $i\in \Lambda$ and $n\in \N$. In particular, 
we have $\oP_n^{(i)}\in \oB_{n-1}$ for $n\in \N$.
For each integer $n\geq -1$, we define 
$\oD_{(n)}$ to be the quotient of $\oB_n$
by the ideal generated by $\oP_{m}^{(i)}$
$(i\in \Lambda, m\in \N, m\leq n+1)$. By comparing this
definition with that of $D_n$, we obtain 
\begin{equation}\label{eq:DeltaEnv-E}
\oD_n=\oD_{(n-1)}[\oS_{i,n}(i\in \Lambda)] \quad (n\in \N).
\end{equation}
By \eqref{eq:DeltaEnv-Cp}, we have 
\begin{align}
\oD_{(-1)}&=\oA/\sum_{i=1}^d \oT_i A,\label{eq:DeltaEnv-Dp}\\
\oD_{(n)}&=\oD_{(n-1)}[\oS_{i,n} (i\in \Lambda)]/(\oP_{n+1}^{(i)},i\in \Lambda)
\quad (n\in \N). \label{eq:DeltaEnv-D}
\end{align}
The inclusion maps $\oB_{n}\hookrightarrow\oB_m\to \oB$
$(m>n\geq -1)$ induce $\oA$-homomorphisms
$\oD_{(n)}\to \oD_{(m)}\to \oD$, and then an isomorphism of $\oA$-algebras
\begin{equation}\label{eq:DeltaEnv-F}
\varinjlim_n \oD_{(n)}\xrightarrow{\cong} \oD.
\end{equation}

By \eqref{eq:DeltaEnv-D},  \eqref{eq:DeltaEnv-C}, and $c\in \oR^{\times}$, 
we see that $\oD_{(n)}$ is a free $\oD_{(n-1)}$-module with a basis 
$\prod_{i=1}^d\oS_{i,n}^{m_i}$
$(m_i\in \N\cap [0,p-1])$ for $n\in \N$.
This implies the claim (1) by \eqref{eq:DeltaEnv-F} and \eqref{eq:DeltaEnv-Dp}. 

For $i\in \N\cap [1,d]$ and $n\in \Z$, $n\geq -1$,
let $\oB_n^{(i)}$ be the $\oA$-subalgebra of $\oB$ generated
by $\oS_{i,l}$ $(l\in \N, l\leq n)$, and let 
$\oD_n^{(i)}$ be the image of $\oB_n^{(i)}$ in $\oD$. 
We have $\oB_{-1}^{(i)}=\oA$. Then
the element $c_{\nu,n}^{(i)}$ $(\nu\in \N\cap [0,p-1])$ appearing
in \eqref{eq:DeltaEnv-C} belong to $\oB^{(i)}_{n-1}$ for $n\in \N$
by Lemma \ref{lem:DeltaIteration}.
Since the image of $\oP_{n+1}^{(i)}$ in $\oD$ vanishes and 
$c\in \oR^{\times}$, we have 
$\oD_n^{(i)}=\sum_{\nu=0}^{p-1}\oD_{n-1}^{(i)}(
\delta^n(\tau_i)^{\nu}\mod ID)$ in $\oD$
for $n\in \N$. By induction on $n$, we see that 
$\oD_n^{(i)}$ is generated by 
$\tau_i^{\{m\}_{\delta}}$ mod $ID$
$(m\in \N\cap [0,p^{n+1}-1])$ as an $\oA$-module.
This completes the proof of the claim (2). 

Let us prove the claim (3). By the claim (1), the assumption for
the second claim
 implies that $\oD$ is flat over $\oR$. Hence the second claim
 follows from the first one by the local criteria of flatness. For the first claim, 
by \eqref{eq:DeltaEnv-B}, it suffices to prove 
$\Tor_1^{R/I^{N+1}}(D_n/I^{N+1}D_n,I^m/I^{m+1})=0$
($N\in \N$, $m\in \N\cap [0,N]$) for all $n\geq -1$. 
By the proof of the claim (1) and \eqref{eq:DeltaEnv-Dp},
$\oD_{(n)}$ is a free $\oA/\sum_{i=1}^dT_i\oA$-module for every 
integer $n\geq -1$. Hence, by 
\eqref{eq:DeltaEnv-C} and \eqref{eq:DeltaEnv-E}, we see that 
$\oD_n/(\oP_{n+1}^{(1)},\ldots, \oP_{n+1}^{(r)})$ is a free
$\oA/\sum_{i=1}^dT_i\oA$-module for $n\in \N$
and $r\in \N\cap [1,d]$. For $n=-1$, we have
$\oD_{-1}/(\oP_0^{(1)},\ldots, \oP_0^{(r)})
=\oA/\sum_{i=1}^rT_i\oA$ $(r\in \N\cap [1,d])$ by 
\eqref{eq:DeltaEnv-Cp}. Hence, the third assumption in (3) implies
\begin{equation}\label{eq:DeltaEnv-G}
\Tor_1^{\oR}(\oD_{n-1}/(\oP_n^{(1)},\ldots, \oP_n^{(r)}),I^m/I^{m+1})=0
\quad(r\in \N\cap [1,d], m\in \N, n\in \N).
\end{equation}
By \eqref{eq:DeltaEnv-Cp} and the first assumption in (3),
the sequence $\oP_0^{(1)},\ldots, \oP_0^{(d)}$ is $\oD_{-1}=\oA$-regular.
By \eqref{eq:DeltaEnv-C}, \eqref{eq:DeltaEnv-E},  and $c\in \oR^{\times}$,
the sequence $\oP_{n+1}^{(1)},\ldots, \oP_{n+1}^{(d)}$
is $\oD_n$-regular for $n\in \N$. 
Since $\Tor_1^{R/I^{N+1}}(D_{-1}/I^{N+1}D_{-1},I^m/I^{m+1})=0$
for $m\in \N\cap [0,N]$ by $D_{-1}=A$ and the
second assumption in (3), \eqref{eq:DeltaEnv-A} and \eqref{eq:DeltaEnv-G} allow
us to prove the desired claim by induction on $n$ by 
applying Lemma \ref{lem:RegSeqLift2} (2) inductively to 
$R/I^{N+1}\to D_{n-1}[S_{i,n}(i\in \Lambda)]/I^{N+1}$, the ideal 
$IR/I^{N+1}$ of $R/I^{N+1}$, and the image
of the sequence $P_n^{(1)},\ldots, P_n^{(d)}$ in $D_{n-1}[S_{i,n}(i\in \Lambda)]/I^{N+1}$,
which is regular modulo $I$ as observed above.

Finally let us prove the claim (4). It suffices to prove the
claim when $J=I^{N+1}$ for $N\in \N$. 
By the assumption in (4),
the claim (1) is equivalent to the claim (4) for $J=I$. 
By the claim (3), the sequences $0\to I^N/I^{N+1}\otimes_{R/I} D/ID
\to D/I^{N+1}D\to D/I^ND\to 0$ $(N\geq 1)$ are exact, and
this allows us to prove the claim (4) for $J=I^{N+1}$ $(N\in\N)$
by induction on $N$. 
\end{proof}

\section{Divided powers on $\delta$-rings}\label{sec:DPDeltaRings}
We discuss a divided power of an element of a $\delta$-ring
which is not necessarily $p$-torsion free.
We begin with a preliminary proposition, which will be also used
later. 

\begin{proposition} \label{prop:DeltaEnvPTF}
Let $A$ be a $\delta$-ring, and let $T_1,\ldots, T_d$ be
elements of $A$. Put $D_0=A[S_1,\ldots, S_d]/(pS_i-T_i, i\in \N\cap [1,d])$,
and let $D$ be the $\delta$-envelope of $D_0$ over $A$. 
Assume that $A$ is $p$-torsion free, and the sequence $T_1,\ldots, T_d$
is $A/pA$-regular. Then $D$ is $p$-torsion free.
\end{proposition}

\begin{lemma}\label{lem:RegSeqLift1}
Let $S$ be a ring.\par
(1) If a sequence $x$, $y$ in $S$ is regular, then $x$ is
$S/yS$-regular.\par
(2)
If a sequence $x, y_1,\ldots, y_d$ in $S$ is $S$-regular, 
then $x$ is $S/\sum_{i=1}^dy_iS$-regular.
\end{lemma}

\begin{proof}
(1) We obtain the claim by applying the snake lemma to 
the multiplication by $y$ on the exact sequence 
$0\to S\xrightarrow{x} S\to S/xS\to 0$.\par
(2) We prove the claim by induction on $d$. The case
$d=1$ is the claim (1). Let $d$ be an integer $\geq 2$,
and assume that $x$ is $S'$-regular, where $S'=S/\sum_{i=1}^{d-1}y_iS$.
Then, since $y_d$ is $S'/xS'$-regular by assumption,
the claim (1) implies that $x$ is $S'/y_dS'$-regular.
\end{proof}

\begin{proof}[Proof of Proposition \ref{prop:DeltaEnvPTF}]
We follow the notation in the first and second paragraphs
of the proof of Proposition \ref{prop:DeltaEnvRegSeq} 
applied to $R=A$ and $\xi=p$. Note that we have
$(\delta(p) \mod pA)=1$, which is invertible, and $I=pA$. 
By \eqref{eq:DeltaEnv-B}, it suffices to show that $D_n$ is $p$-torsion
free for every $n\in \N$. We prove it by induction on $n$.
Since $B_0=A[S_{i,0} (i\in \Lambda)]$ is $p$-torsion free, and 
\eqref{eq:DeltaEnv-Cp} implies that the sequence 
$P_0^{(1)},\ldots, P_0^{(d)}$ is 
$\oB_0=\oA[\oS_{i,0} (i\in \Lambda)]$-regular, 
$D_0=B_0/(P_0^{(i)}(i\in \Lambda))$ is $p$-torsion free by 
Lemma \ref{lem:RegSeqLift1} (2). Let $n\in \N$, 
and assume that $D_n$ is $p$-torsion free. Then 
$D_n[S_{i,n+1}(i\in \Lambda)]$ is $p$-torsion free, and the sequence
$\oP_{n+1}^{(1)},\ldots, \oP_{n+1}^{(d)}$ is 
$\oD_n[\oS_{i,n+1} (i\in \Lambda)]$-regular by 
\eqref{eq:DeltaEnv-E} and \eqref{eq:DeltaEnv-C}.
Hence $D_{n+1}=D_n[S_{i,n+1} (i\in \Lambda)]/(P_{n+1}^{(i)} (i\in \Lambda))$
is $p$-torsion free by Lemma \ref{lem:RegSeqLift1} (2).
\end{proof}

For a ring $R$, let $R[T_1,\ldots, T_d]_{\PD}$ denote the
divided power polynomial ring over $R$ in $d$-variables
$T_1,\ldots, T_d$. For a $\delta$-ring $A$ over $\Z_{(p)}$, and 
$\ut=(t_1,\ldots, t_d)$, $\us=(s_1,\ldots, s_d)\in A^d$
satisfying $t_i^p=ps_i$ $(i\in \N\cap [1,d])$, we have a 
$\delta$-homomorphism 
\begin{equation}\label{eq:DeltaRingPDMap}
\Z_{(p)}[X_i,Y_i (i\in \N\cap [1,d])]_{\delta}/
(X_i^p-pY_i (i\in \N\cap [1,d]))_{\delta}\longrightarrow A
\end{equation}
sending the class $x_i$ of $X_i$ (resp.~$y_i$ of $Y_i$)
 to $t_i$ and $s_i$, respectively. 
Since the domain is $p$-torsion free by Proposition 
\ref{prop:DeltaEnvPTF}, it admits a map from $\Z_{(p)}[T_i (i\in\N\cap [1,d])]_{\PD}$
sending $T_i^{[n]}=\frac{T_i^n}{n!}$ to $\frac{x_i^n}{n!}$ for
$i\in \N\cap [1,d]$ and $n\in \N$ by \cite[Lemma 2.35]{BS}. 
Composing the two maps,  we obtain a homomorphism 
\begin{equation}\label{eq:DeltaRingPDMap2}
\pi_{\ut,\us}\colon \Z_{(p)}[T_1,\ldots, T_d]_{\PD}\longrightarrow A
\end{equation}
sending $T_i$ and $T_i^{[p]}$ to $t_i$ and $(p-1)! s_i$. 

\begin{lemma}\label{lem:DeltaRingPD}
 We keep the notation and assumption as above.\par
(1) The homomorphism $\pi_{\ut,\us}$ coincides with the
tensor product of $\pi_{t_i,s_i}$ $(i\in \N\cap [1,d])$.\par
(2) Assume $d=1$. Then the composition of $\pi_{t_1,s_1}$
with the isomorphism $\Z_{(p)}[T_1]_{\PD}
\xrightarrow{\cong} \Z_{(p)}[T_1]_{\PD};T_1^{[n]}\mapsto
(-T_1)^{[n]}$ coincides with $\pi_{-t_1,(-1)^ps_1}$. \par
(3) Assume $d=2$, and put $t=t_1-t_2$
and $s=s_1+(-1)^ps_2+\sum_{\nu=1}^{p-1}p^{-1}\binom p \nu
t_1^{\nu}(-t_2)^{p-\nu}$, which satisfy $t^p=ps$. Then the composition
of $\pi_{\ut,\us}$ and 
$$\Z_{(p)}[T_1]_{\PD}\to \Z_{(p)}[T_1,T_2]_{\PD};
T_1^{[n]}\mapsto (T_1-T_2)^{[n]}$$ coincides with $\pi_{t,s}$. \par
(4) Let $f\colon A\to A'$ be a $\delta$-homomorphism of 
$\delta$-rings, and put $\ut'=(f(t_1),\ldots, f(t_d))$
and $\us'=(f(s_1),\ldots, f(s_d))$. Then we have 
$f\circ\pi_{\us,\ut}=\pi_{\us',\ut'}$.
\end{lemma}
\begin{proof}
The claim (4) is obvious by definition, and the claims (1) and (2) follow from
the isomorphisms of $\delta$-rings
\begin{align*}
&\Z_{(p)}[X_i,Y_i (i\in \N\cap [1,d])]_{\delta}
/(X_i^p-pY_i)_{\delta}\cong \otimes_{1\leq i\leq d, \Z_{(p)}}
\Z_{(p)}[X_i,Y_i]_{\delta}/(X_i^p-pY_i)_{\delta},\\
&\Z_{(p)}[X_1,Y_1]_{\delta}/(X_1^p-pY_1)_{\delta}
\xrightarrow{\cong}
\Z_{(p)}[X_1,Y_1]_{\delta}/(X_1^p-pY_1)_{\delta};
x_1,y_1\mapsto -x_1, (-1)^py_1.
\end{align*}
We obtain the claim (3) by composing \eqref{eq:DeltaRingPDMap}
with the $\delta$-homomorphism
$$\Z_{(p)}[X,Y]_{\delta}/(X^p-pY)_{\delta}
\to \Z_{(p)}[X_1,X_2,Y_1,Y_2]_{\delta}/(X_1^p-pY_1,X_2^p-pY_2)_{\delta}$$
sending the classes of $X$ and $Y$ to $x_1-x_2$
and $y_1+(-1)^py_2+\sum_{\nu=1}^{p-1}p^{-1}\binom p \nu x_1^{\nu}(-x_2)^{p-\nu}$.
\end{proof}

\begin{lemma}\label{lem:pDivpPower}
Let $R$ be a $\delta$-ring, and let $t_1, t_2, s_1,s_2,$ and $\tau$
be elements of $R$ satisfying $\delta(t_1)=ps_1$,
$\delta(t_2)=ps_2$, and $p\tau=t_2-t_1$.
Then we have 
$$(1-p^{p-1})\tau^p=p\left(s_2-s_1-\delta(\tau)+\sum_{\nu=1}^{p-1}
p^{-1}\binom p\nu t_1^{p-\nu}p^{\nu-1}\tau^{\nu}\right).
$$
\end{lemma}

\begin{proof}
By taking the image of $t_2=t_1+p\tau$ under
$\delta$, we obtain $ps_2=ps_1+\delta(p\tau)
-\sum\limits_{\nu=1}^{p-1}p^{-1}\binom p\nu t_1^{p-\nu}(p\tau)^{\nu}$.
The claim follows from 
$\delta(p\tau)=\delta(p)\tau^p+\varphi(p)\delta(\tau)
=(1-p^{p-1})\tau^p+p\delta(\tau)$.
\end{proof}

\begin{proposition}\label{prop:DeltaPDDiff}
(1) Let $R$ be a $\delta$-$\Z_{(p)}$-algebra, and
let $t_1$, $t_2$, and $\tau_{12}$ be elements of $R$ such that
$\delta(t_1)=\delta(t_2)=0$, and $p\tau_{12}=t_2-t_1$.
Put $\tau_{21}=-\tau_{12}$. For $(i,j)=(1,2)$, $(2,1)$,
let $\sigma_{ij}$ denote 
$(1-p^{p-1})^{-1}(-\delta(\tau_{ij})+\sum_{\nu=1}^{p-1}
p^{-1}\binom p\nu t_i^{p-\nu}p^{\nu-1}\tau_{ij}^{\nu})$,
which satisfies $\tau_{ij}^p=p\sigma_{ij}$ by 
Lemma \ref{lem:pDivpPower}. Then we have 
$$\sigma_{21}=(-1)^p\sigma_{12}.$$
(2) Let $R$ be a $\delta$-$\Z_{(p)}$-algebra, and let 
$t_1$, $t_2$, $t_3$, $\tau_{12}$, and $\tau_{13}$ be elements
of $R$ such that $\delta(t_i)=0$ $(i=1,2,3)$, 
$p\tau_{12}=t_2-t_1$, and $p\tau_{13}=t_3-t_1$. 
Put $\tau_{23}=\tau_{13}-\tau_{12}$, which satisfies 
$p\tau_{13}=t_3-t_2$. 
For $(i,j)=(1,2), (1,3), (2,3)$, we define
an element $\sigma_{ij}$ of $R$ to be
$(1-p^{p-1})^{-1}(-\delta(\tau_{ij})
+\sum_{\nu=1}^{p-1}p^{-1}\binom p \nu
t_i^{p-\nu}p^{\nu-1}\tau_{ij}^{\nu})$,
which satisfies $\tau_{ij}^p=p\sigma_{ij}$ by
Lemma \ref{lem:pDivpPower}. 
Then we have 
$$\sigma_{23}=\sigma_{13}+(-1)^p\sigma_{12}
+\sum_{\nu=1}^{p-1}p^{-1}\binom p\nu\tau_{13}^{\nu}
(-\tau_{12})^{p-\nu}.$$
\end{proposition}

\begin{proof}
We prove the claims by reducing them to the universal case,
where $R$ is $p$-torsion free.\par
(1) Let $S$ be the polynomial algebra
$\Z_{(p)}[\tlt_1,\tlt_2]$ equipped with the $\delta$-structure
defined by $\delta(\tlt_1)=\delta(\tlt_2)=0$, put 
$\tlt_{ij}=\tlt_j-\tlt_i$, $(i,j)=(1,2), (2,1)$,
and let $\tR$ be the $\delta$-envelope of 
$S[\ttau_{12}]/(p\ttau_{12}-\tlt_{12})$
over $S$. Then the homomorphism 
$S\to R$ defined by $\tlt_i\mapsto t_i$
$(i=1,2)$ is a $\delta$-homomorphism, and 
extends uniquely to a $\delta$-homomorphism 
$f\colon \tR\to R$ sending $\ttau_{12}$
to $\tau_{12}$. Put $\ttau_{21}=-\ttau_{12}$,
and define $\tsigma_{ij}$, $(i,j)=(1,2)$, $(2,1)$
in the same way as $\sigma_{ij}$ by using 
$\ttau_{ij}$ and $\tlt_i$. Then we have
$f(\tsigma_{ij})=\sigma_{ij}$. Thus we are 
reduced to proving the claim for 
$\tR$, $\tlt_1$, and $\ttau_{ij}$. Since
$S$ is $p$-torsion free and 
$\tlt_{12}$ is $S/pS$-regular, 
$\tR$ is $p$-torsion free by Proposition \ref{prop:DeltaEnvPTF}.
Hence the claim follows from 
$p\tsigma_{21}=\ttau_{21}^p=(-1)^p\ttau_{12}^p
=(-1)^pp\tsigma_{12}$. \par
(2) Let $S$ be the polynomial algebra 
$\Z_{(p)}[\tlt_1, \tlt_2,\tlt_3]$ equipped with 
the $\delta$-structure defined by $\delta(\tlt_i)=0$
$(i=1,2,3)$, put $\tlt_{ij}=\tlt_j-\tlt_i$ for
$(i,j)=(1,2)$, $(2,3)$, $(1,3)$, 
and let $\tR$ be the $\delta$-envelope of
$S[\ttau_{12},\ttau_{13}]/(p\ttau_{12}-\tlt_{12},
p\ttau_{13}-\tlt_{13})$ over $S$.
Then the ring homomorphism $S\to R$ defined
by $\tlt_i\mapsto t_i$ $(i=1,2,3)$ is 
a $\delta$-homomorphism, and extends
uniquely to a $\delta$-homomorphism 
$f\colon \tR\to R$ sending
$\ttau_{12}$ and $\ttau_{13}$ to $\tau_{12}$ and $\tau_{13}$,
respectively. Put $\ttau_{23}=\ttau_{13}-\ttau_{12}$.
If we define $\tsigma_{ij}$
in the same way as $\sigma_{ij}$ by using 
$\ttau_{ij}$ and $\tlt_i$, 
then we have $f(\tsigma_{ij})=\sigma_{ij}$. 
Hence it suffices to prove the claim 
for $\tR$, $\tlt_i$, and $\ttau_{ij}$. 
Since $S$ is $p$-torsion free and 
the sequence $\tlt_{12}$, $\tlt_{13}$ is
$S/pS$-regular, the $\delta$-ring $\tR$
is $p$-torsion free by Proposition \ref{prop:DeltaEnvPTF}.
Hence the claim follows from 
$$p\tsigma_{23}
=(\ttau_{13}-\ttau_{12})^p
=p\tsigma_{13}+p(-1)^p\tsigma_{12}
+\sum_{\nu=1}^{p-1}\binom p\nu \ttau_{13}^{\nu}
(-\ttau_{12})^{p-\nu}.$$
\end{proof}

\section{Bounded prisms and bounded prismatic envelopes}
\label{sec:PrismPrismEnv}
We recall bounded prisms and define a variant of
prismatic envelopes for bounded prisms.

\begin{definition}\label{def:prism} (\cite[Definition 3.2, Lemma 3.7 (1)]{BS})
(1) A {\it $\delta$-pair} is a pair $(R,I)$ of a $\delta$-ring $R$ 
and an ideal $I$ of $R$. A {\it morphism of $\delta$-pairs} 
$f\colon (R,I)\to (S,J)$ is a $\delta$-homomorphism 
$f\colon R\to S$ such that $f(I)\subset J$.\par
(2) A {\it bounded prism} $(R,I)$ is a $\delta$-pair satisfying the following 
conditions.\par
(i)  $I$ is an invertible ideal.\par
(ii) $R$ is $(pR+I)$-adically complete and separated.\par
(iii) $p\in I+\varphi(I)R$.\par
(iv) $(R/I)[p^{\infty}]=(R/I)[p^N]$ $(N\gg 0)$.\par
A {\it morphism of bounded prisms} $f\colon (R,I)
\to (S,J)$ is a morphism as $\delta$-pairs.
\end{definition}

\begin{proposition} [{\cite[Lemma 3.5]{BS}}]
Let $f\colon (R,I)\to (S,J)$ be a morphism of 
bounded prisms. Then we have $J=f(R)I$.
\end{proposition}

\begin{proposition}
[{\cite[Lemma 2.25]{BS}}]\label{prop:PrismGenerator}
Let $(R,I)$ be a bounded prism. If $I=\xi R$ for $\xi\in R$,
then we have $\delta(\xi)\in R^{\times}$. 
\end{proposition}

\begin{proposition}\label{prop:bddFlatInv}
Let $N$ be a non-negative integer.
Let $(R,I)$ be a pair of a ring $R$ and an invertible ideal
$I$ of $R$, and let $\pi$ be an element of $R$ 
satisfying $(R/I)[\pi^{\infty}]=(R/I)[\pi^N]$.
Let $M$ be an $R$-module $(\pi R+I)$-adically complete and separated
such that the $R/(\pi R+I)^n$-module 
$M/(\pi R+I)^nM$ is flat for every positive integer $n$.
Then the following holds.\par
(1) $(M/IM)[\pi^{\infty}]=(M/IM)[\pi^N]$.\par
(2) $M/I^nM$ is $\pi$-adically complete and separated for every positive integer $n$.\par
(3) The homomorphism $I^n\otimes_RM\to M$ is injective for every positive integer $n$.
\end{proposition}

\begin{proof}
We have 
an exact sequence 
$0\to I\otimes_RR/I^n\to R/I^{n+1}\to R/I\to 0$
for each integer $n>0$,  and $I$ is a flat $R$-module
since $I$ is an invertible ideal of $R$.
Hence we see $(R/I^n)[\pi^{\infty}]=(R/I^n)[\pi^{nN}]$
by induction on $n$. By applying the snake lemma
to the multiplication by $\pi^l$ on the
short exact sequence
$0\to I^m\otimes_RR/I^n
\to R/I^{n+m}\to R/I^m\to 0$,
we see that the kernel $(K_l)_{l\geq 1}$
of the morphism of projective systems
$((I^m\otimes_{R}R/I^n)/\pi^l)_{l\geq 1}
\to ((R/I^{n+m})/\pi^l)_{l\geq 1}$
is Artin-Rees zero; the transition map
$K_{l+mN}\to K_l$ vanishes for every $l\geq 1$.
Since the $R/(I^{n+m}+\pi^lR)$-module
$M/(I^{n+m}M+\pi^lM)$  is flat, the same holds
for the kernel of 
$((I^m\otimes_RM/I^nM)/\pi^l)_{l\geq 1}
\to ((M/I^{n+m}M)/\pi^l)_{l\geq 1}$.
By taking the inverse limit over $l$
and noting that $I^m$ is a finite projective
$R$-module, we obtain a
short exact sequence
$$0\to I^m\otimes_R\varprojlim_l(M/I^nM)/\pi^l
\to \varprojlim_l(M/I^{n+m}M)/\pi^l
\to \varprojlim_l(M/I^mM)/\pi^l\to 0.$$
By taking $\varprojlim_n$, we obtain a
short exact sequence
$$0\to I^m\otimes_RM\to M\to 
\varprojlim_l (M/I^mM)/\pi^l\to 0.$$
This completes the proof of (2) and (3). \par

It is straightforward to verify that the sequence
$0\to R/I[\pi^N]\to (R/I)/\pi^n\xrightarrow{\pi^m}
(R/I)/\pi^{n+m}$ is exact for every 
$n,m\geq N$. 
Since the $R/(I+\pi^{n+m}R)$-module $M/(IM+\pi^{n+m}M)$ is flat,
we see that the above sequence remains exact 
after $-\otimes_RM$, whose inverse limit over $n$ gives
an exact sequence
$0\to R/I[\pi^N]\otimes_RM
\to M/IM\xrightarrow{\pi^m}M/IM$ for $m\geq N$ by the claim (2). Hence the claim (1) holds.
\end{proof}

\begin{corollary}\label{cor:bddPrismRedpcomp}
For any bounded prism $(R,I)$, $R/I^n$ is $p$-adically complete
and separated.
\end{corollary}

\begin{proof}
We just apply Proposition \ref{prop:bddFlatInv} (2) to $(R,I)$, $\pi=p$, and $M=R$.
\end{proof}

\begin{corollary}\label{cor:bddPrismFlatMap}
Let $(R,I)$ be a bounded prism, and let 
$S$ be a $\delta$-$R$-algebra $(pR+I)$-adically complete
and separated such that 
$R\to S$ is $(pR+I)$-adically flat (Definition \ref{def:formallyflat} (1)).
Then $(S,IS)$ is a bounded prism.
\end{corollary}

\begin{corollary}
Let $(R,I)$ be a bounded prism, and let 
$S$ be an $R$-algebra $(pR+I)$-adically complete
and separated such that 
$R\to S$ is $(pR+I)$-adically \'etale 
(Definition \ref{def:formallyflat} (1)). 
Then $(S,IS)$ equipped with the unique $\delta$-$R$-algebra
structure (Proposition \ref{prop:DeltaStrExtEtSm} (1)) is a bounded prism over
$(R,I)$.
\end{corollary}

\begin{definition}\label{def:bddPrsimEnv}
(cf.~\cite[Proposition 3.13]{BS}).
Let $(R,I)$ be bounded prism, and let
$(A,J)$ be a $\delta$-pair over $(R,I)$. 
We say that a bounded prism 
$(D,ID)$ over $(R,I)$ equipped with a
morphism of $\delta$-pairs $g\colon (A,J)\to (D,ID)$
over $(R,I)$ 
is a {\it bounded prismatic envelope}
of $(A,J)$ over $(R,I)$ if it satisfies the following 
universal property: For any bounded prism 
$(B,IB)$ over $(R,I)$ and any morphism of $\delta$-pairs
$h\colon (A,J)\to (B,IB)$ over $(R,I)$, there exists
a unique morphism
$k\colon (D,ID)\to (B,IB)$ of bounded prisms over
$(R,I)$ satisfying $h=k\circ g$.
\end{definition}

\begin{lemma}\label{lem:CompFinGenIdeal}
Let $A$ be a ring, let $\fa$ be a finitely generated ideal of $A$,
and put $A_n=A/\fa^n$ for each positive integer $n$.
Let $(M_n)_{n>0}$ be an inverse system of $A$-modules
such that $\fa^nM_n=0$ and the transition map 
    $M_{n+1}\to M_n$ induces an isomorphism $M_{n+1}\otimes_{A_{n+1}}A_n
\xrightarrow{\cong} M_n$ for every positive integer $n$
(e.g.~$M_n=M/\fa^n M$ for an $A$-module $M$).
Put $\hM=\varprojlim_nM_n$. 
Then, for any positive integer $m$, the submodule 
$\varprojlim_{n\geq m}\fa^mM_n$ of $\hM$ coincides with 
$\fa^m\hM$, whence the projection 
map $\hM\to M_m$ induces an isomorphism 
$\hM/\fa^m\hM\xrightarrow{\cong}M_m$.
In particular, $\hM$ is $\fa$-adically complete and separated.
\end{lemma}

\begin{proof}
Since $\fa^m$ is a finitely generated ideal of $A$
and $\hM=\varprojlim_n M_{nm}$, we may replace
$\fa$ and $M_n$ with $\fa^m$ and $M_{nm}$, respectively,
and assume $m=1$. 
Choose generators
$a_1,\ldots, a_r$ of $\fa$. Put $A_0=0$ and $M_0=0$.
For $a\in \fa$ and a positive integer
$n$, the inclusion $a\fa^{n-1}\subset \fa^n$
implies that there exists a unique $A_n$-linear map
$[a]_n\colon M_{n-1}\to \fa M_n$ whose composition
with the surjective $A_n$-linear map
$M_n\to M_n/\fa^{n-1}M_n\cong M_{n-1}$
is the multiplication by $a$. The maps $[a]_n$ $(n\geq 1)$
define a morphism of inverse systems of $A$-modules
$(M_{n-1})_{n\geq 1}\to (\fa M_n)_{n\geq 1}$.
Let $(x_n)_{n\geq 1}$
be an element of $\varprojlim_{n\geq 1}\fa M_n$.\par
Let $n$ be a positive integer.
Choose $y_{1},\ldots, y_{r}\in M_{n-1}$ such that $x_n=\sum_{i=1}^r[a_i]_ny_{i}$. 
Note $\fa M_n=\sum_{i=1}^ra_iM_n=\sum_{i=1}^r [a_i]_n(M_{n-1})$.
Choose a lifting $\ty_{i}\in M_n$ of $y_{i}$
for each $i\in \N\cap [1,r]$, and put $\tx_n=\sum_{i=1}^r[a_i]_{n+1}\ty_{i}\in M_{n+1}$.
Then $\tx_n$ is a lifting of $x_n$, and therefore $x_{n+1}-\tx_n\in \fa^nM_{n+1}$.
Since $\fa^nM_{n+1}=\sum_{i=1}^ra_i\fa^{n-1}M_{n+1}
=\sum_{i=1}^r[a_i]_{n+1}(\fa^{n-1}M_n)$, 
we can choose $z_{1},\ldots, z_{r}\in \fa^{n-1}M_n$ such that
$x_{n+1}-\tx_n=\sum_{i=1}^r[a_i]_{n+1}z_{i}$. Then
$w_{i}=\ty_{i}+z_{i}\in M_{n}$ is a lifting of $y_{i}$
for $i\in \N\cap [1,r]$, and we have $x_{n+1}=\sum_{i=1}^r[a_i]_{n+1}w_{i}$.\par

For each positive integer $n\geq 1$, let $S_{x,n}$
be the subset of $M_{n-1}^{\oplus r}$ consisting of
$(y_i)\in M_{n-1}^{\oplus r}$ satisfying $x_n=\sum_{i=1}^r[a_i]_{n}y_i$.
Then the above argument shows that the
surjective map $M_{n}^{\oplus r}\to M_{n-1}^{\oplus r}$
induces a surjective map $S_{x,n+1}\to S_{x,n}$ for every $n\geq 1$.
Hence there exists  an element $((y_{i,n})_i)\in \varprojlim_nS_{x,n}\subset
\varprojlim_n M_{n-1}^{\oplus r}=\hM^{\oplus r}$. 
Setting $y_i=(y_{i,n+1})_{n\geq 1}\in \varprojlim_n M_n=\hM$,
we obtain $x=\sum_{i=1}^ra_iy_i$;
for each $n\geq 1$, we have 
$x_n=\sum_{i=1}^r[a_i]_n y_{i,n}=\sum_{i=1}^r a_iy_{i,n+1}$.
\end{proof}

\begin{lemma}\label{lem:bddPrismEnvBase}
(1) Let $(R,I)$ be a bounded prism, let $(R',IR')$ be a bounded
prism over $(R,I)$, and let $(A,J)$ be a $\delta$-pair
over $(R',IR')$. Let $(D,I_D)$ be a bounded prism, and
let $g\colon (A,J)\to (D,I_D)$  be a morphism of $\delta$-pairs.
We regard $(D,I_D)$ as a bounded prism over $(R,I)$
(resp.~$(R',IR')$) by the composition of $g$ with 
the morphism $(R',IR')\to (A,J)$ (resp.~$(R,I)\to (R',IR')\to (A,J)$).
Then $g$ is a bounded prismatic envelope  
over $(R,IR)$ if and only if $g$ is a bounded 
prismatic envelope over $(R',IR')$.\par
(2) Let $(R,I)$ be a bounded prism, and let $f\colon (A,J)\to (A',J')$
be a morphism of $\delta$-pairs over $(R,I)$. 
Let $g\colon (A,J)\to (D,ID)$ be a bounded prismatic envelope 
over $(R,I)$. Let $(D',ID')$ be a bounded prism over $(R,I)$, and
let $g'\colon (A',J')\to (D',ID')$ be a morphism of $\delta$-pairs
over $(R,I)$, which extends uniquely to a 
morphism of $\delta$-pairs
$g'_D\colon (A'\otimes_AD,J'\cdot A'\otimes_AD)
\to (D',ID')$ by the universal property of $g$ applied to 
$g'\circ f$.
Then $g'$ is a bounded prismatic envelope over $(R,I)$
if and only if $g'_D$ is a bounded prismatic envelope over
$(R,I)$ (or equivalently over $(D,ID)$ by (1)).\par
(3) Le $(R,I)$ be a bounded prism, let $(A,J)$ be a $\delta$-pair
over $(R,I)$, let $\hA$ be $\varprojlim_nA/(pA+IA)^n$
equipped with the $\delta$-structure induced by that of $A$,
and let  $\hJ$ be the ideal $\varprojlim_n J(A/(pA+IA)^n)$ of 
$\hA$. Let $(D,ID)$ be a bounded prism over $(R,I)$, and
let $f$ be a morphism of $\delta$-pairs $(A,J)\to (D,ID)$
over $(R,I)$. Then $f$ extends uniquely to a morphism
of $\delta$-pairs $\hf \colon (\hA,\hJ)\to (D,ID)$ over $(R,I)$, 
and $f$ is a bounded prismatic envelope over $(R,I)$
if and only if $\hf$ is a bounded prismatic envelope over $(R,I)$.
\end{lemma}

\begin{proof}
(1) For any bounded prism $(B,IB)$ over $(R,I)$ and any morphism  
$h\colon (A,J)\to (B,IB)$ of $\delta$-pairs over $(R,I)$, we may 
regard $(B,IB)$ as a bounded prism over $(R',IR')$ by the composition of 
$h$ and $(R',IR')\to (A,J)$, for which $h$ becomes a morphism
over $(R',IR')$. This implies the claim.\par
(2) The unique existence of the extension $g'_{D}$ for 
an arbitrary $g'$ implies that $g'$ satisfies the universal property
of bounded prismatic envelope over $(R,I)$ if and only if $g'_D$ does. 

(3) The second claim follows from the first one for an arbitrary $f$, similarly
to the proof of (2). For the first claim, 
the $(pR+I)$-adic completion of $f$ gives the desired extension of 
$f$ since $ID=\varprojlim_n I(D/(pD+ID)^n)$ by Corollary \ref{cor:bddPrismRedpcomp}.
The extension $\hf$ is unique because
$A/(pA+IA)^n\to \hA/(p\hA+I\hA)^n$ is an isomorphism
by Lemma \ref{lem:CompFinGenIdeal}.
\end{proof}

\begin{proposition}\label{prop:PrismEnvEtale}
Let $(R,I)$ be a bounded prism, and let $f\colon (A,J)\to (A',J')$
 be a morphism of $\delta$-pairs over $(R,I)$ such that 
 $f\colon A\to A'$ is $(pR+I)$-adically \'etale
 (Definition \ref{def:formallyflat} (1)) and induces an isomorphism
 $A/J\xrightarrow{\cong}A'/J'$. 
 Then $(A,J)$ has  a bounded prismatic envelope $(D,ID)$
 over $(R,I)$ if and only if $(A',J')$ has a bounded prismatic envelope
 $(D',ID')$ over $(R,I)$. If these equivalent conditions hold, 
then the morphism of prisms
 $(D,ID)\to (D',ID')$ over $(R,I)$ induced by 
 $f$ is an isomorphism. 
\end{proposition}

\begin{proof}
It suffices to prove that, for any bounded prism 
$(B,IB)$ over $(R,I)$, any morphism of $\delta$-pairs
$(A,J)\to (B,IB)$ over $(R,I)$ extends uniquely to 
a morphism of $\delta$-pairs $g\colon (A',J')\to (B,IB)$ over $(R,I)$.
By Proposition \ref{prop:DeltaCompEtExt}, we may forget the
compatibility of the extension with the $\delta$-structures. 
Hence, by the assumption $A/J\xrightarrow{\cong} A'/J'$,
this is reduced to showing that there exists a unique ring
homomorphism $A'\to B$ making the diagram 
\begin{equation*}
\xymatrix{
A'\ar[r]\ar@{-->}[dr]& B/IB\\
A\ar[r]\ar[u]^f&\ar[u] B
}
\end{equation*}
commute, where the upper horizontal map is the composition  of 
$A'\to A'/J'\cong A/J$ and the morphism 
$A/J\to B/IB$ induced by $g$. The claim holds
after taking the reduction mod $(pR+I)^n$ for every
integer $n\geq 1$ because $f$ mod $(pR+I)^n$
is \'etale. By taking the inverse limit over $n$,
we obtain the desired claim because $B/IB$
is $p$-adically complete and separated by 
Corollary \ref{cor:bddPrismRedpcomp}. 
\end{proof}

\begin{proposition}[cf.~{\cite[Proposition 3.13]{BS}}]
\label{prop:PrismEnvRegSeq}
Let $(R,\xi R)$ be a bounded prism, put $\tI=pR+\xi R$,
and let $(A,J)$ be a $\delta$-pair over $(R,\xi R)$ such that 
$A$ is $\tI$-adically flat over $R$ and 
$A/J$ is $p$-adically flat over $R/\xi R$
(Definition \ref{def:formallyflat} (1)). Assume that 
we are given a sequence $T_1,\ldots, T_d\in J$ whose
image in $A/\tI A$ generates 
$J\cdot A/\tI A$ and is transversally regular 
relative to $R/\tI$ (\cite[D\'efinition (19.2.1)]{EGAIV}), i.e.,
forms a regular sequence with  quotients 
$(A/\tI A)/\sum_{i=1}^r T_i (A/\tI A)$ $(r\in \N\cap [0,d])$
flat over $R/\tI$. Then $R$, $\xi$, $A$, and $T_1,\ldots, T_d$
satisfy all the conditions in Proposition \ref{prop:DeltaEnvRegSeq} (3)
by Remark \ref{rmk:DeltaEnvRegSeq} (1).
(Note $\delta(\xi)\in R^{\times}$ by Proposition \ref{prop:PrismGenerator}.)
Let $\hD$ be the $\tI$-adic completion of the $\delta$-$A$-algebra $D$ 
constructed in Proposition \ref{prop:DeltaEnvRegSeq}. Then 
$(\hD,\xi\hD)$ is a bounded prism, $J\hD\subset \xi\hD$,
and the morphism of $\delta$-pairs 
$(A,J)\to (\hD,\xi\hD)$ over $(R,\xi R)$ is 
a bounded prismatic envelope of $(A,J)$ over $(R,\xi R)$. 
\end{proposition}

\begin{proof}
First note that $\hD$ is $\tI$-adically complete and separated
and we have $\hD/\tI^n\hD\cong  D/\tI^n D$ by Lemma \ref{lem:CompFinGenIdeal}. 
By Proposition 
\ref{prop:DeltaEnvRegSeq} (3), $R/\tI^n\to D/\tI^nD$ is
flat for every positive integer $n$. Hence 
$R\to \hD$ is $\tI$-adically flat, and therefore
$(\hD,\xi\hD)$ is a bounded prism by Corollary \ref{cor:bddPrismFlatMap}. 
Since $(A/J)/p^n$ is flat over $(R/\xi R)/p^n$ by assumption,
the homomorphism 
$(J\cdot (A/\xi A)/p^n)/p\to J\cdot (A/\xi A)/p$ is an isomorphism.
This implies that the ideal $J\cdot (A/\xi A)/p^n$
of $(A/\xi A)/p^n$ is generated by the images of
$T_1,\ldots, T_d$. Hence the image of $J$
in $(D/\xi D)/p^n\cong \hD/(\xi \hD+p^n\hD)$ 
vanishes for every positive integer $n$.
This shows $J\hD\subset \xi\hD$ as
$\hD/\xi\hD\cong\varprojlim_n (\hD/\xi \hD)/p^n$
by Corollary \ref{cor:bddPrismRedpcomp}. 
It remains to prove that $(A,J)\to (\hD,\xi\hD)$ is a bounded
prismatic envelope of $(A,J)$ over $(R,\xi R)$. 
Any $(R,\xi R)$-morphism of $\delta$-pairs
$g\colon (A,J)\to (B,\xi B)$ to a bounded prism over $(R,\xi R)$
extends uniquely to an $A$-homomorphism 
$A[S_1,\ldots, S_d]/(\xi S_i-T_i)\to B$ because
$g(T_i)\in\xi B$ and $\xi$ is regular on $B$. 
By the universality of $\delta$-envelope and the $\tI$-adic
completeness of $B$, it has a unique extension to a 
$\delta$-$A$-homomorphism $\hD\to B$. 
This completes the proof. 
\end{proof}

\begin{proposition}\label{prop:bddPrismEnvSmSmEmb}
Let $(R,\xi R)$ be a bounded prism, put $\tI= pR+\xi R$,
and let $(A,J)$ be a $\delta$-pair over $(R,\xi R)$
such that $A$ is $\tI$-adically smooth over $R$
and $A/J$ is $p$-adically smooth over $R/\xi R$
(Definition \ref{def:formallyflat} (1)).  
Assume that $J\cdot (A/\tI A)/J^2\cdot (A/\tI A)$ is a finite
free $A/(\tI A+J)$-module, and let $T_1,\ldots, T_d$ be 
elements of $J$ whose images in $J\cdot (A/\tI A)/J^2\cdot (A/\tI A)$ 
form a basis. Then there exists $f\in 1+J$ such that 
$(R,\xi R)$, $(\widehat{A_f}, J\widehat{A_f})$, and $T_1,\ldots, T_d$ satisfy
the assumptions in Proposition \ref{prop:PrismEnvRegSeq},
where $\widehat{A_f}$ denotes the $\tI$-adic completion 
$\varprojlim_nA_f/\tI^nA_f$ of $A_f$ equipped with the unique 
$\delta$-$A$-algebra structure (Proposition \ref{prop:DeltaStrExtEtSm} (1)).
(Thanks to Propositions \ref{lem:bddPrismEnvBase} (3) and \ref{prop:PrismEnvEtale}, 
we obtain a bounded prismatic
envelope of $(A,J)$ over $(R,\xi R)$ by applying the construction
of Proposition \ref{prop:PrismEnvRegSeq} to $(\widehat{A_f},J\widehat{A_f})$. )
\end{proposition}

\begin{proof}
Put $\oR=R/\tI$, $\oA=A/\tI A$, $\oJ=J\cdot \oA$, and
$\oT_i=(T_i\mod \tI)\in \oJ$. Since $\oA$ and $\oA/\oJ$
are smooth over $\oR$, $\oJ$ is finitely generated.
By Nakayama's lemma, $\oJ\cdot \oA_{1+\oJ}$ is generated
by $\oT_1,\ldots, \oT_d$ as $\oJ\cdot \oA_{1+\oJ}$ is
contained in the Jacobson radical of $\oA_{1+\oJ}$. 
Hence there exists $g\in 1+\oJ$ such that $\oJ\cdot \oA_g$
is also generated by $\oT_1,\ldots, \oT_d$.\par
Since $\oA$ and $\oA/\oJ$ are smooth over
$\oR$, $\Omega_{\oA/\oR}$ is a finite projective $\oA$-module,
and the images of $d\oT_i\in \Omega_{\oA/\oR}$ $(1\leq i\leq d)$
in $\Omega_{\oA/\oR}\otimes_{\oA}\oA/\fm$ form
a part of a basis of $\Omega_{\oA/\oR}\otimes_{\oA}\oA/\fm$
for each maximal ideal $\fm$ of $\oA$ containing $\oJ$. 
Let $N$ be the free $\oA$-module $\oA^{\oplus d}$, 
let $\iota$ be the $\oA$-linear map $N\to \Omega_{\oA/\oR};
(a_i)_{1\leq i\leq d}\mapsto \sum_{1\leq i\leq d} a_id\oT_i$,
and let $M$ be the cokernel of $\iota$.
Then, we see that $\iota_{1+\oJ}\colon 
\colon N_{1+\oJ}\to (\Omega_{\oA/\oR})_{1+\oJ}$
is injective, and $M_{1+\oJ}$ is 
a finite projective $\oA_{1+\oJ}$-module.
Since $M$ and $\Omega_{\oA/\oR}$ are finitely presented $\oA$-modules,
the claim still holds after replacing $(-)_{1+\oJ}$ with $(-)_h$
for some $h\in 1+\oJ$; splittings
$M_{1+\oJ}\to (\Omega_{\oA/\oR})_{1+\oJ}$
and $(\Omega_{\oA/\oR})_{1+\oJ}
\to N_{1+\oJ}$
extend to splittings 
$M_h\to (\Omega_{\oA/\oR})_h$
and $(\Omega_{\oA/\oR})_h
\to N_h$ for some $h\in 1+\oJ$. For $h$ as above, the homomorphism 
$\oR[U_1,\ldots, U_d]\to \oA_h;U_i\mapsto T_i$ is smooth.
Hence the image of the sequence $T_1,\ldots, T_d$
in $\oA_h$ is transversally regular relative to $\oR$. 
Now any lifting $f\in 1+J$ of $gh\in 1+\oJ$
satisfies the desired properties. 
Note that we have $\widehat{A_f}/\tI^{n+1}\widehat{A_f}
\cong A_f/\tI^{n+1}A_f$ (Lemma \ref{lem:CompFinGenIdeal}). 
\end{proof}

\begin{corollary}\label{cor:PrismEnvSection}
Let $(R,\xi R)$ be a bounded prism, put $\tI=pR+\xi R$,
and let $f\colon (A,J)\to (A',J')$ be a morphism of $\delta$-pairs
over $(R,\xi R)$ such that the underlying $R$-algebra 
homomorphism of $f$ is $\tI$-adically smooth 
(Definition \ref{def:formallyflat} (1)),
and the homomorphism $A/J\to A'/J'$ induced by $f$ is an
isomorphism. Suppose that there exist a bounded prismatic
envelope $(D,\xi D)$ of $(A,J)$ over $(R,\xi R)$ and $\tI$-adic coordinates
(Definition \ref{def:formallyflat} (2)) of $A'$ over $A$. Then there exists
a bounded prismatic envelope $(D',\xi D')$ of $(A',J')$ over $(R,\xi R)$. Moreover,
if we choose $\tI$-adic coordinates $t_1,\ldots, t_d\in A'$ 
of $A'$ over $A$ and elements $a_1,\ldots, a_d\in A$
whose images in $A'/J'\cong A/J$ coincide with
the images of $t_1,\ldots, t_d$, then $D'/\tI^{n+1}D'$
is  a free $D/\tI^{n+1}D$-module with a basis
$\prod_{i=1}^d\tau_i^{\{m_i\}_{\delta}}$ mod $\tI^{n+1}D'$
$((m_i)\in \N^d)$, where 
$\tau_i=\xi^{-1}(t_i-a_i)\in D'$ and 
$\tau_i^{\{m\}_{\delta}}$ $(m\in \N)$ is
defined as in Proposition \ref{prop:DeltaEnvRegSeq}.
\end{corollary}

\begin{proof}
By Lemma \ref{lem:bddPrismEnvBase} (2), 
it suffices to prove the claim for a bounded prismatic 
envelope of $(A'\otimes_AD, J'(A'\otimes_AD))$ over $(D,\xi D)$. 
Since $(A'\otimes_AD)/J'\cdot (A'\otimes_AD)
\cong A'/J'\otimes_{A/J} D/\xi D\cong D/\xi D$
and $t_i\otimes 1$ $(i\in \N\cap [1,d])$ are $\tI$-adic
coordinates of $A'\otimes_AD$ over $D$,
one can apply Proposition \ref{prop:bddPrismEnvSmSmEmb} to 
$T_i=t_i\otimes 1-1\otimes a_i\in J' (A'\otimes_AD)$
and obtain the claims from Proposition \ref{prop:PrismEnvEtale} and 
Proposition \ref{prop:DeltaEnvRegSeq} (4). 
Note that the under the notation in Proposition \ref{prop:bddPrismEnvSmSmEmb},
we have $\widehat{A_f}/\tI\widehat{A_f}\cong A_f/\tI A_f$
and 
$A_f/(\tI A_f+\sum_{i=1}^d T_iA_f)
\cong (A_f/\tI A_f)/J(A_f/\tI A_f)
\cong (A_f/J A_f)/\tI(A_f/J A_f)
\cong (A/J)/\tI (A/J)$. 
\end{proof}

\section{Twisted derivations on a ring}\label{sec:TwistedDeriv}
We first introduce a notion of a derivation on a module twisted by 
a ring endomorphism and summarize
its basic properties.

\begin{definition}\label{def:gammaDeriv}
Let $R$ be a ring, let $A$ be an $R$-algebra, let $\gamma$
be an endomorphism of the $R$-algebra $A$, and let $M$ 
be an $A$-module. A {\it $\gamma$-derivation $\partial$ of $A$
into $M$ over $R$} is an $R$-linear map 
$\partial\colon A\to M$ satisfying $\partial(xy)=\gamma(x)\partial(y)+y\partial(x)$
for all $x, y\in A$. We write $\Der_R^{\gamma}(A,M)$
for the set of all $\gamma$-derivations of $A$ into $M$ over $R$,
which naturally forms an $A$-module. 
\end{definition}

\begin{remark}
Let $R$, $A$, $\gamma$, and $M$ be the same as in Definition 
\ref{def:gammaDeriv}.\par
(1)  For $\partial\in \Der_R^{\gamma}(A,M)$, we have 
$\partial(1)=\partial(1\cdot 1)=\partial(1)+\partial(1)$, 
which implies $\partial(1)=0$, and hence
$\partial(r\cdot 1_A)=0$ for all $r\in R$.\par
(2) The composition with a homomorphism of $A$-modules
$M\to M'$ induces an $A$-linear homomorphism 
$\Der_R^{\gamma}(A,M)\to \Der_R^{\gamma}(A,M')$.\par
(3) Let $R'$ be an $R$-algebra. Put $A'=R'\otimes_RA$,
$\gamma'=\id_{R'}\otimes\gamma$,
and $M'=R'\otimes_RM$. Then the scalar extension under
$R\to R'$ induces an $A$-linear homomorphism 
$\Der_R^{\gamma}(A,M)\to \Der_{R'}^{\gamma'}(A',M'),$
where we regard the target as an $A$-module via
$A\to R'\otimes_RA=A'; a\mapsto 1\otimes a$. 
\end{remark}

\begin{example}\label{ex:gammaDeriv}
Let $R$, $A$, and $\gamma$ be as in Definition \ref{def:gammaDeriv}.
Let $\alpha\in A$ be a non-zero divisor, and assume $(\gamma-1)(A)
\subset \alpha A$. Then $\partial=\alpha^{-1}(\gamma-1)\colon A\to A$
is a $\gamma$-derivation of $A$ into $A$ over $R$ since we have
$(\gamma-1)(xy)=\gamma(x)(\gamma(y)-y)+(\gamma(x)-x)y$ for
$x,y\in A$.
\end{example}

The $\gamma$-derivation $\partial\colon A\to A$ over $R$ in 
Example \ref{ex:gammaDeriv} satisfies 
\begin{equation}\label{eq:divdedDifference}
\partial(xy)=\gamma(x)\partial(y)+y\partial(x)
=\partial(x)y+x\partial(y)+\alpha\partial(x)\partial(y)\quad (x,y\in A).
\end{equation}
Based on this observation, we define a twisted derivation
on a ring $A$ with respect to $\alpha$ without assuming
that $\alpha $ is a non-zero divisor as follows.

\begin{definition}\label{def:alphaDerivation}
 (1) Let $R$ be a ring, let $A$ be an $R$-algebra, and 
let $\alpha$ be an element of $A$. {\it An $\alpha$-derivation 
of $A$ over $R$} is an $R$-linear map $\partial\colon A\to A$
satisfying $\partial(1)=0$ and $\partial(xy)=\partial(x)y+x\partial(y)+\alpha \partial(x)\partial(y)$ for all $x$, $y\in A$.
We write $\Der^{\alpha}_R(A)$ for the set of all $\alpha$-derivations
of $A$ over $R$. When we are given a homomorphism of rings
$S\to A$ and $\alpha$ is the image of an element
$\sigma$ of $S$, we also say a $\sigma$-derivation instead of 
an $\alpha$-derivation and write $\Der_R^{\sigma}(A)$ for
$\Der_R^{\alpha}(A)$. \par
(2) Let $f\colon R\to R'$ be a homomorphism of rings, and let
$g\colon A\to A'$ be a homomorphism over $f$ from an $R$-algebra
to an $R'$-algebra. Let $\alpha$ be an element of $A$.
For $\partial\in \Der_R^{\alpha}(A)$ and $\partial'\in \Der_{R'}^{\alpha}(A')$,
we say that {\it $g$ is compatible with $\partial$ and $\partial'$} if
$\partial'\circ g=g\circ\partial$. We also say that
{\it $\partial'$ is an extension of $\partial$ (along $g$). }
\end{definition}

\begin{remark}\label{rmk:alphaDerivBC}
Let $R$, $A$, and $\alpha$ be the same as in Definition \ref{def:alphaDerivation}, and let $\partial\colon A\to A$ be
an $\alpha$-derivation over $R$. \par
(1) Let $R'$ be an $R$-algebra, put 
$A'=A\otimes_RR'$, and let $\alpha'$ be the image $\alpha\otimes1$
of $\alpha$ in $A'$. Then the $R'$-linear extension
$\partial'\colon A'\to A'$ of $\partial$ is an $\alpha'$-derivation since both 
$\partial'(xy)$ and $\partial'(x)y+x \partial'(y)+\alpha' \partial'(x)\partial'(y)$
for a pair $(x,y)\in A'\times A'$ are $R'$-bilinear, and they coincide
when $(x,y)$ is in the image of $A\times A$. \par
(2) Let $I$ be an ideal of $R$, put $\hR=\varprojlim_n R/I^n$
and $\hA=\varprojlim_n A/I^nA$, and let $\halpha$ be 
the image of $\alpha$ in $\hA$. Then 
$\hpartial=\varprojlim_n\partial\otimes \id_{R/I^n}\colon \hA\to \hA$ is an 
$\halpha$-derivation of $\hA$ over $\hR$. \par
(3) Let $\fa$ be an ideal of $A$, put $\oA=A/\fa$, and let $\oalpha$
be the image of $\alpha$ in $\oA$. 
If $\partial(\fa)\subset \fa$, then the $R$-linear map 
$\opartial\colon \oA\to \oA$ induced by $\partial$ is an $\oalpha$-derivation of
$\oA$ over $R$. If the ideal $\fa$ is generated by a subset $\scrS$ of 
$\fa$, then $\partial(\scrS)\subset \fa$ implies $\partial(\fa)\subset \fa$.\par
(4) Put $A^\partial=\{x\in A\,\vert\, \partial(x)=0\}$. Then $A^\partial$ is an $R$-subalgebra
of $A$, and the map $\partial\colon A\to A$ is $A^\partial$-linear.\par
(5) Let $A'$ be an $R$-subalgebra of $A$ containing $\alpha$. 
If $\partial(A')\subset A'$, then $\partial\vert_{A'}\colon A'\to A'$
is an $\alpha$-derivation over $R$. In general,
we see that the subset $A''=\{x\in A'\,\vert\, \partial(x)\in A'\}$
is an $R$-subalgebra of $A'$. Hence, if $A'$ is
generated by a subset $\scrS$ of $A'$ as an $R$-algebra,
then the assumption $\partial(A')\subset A'$ above holds if 
$\partial(\scrS)\subset A'$.\par
(6) Let $I$ be an ideal of $A$ satisfying $\alpha\partial(I)\subset I$ (e.g.~$\alpha\in I$).
Then noting that $\partial$ is an additive map, we see $\partial(I^{n+1})\subset I^{n}$
$(n\in\N)$ by induction on $n$. Hence $\partial$ induces a 
homomorphism of inverse systems of $R$-modules $(\partial_n)_{n\in\N}
\colon (A/I^{n+1})_{n\in \N}\to (A/I^n)_{n\in \N}$. By taking the
inverse limit over $n$, we obtain an $\halpha$-derivation $\hpartial$
of  $\hA=\varprojlim_nA/I^n$ over $R$, where $\halpha$ denotes the
image of $\alpha$ in $\hA$.
\end{remark}

\begin{lemma}\label{lem:alphaDergammaDer}
Let $R$, $A$, and $\alpha$ be as in Definition \ref{def:alphaDerivation}.\par
(1) Let $\partial\colon A\to A$ be an $\alpha$-derivation of $A$ over $R$.
Then the map $\gamma=1+\alpha \partial\colon A\to A$ is an $R$-algebra
homomorphism, and $\partial$ is a $\gamma$-derivation of $A$ into
$A$ over $R$ (Definition \ref{def:gammaDeriv}). \par
(2) Let $\gamma\colon A\to A$ be an $R$-algebra endomorphism such that
$(\gamma-1)(A)\subset \alpha A$, and suppose that $\alpha$ is regular
on $A$. Then the $R$-linear map $\partial=\alpha^{-1}(\gamma-1)\colon A\to A$
is an $\alpha$-derivation of $A$ over $R$.
\end{lemma}

\begin{proof}
(1) The claim immediately follows from $(1+\alpha\partial)(1)=1$
and the following equalities for $x$, $y\in A$.
\begin{align*}
&(1+\alpha\partial)(x)\cdot (1+\alpha \partial)(y)=
xy+\alpha(\partial(x)y+x\partial(y)+\alpha \partial(x)\partial(y)),\\
&\partial(x)y+x\partial(y)+\alpha \partial(x)\partial(y)=\gamma(x) \partial(y)+\partial(x)y.
\end{align*}
(2) This is what we observed  \eqref{eq:divdedDifference} before Definition \ref{def:alphaDerivation}. 
\end{proof}

Similarly to a usual derivation on a module, we can interpret a twisted
derivation defined above in terms of a section of 
a certain augmented algebra. 

For a ring $A$ and an element $\alpha$ of $A$, we define
the twisted extension $E^{\alpha}(A)$ of $A$ by $A$ with 
respect to $\alpha$ to be the $A$-module $A\oplus A$ 
equipped with the product
$E^{\alpha}(A)\times E^{\alpha}(A)\to E^{\alpha}(A)$ defined by 
$(x_0,x_1)\cdot (y_0,y_1)=(x_0x_1,x_0y_1+x_1y_0+\alpha x_1y_1)$.
Then we have an $A$-linear bijection 
\begin{equation}\label{eq:TwExtIsom}
A[T]/(T^2-\alpha T)
\xrightarrow{\cong}E^{\alpha}(A);x_0+x_1T\mapsto (x_0,x_1)
\end{equation}
compatible with products. This implies that $E^{\alpha}(A)$ is
a commutative ring, and the maps 
$A\xrightarrow{\iota} E^{\alpha}(A)
\overset{\pi_0}{\underset{\pi_{\alpha}}\rightrightarrows} A$ defined by 
$\iota(x)=(x,0)$, $\pi_0(x_0,x_1)=x_0$, and $\pi_{\alpha}(x_0,x_1)=x_0+\alpha x_1$
are ring homomorphisms; the composition of  \eqref{eq:TwExtIsom} with 
the ring homomorphism $A\to A[T]/(T^2-\alpha T)$ coincides with 
$\iota$, and the compositions of \eqref{eq:TwExtIsom} with
$\pi_0$ and $\pi_{\alpha}$ are 
the $A$-algebra homomorphisms 
$A[T]/(T^2-\alpha T)\to A$ defined by $T\mapsto 0$
and by $T\mapsto \alpha$, respectively.

We define an additive  map $\scrD\colon E^{\alpha}(A)\to A$ by 
$\scrD(x_0,x_1)=x_1$. We have $\pi_{\alpha}=\pi_0+\alpha\scrD$.

\begin{lemma}\label{lem:UnivAlphaDeriv}
For $x$, $y\in E^{\alpha}(A)$, we have 
$$\scrD(xy)
=\pi_0(x)\scrD(y)+\scrD(x)\pi_0(y)+\alpha
\scrD(x)\scrD(y)
=\pi_{\alpha}(x)\scrD(y)+\pi_0(y)\scrD(x).$$
\end{lemma}

\begin{proof}
Put $x=(x_0,x_1)$ and $y=(y_0,y_1)$. Then we have 
$\scrD(xy)=x_0y_1+x_1y_0+\alpha x_1y_1
=(x_0+\alpha x_1)y_1+y_0x_1$. 
\end{proof}

\begin{lemma}\label{lem:UnivAlphaDerivMonom}
For a positive integer $n$, we have
$$\scrD(x^n)=\sum_{m=1}^n\binom n m\pi_0(x)^{n-m}\alpha^{m-1}\scrD(x)^m,
\qquad x\in E^{\alpha}(A).$$
\end{lemma}

\begin{proof}
In $A[T]/(T^2-\alpha T)$, 
we have
$$(x_0+x_1T)^n=
x_0^n+\sum_{m=1}^n\binom n m x_0^{n-m}(x_1T)^m
=x_0^n+\left(\sum_{m=1}^n\binom n mx_0^{n-m}x_1^m\alpha^{m-1}\right)T
$$
for $x_0, x_1\in A$. This completes the proof by \eqref{eq:TwExtIsom}.
\end{proof}

The construction of $E^{\alpha}(A)$ with 
$\iota$, $\pi_0$, $\pi_{\alpha}$, and $\scrD$ is functorial 
in $(A,\alpha)$ in the obvious sense.\par

Suppose that $A$ is an algebra over a ring $R$. 
Then an $\alpha$-derivation of $A$ over $R$ is interpreted
as follows.

\begin{proposition}\label{prop:AlphaDerivInterpret}
For $\partial\in \Der_R^{\alpha}(A)$,  the map $s\colon A\to E^{\alpha}(A)$
defined by $s(x)=(x,\partial(x))$ is an $R$-algebra homomorphism.
This gives a bijection from the set $\Der_R^{\alpha}(A)$
to the set of $R$-algebra homomorphisms
$A\to E^{\alpha}(A)$ whose composition with $\pi_0$ is
the identity. 
\end{proposition}

\begin{proof}
Let $\partial\colon A\to A$ be an additive map and define a map
$s\colon A\to E^{\alpha}(A)$ by $s(x)=(x,\partial(x))$. 
Then we have 
\begin{gather*}
s(x)s(y)=(xy, \partial(x)y+x\partial(y)+\alpha\partial(x)\partial(y))
\quad(x,y\in A),\quad
s(1)=(1,\partial(1)),\\
\quad s(x)=(x,\partial(x))\quad(x\in R).
\end{gather*}
Hence $s$ is an $R$-algebra homomorphism if and only if
\begin{equation*}
\partial(xy)=\partial(x)y+x\partial(y)+\alpha\partial(x)\partial(y)\quad
(x,y\in A),\quad
\partial(1)=0,\quad \partial(x)=0\quad (x\in R).
\end{equation*}
Under the first and second conditions, the last condition
is equivalent to assuming that $\partial$ is $R$-linear.
This completes the proof.
\end{proof}

By using Proposition \ref{prop:AlphaDerivInterpret}, we obtain the following formula
for an $\alpha$-derivation from Lemma \ref{lem:UnivAlphaDerivMonom}.

\begin{lemma}\label{lem:alphaDerivMonom}
For $\partial\in \Der_R^{\alpha}(A)$ and a positive integer $n$, we have
$$\partial(x^n)=\sum_{m=1}^n\binom n m x^{n-m}\alpha^{m-1}\partial(x)^m,
\quad x\in A.$$
\end{lemma}

\begin{proof}
Let $s$ be the $R$-homomorphism $A\to E^{\alpha}(A);
x\mapsto (x,\partial(x))$ corresponding to $\partial$ by Proposition \ref{prop:AlphaDerivInterpret}.
Then we obtain the desired equality by applying Lemma \ref{lem:UnivAlphaDerivMonom}
to $s(x)\in E^{\alpha}(A)$ and using $\scrD\circ s=\partial$, 
$\pi_0\circ s=\id_A$, and $s(x)^n=s(x^n)$. 
\end{proof}

\begin{lemma}\label{lem:TwExtNilp}
Let $A$ be a ring, let $N$ be a positive integer, and 
let $\alpha$ be an element of $A$ such that $\alpha^N=0$.
Let $J$ be the kernel of $\pi_0\colon E^{\alpha}(A)\to A$. Then
we have $J^{N+1}=0$. 
\end{lemma}

\begin{proof}
The lemma follows from \eqref{eq:TwExtIsom} and 
$T^{N+1}=\alpha^NT=0$ in $A[T]/(T^2-\alpha T)$. 
\end{proof}

Similarly to the case of usual derivations, we see that
the following two propositions hold for our $\alpha$-derivations.

\begin{proposition}\label{prop:TwistDerEtaleExt}
Let $R$ be a ring, let $A$ be an $R$-algebra, let $I$ be an ideal of $A$, and 
let $f\colon A\to A'$ be an $I$-adically \'etale 
(Definition \ref{def:formallyflat} (1)) homomorphism of $R$-algebras
such that $A'$ is $I$-adically complete and separated.
Let $\alpha$ be an $I$-adically nilpotent
element of $A$, and let $\alpha'$ be the 
image of $\alpha$ in $A'$. Then, for
each $\partial\in \Der_R^{\alpha}(A)$,
there exists a unique extension 
$\partial'\in \Der_R^{\alpha'}(A)$ of $\partial$
along $f$ (Definition \ref{def:alphaDerivation} (2)).
\end{proposition}

\begin{proof}
Since $\alpha$ is $I$-adically nilpotent, the $I$-adic topology
of an $A$-module is the same as the $\alpha A+I$-adic topology. 
Therefore, by replacing $I$ with $\alpha A+I$, we may 
assume $\alpha\in I$. 
Put $A/I^nA$, $A'_n=A'/I^nA'$,
and $f_n=(f\mod I^n)\colon A_n\to A'_n$
for a positive integer $n$. Then, by Remark \ref{rmk:alphaDerivBC} (6), 
we have $\partial(I^{n+1})\subset I^n$, which implies that the
$R$-homomorphism $s\colon A\to E^{\alpha}(A)$ corresponding
to $\partial$ by Proposition \ref{prop:AlphaDerivInterpret}
induces an $R$-homomorphism $s_n\colon A_{n+1}\to E^{\alpha}(A_n)$,
which fits into the following commutative diagram of $R$-algebras.
\begin{equation*}
\xymatrix@R=15pt{
A'_{n+1}\ar[rr]^{\text{pr}_{A',n}}&& A'_n\\
A_{n+1}\ar[u]^{f_{n+1}} \ar[r]^(.4){s_n}\ar[r]& 
E^{\alpha}(A_n)\ar[r]& E^{\alpha'}(A'_n)\ar[u]_{\pi_{A'_n,0}}
}
\end{equation*}
Here $\text{pr}_{A',n}$ denotes the projection map.
As $f_{n+1}\colon A_{n+1}\to A'_{n+1}$ is \'etale, and the kernel of $\pi_{A'_n,0}$ is 
nilpotent by Lemma \ref{lem:TwExtNilp}, 
there exists a unique
homomorphism $s_n'\colon A'_{n+1}\to E^{\alpha'}(A'_n)$ 
making the above diagram with $s_n'$ added commutative.
We obtain the proposition by taking the inverse limit over $n$
and using Proposition \ref{prop:AlphaDerivInterpret}.
\end{proof}

\begin{proposition}\label{prop:TwistDerivCoord}
Let $R$ be a ring, let $I$ be an ideal of $R$, and 
let $A$ be an $R$-algebra $I$-adically complete
and separated such that $R\to A$ is $I$-adically smooth
(Definition \ref{def:formallyflat} (1)).
Let $\alpha$ be an $I$-adically nilpotent element of 
$A$. Suppose that we are given $I$-adic coordinates (Definition \ref{def:formallyflat} (2))
$T_1,\ldots, T_d$ of $A$ over $R$ and elements $a_1,\ldots, a_d\in A$.
Then there exists a unique $\partial\in \Der_{R}^{\alpha}(A)$
such that
$\partial(T_i)=a_i$.
\end{proposition}

\begin{proof}
For a positive integer $n$, put $R_n=R/I^n$ and $A_n=A/I^nA$, 
and let $B_n$ be the polynomial algebra $R_n[S_1,\ldots, S_d]$
over $R_n$. Then the $R_n$-homomorphism 
$B_n\to A_n;S_i\mapsto (T_i\text{ mod }I^n)$ is \'etale by the assumptions 
on $R\to A$ and $T_i$, and we have the following commutative
diagram of rings
\begin{equation*}
\xymatrix@R=10pt{
A_n\ar[r]^{\text{id}_{A_n}}& A_n\\
B_n\ar[u] &\\
R_n\ar[u] \ar[r]&E^{\alpha}(A_n)\ar[uu]_{\pi_{A_n,0}}.
}
\end{equation*}
The $R_n$-homomorphism 
$s_n\colon B_n\to 
E^{\alpha}(A_n)$ sending $S_i$ to $(T_i, a_i)$ is a unique
homomorphism making the above diagram with $s_n$ added commutative
and sending $T_i$ to $a_i$ after composed with 
$\scrD_{A_n}\colon E^{\alpha}(A_n)\to A_n\colon (x_0,x_1)\mapsto x_1$.
As $B_n\to A_n$ is \'etale, and the kernel $\pi_{A_n,0}$ is 
nilpotent by Lemma \ref{lem:TwExtNilp}, there exists a unique
extension $s_n'\colon A_n\to E^{\alpha}(A_n)$ of 
$s_n$ satisfying $\pi_{A_n,0}\circ s'_n=\id_{A_n}$. 
We obtain the proposition by taking the inverse limit over $n$
and using Proposition \ref{prop:AlphaDerivInterpret}.
\end{proof}

\begin{lemma}\label{lem:alphaDerivEq}
Let $f\colon R\to R'$ be a homomorphism of rings, and let $g\colon A\to A'$ be a homomorphism 
over $f$ from an $R$-algebra to an $R'$-algebra. 
Let $\alpha$ be an element of $A$, and put $\alpha'=g(\alpha)$.
Let $I$ be an ideal of $A$ such that $A'$ is $I$-adically separated,
and let $\scrS$ be a subset of $A$ such that the $R$-subalgebra
$R[\scrS]$ of $A$ is dense with respect to the $I$-adic topology.
Then, for $\partial\in \Der^{\alpha}_R(A)$ and 
$\partial'\in \Der^{\alpha'}_{R'}(A')$ satisfying 
$\alpha\partial(I)\subset I$ and $\alpha'\partial'(IA')\subset IA'$
(which always hold when $I$ is generated by elements of $R$, or $\alpha\in I$),
$g$ is compatible with $\partial$ and $\partial'$
(Definition \ref{def:alphaDerivation} (2)) if and only if $\partial'\circ g(s)=g\circ \partial(s)$
for all $s\in \scrS$. 
\end{lemma}

\begin{proof}
The necessity is trivial. We prove the sufficiency.
Since $\partial'$ is $R$-linear with respect to the action of $R$ via $f$,
we may assume $R'=R$ and $f=\id$. 
Let $s$ (resp.~$s'$) be the $R$-homomorphism
 $A\to E^{\alpha}(A)$ 
(resp.~$A'\to E^{\alpha'}(A')$) corresponding to $\partial$
(resp.~$\partial'$) by Proposition \ref{prop:AlphaDerivInterpret}. 
Put $A_n=A/I^n$,  $A'_n=A'/I^nA'$, and $g=(g\mod I^n)\colon A_n\to A_n'$
 for each integer $n\geq 1$.
It suffices to prove the the left diagram below is commutative.
\begin{equation*}
\xymatrix{
A\ar[d]_{g}\ar[r]^(.4)s& E^{\alpha}(A)\ar[d]^{E^{\alpha}(g)}\\
A'\ar[r]^(.4){s'}&E^{\alpha'}(A')
}
\qquad
\xymatrix{
A_{n+1}\ar[d]_{g_{n+1}}\ar[r]^(.45){s_n}& E^{\alpha}(A_n)\ar[d]^{E^{\alpha}(g_n)}\\
A'_{n+1}\ar[r]^(.45){s'_n}&E^{\alpha'}(A'_n)
}
\end{equation*}
By Remark \ref{rmk:alphaDerivBC} (6) and the assumption on $\partial$ and $\partial'$, 
the left diagram induces the right one for each integer $n\geq 1$.
The right diagram is commutative since the $R$-algebra
$A_{n+1}$ is generated by the image of $\scrS$.
Since $A'$ is $I$-adically separated, this implies that the
left diagram is also commutative.
\end{proof}

\begin{lemma}\label{lem:TwDerivEtExtComp}
Let $R$ be a ring, let $A\xrightarrow{f} A'\xrightarrow{g} A''$
be homomorphisms of $R$-algebras, and let $\alpha$ be
an element of $A$. Let $I$ be an ideal of $A$, and assume
that $\alpha$ is $I$-adically nilpotent, $f$ is $I$-adically \'etale,
and $A''$ is $I$-adically separated. Let $\partial$, $\partial'$,
and $\partial''$ be $\alpha$-derivations of $A$, $A'$, and $A''$,
respectively, over $R$. If $\partial'$ and $\partial''$ are extensions
of $\partial$ along $f$ and $g\circ f$, respectively, 
then $\partial''$ is an extension of $\partial'$ along $g$.
\end{lemma}
\begin{proof}
Since $\alpha$ is $I$-adically nilpotent, the $I$-adic topology of an
$A$-module is the same as the $(\alpha A+I)$-adic topology. 
Therefore we may assume $\alpha\in I$ by replacing $I$ with $\alpha A+I$.
Let $s'$ and $s''$ be the $R$-homomorphisms $A'\to E^{\alpha}(A')$
and $A''\to E^{\alpha}(A'')$ corresponding to $\partial'$
and $\partial''$, respectively, by Proposition \ref{prop:AlphaDerivInterpret}.
For an integer $n\geq 1$, put $A_n=A/I^n$, $A'_n=A'/I^nA'$, and $A''_n=A''/I^nA''$,
and let $f_n\colon A_n\to A'_n$ and $g_n\colon A'_n\to A''_n$ denote
the homomorphisms induced by $f$ and $g$, respectively. By Remark \ref{rmk:alphaDerivBC} (6),
$s'$ and $s''$ induce $R$-homomorphisms
$s'_n\colon A'_{n+1}\to E^{\alpha}(A'_n)$ 
and $s''_n\colon A''_{n+1}\to E^{\alpha}(A''_n)$ for each integer $n\geq 1$.
Since $A''\to \varprojlim_n A''_n$ is injective,
it suffices to prove that the two homomorphisms
$E^{\alpha}(g_{n})\circ s'_n\colon A'_{n+1}\to E^{\alpha}(A'_n)
\to E^{\alpha}(A''_n)$ and $s''_n\circ g_{n+1}\colon A'_{n+1}\to A''_{n+1}
\to E^{\alpha}(A''_n)$ coincide for each positive integer $n$. 
They coincide after composing with the \'etale
homomorphism $f_{n+1}\colon A_{n+1}\to A'_{n+1}$
(resp.~the homomorphism $\pi_0\colon E^{\alpha}(A''_n)\to A''_n$ with
nilpotent kernel (Lemma \ref{lem:TwExtNilp})) since $\partial'$ and 
$\partial''$ are extensions of $\partial$
(resp.~by the construction of $s'$ and $s''$).
Hence we have $E^{\alpha}(g_n)\circ s'_n=s''_n\circ g_{n+1}$. 
\end{proof}

Finally we discuss the commutativity of two twisted derivations.
Let $A$ be a ring, and let $\alpha_1$ and $\alpha_2$ be elements
of $A$. Then we have the following isomorphisms by 
\eqref{eq:TwExtIsom}. 
\begin{align}\label{eq:TwExtCompDescription}
E^{\alpha_2}(E^{\alpha_1}(A))
&\cong E^{\alpha_1}(A)[T_2]/(T_2(T_2-\alpha_2))
\cong (A[T_1]/(T_1(T_1-\alpha_1)))[T_2]/(T_2(T_2-\alpha_2))\\
&\cong A[T_1,T_2]/(T_1(T_1-\alpha_1),T_2(T_2-\alpha_2)).\notag
\end{align}
By comparing it with the isomorphisms obtained by exchanging
$\alpha_1$ and $\alpha_2$, we obtain a canonical
isomorphism of $A$-algebras
\begin{equation}\label{eq:TwExtComp}
E^{\alpha_2}(E^{\alpha_1}(A))\cong E^{\alpha_1}(E^{\alpha_2}(A)).
\end{equation}
We identify the two $A$-algebras by \eqref{eq:TwExtComp}.

\begin{lemma}\label{lem:TwDerivSectionCommutativity}
Let $A$ be a ring, and let $\alpha_1$ and $\alpha_2$ be elements
of $A$. For $i\in \{1,2\}$, let $\partial_i$ be an $\alpha_i$-derivation 
of $A$ over $R$, and let $s_i$ be the $R$-homomorphism
$A\to E^{\alpha_i}(A)$ corresponding to $\partial_i$
(Proposition \ref{prop:AlphaDerivInterpret}). 
Assume that $\partial_1(\alpha_2)=0$ and $\partial_2(\alpha_1)=0$,
which are equivalent to $s_1(\alpha_2)=\alpha_2$ and
$s_2(\alpha_1)=\alpha_1$, respectively. 
We consider the following two compositions.
\begin{equation*}
A\xrightarrow{s_1} E^{\alpha_1}(A)\xrightarrow
{E^{\alpha_1}(s_2)}E^{s_2(\alpha_1)}(E^{\alpha_2}(A)),\qquad
A\xrightarrow{s_2} E^{\alpha_2}(A)\xrightarrow{E^{\alpha_2}(s_1)}
E^{s_1(\alpha_2)}(E^{\alpha_1}(A)).
\end{equation*}

(1) For $x\in A$, we have $\partial_2\circ\partial_1(x)=\partial_1\circ\partial_2(x)$
if and only if $E^{\alpha_1}(s_2)\circ s_1(x)=E^{\alpha_2}(s_1)\circ s_2(x)$. \par
(2) We have
$\partial_1\circ\partial_2=\partial_2\circ\partial_1$ if and only if 
$E^{\alpha_1}(s_2)\circ s_1=E^{\alpha_2}(s_1)\circ s_2$. 
\end{lemma}

\begin{proof}
The claim (2) follows from the claim (1). Let us prove (1). 
Under the description \eqref{eq:TwExtCompDescription} 
of $E^{\alpha_2}(E^{\alpha_1}(A))$,
we have
$$
E^{\alpha_2}(s_1)\circ s_2(x)=
E^{\alpha_2}(s_1)(x+\partial_2(x)T_2)
=(x+\partial_1(x)T_1)+(\partial_2(x)+\partial_1\circ\partial_2(x)T_1)T_2$$
for $x\in A$. Exchanging $(s_1,\alpha_1)$ and $(s_2,\alpha_2)$, we obtain
$$E^{\alpha_1}(s_2)\circ s_1(x)
=(x+\partial_2(x)T_2)+(\partial_1(x)+\partial_2\circ\partial_1(x)T_2)T_1.$$
We obtain the desired equivalence since 
$1$, $T_1$, $T_2$, and $T_1T_2$ form a basis of 
$E^{\alpha_1}(E^{\alpha_2}(A))=E^{\alpha_2}(E^{\alpha_1}(A))$
as an $A$-module.
\end{proof}

\begin{lemma}\label{lem:EtMapDerivComm}
Let $R$ be a ring, let $f\colon A\to A'$ be a homomorphism of 
$R$-algebras, and let $\alpha_1$ and $\alpha_2$ be elements
of $A'$. Let $I$ be an ideal of $A$ such that $\alpha_1$ and $\alpha_2$
are $I$-adically nilpotent, $f$ is $I$-adically \'etale
(Definition \ref{def:formallyflat} (1)), and $A'$ is $I$-adically separated.
Let $\scrS$ be a subset of $A$ such that the $R$-subalgebra 
$R[\scrS]$ of $A$ is dense with respect to the $I$-adic topology. 
Then, for $\partial_i\in \Der_R^{\alpha_i}{A'}$ $(i=1,2)$
satisfying $\partial_1(\alpha_2)=0$ and $\partial_2(\alpha_1)=0$,
we have $\partial_1\circ\partial_2=\partial_2\circ\partial_1$
if $\partial_1\circ\partial_2(f(s))=\partial_2\circ\partial_1(f(s))$
for all $s\in \scrS$.
\end{lemma}

\begin{proof}
For $i\in \{1,2\}$, let $s_i$ be the $R$-homomorphism
$A'\to E^{\alpha_i}(A')$ corresponding to $\partial_i$
by Proposition \ref{prop:AlphaDerivInterpret}. 
By Lemma \ref{lem:TwDerivSectionCommutativity} (2), 
it suffices to prove that the diagram
\begin{equation*}
\xymatrix@C=50pt{
A'\ar[r]^(.45){s_2}\ar@{=}[d]&
E^{\alpha_2}(A')\ar[r]^(.45){E^{\alpha_2}(s_1)}&
E^{\alpha_2}(E^{\alpha_1}(A'))\ar@{=}[d]^{\eqref{eq:TwExtComp}}\\
A'\ar[r]^(.45){s_1}&
E^{\alpha_1}(A')\ar[r]^(.45){E^{\alpha_1}(s_2)}&
E^{\alpha_1}(E^{\alpha_2}(A'))
}
\end{equation*}
is commutative.
Since $\alpha_1$ and $\alpha_2$ are $I$-adically nilpotent,
the $I$-adic topology of an $A$-module is the same as the
$(\alpha_1A+\alpha_2A+I)$-adic topology. Hence, we may replace
$I$ by $\alpha_1A+\alpha_2 A+I$ and assume 
$\alpha_1,\alpha_2\in I$. Put $A_n=A/I^n$, $A'_n=A'/I^nA'$,
and $f_n=(f\mod I^n)\colon A_n\to A'_n$ for each integer $n\geq 1$.
Then, by Remark \ref{rmk:alphaDerivBC} (6), the homomorphism
$s_i$ induces a homomorphism $s_{i,n}\colon A'_{n+1}\to E^{\alpha_i}(A'_{n})$
for $n\geq 1$. Since $A'\to \varprojlim_n A'/I^nA'$ is injective by assumption,
it suffices to prove that the following diagram is commutative.
\begin{equation}
\xymatrix@C=50pt{
A'_{n+2}\ar[r]^(.45){s_{2,n+1}}\ar@{=}[d]&
E^{\alpha_2}(A'_{n+1})\ar[r]^(.45){E^{\alpha_2}(s_{1,n})}&
E^{\alpha_2}(E^{\alpha_1}(A'_n))\ar@{=}[d]^{\eqref{eq:TwExtComp}}\\
A'_{n+2}\ar[r]^(.45){s_{1,n+1}}&
E^{\alpha_1}(A'_{n+1})\ar[r]^(.45){E^{\alpha_1}(s_{2,n})}&
E^{\alpha_1}(E^{\alpha_2}(A'_n))
}\tag{$*$}
\end{equation}
Since the $R$-algebra $A_{n+2}$ is generated by the image of $\scrS$,
Lemma \ref{lem:TwDerivSectionCommutativity} (1)
shows that the diagram $(*)$ is commutative after
composing with the \'etale homomorphism 
$f_{n+2}\colon A_{n+2}\to A'_{n+2}$. The homomorphisms 
$E^{\alpha_i}(E^{\alpha_j}(A'_n))
\xrightarrow{E^{\alpha_i}(\pi_0)}
E^{\alpha_i}(A'_n)\xrightarrow{\pi_0} A'_n$
for $(i,j)\in \{(1,2), (2,1)\}$ are the same, its
kernel is nilpotent by Lemma \ref{lem:TwExtNilp}, 
and its compositions with the two horizontal maps in 
$(*)$ both become the projection map $A'_{n+2}\to A'_n$.
Having these two observations, we see that
$(*)$ is commutative. 
\end{proof}

\section{Twisted derivations on a {$\delta$}-ring}\label{sec:TwistedDerivDelta}
Let $R$ be a $\delta$-ring, let $A$ be a $\delta$-$R$-algebra,
and let $\alpha$ be an element of $A$. We define
a compatibility of an $\alpha$-derivation of $A$ over $R$
with the $\delta$-structure on $A$, and study its basic properties.\par
For $\partial\in \Der_R^{\alpha}(A)$, it is natural to ask
if the $R$-endomorphism $\gamma=1+\alpha\partial$ of $A$
(Lemma \ref{lem:alphaDergammaDer} (1)) is a $\delta$-homomorphism.
For $x\in A$, we have
\begin{align*}
\gamma(\delta(x))&=\delta(x)+\alpha\partial(\delta(x)),\\
\delta(\gamma(x))&=
\delta(x)+\delta(\alpha)\partial(x)^p+\alpha^p\delta(\partial(x))+p\delta(\alpha)\delta(\partial(x))-\sum_{\nu=1}^{p-1}p^{-1}
\binom p\nu x^{p-\nu}\alpha^{\nu}\partial(x)^{\nu}.
\end{align*}
This computation naturally leads us to define a compatibility of $\partial$
with the $\delta$-structure on $A$ as follows.

\begin{definition}\label{def:alphaDerivDeltaComp}
Let $R$ be a $\delta$-ring, let $A$ be a $\delta$-$R$-algebra,
let $\alpha$ be an element of $A$ satisfying $\delta(\alpha)\in \alpha A$,
and let $\beta$ be an element of $A$ such that $\delta(\alpha)=\beta\alpha$.
Let $\partial\colon A\to A$ be an $\alpha$-derivation over $R$
(Definition \ref{def:alphaDerivation} (1)). For 
$x\in A$, we say that {\it $x$ is $\delta$-compatible with respect to 
$\partial$ and $\beta$} if the following equality holds.
\begin{equation}\label{eq:AlphaDerivDerComp}
\partial(\delta(x))=(\alpha^{p-1}+p\beta)\delta(\partial(x))
+\beta \partial(x)^p-\sum_{\nu=1}^{p-1}p^{-1}\binom p \nu x^{p-\nu}
\alpha^{\nu-1}\partial(x)^{\nu}. 
\end{equation}
We say that $\partial$ is {\it $\delta$-compatible with respect to $\beta$}
if every $x\in A$ satisfies \eqref{eq:AlphaDerivDerComp}. 
We write $\Der_{R,\delta}^{\alpha,\beta}(A)$ for the set of $\alpha$-derivations
of $A$ over $R$ $\delta$ compatible with respect to $\beta$. 
When we are given a homomorphism of $\delta$-rings $S\to A$,
and $\alpha$ and $\beta$  are the images of elements $\sigma$
and $\tau$ of $S$ satisfying $\delta(\sigma)=\sigma\tau$,
we also call $\partial\in \Der_{R,\delta}^{\alpha,\beta}(A)$ a $\sigma$-derivation of $A$ over $R$ 
$\delta$-compatible with respect to $\tau$, and
write $\Der_{R,\delta}^{\sigma,\tau}(A)$
for $\Der_{R,\delta}^{\alpha,\beta}(A)$.
\end{definition}

\begin{example}
Let $A$ be a $\delta$-ring, and let $t$ and $q$ be elements of $A$
such that $\delta(t)=\delta(q)=0$. Then 
$\alpha=t(q-1)$ and $\beta=t^{p-1}\sum_{\nu=1}^{p-1}p^{-1}\binom p\nu(q-1)^{\nu-1}$ satisfy $\delta(\alpha)=\alpha\beta$
by $\delta(\alpha)=t^p\delta(q-1)$ and
$0=\delta((q-1)+1)=\delta(q-1)-\sum_{\nu=1}^{p-1}p^{-1}\binom p\nu(q-1)^{\nu}$. If $t$ and $q-1$ are regular on $A$, and $\gamma\colon A\to A$
is a $\delta$-homomorphism satisfying $(\gamma-1)(A)\subset t(q-1)A$,
then we have $\partial=t^{-1}(q-1)^{-1}(\gamma-1)\in \Der_{\Z, \delta}^{\alpha,\beta}(A)$ by Lemma \ref{lem:alphaDergammaDer} (2) and the computation before Definition 
\ref{def:alphaDerivDeltaComp}.
\end{example}

\begin{lemma}\label{lem:alphaDerivDeltaEquiv}
Let $R$, $A$, $\alpha$, and $\beta$ be as in Definition \ref{def:alphaDerivDeltaComp}, let $\partial$ be an $\alpha$-derivation of 
$A$ over $R$, and let $\gamma$ be the $R$-algebra endomorphism 
$1+\alpha \partial$ of $A$ (Lemma \ref{lem:alphaDergammaDer} (1)). 
Let $x$ be an element of $A$. \par
(1) If $x$ is $\delta$-compatible with respect to $\partial$ and $\beta$,
then we have $\gamma(\delta(x))=\delta(\gamma(x))$. 
The converse is also true if $\alpha$ is a non-zero-divisor of $A$.\par
(2) If $x$ is $\delta$-compatible with respect to $\partial$ and $\beta$,
then we have $\partial(\varphi(x))=(\alpha^{p-1}+p\beta)\varphi(\partial(x))$.
The converse is also true if $A$ is $p$-torsion free.
\end{lemma}

\begin{proof}
(1) This follows from the computation given before Definition 
\ref{def:alphaDerivDeltaComp}. \par
(2) We have 
\begin{align*}
\partial(\varphi(x))&=\partial(x^p)+p\partial(\delta(x)),\\
(\alpha^{p-1}+p\beta)\varphi(\partial(x))
&=(\alpha^{p-1}+p\beta)(\partial(x)^p+p\delta(\partial(x)))\\
&=\alpha^{p-1}\partial(x)^p+p\{\beta \partial(x)^p+(\alpha^{p-1}+p\beta)\delta(\partial(x))\}.
\end{align*}
By Lemma \ref{lem:alphaDerivMonom},  
we have $\partial(x^p)-\alpha^{p-1}\partial(x)^p=p\sum_{\nu=1}^{p-1}p^{-1}
\binom p\nu x^{p-\nu}\alpha^{\nu-1}\partial(x)^{\nu}.$
Hence the difference between $\partial(\varphi(x))$ and 
$(\alpha^{p-1}+p\beta)\varphi(\partial(x))$ coincides with 
that of both sides of \eqref{eq:AlphaDerivDerComp} multiplied by $p$.
\end{proof}

\begin{remark}\label{rmk:alphaDerivDeltaBC} 
Let $R$, $A$, $\alpha$, and $\beta$ be the same as in Definition
\ref{def:alphaDerivDeltaComp}, and let $\partial\colon A\to A$ be
an $\alpha$-derivation over $R$ $\delta$-compatible with respect
to $\beta$.\par
(1) Let $I$ be an ideal of $R$ containing $p$, and let
$\hR$ (resp.~$\hA$) be $\varprojlim_n R/I^n$
(resp.~$\varprojlim_nA/I^nA$) equipped with the completion
of the $\delta$-structure of $R$ (resp.~$A$). 
Let $\halpha$ and $\hbeta$ be the images of 
$\alpha$ and $\beta$ in $\hA$. Then 
the $\halpha$-derivation $\hpartial=\varprojlim_n \partial\otimes_RR/I^n
\colon \hA\to \hA$ over $\hR$ (Remark \ref{rmk:alphaDerivBC} (2))
is $\delta$-compatible with respect to $\hbeta$. 
This follows from the equality in $A/I^nA$
$(n\in \N, n>0)$ for $x\in A/I^{n+1}A$:
\begin{equation*}
\partial_n(\delta_{n}(x))
=(\alpha^{p-1}+p\beta)\delta_{n}(\partial_{n+1}(x))
+\beta(\pi_n\circ \partial_{n+1}(x))^p-
\sum_{\nu=1}^{p-1}p^{-1}\binom p \nu
\pi_n(x)^{p-\nu}\alpha^{\nu-1}(\pi_n\circ \partial_{n+1}(x))^{\nu},
\end{equation*}
where $\partial_m\colon A/I^mA\to A/I^mA$ and 
$\delta_{m}\colon A/I^{m+1}A\to A/I^mA$ denote the
maps induced by $\partial$ and $\delta$, and $\pi_m$
denotes the canonical projection $A/I^{m+1}A\to A/I^mA$. \par
(2) Let $\fa$ be a $\delta$-ideal of $A$, 
let $\oA$ be the $\delta$-$R$-algebra $A/\fa$,
and let $\oalpha$ and $\obeta$ be the images of $\alpha$ and $\beta$
in $\oA$. If $\partial(\fa)\subset \fa$, then 
the $\oalpha$-derivation $\opartial\colon \oA\to \oA$ induced by $\partial$
over $R$ 
(Remark \ref{rmk:alphaDerivBC} (3)) is $\delta$-compatible with
respect to $\obeta$. For an element $a$ of $\fa$,
$\partial(a)\in \fa$ implies $\partial(\delta(a))\in \fa$ by the formula
\eqref{eq:AlphaDerivDerComp}. Hence, if $\fa$ is
generated by a subset $\scrS$ of $\fa$ as a $\delta$-ideal of $A$,
then $\partial(\scrS)\subset \fa$ implies 
$\partial(\cup_{n\in \N}\delta^n(\scrS))\subset \fa$, whence 
$\partial(\fa)\subset \fa$. \par
(3) For $x\in A$, $\partial(x)=0$ implies $\partial(\delta(x))=0$.
This implies that the $R$-subalgebra $A^\partial=\{x\in A\,\vert\, \partial(x)=0\}$
of $A$ (Remark \ref{rmk:alphaDerivBC} (4)) is a $\delta$-subalgebra.
\par
(4) Let $A'$ be a $\delta$-$R$-subalgebra of $A$ containing
$\alpha$ and $\beta$. If $\partial(A')\subset A'$, then 
$\partial\vert_{A'}\colon A'\to A'$ is an $\alpha$-derivation over $R$
$\delta$-compatible with respect to $\beta$. 
In general, we see that the $R$-subalgebra
$A''=\{x\in A'\,\vert\,\partial(x)\in A'\}$ (Remark \ref{rmk:alphaDerivBC} (5))
is a $\delta$-$R$-subalgebra of $A'$, i.e., $\delta(A'')\subset A''$.
Hence, if $A'$ is generated by a subset $\scrS$ of $A'$ as a 
$\delta$-$R$-algebra, then 
the assumption $\partial(A')\subset A'$ above follows from 
$\partial(\scrS)\subset A'$. 
\par
(5)  Let $I$ be an ideal of $A$ containing $p$ and satisfying $\alpha\partial(I)\subset I$ 
(e.g.~$\alpha\in I$), let $\hA$ be $\varprojlim_n A/I^n$ equipped with the completion
of the $\delta$-structure of $A$, and let $\halpha$ and $\hbeta$ be the images of $\alpha$ and $\beta$ in $\hA$.
Then we see that the $\halpha$-derivation $\hpartial\colon \hA\to \hA$ over $R$
induced by $\partial$ (Remark \ref{rmk:alphaDerivBC} (6)) is $\delta$-compatible with 
respect to $\hbeta$ similarly to the remark (1) above
by using the homomorphisms $\partial_n\colon A/I^{n+1}\to A/I^n$ $(n\geq 1)$
induced by $\partial$ (Remark \ref{rmk:alphaDerivBC} (6)).
\end{remark}

We can interpret the $\delta$-compatibility of an $\alpha$-derivation
of $A$ over $R$ in terms of the $\delta$-compatibility
of the corresponding section of the twisted extension
$E^{\alpha}(A)$ with respect to a  $\delta$-structure
on $E^{\alpha}(A)$, which we define below.
Let $A$ be a $\delta$-ring, and let $\alpha$ and $\beta$
be elements of $A$ satisfying $\delta(\alpha)=\beta\alpha$.

\begin{lemma}\label{lem:DeltaStrTwExt}
There exists a unique $\delta$-algebra structure on 
$E^{\alpha}(A)$ such that $\iota\colon A\to E^{\alpha}(A)$
is a $\delta$-homomorphism and $\delta(0,1)=(0,\beta)$.
This $\delta$-structure is explicitly given by the following formula.
\begin{equation}\label{eq:DeltaStrTwExt}
\delta(x_0,x_1)=(\delta(x_0),
(\alpha^{p-1}+p\beta)\delta(x_1)+\beta x_1^p-
\sum_{\nu=1}^{p-1}p^{-1}\binom p\nu x_0^{p-\nu}\alpha^{\nu-1}
x_1^{\nu})
\end{equation}
\end{lemma}

\begin{proof}
We identify $E^{\alpha}(A)$ with $A[T]/(T^2-\alpha T)$ by \eqref{eq:TwExtIsom}.
The composition of the ring homomorphism 
$A\to W_2(A); a\mapsto (a,\delta(a))$ with 
the ring homomorphism $W_2(\iota)\colon W_2(A)\to W_2(A[T]/(T^2-\alpha T))$
extends uniquely to a ring homomorphism 
$\widetilde{\epsilon}\colon A[T]\to W_2(A[T]/(T^2-\alpha T))$
sending $T$ to $(T,\beta T)$. 
We see $\widetilde{\epsilon}(T^2)=\widetilde{\epsilon}(\alpha T)$
by the following computation.
\begin{align*}
\widetilde{\epsilon}(T^2)&=
(T,\beta T)^2=(T^2,2T^p\beta T+p(\beta T)^2)
=(\alpha T, 2\alpha^p\beta T+ p\alpha\beta^2 T),\\
\widetilde{\epsilon}(\alpha T)&=(\alpha,\beta\alpha)(T,\beta T)
=(\alpha T,\alpha^p \beta T+\beta \alpha T^p+p\beta\alpha\beta T)
=(\alpha T, 2\alpha^p\beta T+p\alpha\beta^2 T).
\end{align*}
Hence $\widetilde{\epsilon}$ factors through the quotient 
$A[T]/(T^2-\alpha T)$. This implies the first claim.
We obtain the second claim by the following computation in 
$A[T]/(T^2-\alpha T)$ for $x_0, x_1\in A$. 
\begin{align*}
\delta(x_0+x_1T)&=
\delta(x_0)+\delta(x_1 T)-\sum_{\nu=1}^{p-1}p^{-1}\binom p \nu
x_0^{p-\nu}(x_1T)^{\nu},\\
(x_1T)^{\nu}&=\alpha^{\nu-1}x_1^{\nu} T,\\
\delta(x_1T)&=\delta(x_1)T^p+x_1^p\beta T+p\delta(x_1)\beta T
=\{\alpha^{p-1}\delta(x_1)+\beta x_1^p+p\beta \delta(x_1)\}T.
\end{align*}
\end{proof}

We define the twisted $\delta$-extension
$E_{\delta}^{\alpha,\beta}(A)$
of $A$ by $A$ with respect to $(\alpha,\beta)$
to be $E^{\alpha}(A)$ equipped with the unique  $\delta$-structure
in Lemma \ref{lem:DeltaStrTwExt}.
The homomorphism $\iota\colon A\to E_{\delta}^{\alpha,\beta}(A)$
is a $\delta$-homomorphism by definition. Since 
$\pi_0, \pi_{\alpha}\colon E_{\delta}^{\alpha,\beta}(A)\to A$
are homomorphisms of $A$-algebras, and the $A$-algebra
$E_{\delta}^{\alpha,\beta}(A)$ is generated by the element $(0,1)$
by \eqref{eq:TwExtIsom}, we see that $\pi_0$ and $\pi_{\alpha}$
are also $\delta$-homomorphisms by the following simple computation:
$\pi_0(\delta(0,1))=\pi_0(0,\beta)=0$,
$\delta(\pi_0(0,1))=\delta(0)=0$, 
$\pi_{\alpha}(\delta(0,1))=\pi_{\alpha}(0,\beta)=\beta\alpha$,
$\delta(\pi_{\alpha}(0,1))=\delta(\alpha)=\beta\alpha$.

The construction of $E_{\delta}^{\alpha,\beta}(A)$ is obviously functorial in $(A,\alpha,\beta)$
as follows. Let $f\colon A\to A'$ be a $\delta$-homomorphism of 
$\delta$-rings, and put $\alpha'=f(\alpha)$ and $\beta'=f(\beta)$,
which satisfy $\delta(\alpha')=\alpha'\beta'$. Then the 
homomorphism of rings $E^{\alpha}(f)\colon E^{\alpha}(A)\to E^{\alpha'}(A');
(x_0,x_1)\mapsto (f(x_0),f(x_1))$ induced by $f$
defines a $\delta$-homomorphism
$E^{\alpha,\beta}_{\delta}(f)\colon 
E_{\delta}^{\alpha,\beta}(A)\to E_{\delta}^{\alpha',\beta'}(A')$.

We define the endomorphism $\varphi$ of $E_{\delta}^{\alpha,\beta}(A)$
by $\varphi(x)=x^p+p\delta(x)$. Recall that we have defined an $A$-linear map
$\scrD\colon E^{\alpha}(A)\to A; (x_0,x_1)=x_1$.

\begin{lemma}\label{lem:DeltaTwExtComp}
The following equality of homomorphisms $E_{\delta}^{\alpha,\beta}(A)\to A$ holds.
$$\scrD\circ\varphi=(\alpha^{p-1}+p\beta)\varphi\circ\scrD.$$
\end{lemma}

\begin{proof}
Let $x\in E_{\delta}^{\alpha,\beta}(A)$. By Lemma 
\ref{lem:UnivAlphaDerivMonom} and \eqref{eq:DeltaStrTwExt}, we have 
\begin{align*}
\scrD(\varphi(x))&=
\scrD(x^p)+p\scrD(\delta(x))\\
&=\sum_{\nu=1}^p\binom p \nu \pi_0(x)^{p-\nu}\alpha^{\nu-1}\scrD(x)^{\nu}\\
&+p\{(\alpha^{p-1}+p\beta)\delta(\scrD(x))+\beta\scrD(x)^p
-\sum_{\nu=1}^{p-1}p^{-1}\binom p\nu \pi_0(x)^{p-\nu}\alpha^{\nu-1}
\scrD(x)^{\nu}\}\\
&=(\alpha^{p-1}+p\beta)\{\scrD(x)^p+p\delta(\scrD(x))\}\\
&=(\alpha^{p-1}+p\beta)\varphi(\scrD(x)).
\end{align*}
\end{proof}

Now we can interpret the $\delta$-compatibility
of an $\alpha$-derivation in terms of $E_{\delta}^{\alpha,\beta}(A)$.

\begin{proposition}\label{prop:DeltaAlphaDerivInt}
Let $R$, $A$, $\alpha$, and $\beta$ be as in Definition \ref{def:alphaDerivDeltaComp},
let $\partial\colon A\to A$ be an $\alpha$-derivation of $A$ over $R$,
and let $s\colon A\to  E_{\delta}^{\alpha,\beta}(A)$ be the $R$-homomorphism
corresponding to $\partial$ by Proposition \ref{prop:AlphaDerivInterpret}.\par
(1) An element $x$ of $A$ is $\delta$-compatible with respect 
to $\partial$ and $\beta$ if and only if $\delta(s(x))=s(\delta(x))$.\par
(2) The $\alpha$-derivation $\partial$ of $A$ over $R$ is 
$\delta$-compatible with respect to $\beta$ if and only if
$s$ is a $\delta$-homomorphism.
\end{proposition}

\begin{proof}
For $x\in A$, we have $s(\delta(x))=(\delta(x),\partial(\delta(x)))$, and
$$\delta(s(x))
=\delta(x,\partial(x))=
(\delta(x), (\alpha^{p-1}+p\beta)\delta(\partial(x))+\beta\partial(x)^p
-\sum_{\nu=1}^{p-1}p^{-1}\binom p\nu x^{p-\nu}\alpha^{\nu-1}
\partial(x)^{\nu})$$
by \eqref{eq:DeltaStrTwExt}. This implies (1). The claim (2) immediately follows
from the claim (1).
\end{proof}

\begin{lemma}\label{lem:DerivDeltaCompDeltaGen}
Let $R$, $A$, $\alpha$, and $\beta$ be as in 
Definition \ref{def:alphaDerivDeltaComp}, and
let $\partial\colon A\to A$ be an $\alpha$-derivation 
of $A$ over $R$. 
Let $I$ be an ideal of $A$ containing $p$ such that $A$ is $I$-adically
separated and $\alpha\partial(I)\subset I$ (e.g.~$I$ is generated by elements
of $R$, or $\alpha\in I$), and let $\scrS$ be a subset of $A$
such that the $R$-subalgebra $R[\scrS]$
of $A$ is $I$-adically dense in $A$. 
Then $\partial$ is $\delta$-compatible with respect
to $\beta$ if and only if every element of $\scrS$ is 
$\delta$-compatible with respect to $\partial$ and $\beta$.
\end{lemma}

\begin{proof}
Put $E=E_{\delta}^{\alpha,\beta}(A)$.
Let $s\colon A\to E$ be the 
$R$-homomorphism corresponding to $\partial$
by Proposition \ref{prop:AlphaDerivInterpret}. 
It suffices to prove that the left diagram below is commutative,
where the vertical maps are defined by the $\delta$-structures
of $A$ and $E$. 
\begin{equation*}
\xymatrix@C=50pt{
A\ar[d]\ar[r]^s& E\ar[d]\\
W_2(A)\ar[r]^{W_2(s)}& W_2(E)
}\qquad
\xymatrix{
A/I^{n+1}\ar[d]\ar[r]& E/s(I)^{n+1}E\ar[d]\\
W_2(A/I^n)\ar[r]& W_2(E/s(I)^{n}E)
}
\end{equation*}
By $p\in I$, 
the left diagram induces the right one for each integer $n\geq 1$.
By Proposition \ref{prop:DeltaAlphaDerivInt} (1), 
the left diagram is commutative after composed with 
the inclusion map $R[\scrS]\hookrightarrow A$.
Since the $R$-algebra $A/I^{n+1}$ is generated the image
of $\scrS$, this implies that the right diagram is commutative
for every integer $n\geq 1$. By the assumption $\alpha\partial(I)\subset I$
and Remark \ref{rmk:alphaDerivBC} (6), we have $s(I^{n})E\subset I^{n-1}E$.
Since $A$ is assumed to be $I$-adically separated, this implies 
$\cap_{n\geq 1}s(I^n)E=0$. Therefore the left diagram is also commutative.
\end{proof}

\begin{lemma}\label{lem:AlphaCompDeltaCompBC}
Let $R$, $A$, $\alpha$, and $\beta$ be as in Definition \ref{def:alphaDerivDeltaComp},
let $R'$ be a $\delta$-$R$-algebra, let $A'$ be
$A\otimes_RR'$ equipped with the induced $\delta$-structure,
and let $\alpha'$ and $\beta'$ be the images of $\alpha$
and $\beta$ in $A'$, respectively. Let $\partial$ be an $\alpha$-derivation
of $A$ over $R$ $\delta$-compatible with respect to $\beta$.
Then the $\alpha'$-derivation $\partial'=\id_{R'}\otimes \partial\colon 
A'\to A'$ (Remark \ref{rmk:alphaDerivBC} (1)) is $\delta$-compatible with respect
to $\beta'$.
\end{lemma}

\begin{proof}
Put $E=E_{\delta}^{\alpha,\beta}(A)$ and $E'=E_{\delta}^{\alpha',\beta'}(A')$,
and let $s$ (resp.~$s'$) be the $\delta$-$R$-homomorphism 
$A\to E$ (resp.~the $R'$-homomorphism $A'\to E'$) 
corresponding to $\partial$ (resp.~$\partial'$) by
Proposition \ref{prop:DeltaAlphaDerivInt}
 (resp.~\ref{prop:AlphaDerivInterpret}). The $\delta$-homomorphism 
$E\to E'$ induced by the $\delta$-homomorphism $A\to A'$
extends to a $\delta$-$R'$-isomorphism $E\otimes_RR'
\xrightarrow{\cong} E'$. We see that the
$R'$-homomorphism $s'$ coincides with the composition 
of the $\delta$-$R'$-homomorphism 
$s\otimes \id_{R'}\colon A'=A\otimes_RR'
\to E\otimes_RR'$ and the above $\delta$-$R'$-isomorphism.
This completes the proof by Proposition \ref{prop:DeltaAlphaDerivInt} (2).
\end{proof}

\begin{proposition}\label{prop:TwistDerDeltaEtaleExt}
Let $R$ be a $\delta$-ring, let $A$ be a $\delta$-$R$-algebra,
let $I$ be an ideal of $A$ containing $p$,
and let $f\colon A\to A'$ be a homomorphism of $\delta$-$R$-algebars
such that $A'$ is $I$-adically separated and 
the homomorphism of rings $f$ is $I$-adically \'etale
(Definition \ref{def:formallyflat} (1)).
Let $\alpha$ and $\beta$ be elements of $A$ such that 
$\delta(\alpha)=\alpha\beta$, and put 
$\alpha'=f(\alpha)$ and $\beta'=f(\beta)$, which satisfy
$\delta(\alpha')=\alpha'\beta'$. Let 
$\partial\in \Der^{\alpha}_R(A)$, 
and let $\partial'\in \Der^{\alpha}_R(A')$ be an
extension of $\partial$ (Definition \ref{def:alphaDerivation}
 (2)). Suppose that $\partial$ satisfies $\alpha\partial(I)\subset I$
 (e.g.~$I$ is generated by elements of $R$, or $\alpha\in I$). 
If $\partial$ is $\delta$-compatible with respect to $\beta$, then 
$\partial'$ is $\delta$-compatible with respect to $\beta'$.
\end{proposition}

\begin{proof}
Put $E=E_{\delta}^{\alpha,\beta}(A)$ and $E'=E_{\delta}^{\alpha',\beta'}(A')$.
Then we have a diagram below, where the
left four horizontal maps are defied by the given $\delta$-structures,
the right horizontal ones are the projection to the first component,
and the vertical ones are defined by $\partial$ and $\partial'$, 
and the functoriality of $W_2(-)$. 
\begin{equation*}
\xymatrix@R=10pt{
&A'\ar[rr]\ar'[d][dd]&&W_2(A')\ar[rr]\ar[dd] &&A'\ar[dd]\\
A\ar[ur]\ar[rr]\ar[dd]&&W_2(A)\ar[ur]\ar[dd]&&&\\
&E'\ar'[r][rr]&&W_2(E')\ar[rr] &&E'\\
E\ar[rr]\ar[ur]&&W_2(E)\ar[ur]&&&}
\end{equation*}
By assumption, the diagram is commutative except for
the left square at the back of the diagram.
Put $A_n=A/I^n$, $A'_n=A'/I^nA'$, $E_n=E/I^nE$, and $E'_n=E'/I^nE'$, 
where we regard $E$ (resp.~$E'$) as an algebra over $A$ (resp.~$A'$) by 
the homomorphism $a\mapsto (a,0)$. By
$p\in I$, $\alpha\partial(I)\subset I$, which implies $\alpha'\partial'(IA')\subset IA'$,
and Remark \ref{rmk:alphaDerivBC} (6), the diagram above induces a diagram
\begin{equation*}
\xymatrix@R=10pt{
&A'_{n+2}\ar[rr]\ar'[d][dd]&&W_2(A'_{n+1})\ar[rr]\ar[dd] &&A'_{n+1}\ar[dd]\\
A_{n+2}\ar[ur]\ar[rr]\ar[dd]&&W_2(A_{n+1})\ar[ur]\ar[dd]&&&\\
&E'_{n+1}\ar'[r][rr]&&W_2(E'_{n})\ar[rr] &&E'_n\\
E_{n+1}\ar[rr]\ar[ur]&&W_2(E_{n})\ar[ur]&&&}
\end{equation*}
for each positive integer $n$. Since $A'$ is $I$-adically separated,
it suffices to prove that 
the two compositions 
$A'_{n+2}\to W_2(A'_{n+1})\to W_2(E'_n)$ and
$A_{n+2}'\to E_{n+1}'\to W_2(E'_n)$
coincide for every $n\geq 1$.
By the commutativity of the diagram except the back left square,
we see that  the compositions of the above two maps
with the \'etale map $A_{n+2}\to A_{n+2}'$
(resp.~the map $W_2(E'_n)\to E'_n$, whose kernel is nilpotent by the assumption
$p\in I$) coincide. Hence the two maps in question are the same.
\end{proof}

We obtain the following corollary by Proposition \ref{prop:TwistDerEtaleExt}.

\begin{corollary}\label{cor:DeltaCompDerivEtExt}
Let $R$ be a $\delta$-ring, let $A$ be a $\delta$-$R$-algebra,
let $I$ be an ideal of $A$ containing $p$,
and let $f\colon A\to A'$ be a homomorphism of $\delta$-$R$-algebars
such that $A'$ is $I$-adically complete and separated, and that 
the homomorphism of rings $f$ is $I$-adically \'etale
(Definition \ref{def:formallyflat} (1)).
Let $\alpha$ and $\beta$ be elements of $A$ such that 
$\delta(\alpha)=\alpha\beta$, and put 
$\alpha'=f(\alpha)$ and $\beta'=f(\beta)$, which satisfy
$\delta(\alpha')=\alpha'\beta'$. 
Assume that $\alpha$ is $I$-adically nilpotent.
Then, every $\partial\in \Der_{R,\delta}^{\alpha,\beta}(A)$ extends 
uniquely to $\partial'\in \Der_{R,\delta}^{\alpha',\beta'}(A')$ along $f$.
\end{corollary}

\begin{proof}Since $\alpha$ is $I$-adically nilpotent,
the $I$-adic topology of an $A$-module is the same as the
$(\alpha A+I)$-adic topology. Therefore we may replace
$I$ by $\alpha A+I$ and assume $\alpha\in I$.
By Proposition \ref{prop:TwistDerEtaleExt}, there exists a unique extension 
$\partial'\in \Der^{\alpha'}_{R'}(A')$ of $\partial$ along $f$. 
By Proposition \ref{prop:TwistDerDeltaEtaleExt}, 
we see that $\partial'$ is $\delta$-compatible
with respect to $\beta'$.
\end{proof}

\begin{proposition}\label{prop:TwistDerivDeltaCompCoord}
Let $R$ be a $\delta$-ring, let $I$ be an ideal of $R$
containing $p$, and let $A$ be a $\delta$-$R$-algebra $I$-adically separated
such that $A$ is $I$-adically smooth over $R$
(Definition \ref{def:formallyflat} (1)).
Suppose that we are given $I$-adic coordinates $T_1,\ldots, T_d$ of $A$
over $R$ (Definition \ref{def:formallyflat} (2)). 
Let $\alpha$, $\beta$
be elements of $A$ such that $\delta(\alpha)=\alpha\beta$,
and let $\partial\in \Der_{R}^{\alpha}(A)$.
Then $\partial$ is $\delta$-compatible with respect to $\beta$
if and only if $T_i$ is $\delta$-compatible with respect to $\partial$ and
$\beta$ for every $i\in \N\cap [1,d]$. 
\end{proposition}

\begin{proof}
The necessity is trivial. Let us prove the sufficiency.
Put $A_n=A/I^nA$, $A'_n=A'/I^nA'$, 
$E=E_{\delta}^{\alpha,\beta}(A)$, and $E_n=E/I^nE$.
Let $s\colon A\to E$ be the $R$-homomorphism corresponding  
to $\partial$ by Proposition \ref{prop:AlphaDerivInterpret}, 
and let $w_A$ (resp.~$w_E$) be the homomorphism
$A\to W_2(A)$ (resp.~$E\to W_2(E)$)$\colon x\mapsto (x,\delta(x))$
corresponding to the $\delta$-structure on $A$ (resp.~$E$).
Then it suffices to prove that the following diagram is commutative.
\begin{equation*}
\xymatrix@C=50pt@R=15pt
{A\ar[r]^(.4){w_A}\ar[d]_s&W_2(A)\ar[d]^{W_2(s)}\\
E\ar[r]^(.4){w_E}&W_2(E)
}
\end{equation*}
Since $p\in I$, this diagram induces a diagram
\begin{equation*}
\xymatrix@R=15pt
{A_{n+1}\ar[r]\ar[d]&W_2(A_n)\ar[d]\\
E_{n+1}\ar[r]&W_2(E_n)
}
\end{equation*}
for each positive integer $n$. It becomes
commutative after composed with the \'etale $R_{n+1}$-homomorphism
$R_{n+1}[S_1,\ldots, S_n]\to A_{n+1}$ sending 
$S_i$ to $T_i$ (resp.~the projection to the first component
$W_2(E_n)\to E_n$, whose kernel is nilpotent by $p\in I$) by
the assumption on $T_i$,  Proposition 
\ref{prop:DeltaAlphaDerivInt} (1),
and $\delta(s(x))=\delta(x,0)=(\delta(x),0)=s(\delta(x))$ for 
$x\in R$ 
(resp.~the definition of $w_E$ and $w_A$).
Hence the diagram is commutative for every 
positive integer $n$. This implies 
$W_2(s)\circ w_A=w_E\circ s$ because
$A$ is $I$-adically separated.
\end{proof}

We obtain the following corollary by Proposition \ref{prop:TwistDerivCoord}.
\begin{corollary}\label{cor:SmoothDeltCompDeriv}
Let $R$, $I$, $A$, $T_1,\ldots, T_d$, $\alpha$, and $\beta$ be
the same as in Proposition \ref{prop:TwistDerivDeltaCompCoord}. 
We further assume that 
$A$ is $I$-adically complete and separated, $\alpha$
is $I$-adically nilpotent, and $\delta(T_i)=0$ for every 
$i\in \N\cap [1,d]$.  Let $a_1,\ldots, a_d$ be elements
of $A$ satisfying
$$(\alpha^{p-1}+p\beta)\delta(a_i)+\beta a_i^p
-\sum_{\nu=1}^{p-1}p^{-1}\binom p\nu T_i^{p-\nu}\alpha^{\nu-1}
a_i^{\nu}=0.$$
Then there exists a unique $\partial\in \Der^{\alpha,\beta}_{R,\delta}(A)$
satisfying $\partial(T_i)=a_i$ for every $i\in \N\cap [1,d]$.
\end{corollary}

An $\alpha$-derivation extends to 
a $\delta$-compatible $\alpha$-derivation on the $\delta$-envelope
uniquely as follows.

\begin{proposition} \label{eq:DeltaEnvTwDerivExt}
(1) 
Let $R$ be a $\delta$-ring, let $A$ be an $R$-algebra,
and let $A_{\delta}$ be the $\delta$-envelope of $A$ over $R$.
Let $\alpha$  be an element of $A$,  and let $\beta$ be an element of $A_{\delta}$ 
satisfying $\delta(\alpha)=\alpha\beta$.
Then any $\alpha$-derivation $\partial\colon A\to A$ of $A$ over $R$
extends uniquely to an $\alpha$-derivation 
$\partial_{\delta}\colon A_{\delta}
\to A_{\delta}$ of $A_{\delta}$ over $R$
$\delta$-compatible with respect to $\beta$.\par
(2) Let $R$ be a $\delta$-ring, let $A$ be a $\delta$-$R$-algebra,
let $B$ be an $A$-algebra, and let $B_{\delta}$
be the $\delta$-envelope of $B$ over $A$. Let 
$\alpha$ and $\beta$ be elements of $A$ satisfying 
$\delta(\alpha)=\alpha\beta$, and let $\partial_A\colon A\to A$
be an $\alpha$-derivation of $A$ over $R$ $\delta$-compatible
with respect to $\beta$. Then any $\alpha$-derivation 
$\partial_B\colon B\to B$ of $B$ over $R$
compatible with $\partial_A$ extends uniquely
to an $\alpha$-derivation $B_{\delta}\to B_{\delta}$
of $B_{\delta}$ over $R$ $\delta$-compatible with
respect to $\beta$. 
\end{proposition}

\begin{proof} (1)
The composition $R\to A\xrightarrow{\iota}E^{\alpha}(A)
\to E_{\delta}^{\alpha,\beta}(A_{\delta})$ coincides
with the composition $R\to A_{\delta}\xrightarrow{\iota}
E_{\delta}^{\alpha,\beta}(A_{\delta})$, which is
a $\delta$-homomorphism.
Hence the $R$-homomorphism $s\colon A\to E^{\alpha}(A)
\colon x\mapsto (x,\partial(x))$ (Proposition \ref{prop:AlphaDerivInterpret})
extends uniquely to a $\delta$-homomorphism 
$s_{\delta}\colon A_{\delta}\to E^{\alpha,\beta}_{\delta}(A_{\delta})$
by the universality of $\delta$-envelopes of $R$-algebras,
and the composition $\pi_{0}\circ s_{\delta}\colon 
A_{\delta}\to A_{\delta}$ is the unique $\delta$-homomorphism
extending $\id_A\colon A\to A$, i.e., $\id_{A_{\delta}}$. 
Hence the claim follows from Proposition \ref{prop:DeltaAlphaDerivInt} (2). \par
(2) Let $s_A$ (resp.~$s_B$) be the $\delta$-$R$-homomorphism
$A\to E_{\delta}^{\alpha,\beta}(A)$
(resp.~the $R$-homomorphism $B\to E^{\alpha}(B)$) corresponding
to $\partial_A$ (resp.~$\partial_B$) by Proposition 
\ref{prop:DeltaAlphaDerivInt} (2)
(resp.~Proposition \ref{prop:AlphaDerivInterpret}). Since $\partial_B$ is compatible with $\partial_A$,
the composition $A\to B\xrightarrow{s_B}E^{\alpha}(B)\to
 E_{\delta}^{\alpha,\beta}(B_{\delta})$ coincides with the composition 
$A\xrightarrow{s_A}E_{\delta}^{\alpha,\beta}(A)\to E^{\alpha}(B)
\to E_{\delta}^{\alpha,\beta}(B_{\delta})$. The latter composition 
is a $\delta$-homomorphism because $A\to B_{\delta}$
is a $\delta$-homomorphism. Hence, by the universality 
of $\delta$-envelopes of $A$-algebras, 
the homomorphism $s_B$ extends uniquely to a
$\delta$-homomorphism $s_{B_{\delta}}\colon 
B_{\delta}\to E_{\delta}^{\alpha,\beta}(B_{\delta})$.
The composition $\pi_0\circ s_{B_{\delta}}\colon 
B_{\delta}\to B_{\delta}$ is the unique extension 
of $\pi_0\circ s_B=\id_B\colon B\to B$
to a $\delta$-homomorphism, which  is $\id_{B_{\delta}}$.
Thus the claim follows from Proposition \ref{prop:DeltaAlphaDerivInt} (2).
\end{proof}

A $\delta$-compatible $\alpha$-derivation extends uniquely to
the bounded prismatic envelope as follows.
\begin{proposition}\label{prop:TwDerivPrismExt}
Let $(R,I)$ be a bounded prism, let $(A,J)$ be a
$\delta$-pair over $(R,I)$, 
and let $g\colon (A,J)\to (D,ID)$ be
the bounded prismatic envelope of $(A,J)$ over $(R,I)$ 
(assuming it exists) (Definition \ref{def:bddPrsimEnv}).
Let $\alpha$ and $\beta$ be elements of $A$ satisfying
$\delta(\alpha)=\alpha\beta$. Then 
$\partial\in \Der_{R,\delta}^{\alpha,\beta}(A)$ satisfying
$\partial(J)\subset J$ extends uniquely to 
$\partial'\in \Der_{R,\delta}^{\alpha,\beta}(D)$. 
\end{proposition}

\begin{proof}
Let $s$ be a $\delta$-$R$-homomorphism $A\to E_{\delta}^{\alpha,\beta}(A)$
corresponding to $\partial$ (Proposition \ref{prop:DeltaAlphaDerivInt} (2)).
Then the composition 
$A\xrightarrow{s}E_{\delta}^{\alpha,\beta}(A)
\xrightarrow{E^{\alpha,\beta}_{\delta}(g)} E_{\gamma}^{\alpha,\beta}(D)$
is a $\delta$-$R$-homomorphism, which sends $J$ into $ID$ by 
the assumption $\partial(J)\subset J$, and the $\delta$-pair
$(E_{\delta}^{\alpha,\beta}(D),IE_{\delta}^{\alpha,\beta}(D))$
is a bounded prism since $E_{\delta}^{\alpha,\beta}(D)\cong 
D\oplus D$ as a $D$-module. Hence the composition 
above extends uniquely to a morphism of bounded prisms
$s_D\colon (D,ID)\to 
(E_{\delta}^{\alpha,\beta}(D),IE_{\delta}^{\alpha,\beta}(D))$ by the universality
of the  bounded prismatic envelope $g\colon (A,J)\to (D,ID)$. 
By Proposition \ref{prop:DeltaAlphaDerivInt} (2), 
it remains to show that the composition  
$D\xrightarrow{s_D}E_{\delta}^{\alpha,\beta}(D)
\xrightarrow{\pi_0}D$ is the identity. We have a commutative diagram
of $\delta$-pairs over $(R,I)$.
\begin{equation*}
\xymatrix{
(A,J)\ar[r]^(.3)s\ar[d]_g&(E_{\delta}^{\alpha,\beta}(A), JE_{\delta}^{\alpha,\beta}(A))\ar[r]^(.65){\pi_0}
\ar[d]^{E_{\delta}^{\alpha,\beta}(g)}&(A,J)\ar[d]^g\\
(D,ID)\ar[r]^(.35){s_D} &(E_{\delta}^{\alpha,\beta}(D),IE_{\delta}^{\alpha,\beta}(D))
\ar[r]^(.65){\pi_0}&(D,ID),
}
\end{equation*} 
and the composition of the upper two homomorphisms is the
identity. Hence the composition of the lower ones is also 
the identity by the universality of $g$.
\end{proof}

We have the following variant of Lemma \ref{lem:alphaDerivEq}. 

\begin{lemma}\label{lem:alphaDerivDeltaEq}
Let $f\colon R\to R'$ be a $\delta$-homomorphism of $\delta$-rings,
and let $g\colon A\to A'$ be a $\delta$-homomorphism over $f$
from a $\delta$-$R$-algebra to a $\delta$-$R'$-algebra.
Let $\alpha$ and $\beta$ be elements of $A$ satisfying
$\delta(\alpha)=\alpha\beta$, and put 
$\alpha'=g(\alpha)$ and $\beta'=g(\beta)$. 
Let $I$ be an ideal of $A$ such that $A'$ is $I$-adically
separated, and let $\scrS$ be a subset of $A$ such that 
the $\delta$-$R$-subalgebra $R[\scrS]_{\delta}$
of $A$ generated by $\scrS$ is dense with respect to the
$I$-adic topology. Then, for 
$\partial\in \Der_{R,\delta}^{\alpha,\beta}(A)$ and 
$\partial'\in \Der_{R',\delta}^{\alpha',\beta'}(A)$
satisfying $\alpha\partial(I)\subset I$ and
$\alpha'\partial'(IA')\subset IA'$
(e.g.~$I$ is generated by elements of $R$, or 
$\alpha\in I$), $g$ is compatible with $\partial$ and $\partial'$
(Definition \ref{def:alphaDerivation} (2))
if and only if 
$\partial'\circ g(s)=g\circ \partial(s)$ for all $s\in \scrS$.
\end{lemma}

\begin{proof}
One can prove the lemma in the same way as 
Lemma \ref{lem:alphaDerivEq} by using Proposition \ref{prop:DeltaAlphaDerivInt}
instead of Proposition \ref{prop:AlphaDerivInterpret} as follows. 
The necessity is trivial. We prove the sufficiency.
We may assume $R'=R$ and $f=\id_R$.
Put $E=E_{\delta}^{\alpha,\beta}(A)$ and $E'=E_{\delta}^{\alpha',\beta'}(A')$.
Let $h\colon E\to E'$ be the $\delta$-homomorphism
induced by $g$, and let $s$ (resp.~$s'$)
be the $\delta$-$R$-homomorphism $A\to E$ (resp.~$A'\to E'$) corresponding to $\partial$ (resp.~$\partial'$)
by Proposition \ref{prop:DeltaAlphaDerivInt} (2).
By assumption, the $\delta$-$R$-homomorphisms
$h\circ s$ and $s'\circ g$ coincide on the $\delta$-$R$-subalgebra
$R[\scrS]_{\delta}$ of $A$ generated by $\scrS$. By 
$\alpha\partial(I)\subset I$, $\alpha'\partial'(IA')\subset IA'$,
and Remark \ref{rmk:alphaDerivBC} (6), $s$ and $s'$ induce $R$-homomorphisms
$s_n\colon A/I^{n+1}\to E/I^nE$ and $s'_n\colon A'/I^{n+1}A'\to E'/I^nE'$,
respectively. Let $g_n$ and $h_n$ denote the reduction modulo $I^n$
of $g$ and $h$. Then we have $h_n\circ s_n=s'_n\circ g_{n+1}$
since the composition $R[\scrS]_{\delta}\to A\to A_{n+1}$ is surjective.
Since $E'$ is $I$-adically separated,
this implies $h\circ s=s'\circ g$. 
\end{proof}

\begin{remark}
One can derive Lemma \ref{lem:alphaDerivDeltaEq} from Lemma \ref{lem:alphaDerivEq}
by using the fact:
$\partial'\circ g(x)=g\circ\partial(x)$ implies $\partial'\circ g(\delta(x))=g\circ\partial(\delta(x))$, which 
follows from the description of $\partial'(\delta(g(x))$
(resp.~$\partial(\delta(x))$) in terms of $g(x)$, 
$\delta(\partial'(g(x)))$ and $\partial'(g(x))$
(resp.~$x$, $\delta(\partial(x))$ and $\partial(x)$).
\end{remark}

\begin{lemma}
Let $(R,I)$ be a bounded prism, let $(A,J)$ be a $\delta$-pair
over $(R,I)$, and let $f\colon (A,J)\to (D,ID)$ be the 
bounded prismatic envelope of $(A,J)$ over $(R,I)$ (assuming it exists). 
Let $\alpha$ and $\beta$ be elements of $A$ satisfying
$\delta(\alpha)=\alpha\beta$. 
Let $g\colon (D,ID)\to (B,IB)$ be a morphism of 
bounded prisms over $(R,I)$, and let $\partial_A$,
$\partial_D$, and $\partial_B$ be $\alpha$-derivations
over $R$ on $A$, $D$, and $B$, respectively, 
$\delta$-compatible with respect to $\beta$.
If $\partial_B$ and $\partial_D$ are extensions of $\partial_A$
along $g\circ f$ and $f$, respectively, then 
$\partial_B$ is an extension of $\partial_D$ along $g$. 
\end{lemma}

\begin{proof}
Let $s_B$ and $s_D$ be the $\delta$-$R$-homomorphisms
$B\to E_{\delta}^{\alpha,\beta}(B)$ and $D\to E_{\delta}^{\alpha,\beta}(D)$
corresponding to $\partial_B$ and $\partial_D$, respectively,
by Proposition \ref{prop:DeltaAlphaDerivInt} (2). It suffices to show
that the two morphisms of bounded prisms over $(R,I)$
$E_{\delta}^{\alpha,\beta}(g)\circ s_D\colon 
D\to E_{\delta}^{\alpha,\beta}(D)\to E_{\delta}^{\alpha,\beta}(B)$
and $s_B\circ g\colon D\to B\to E_{\delta}^{\alpha,\beta}(B)$
are the same. By the universal property of the bounded
prismatic envelope $f$, this follows from the fact that
the two morphisms coincide after composing with $f$
by the assumption on $\partial_A$, $\partial_D$, and $\partial_B$.
\end{proof}

Finally we discuss the commutativity of two twisted
derivations compatible with the given $\delta$-structure. 
Let $A$ be a $\delta$-ring, and let $\alpha_1$,
$\alpha_2$, $\beta_1$, and $\beta_2$ be elements
of $A$ satisfying $\delta(\alpha_i)=\alpha_i\beta_i$
$(i=1,2)$. Then the isomorphism \eqref{eq:TwExtComp}
\begin{equation}\label{eq:TwExtDeltacomp}
E_{\delta}^{\alpha_2,\beta_2}(E_{\delta}^{\alpha_1,\beta_1}(A))
\cong E_{\delta}^{\alpha_1,\beta_1}(E_{\delta}^{\alpha_2,\beta_2}(A))
\end{equation}
is compatible with the $\delta$-structures because both sides
are $\delta$-$A$-algebras and satisfy $\delta(T_1)=\beta_1T_1$
and $\delta(T_2)=\beta_2T_2$ under the presentation
\eqref{eq:TwExtCompDescription}. We identify the two
$\delta$-$A$-algebras by \eqref{eq:TwExtDeltacomp}. 

\begin{lemma}\label{lem:TwDerivCommGen}
Let $R$ be a $\delta$-ring, let $A$ be a $\delta$-$R$-algebra,
and let $\alpha_1$, $\beta_1$, $\alpha_2$, and $\beta_2$ be elements
of $A$ satisfying $\delta(\alpha_i)=\alpha_i\beta_i$ $(i=1,2)$.
Let $I$ be an ideal of $A$ such that $A$ is $I$-adically separated,
and let $\scrS$ be a subset of $A$ such that the $\delta$-$R$-subalgebra
$R[\scrS]_{\delta}$ of $A$ generated by $\scrS$ is dense with
respect to the $I$-adic topology. Then, for $\partial_i\in \Der_{R,\delta}^{\alpha_i,\beta_i}(A)$ $(i=1,2)$
satisfying $\partial_i(\alpha_j)=\partial_i(\beta_j)=0$ for 
$(i,j)=(1,2), (2,1)$ and $\alpha_i\partial_i(I)\subset I$ for $i=1,2$ (e.g.~$I$ is generated
by elements of $R$, or $\alpha_1,\alpha_2\in I$), we have 
$\partial_1\circ\partial_2=\partial_2\circ\partial_1$ if 
$\partial_1\circ \partial_2(s)=\partial_2\circ \partial_1(s)$ for
all $s\in \scrS$. 
\end{lemma}

\begin{proof}
For $i\in \{1,2\}$, let $s_i$ be the $\delta$-$R$-homomorphism 
$A\to E_{\delta}^{\alpha_i,\beta_i}(A)$ corresponding to 
$\partial_i$ by Proposition \ref{prop:DeltaAlphaDerivInt} (2).
We have $s_i(\alpha_j)=\alpha_j$ and $s_i(\beta_j)=\beta_j$
for $(i,j)=(1,2), (2,1)$ by the vanishing of $\partial_i$
on $\alpha_j$ and $\beta_j$.
Let us consider the following two compositions of 
$\delta$-$R$-homomorphisms 
\begin{align*}
&A\xrightarrow{s_1} E_{\delta}^{\alpha_1,\beta_1}(A)
\xrightarrow{E_{\delta}^{\alpha_1,\beta_1}(s_2)}
 E_{\delta}^{s_2(\alpha_1),s_2(\beta_1)}(E_{\delta}^{\alpha_2,\beta_2}(A))
= E_{\delta}^{\alpha_1,\beta_1}(E_{\delta}^{\alpha_2,\beta_2}(A)),\\
 &A\xrightarrow{s_2} E_{\delta}^{\alpha_2,\beta_2}(A)
 \xrightarrow{E_{\delta}^{\alpha_2,\beta_2}(s_1)}
 E_{\delta}^{s_1(\alpha_2),s_1(\beta_2)}(E_{\delta}^{\alpha_1,\beta_1}(A))
=E_{\delta}^{\alpha_2,\beta_2}(E_{\delta}^{\alpha_1,\beta_1}(A)).
\end{align*}
By Lemma \ref{lem:TwDerivSectionCommutativity}  (1) 
and the assumption $\partial_1\circ\partial_2(s)=
\partial_2\circ\partial_1(s)$ $(s\in \scrS)$, these $\delta$-$R$-homomorphisms
coincide on $\scrS$ and hence on $R[\scrS]_{\delta}$.
By Remark \ref{rmk:alphaDerivBC} (6), we have $s_i(I^{n+1})\subset I^nE^{\alpha_i}(A)$
$(i=1,2, n\in \N)$, 
which imply that the two compositions above are continuous with respect
to the $I$-adic topology. 
Since $A$ is $I$-adically separated, we see that they are the same.
Hence $\partial_1\circ\partial_2=\partial_2\circ\partial_1$
by Lemma \ref{lem:TwDerivSectionCommutativity} (2).
\end{proof}

\begin{lemma}\label{lem:TwDervPrismExtComm}
Let $(R,I)$ be a bounded prism, let $(A,J)$ be a $\delta$-pair
over $(R,I)$, and let $g\colon (A,J)\to (D,ID)$ be the
bounded prismatic envelope of $(A,J)$ over $(R,I)$
(assuming it exists) (Definition \ref{def:bddPrsimEnv}). 
For $i\in \{1,2\}$, let $\alpha_i$ and $\beta_i$ be
elements of $A$ satisfying $\delta(\alpha_i)=\alpha_i\beta_i$,
let $\partial_i\in \Der^{\alpha_i,\beta_i}_{R,\delta}(A)$
be an element satisfying $\partial_i(J)\subset J$, and 
let $\partial_i'\in \Der_{R,\delta}^{\alpha_i,\beta_i}(D)$ be the
unique extension of $\partial_i$
 (Proposition \ref{prop:TwDerivPrismExt}). 
Assume $\partial_i(\alpha_j)=\partial_i(\beta_j)=0$ for
$(i,j)=(1,2), (2,1)$. 
If $\partial_1\circ\partial_2=\partial_2\circ\partial_1$,
then $\partial_1'\circ\partial_2'=\partial_2'\circ\partial_1'$.
\end{lemma}

\begin{proof}
For $i\in \{1,2\}$, let $s'_i$ be the $\delta$-$R$-homomorphism 
$D\to E_{\delta}^{\alpha_i,\beta_i}(D)$ corresponding to 
$\partial_i'$ (Proposition \ref{prop:DeltaAlphaDerivInt} (2)). 
We have $s_i'(\alpha_j)=\alpha_j$ and $s_i'(\beta_j)=\beta_j$
for $(i,j)=(1,2), (2,1)$. 
By Lemma \ref{lem:TwDerivSectionCommutativity} (2), 
it suffices to prove that the two compositions
$E_{\delta}^{\alpha_1,\beta_1}(s'_2)\circ s_1'$ 
and $E_{\delta}^{\alpha_2,\beta_2}(s'_1)\circ s_2'$ are the same.
Since $\partial_1'\circ\partial_2'(g(a))=
\partial_2'\circ\partial_1'(g(a))$ for all $a\in A$
by assumption, we see that they coincide
after composing with $g$ by Lemma 
\ref{lem:TwDerivSectionCommutativity} (1). 
Since the two compositions in question both 
give morphisms of bounded prisms over $(R,I)$, 
the universality of the bounded prismatic envelope $g$
implies the desired equality. 
\end{proof}

\section{Twisted derivations on divided $\delta$-envelopes
and Poincar\'e lemma}
\label{sec:TwDeivDivDeltaEnv}
In this  section, we consider $\delta$-algebras over 
the polynomial ring $\Z[q]$ in one variable $q$ equipped
with the $\delta$-structure associated to the lifting of
Frobenius $f(q)\mapsto f(q^p)$. We put $\mu=q-1$,
$\xi=[p]_q=\varphi(\mu)\mu^{-1}=1+q+\cdots+q^{p-1}$,
and $I=(p,\pq)=(p,\mu^{p-1})$.
We have $\delta(\mu)=p^{-1}\{(q^p-1)-(q-1)^p\}
=p^{-1}\{(1+\mu)^p-1-\mu^p\}
=\sum_{\nu=1}^{p-1}p^{-1}\binom p\nu \mu^{\nu}
\in \mu \Z[q]$. Hence $\mu\Z[q]$ is a $\delta$-ideal
of $\Z[q]$. Since $\pq\equiv p\mod \mu$, we have
$\delta(\pq)\equiv \delta(p)=1-p^{p-1}\mod \mu$,
which implies that $\delta(\pq)$ is invertible modulo $I$. We put 
$\eta=\delta(\mu)\mu^{-1}=\sum_{\nu=1}^{p-1}p^{-1}\binom p\nu
\mu^{\nu-1}\in \Z[q]$.

\begin{proposition}\label{prop:PrismEnvDeriv}
Let $R$ be a $\delta$-$\Z[q]$-algebra, let $A$ be a
$\delta$-$R$-algebra, and let $T_1,\ldots, T_d$ be elements of $A$.
Put $I=pR+\xi R$. 
We define $D_0$, $D$, and $\tau_i,\tau_i^{\{m\}_{\delta}}\in D$ 
$(m\in \N)$ as in Proposition \ref{prop:DeltaEnvRegSeq}, and $\hD$
to be the $I$-adic completion $\varprojlim_nD/I^nD$ of $D$
 equipped with the completion of the $\delta$-structure of $D$.
We assume that we are given $t_i\in A$ and $a_i\in R$ $(i\in \N\cap [1,d])$
satisfying $T_i=t_i-a_i$ and $\delta(t_i)=0$, and that 
the $\delta$-$R$-algebra $D$ with $\tau_i\in D$
satisfies the conclusion of Proposition \ref{prop:DeltaEnvRegSeq} (4),
i.e., the following holds.
\begin{align}
&\text{The $R/I^{n+1}R$-module $D/I^{n+1}D$
is free with a basis ${\textstyle \prod\limits_{i=1}^d\tau_i^{\{m_i\}_{\delta}}}$ mod $I^{n+1}D$}
\;\;((m_i)\in \N^d)
\label{cond:DivDeltEnvBasis}\\
& \text{for every $n\in \N$} \notag
\end{align}
Then, for each $i\in \N\cap[1,d]$, there exists a unique 
$t_i\mu$-derivation $\theta_{\hD,i}$ on $\hD$ over $R$ $\delta$-compatible
with respect to $t_i^{p-1}\eta$ (Definition \ref{def:alphaDerivDeltaComp}) satisfying 
$\theta_{\hD,i}(\tau_i)=1$ and $\theta_{\hD,i}(\tau_j)=0$ 
$(j\in \N\cap [1,d], j\neq i)$. 
 (Note $\delta(t_i\mu)=t_i^{p}\delta(\mu)=(t_i^{p-1}\eta)t_i\mu$.)
Moreover we have $\theta_{\hD,i}\circ\theta_{\hD,j}=\theta_{\hD,j}\circ\theta_{\hD,i}$
$(i, j\in \N\cap [1,d])$. 
\end{proposition}

\begin{remark}\label{rmk:DerivDivDeltaEnv}
(1) Under the notation and assumption in Proposition 
\ref{prop:PrismEnvDeriv}, we see
that the $\delta$-$R$-algebra endomorphism $\gamma_{\hD,i}
=1+t_i\mu\theta_{\hD,i}$ of $\hD$ associated to $\theta_{\hD,i}$
(Lemma \ref{lem:alphaDergammaDer} (1), Lemma \ref{lem:alphaDerivDeltaEquiv} (1)) satisfies 
$\gamma_{\hD,i}(t_i)=q^pt_i$ and $\gamma_{\hD,i}(t_j)=t_j$ $
(j\neq i)$ as follows. From $t_j=\pq \tau_j+a_j$, we obtain 
$\theta_{\hD,i}(t_j)=\pq\theta_{\hD,i}(\tau_j)+\theta_{\hD,i}(a_j)=
\pq$ if $j=i$ and $0$ otherwise. This implies
$\gamma_{\hD,i}(t_i)=t_i+t_i\mu\pq=t_i+\varphi(\mu)t_i=q^pt_i$
and $\gamma_{\hD,i}(t_j)=t_j$ $(j\neq i)$. \par
(2) By Proposition \ref{prop:DeltaEnvRegSeq} (1), 
the assumption \eqref{cond:DivDeltEnvBasis} implies
that the homomorphism $R/I\to A/(IA+\sum_{i=1}^dT_i A)$
is an isomorphism. 
\end{remark}

\begin{proof}[Proof of Proposition \ref{prop:PrismEnvDeriv}]
By the construction of $D$ and $\tau_i$, and $\hD/I^n\hD\cong D/I^nD$
$(n\geq 1)$ (Lemma \ref{lem:CompFinGenIdeal}), the $\delta$-$R$-subalgebra
$R[\tau_1,\ldots, \tau_d]_{\delta}$ of $\hD$ 
is $I$-adically dense. Hence $\theta_{\hD,i}$ is unique by 
Lemma \ref{lem:alphaDerivDeltaEq}
for $f=\id_R$ and $g=\id_A$.  \par

Let us prove the existence of $\theta_{\hD,i}$. 
Let $\tA$ be the polynomial algebra $R[\tlt_1,\ldots, \tlt_d]$
over $R$ equipped with the $\delta$-$R$-algebra structure
defined by $\delta(\tlt_i)=0$, and put 
$\tT_i=\tlt_i-a_i\in \tA$. Then $R$, $\xi=\pq$, $\tA$, and $\tT_i$
satisfy the assumptions in Proposition \ref{prop:DeltaEnvRegSeq}  (3) and (4)
by Remark \ref{rmk:DeltaEnvRegSeq} (1). 
Put $\tD_0=\tA[\tS_1,\ldots, \tS_d]/(\pq \tS_i-\tT_i)$,
and let $\tD$ be the $\delta$-envelope of $\tD_0$ over $\tA$.
Then the $R$-homomorphism $f\colon \tA\to A$
defined by $\tT_i\to T_i$ is compatible with the $\delta$-structures.
It extends to a homomorphism $\tD_0\to D_0$ sending 
$\tS_i$ to $S_i$ and then to a $\delta$-homomorphism 
$\tD\to D$ by the universal property of the $\delta$-envelope
$\tD_0\to \tD$ over $\tA$. By the assumption on $D$ and
$\tau_i$,  and Proposition \ref{prop:DeltaEnvRegSeq}
 (4) applied to $\tD$, the $I$-adic completion of the last map $\tD\to D$ is an isomorphism.
 Hence we may replace $A$ by $\tA$ and assume
 $A=R[T_1,\ldots, T_d]$.\par
 
 Put $\alpha_i=t_i\mu$ and $\beta_i=t_i^{p-1}\eta$.
 There exists a unique $R$-homomorphism 
 $s_{A,i}\colon A\to E^{\alpha_i}(A)$ sending
 $T_i$ to $(T_i,\pq)$ and $T_j$ to 
 $(T_j,0)$ for $j\neq i$. By Proposition \ref{prop:AlphaDerivInterpret}, 
 this defines an 
 $\alpha_i$-derivation $\theta_{A,i}$ of $A$ over $R$
 satisfying $\theta_{A,i}(T_i)=\pq$ and $\theta_{A,i}(T_j)=0$
 $(j\neq i)$, which is equivalent to
 $\theta_{A,i}(t_i)=\pq$ and $\theta_{A,i}(t_j)=0$ $(j\neq i)$. 
 By Lemma \ref{lem:TwDerivxiDeltaComp} 
 below and $\delta(t_j)=0$ $(1\leq j\leq d)$, 
 we see that $t_j$ is $\delta$-compatible with respect to
 $\theta_{A,i}$ and $\beta_i$ (Definition \ref{def:alphaDerivDeltaComp}). 
 Hence $\theta_{A,i}$ is 
 $\delta$-compatible with respect to $\beta_i$ by 
 Lemma \ref{lem:DerivDeltaCompDeltaGen}. 
 The $R$-homomorphism $s_{A,i}$ extends to a 
 homomorphism $s_{D_0,i}\colon 
 D_0\to E^{\alpha_i}(D_0)$ sending $S_i$ to $(S_i,1)$
 and $S_j$ to $(S_j,0)$ for $j\neq i$
 because $\pq(S_i,1)=(T_i,\pq)$ and $\pq(S_j,0)=(T_j,0)$
 $(j\neq i)$  in $E^{\alpha_i}(D_0)$.
 The corresponding $\alpha_i$-derivation 
 $\theta_{D_0,i}$ of $D_0$ over $R$ uniquely 
 extends to $\theta_{D,i}\in\Der^{\alpha_i,\beta_i}_{R,\delta}(D)$ by
 Proposition \ref{eq:DeltaEnvTwDerivExt} (2). 
 By taking the $I$-adic completion, we obtain the desired
 $\theta_{\hD,i}\in \Der_{R,\delta}^{\alpha_i,\beta_i}(\hD)$
 (Remark \ref{rmk:alphaDerivDeltaBC} (1)).
 
 It remains to show $\theta_{\hD,i}\circ\theta_{\hD,j}=
 \theta_{\hD,j}\circ\theta_{\hD,i}$ for $i\neq j$. 
 We have
 $\theta_{\hD,i}\circ\theta_{\hD,j}(\tau_k)
 =\theta_{\hD,j}\circ\theta_{\hD,i}(\tau_k)=0$
 for every $k\in\N\cap [1,d]$. Since the 
 $\delta$-$R$-subalgebra $R[\tau_1,\ldots, \tau_d]_{\delta}$
 of $\hD$ is $I$-adically dense, this implies
 $\theta_{\hD,i}\circ\theta_{\hD,j}=
 \theta_{\hD,j}\circ\theta_{\hD,i}$ by 
Lemma \ref{lem:TwDerivCommGen}.
Note that we have $\theta_{\hD,i}(\alpha_j)=\theta_{\hD,i}(\beta_j)=0$
for $i\neq j$ because $\theta_{\hD,i}$ is $R$-linear
and $\theta_{\hD,i}(t_j)=0$. 
\end{proof}

\begin{lemma}\label{lem:TwDerivxiDeltaComp}
We have $(\mu^{p-1}+p\eta)\delta(\pq)+\eta(\pq)^p
-\sum_{\nu=1}^{p-1}p^{-1}\binom p\nu \mu^{\nu-1}(\pq)^{\nu}=0$
in $\Z[q]$.
\end{lemma}

\begin{proof}
By taking the images under $\delta$ of  both sides
of the equality $q^p=1+\mu\pq$ in $\Z[q]$, 
we obtain $0=\delta(\mu\pq)-\sum_{\nu=1}^{p-1}p^{-1}\binom p\nu
\mu^{\nu}(\pq)^{\nu}$. The desired equality follows from 
$\delta(\mu\pq)=\mu^p\delta(\pq)+\delta(\mu)(\pq)^p+p\delta(\mu)\delta(\pq)
=\mu(\mu^{p-1}\delta(\pq)+\eta(\pq)^{p}+p\eta\delta(\pq)).$
\end{proof}

\begin{proposition}\label{prop:DivDeltaEnvPDEnv}
Under the notation and assumption in Proposition \ref{prop:PrismEnvDeriv}, 
we further assume $\mu R=0$, $R$ is a $\Z_{(p)}$-algebra,
and we are given $b_i\in R$ such that 
$\delta(a_i)=pb_i$ for each $i\in \N\cap [1,d]$.
Put 
$$\sigma_i=(1-p^{p-1})^{-1}\left(-b_i-\delta(\tau_i)+\sum_{\nu=1}^{p-1}
p^{-1}\binom p\nu a_i^{p-\nu}p^{\nu-1}\tau_i^{\nu}\right)\quad \in D$$
for each $i\in \N\cap [1,d]$. Then we have $\tau_i^p=p\sigma_i$
and the $p$-adic completion of the $R$-homomorphism 
$$\pi_{\utau,\usigma, R}\colon R[X_1,\ldots, X_d]_{\PD}\longrightarrow D$$
induced by the homomorphism $\pi_{\utau,\usigma}$ 
\eqref{eq:DeltaRingPDMap2} 
associated to $\utau=(\tau_1,\ldots, \tau_d)$
and $\usigma=(\sigma_1,\ldots, \sigma_d)$ is an isomorphism.
Let $\tau_i^{[n]}$ be the image
of $X_i^{[n]}$ under $\pi_{\utau,\usigma}$ for 
$i\in \N\cap [1,d]$ and $n\in \N$. Let $\theta_{\hD,i}$ be the
derivation of $\hD$ over $R$ in Proposition \ref{prop:PrismEnvDeriv}. (Note that 
$t_i\mu A=0$ by assumption.) 
Then we have $\theta_{\hD,i}(\tau_i^{[n+1]})=\tau_i^{[n]}$
$(n\in \N)$ and $\theta_{\hD,i}(\tau_j^{[n]})=0$ $(j\neq i, n\in \N)$.
\end{proposition}

\begin{remark}\label{rmk:DeltaEnvPDEnv}
Under the notation and assumption of Proposition \ref{prop:PrismEnvDeriv},
we further assume that $R$ is a $\Z_{(p)}$-algebra.
Let $\oR$ be the $\delta$-$R$-algebra $R/\mu R$, 
let $\oA$ be the $\delta$-$\oR$-algebra $A/\mu A=A\otimes_R\oR$,
let $\oT_i$ and $\ot_i$ $(i\in \N\cap [1,d])$ be the images
of $T_i$ and $t_i$ in $\oA$, and let $\oa_i$ $(i\in \N\cap [1,d])$ be
the images of $a_i$ in $\oR$. We define $\oD_0$, 
$\oD$, and $\otau_i$ in the same way as 
$D_0$, $D$, $\tau_i$ by using $\oR$, $\oA$, and $\oT_i$.
Let $\hoD$ be the $\delta$-$R$-algebra
$\varprojlim_n \oD/I^n\oD=\varprojlim_n \oD/p^n\oD$.
Then, by Remark \ref{rmk:DeltaEnvRegSeq} (2), 
we have a canonical isomorphism of 
$\delta$-rings $D/\mu D\xrightarrow{\cong}\oD$ and we can apply
Proposition \ref{prop:PrismEnvDeriv} also to 
$\oR$, $\oA$, $\oT_i$, $\ot_i$, and $\oa_i$.
Since $R[\tau_1,\ldots,\tau_d]_{\delta}\subset \hD$ is 
$I$-adically dense, Lemma \ref{lem:alphaDerivDeltaEq} shows the that the
 homomorphism $\hD\to \hoD$ induced by 
 $D\to D/\mu D\xrightarrow{\cong} \oD$ is
 compatible with the $t_i\mu$-derivation $\theta_{\hD,i}$
 on $\hD$ and the $t_i\mu$-derivation $\theta_{\hoD,i}$
 on $\hoD$ given by Proposition \ref{prop:PrismEnvDeriv}. \par
 Hence, if we are given $\ob_i\in \oR$ satisfying 
 $\delta(\oa_i)=p\ob_i$ for each $i\in \N\cap [1,d]$,
 then we can apply Proposition \ref{prop:DivDeltaEnvPDEnv} to 
 $\oR$, $\oA$, $\oT_i$, $\ot_i$, $\oa_i$ and $\ob_i$,
 and obtain a presentation of $(D/\mu D)/p^n$
and $\theta_{\hD,i}$ mod $p^n R+\mu R$ .
\end{remark}

\begin{proof}[Proof of Proposition \ref{prop:DivDeltaEnvPDEnv}]
Note that the image of $\pq$ in $R$ is $p$ by the assumption $\mu R=0$.
By taking the image of $t_i=p\tau_i+a_i$ under $\delta\colon A\to A$, 
we obtain
$$
0=\delta(p\tau_i)+\delta(a_i)-\sum_{\nu=1}^{p-1}p^{-1}
\binom p\nu (p\tau_i)^{\nu}a_i^{p-\nu}
=(1-p^{p-1})\tau_i^p+p\delta(\tau_i)+pb_i-
p\sum_{\nu=1}^{p-1}p^{-1}\binom p\nu
a_i^{p-\nu}p^{\nu-1}\tau_i^{\nu}.
$$
This implies $\tau_i^p=p\sigma_i$. \par
We prove the remaining claims by considering the universal case.
Let $\tR$ be the $\delta$-ring 
$\Z_{(p)}[\ta_1,\ldots, \ta_d,\tb_1,\ldots, \tb_d]_{\delta}/(\delta(\ta_i)-p\tb_i
\;(i\in \N\cap [1,d]))_{\delta}$
regarded as a $\delta$-$\Z[q]$-algebra via the $\delta$-homomorphism 
$\Z[q]\to \Z;q\mapsto 1$, and let $\tA$ be the $\delta$-$\tR$-algebra
$\tR[\tlt_1,\ldots, \tlt_d]$ whose $\delta$-structure is defined by $\delta(\tlt_i)=0$
$(i\in \N\cap [1,d])$. Since $\tR$ is the $\delta$-envelope of 
$\Z_{(p)}[\ta_1,\ldots, \ta_d]_{\delta}[\tb_1,\ldots,\tb_d]/
(\delta(a_i)-p\tb_i\;(i\in \N\cap [1,d]))$ over 
$\Z_{(p)}[\ta_1,\ldots, \ta_d]_{\delta}$ and the sequence
$\delta(\ta_1),\ldots, \delta(\ta_d)$ is $\Z_{(p)}[\ta_1,\ldots, \ta_d]_{\delta}/p$-regular,
we see that $\tR$ is $p$-torsion free by Proposition \ref{prop:DeltaEnvPTF}.

Let $f\colon \tR\to R$ be the $\delta$-$\Z[q]$-homomorphism 
defined by $f(\ta_i)=a_i$ and $f(\tb_i)=b_i$, and let 
$g\colon \tA\to A$ be the ring homomorphism compatible with $f$ defined by
$g(\tlt_i)=t_i$. Since $\delta(t_i)=0$ by assumption, $g$ is a
$\delta$-homomorphism. We define $\tD$ to be the $\delta$-envelope
of $\tD_0=\tA[\tS_1,\ldots, \tS_d]/(p\tS_i-\tT_i\;(i\in \N\cap [1,d]))$
over $\tA$, where $\tT_i=\tlt_i-\ta_i$. Let $\ttau_i$ be the images of
$\tS_i$ in $\tD_0$ and $\tD$. Then the homomorphism 
$h_0\colon \tD_0\to D_0$ compatible with $g$ defined by 
$h_0(\ttau_i)=\tau_i$ extends uniquely to a $\delta$-homomorphism
$h\colon \tD\to D$. Let $\widehat{\tD}$ be the $p$-adic completion 
of $\tD$. The sequence $\tT_1,\ldots, \tT_d$ is 
$\tA/p\tA$-regular, the $\tR$-algebras $\tA$ and 
$\tA/\sum_{i=1}^r\tT_i\tA$ $(r\in \N\cap [1,d])$ are flat,
and the homomorphism $\tR\to \tA/\sum_{i=1}^d\tT_i\tA$
is an isomorphism.  Hence, one can apply 
Proposition \ref{prop:DeltaEnvRegSeq} (4) to 
$(\tR,\xi=\pq, \tA,\tT_i)$ by Remark \ref{rmk:DeltaEnvRegSeq} (1), 
and see that the homomorphism
of $\delta$-$R$-algebars $h_R\colon \tD\otimes_{\tR}R
\to D$ induced by $h$ becomes an isomorphism after
taking the $p$-adic completion, and that
$\delta$-$\tR$-subalgebra of $\widehat{\tD}$
generated by $\ttau_i$ $(i\in \N\cap [1,d])$
is dense in $\widehat{\tD}$ with respect to the $p$-adic
topology. One can also apply Proposition \ref{prop:PrismEnvDeriv} to
$(\tR,\xi=\pq,\tA, \tlt_i, \ta_i)$,  and see that the derivation 
$\theta_{\widehat{\tD},i}\colon \widehat{\tD}\to
\widehat{\tD}$ is compatible with the 
derivation $\theta_{\hD,i}\colon \hD\to \hD$ via $h$
by Lemma \ref{lem:alphaDerivDeltaEq}. 
By Lemma \ref{lem:DeltaRingPD} (4),
we may replace $R$, $A$, $t_i$, $a_i$, and $b_i$ with
$\tR$, $\tA$, $\tlt_i$, $\ta_i$, and $\tb_i$,
and assume that $R$ is $p$-torsion free
and  $A$ is the polynomial algebra over $R$
in variables $t_1,\ldots, t_d$. 
 
By Proposition \ref{prop:DeltaEnvPTF}, $D$ is $p$-torsion free. Hence $\hD$ is $p$-torsion
free, and the last claim on $\theta_{\hD,i}$ follows from 
$\theta_{\hD,i}(\tau_j^{n+1})=(n+1)\tau_j^n\theta_{\hD,i}(\tau_j)
=(n+1)\tau_i^n$ (if $j=i$), $0$ (otherwise). We show that
$\pi_{\utau,\usigma,R}$ is an isomorphism, which implies the second
claim,  i.e., completes the proof. Since $D[\frac{1}{p}]$ is the
$\delta$-envelope of $A[\frac{1}{p}][S_i]/(pS_i-T_i)\cong A[\frac{1}{p}]$
over $A[\frac{1}{p}]$, the homomorphism 
$A[\frac{1}{p}]\to D[\frac{1}{p}]$ is an isomorphism. 
Hence $D[\frac{1}{p}]$ is the polynomial algebra over $R[\frac{1}{p}]$
in variables $\tau_1,\ldots,\tau_d$ because $t_i=p\tau_i+a_i$ and
$a_i\in R$. This implies that the homomorphism $\pi_{\utau,\usigma,R}$ is injective.
By $t_i=p\tau_i+a_i$ and $a_i\in R$ again, we see that
$D$ is the $\delta$-$R$-subalgebra of $D[\frac{1}{p}]$
generated by $\tau_1,\ldots, \tau_d$. Let $D'$
be the image of $\pi_{\utau,\usigma,R}$. Since 
$\tau_1,\ldots, \tau_d\in D'$, it suffices to show that $D'$ is a 
$\delta$-$R$-subalgebra of $D[\frac{1}{p}]$, i.e.,
$\delta(\frac{\tau_i^{n}}{n!})\in D'$ for every $n\in \N$
and $i\in \N\cap [1,d]$. By the computation in the first paragraph, we
have 
$$\delta(\tau_i)=-(1-p^{p-1})\frac{\tau_i^p}{p}-b_i
+\sum_{\nu=1}^{p-1}p^{-1}\binom p\nu a_i^{p-\nu}p^{\nu-1}\tau_i^{\nu}\in D'.$$
Hence $\varphi(\tau_i)=\tau_i^p+p\delta(\tau_i)$ is of the form $ps_i$
$(s_i\in D'$), and we obtain 
$\delta(\frac{\tau_i^n}{n!})=\frac{p^{n-1}}{n!}s_i^n-
\frac{1}{p}(\frac{\tau_i^n}{n!})^p\in D'$. 
 \end{proof}

In the rest of this section, we prove an analogue of Poincar\'e lemma for the twisted
derivations $\theta_{\hD,i}$ on $\hD$ constructed in
Proposition \ref{prop:PrismEnvDeriv}. 
We follow the notation and assumption in Proposition \ref{prop:PrismEnvDeriv}. 
Let $\Omega$ be the free $\Z$-module with formal basis
$dt_i$ $(i\in \N\cap[1,d])$, and put $\Omega^k=\wedge^k_{\Z}\Omega$ for $k\in \N$.
Then we can define an $R$-linear map 
$\theta^k_{\hD}\colon \hD\otimes_{\Z}\Omega^k\to\hD\otimes_{\Z}\Omega^{k+1}$
for $k\in \N$ by 
\begin{equation*}
\theta^k_{\hD}(x\otimes dt_{i_1}\wedge\cdots\wedge dt_{i_k})=
\sum_{i=1}^d\theta_{\hD,i}(x)\otimes dt_i\wedge dt_{i_1}\wedge\cdots
\wedge dt_{i_k}\quad (x\in \hD,\; i_j\in \N\cap [1,d]).
\end{equation*}
Since $\theta_{\hD,i}\circ \theta_{\hD,j}
=\theta_{\hD,j}\circ\theta_{\hD,i}$ $(i\neq j)$
by Proposition \ref{prop:PrismEnvDeriv}, 
we have $\theta^{k+1}_{\hD}\circ \theta^k_{\hD}=0$ $(k\in \N)$. 
Let $\gamma_{\hD,i}$ be the endomorphism $1+t_i\mu\theta_{\hD,i}$
of the $R$-algebra $\hD$ (Lemma \ref{lem:alphaDergammaDer} (1)), which is a 
$\delta$-$R$-endomorphism by Lemma \ref{lem:alphaDerivDeltaEquiv} (1). By 
Lemma \ref{lem:alphaDergammaDer} (1), we have 
\begin{equation}\label{eq:ThetaGammaDeriv}
\theta_{\hD,i}(xy)=\gamma_{\hD,i}(x)\theta_{\hD,i}(y)+
\theta_{\hD,i}(x)y \quad (x,y\in \hD).
\end{equation}

\begin{theorem}\label{thm:PoincareLem}
For any ideal $J$ of $R$ containing some
power of $I$ and any $R/J$-module $M$,
the complex $(M\otimes_{R}\hD\otimes_{\Z}\Omega^{\bullet},
\id_M\otimes \theta_{\hD}^{\bullet})$ gives a resolution of $M$.\par
\end{theorem}

Let $I_0$ be the ideal $(p,\mu)$ of $\Z[q]$. 
Theorem \ref{thm:PoincareLem} is reduced to the special case
$M=R/I_0R$ as follows. 
Since $I=pR+\mu^{p-1}R$, we have $I_0^{p-1}R\subset I$.
Hence there exists $N\in \N$ such that $I_0^{N+1}M=0$. 
We have an exact sequence of $R/I_0^{n+1}R$-modules
$0\to I_0^nM/I_0^{n+1}M\to M/I_0^{n+1}M
\to M/I_0^nM\to 0$ for $n\in \N\cap [1,N]$, and it remains
exact after taking $-\otimes_R\hD\otimes_{\Z}\Omega^{\bullet}$
because $\hD/I_0^{n+1}\hD=D/I_0^{n+1}D$ is flat over 
$R/I_0^{n+1}R$ by the assumption \eqref{cond:DivDeltEnvBasis}. 
Hence we are reduced to the case $I_0M=0$. 
Since $\hD/I_0\hD=D/I_0D$ is flat over $R/I_0R$,
it suffices to show that $R/I_0R\otimes_R(\hD\otimes_{\Z}\Omega^{\bullet})$
gives a resolution of $R/I_0R$. 

We will prove Theorem \ref{thm:PoincareLem} for $M=R/I_0R$ in the following.
We reduce the proof to the case $\delta(a_i)=0$ $(i\in \N\cap[1,d])$, when the last assertion 
of Remark \ref{rmk:DeltaEnvPDEnv} 
is applicable, by faithfully flat descent with respect 
to a base change faithfully flat modulo $I_0$. 

We start by discussing the faithfully flat descent of the claim
of Theorem \ref{thm:PoincareLem} for $M=R/I_0R$. 
Let $f\colon R\to R'$ be a $\delta$-homomorphism of $\delta$-rings, 
put $A'=A\otimes_RR'$, $T_i'=T_i\otimes 1\in A'$, and $t_i'=t_i\otimes 1\in A'$.
Let $a_i'$ be the image of $a_i$ under $f$. We have $T_i'=t_i'-a_i'$.
We define $D_0'$ $D'$ and $\tau_i'$ in the same way as
$D_0$, $D$ and $\tau_i$ by using $R'$, $A'$ and $T_1',\ldots, T_d'$.
Then, by Remark \ref{rmk:DeltaEnvRegSeq} (2), we have a $\delta$-homomorphism $g\colon D
\to D'$ extending the $\delta$-homomorphism 
$A\to A';a\mapsto a\otimes 1$ and sending
$\tau_i$ to $\tau_i'$ for each $i\in \N\cap [1,d]$, and it induces an 
isomorphism $D\otimes_RR'\xrightarrow{\cong}{D'}$. Hence 
$R'$, $A'$, $T_i'$, $t_i'$ and $a_i'$ also satisfy the assumption
\eqref{cond:DivDeltEnvBasis} of Proposition \ref{prop:PrismEnvDeriv},
and we obtain a $t_i'\mu$-derivation 
$\theta_{\hD',i}$ on $\hD'=\varprojlim_n D'/I^nD'$ over $R$ $\delta$-compatible
with respect to $(t_i')^{p-1}\eta$ for each $i\in \N\cap [1,d]$. 
The $\delta$-homomorphism $\hg\colon \hD\to \hD'$ induced by $g$
satisfies $\theta_{\hD',i}\circ \hg(\tau_i)
=\theta_{\hD',i}(\tau_i')=1=\hg\circ\theta_{\hD,i}(\tau_i)$
and $\theta_{\hD',i}\circ \hg(\tau_j)
=\theta_{\hD',i}(\tau'_j)=0=\hg\circ \theta_{\hD,i}(\tau_j)$
for $j\neq i$.  Since the $\delta$-$R$-subalgebra
$R[\tau_1,\ldots,\tau_d]_{\delta}$ of $\hD$ is
$I$-adically dense by \eqref{cond:DivDeltEnvBasis},
this implies that $\hg$ is compatible with $\theta_{\hD,i}$
and $\theta_{\hD',i}$ by Lemma \ref{lem:alphaDerivDeltaEq}. 
We can define the complex 
$(\hD'\otimes_{\Z}\Omega^{\prime\bullet},\theta_{\hD'}^{\bullet})$,
where $\Omega^{\prime}=\oplus_{i=1}^d\Z dt_i'$
and $\Omega^{\prime k}=\wedge^k_{\Z}\Omega^{\prime}$ $(k\in \N)$,
in the same way as $\hD$ by using $\hD'$ and $\theta_{\hD',i}$, 
and the above compatibility implies that the homomorphism 
$\hg$ and the isomorphism $\Omega\xrightarrow{\cong}
\Omega^{\prime};dt_i\mapsto dt_i'$ give a homomorphism of 
complexes
\begin{equation}
\hD\otimes_{\Z}\Omega^{\bullet}\longrightarrow 
\hD'\otimes_{\Z}\Omega^{\prime\bullet},
\end{equation}
which induces isomorphisms
\begin{equation}\label{eq:qHiggsCpxBCIsom}
R'/I_0 R'\otimes_{R}\hD\otimes_{\Z}\Omega^{\bullet}
\xrightarrow{\;\cong\;} R'/I_0 R'\otimes_{R'}\hD'\otimes_{\Z}
\Omega^{\prime\bullet}.
\end{equation}

\begin{lemma}\label{lem:PLFlatDescent}
 Suppose that the reduction modulo $I_0$ of the homomorphism
$f\colon R\to R'$ is faithfully flat. 
If $R'/I_0 R'\otimes_{R'}\hD'\otimes_{\Z}\Omega^{\prime\bullet}$
is a resolution of $R'/I_0R'$, then 
$R/I_0R\otimes_R\hD\otimes_{\Z}\Omega^{\bullet}$
is a resolution of $R/I_0R$. 
\end{lemma}

\begin{proof}
This follows from \eqref{eq:qHiggsCpxBCIsom}
by faithfully flat descent with respect to $R/I_0R\to R'/I_0R'$. 
\end{proof}

For the reduction to the case $\delta(a_i)=0$, we also need a change
of variables $T_i$ by adding elements of $\pq R$ after a base change
faithfully flat modulo $I_0$. Let $c_i$ $(i\in \N\cap [1,d])$ be elements
of $R$, and put $\tT_i=T_i-\pq c_i\in A$ and $\ta_i=a_i+\pq c_i\in R$.
We have $\tT_i=t_i-\ta_i$. We define $\tD_0$ and $\tD$ to be 
$A[\tS_1,\ldots \tS_d]/(\pq \tS_i-\tT_i)$ and its $\delta$-envelope
over $A$. Let $\ttau_i$ be the image of $\tS_i$ in $\tD$.
Then the $A$-isomorphism $A[\tS_1,\ldots, \tS_d]
\xrightarrow{\cong} A[S_1,\ldots, S_d]; \tS_i\mapsto S_i-c_i$ induces
an $A$-isomorphism $\tD_0\xrightarrow{\cong} D_0$ and then a 
$\delta$-$A$-isomorphism $h\colon \tD\xrightarrow{\cong}D$ because
$\pq(S_i-c_i)-\tT_i=\pq S_i-T_i$. 

\begin{lemma}\label{lem:PLChangeVar} The $\delta$-$R$-algebra $\tD$ and $\ttau_i\in \tD$
$(i\in \N\cap [1,d])$ satisfy \eqref{cond:DivDeltEnvBasis} and the 
$t_i\mu$-derivation
$\theta_{\widehat{\tD},i}$ on $\widehat{\tD}$ over $R$ given by Proposition 
\ref{prop:PrismEnvDeriv} is compatible
with $\theta_{\hD,i}$ via the isomorphism $\hh\colon \widehat{\tD}\xrightarrow{\cong}
\hD$ induced by $h$ above.
\end{lemma}

\begin{proof}
We have an isomorphism 
$R/IR\xrightarrow{\cong} A/(IA+\sum_{i=1}^d T_iA)
=A/(IA+\sum_{i=1}^d\tT_i A)$ by 
the assumption \eqref{cond:DivDeltEnvBasis} (Remark \ref{rmk:DerivDivDeltaEnv} (2)).
Since $D/I^{n+1}D$ is flat over $R/I^{n+1}R$ by 
\eqref{cond:DivDeltEnvBasis}, 
we see that 
\eqref{cond:DivDeltEnvBasis} also holds for $\tD$ and $\ttau_i$
by Proposition \ref{prop:DeltaEnvRegSeq} (1) for $R$, $A$, and $\tT_i$.
We have $\theta_{\hD,i}\circ \hh(\ttau_i)=\theta_{\hD,i}(\tau_i-c_i)
=1=\hh\circ\theta_{\widehat{\tD},i}(\ttau_i)$ and 
$\theta_{\hD,i}\circ\hh(\ttau_j)=\theta_{\hD,i}(\tau_j-c_j)=0
=\hh\circ\theta_{\widehat{\tD},i}(\ttau_j)$ $(j\neq i)$. Since 
the $\delta$-$R$-subalgebra $R[\ttau_1,\ldots, \ttau_d]_{\delta}$
of $\widehat{\tD}$ is $I$-adically dense, this implies the second 
claim by Lemma \ref{lem:alphaDerivDeltaEq}.
\end{proof}

\begin{proof}[Proof of Theorem \ref{thm:PoincareLem} for $M=R/I_0R$]
Let $R'$ be the $\delta$-$R$-algebra $\hD$. Then $R'/I_0 R'$
is faithfully flat over $R/I_0R$ 
by the assumption \eqref{cond:DivDeltEnvBasis}, and we have
$t_i=a_i+\pq \tau_i$ $(\tau_i\in R')$ and $\delta(t_i)=0$ in $R'$.
Hence, by Lemmas \ref{lem:PLFlatDescent} and 
\ref{lem:PLChangeVar}, we may  replace $R$ by $R'$
and then $a_i$ by $t_i$, and may assume $\delta(a_i)=0$
and $R$ is a $\Z_{(p)}$-algebra.\par
Since $\delta(a_i)=0$,
we can apply the last assertion in Remark \ref{rmk:DeltaEnvPDEnv} by setting 
$\ob_i=0$, and obtain an $R/I_0R$-isomorphism 
$R/I_0R[X_1,\ldots,X_d]_{\PD}\xrightarrow{\cong}
D/I_0D=\hD/I_0\hD$ such that $\theta_{\hD,i} \mod I_0$ 
corresponds to the $R/I_0R$-linear derivation $\partial_i$ of the domain
defined by $\partial_i(X_i^{[n+1]})=X_i^{[n]}$ $(n\in \N)$
and $\partial_i(X_j^{[n]})=0$ $(n\in \N,j\neq i)$.
Hence the desired claim holds by 
\cite[V.~Lemme 2.1.2]{BerthelotCrisCoh}
\end{proof}

\section{Connections}\label{sec:connection}
\begin{definition}[{cf.~\cite[2.2]{Andre}}]\label{def:ConnectionTwDeriv}
Let $A$ be an algebra over a ring $R$, and let $\alpha\in A$ and
$\partial\in\Der^{\alpha}_R(A)$ (Definition \ref{def:alphaDerivation} (1)).
Let $\gamma$ be the endomorphism $1+\alpha\partial$ of the
$R$-algebra $A$ (Lemma \ref{lem:alphaDergammaDer} (1)).
{\it A $(\alpha,\partial)$-connection on an $A$-module} $M$ is an additive map
$\nabla\colon M\to M$ satisfying $\nabla(ax)=\gamma(a)\nabla(x)+
\partial(a)x$ for $a\in A$ and $x\in M$. 
\end{definition}

\begin{lemma}\label{lem:ConnectionTwDeriv}
Let $R$, $A$, $\alpha$, $\partial$, and $\gamma$ be as in Definition 
\ref{def:ConnectionTwDeriv}.\par
(1) Let $(M,\nabla_M)$ be an $A$-module with a $(\alpha,\partial)$-connection.
Then $\nabla_M$ is $R$-linear, and the map $\gamma_M:=1+\alpha\nabla_M
\colon M\to M$ is $\gamma$-semilinear.\par
(2) Let $M$ be an $A$-module such that $\alpha$ is regular on $M$. 
Then, for a $\gamma$-semilinear endomorphism $\gamma_M$ of $M$
being the identity modulo $\alpha$, i.e., satisfying $(\gamma-1)(M)\subset
\alpha M$, the map $\nabla_M:=\alpha^{-1}(\gamma_M-1)\colon M\to M$
is a $(\alpha,\partial)$-connection on $M$. This gives a bijection from
the set of $\gamma$-semilinear endomorphisms of $M$ trivial modulo $\alpha$
to that of $(\alpha,\partial)$-connections on $M$. 
\end{lemma}

\begin{proof}
(1) The first claim follows from $\partial(a\cdot 1)=a\partial(1)=0$ for $a\in R$.
The second one is verified as follows.
$$\gamma_M(ax)=ax+\alpha(\gamma(a)\nabla_M(x)+\partial(a)x)
=\gamma(a)x+\alpha\gamma(a)\nabla_M(x)=\gamma(a)\gamma_M(x)\;\;
(a\in A, x\in M)$$
(2) The second claim follows from (1) and the first claim, which is an immediate
consequence of the following computation for $a\in A$ and $x\in M$:
$\gamma_M(ax)-ax=\gamma(a)(\gamma_M(x)-x)+(\gamma(a)-a)x$
and $\gamma(a)-a=\alpha\partial(a)$. 
\end{proof}

Let $A$ be an algebra over a ring $R$, let $\ualpha=(\alpha_i)_{i\in\Lambda}$
be a finite family of elements of $A$, and suppose that 
we are given an $\alpha_i$-derivation $\partial_i$ of $A$ over $R$
(Definition \ref{def:alphaDerivation} (1)) for each $i\in \Lambda$ satisfying the following.
\begin{align}
\partial_i(\alpha_j)&=0\qquad(i,j\in\Lambda, i\neq j),\label{eq:TwDerivFamilyCond1}\\
\partial_i\circ\partial_j&=\partial_j\circ\partial_i\qquad(i,j\in\Lambda).
\label{eq:TwDerivFamilyCond2}
\end{align}
We define a family of $R$-endomorphisms $\ugamma=(\gamma_i)_{i\in\Lambda}$
of $A$ by $\gamma_i=\id_A+\alpha_i\partial_i$ (Lemma \ref{lem:alphaDergammaDer} (1)), an 
$A$-bimodule $\Omega_{A,\ugamma}$ to be $\oplus_{i\in\Lambda}A\omega_i$
equipped with the bimodule structure:
$b(a\omega_i)c=ba\gamma_i(c)\omega_i$ $(a,b,c\in A, i\in\Lambda)$,
and the $R$-linear map $d\colon A\to \Omega_{A,\ugamma}$ by 
$d(a)=\sum_{i\in\Lambda}\partial_i(a)\otimes\omega_i$. 
The map $d$ satisfies the Leibniz rule $d(ab)=ad(b)+d(a)b $ $(a,b\in A)$ since
$\partial_i$ is a $\gamma_i$-derivation by Lemma \ref{lem:alphaDergammaDer} (1)).
By \eqref{eq:TwDerivFamilyCond1}, \eqref{eq:TwDerivFamilyCond2}, and 
Remark \ref{rmk:alphaDerivBC} (4), we have
\begin{align}
\gamma_i\circ\partial_j&=\partial_j\circ\gamma_i\qquad(i,j\in\Lambda,i\neq j),
\label{eq:TwDerivGammaComm1}\\
\gamma_i\circ\gamma_j&=\gamma_j\circ\gamma_i\qquad(i,j\in\Lambda).
\label{eq:TwDerivGammaComm2}
\end{align}

We define $\Omega^{\bullet}_{A,\ugamma}$ to be the graded left $A$-module
whose degree $q$-part is the $q$th exterior tensor product 
$\wedge^q_A\Omega_{A,\ugamma}$ of $\Omega_{A,\ugamma}$
with respect to the left $A$-module structure. We write $-\wedge-$ for
the product with respect to the exterior algebra structure. 
For $ q\in \N$ and $\bmI=(i_1,\ldots,i_q)\in \Lambda^q$, we define 
$\omega_{\bmI}$ to be 
$\omega_{i_1}\wedge\cdots\wedge\omega_{i_q}\in\Omega^q_{A,\ugamma}$,
and $\gamma_{\bmI}$ to be $\gamma_{i_1}\circ\gamma_{i_2}\circ
\cdots \circ\gamma_{i_q}$.  Note that $\omega_{\bmI}\neq 0$ if and only if
the components of $\bmI$ are mutually different.
We can equip $\Omega^{\bullet}_{A,\ugamma}$ with the associative
$R$-algebra structure with unity by
\begin{equation}\label{eq:DGAProdDef}
a\omega_{\bmI}\cdot a'\omega_{\bmI'}=
a\gamma_{\bmI}(a')\omega_{\bmI}\wedge\omega_{\bmI'}
\quad(q,q'\in\N, a,a'\in A, \bmI\in \Lambda^q, \bmI'\in \Lambda^{q'}).
\end{equation}
We define an $R$-linear map $d^q\colon \Omega^q_{A,\ugamma}\to 
\Omega^{q+1}_{A,\ugamma}$ for each $q\in \N$ by 
$$d^q(a\omega_{\bmI})=\sum_{i\in\Lambda}\partial_i(a)\omega_i\wedge\omega_{\bmI}
\quad(a\in A, \bmI\in\Lambda^q).$$
We have $d^0=d$. Then $\Omega^{\bullet}_{A,\ugamma}$ equipped with 
$d^{\bullet}=(d^q)_{q\in \N}$ is a differential graded algebra over $R$, i.e., the following
hold: 
$d^0(R\cdot 1)=0$, $d^{q+1}\circ d^q=0$ $(q\in \N)$,
and $d^{q+q'}(\omega\omega')=
d^q(\omega)\omega'+(-1)^q\omega d^{q'}(\omega')$
$(q,q'\in \N,\omega\in \Omega^q_{A,\ugamma},
\omega'\in \Omega^{q'}_{A,\ugamma})$. 
The first property follows from $\partial_i(1)=0$ and the $R$-linearity of 
$\partial_i$. The second one is an immediate consequence of 
\eqref{eq:TwDerivFamilyCond2}. The last equality is verified as follows.
We may assume $\omega=a\omega_{\bmI}$ and $\omega'=a'\omega_{\bmI'}$
for $a, a'\in A$, $\bmI\in \Lambda^q$, and $\bmI'\in \Lambda^{q'}$ 
such that $\bmI$ and $\bmI'$ have no common components. Then we have 
\begin{align*}
d^{q+q'}(\omega\omega')&=
\sum_{i\in\Lambda}\partial_i(a\gamma_{\bmI}(a'))\omega_i\wedge\omega_{\bmI}
\wedge\omega_{\bmI'},\\
d^q(\omega)\omega'&=\biggl(\sum_{i\in\Lambda}\partial_i(a)\omega_i\wedge\omega_{\bmI}\biggr)a'\omega_{\bmI'}
=\sum_{i\in\Lambda}\partial_i(a)\gamma_i\gamma_{\bmI}(a')
\omega_i\wedge\omega_{\bmI}\wedge\omega_{\bmI'},\\
\omega d^{q'}(\omega')&=
a\omega_{\bmI}\sum_{i\in\Lambda}\partial_i(a')\omega_i\wedge\omega_{\bmI'}
=(-1)^q\sum_{i\in\Lambda}a\gamma_{\bmI}(\partial_i(a'))
\omega_i\wedge \omega_{\bmI}\wedge \omega_{\bmI'}.
\end{align*}
Hence the desired equality follows from 
$\partial_i(a\gamma_{\bmI}(a'))=\partial_i(a)\gamma_i\gamma_{\bmI}(a')
+a\partial_i(\gamma_{\bmI}(a'))$ $(i\in\Lambda)$
and $\partial_i(\gamma_{\bmI}(a'))=\gamma_{\bmI}(\partial_i(a'))$
for $i\in \Lambda$ not appearing in $\bmI$. 
The latter equality holds by \eqref{eq:TwDerivGammaComm1}. 

\begin{definition}[{cf.~\cite[2.2]{Andre}}]\label{def:connection}
{\it A module with connection over} $(A, d\colon A\to \Omega_{A,\ugamma})$
(or $\Omega_{A,\ugamma}^{\bullet})$
is an $A$-module $M$ equipped with an additive map 
$\nabla_M\colon M\to M\otimes_A\Omega_{A,\ugamma}$ satisfying
the Leibniz rule $\nabla_M(am)=\nabla_M(m)a+m\otimes d(a)$ for
$m\in M$ and $a\in A$. We define the endomorphisms 
$\nabla_{M,i}$ and $\gamma_{M,i}$ $(i\in \Lambda)$
of $M$ associated to $\nabla_M$ by 
$\nabla_M(m)=\sum_{i\in \Lambda}\nabla_{M,i}(m)\otimes\omega_i$
and $\gamma_{M,i}(m)=m+\alpha_i\nabla_{M,i}(m)$. 
Then the Leibniz rule of $\nabla_M$ is equivalent to 
saying that $\nabla_{M,i}$ is a $(\alpha_i,\partial_i)$-connection 
(Definition \ref{def:ConnectionTwDeriv}) for every $i\in \Lambda$. Hence
$\gamma_{M,i}$ is $\gamma_i$-semilinear by Lemma 
\ref{lem:ConnectionTwDeriv} (1).
{\it A morphism $f\colon (M,\nabla_M)\to (M',\nabla_{M'})$ of modules
with connection over} $(A,d)$ is an $A$-linear map $f\colon M\to M'$
satisfying $(f\otimes\id_{\Omega_{A,\ugamma}})\circ\nabla_M=
\nabla_{M'}\circ f$. 
\end{definition}

Let $(M,\nabla_M)$ be a module with connection over $(A,d)$. 
For each $q\in \N$, we can define an additive map 
$\nabla_M^q\colon M\otimes_A\Omega^q_{A,\ugamma}\to 
M\otimes_A\Omega_{A,\ugamma}^{q+1}$ by 
$\nabla_M^q(m\otimes\omega)=\nabla_M(m)\wedge\omega+m\otimes d^q(\omega)$
thanks to the Leibniz rule of $\nabla_M$ and 
$d^q(a\omega)=d(a)\wedge\omega+ad^q(\omega)$ $(a\in A,\omega\in \Omega^q_{A,\ugamma})$.
Using the fact that $(\Omega_{A,\ugamma}^{\bullet},d^{\bullet})$
is a differential graded algebra, we obtain 
$\nabla_M^{q+q'}(x\omega)=\nabla_M^q(x)\wedge\omega+(-1)^qx\wedge d^{q'}(\omega)$
for $q,q'\in \N$, $x\in M\otimes_A\Omega^q_{A,\ugamma}$,
and $\omega\in \Omega_{A,\ugamma}^{q'}$.
We have $\nabla_M^0=\nabla_M$. 

\begin{definition}\label{def:intconnection}
For a module with connection $(M,\nabla_M)$ over $(A,d)$, 
we say that $\nabla_M$ is {\it integrable} if $\nabla_M^1\circ\nabla_M=0$. 
We write $\MIC(A,d)$ for the category of modules with integrable connection
over $(A,d)$. 
\end{definition}

For $m\in M$, $q\in \N$, and $\bmI\in \Lambda^q$, we have 
$$\nabla_M^{q+1}\circ\nabla_M^q(m\otimes \omega_{\bmI})
=\nabla_M^{q+1}\biggl(\sum_{i\in \Lambda}\nabla_{M,i}(m)
\otimes\omega_i\wedge\omega_{\bmI}\biggr)
=\sum_{j\in\Lambda}\sum_{i\in\Lambda}
\nabla_{M,j}\circ\nabla_{M,i}(m)\otimes\omega_j\wedge\omega_i\wedge
\omega_{\bmI}.$$
This implies that $\nabla_M$ is integrable if and only if 
$\nabla_{M,i}\circ \nabla_{M,j}=\nabla_{M,j}\circ\nabla_{M,i}$
for every $i,j\in \Lambda$, $i\neq j$, and that
$\nabla_{M}^{q+1}\circ\nabla_{M}^q=0$ for every $q\in \N$ if $\nabla_M$ is 
integrable. When $\nabla_M$ is integrable, we call 
the complex $(M\otimes_{A}\Omega^{\bullet}_{A,\ugamma},\nabla_M^{\bullet})$
the {\it de Rham complex of $(M,\nabla_M)$. }
For $i,j\in \Lambda$, $i\neq j$, and $m\in M$, we have
$\gamma_{M, i}\circ\gamma_{M,j}(m)
=m+\alpha_j\nabla_{M,j}(m)+\alpha_i\nabla_{M,i}(m)+
\alpha_i\alpha_j\nabla_{M,i}\circ\nabla_{M,j}(m)$
and 
$\gamma_{M,i}\circ\nabla_{M,j}(m)=\nabla_{M,i}\circ\gamma_{M,j}(m)=
\nabla_{M,i}(m)+\alpha_j\nabla_{M,i}\circ\nabla_{M,j}(m)$
by \eqref{eq:TwDerivFamilyCond1}. Hence we have 
$\gamma_{M,i}\circ\gamma_{M,j}=\gamma_{M,j}\circ\gamma_{M,i}$
and $\gamma_{M,i}\circ\nabla_{M,j}=\nabla_{M,j}\circ\gamma_{M,i}$
for every $i,j\in\Lambda$, $i\neq j$ if $\nabla_M$ is integrable. 

Let $(M,\nabla_M)$ and $(M',\nabla_{M'})$ be modules with connection
over $(A,d)$. Then it is straightforward to verify that the map
$$M\times M'\to M\otimes_AM'\otimes_A\Omega_{A,\ugamma};
(m,m')\mapsto \sum_{i\in\Lambda}
(\nabla_{M,i}(m)\otimes\gamma_{M',i}(m')+m\otimes\nabla_{M',i}(m'))
\otimes \omega_i$$
is $A$-bilinear and defines a connection
$\nabla_{M\otimes M'}\colon M\otimes_AM'\to (M\otimes_AM')\otimes_A\Omega_{A,\ugamma}$. 
We have
\begin{equation}\label{eq:ConnProdSymmForm}
\nabla_{M\otimes M',i}(m\otimes m')=\nabla_{M,i}(m)\otimes m'
+m\otimes\nabla_{M',i}(m')+\alpha_i\nabla_{M,i}(m)\otimes\nabla_{M',i}(m')
\end{equation}
for $i\in \Lambda$, $m\in M$, and $m'\in M'$. 
This implies that the isomorphism $M\otimes_A M'\cong M'\otimes_A M; 
m\otimes m'\mapsto m'\otimes m$ is compatible with
$\nabla_{M\otimes M'}$ and $\nabla_{M'\otimes M}$.
\par
For $m\in M$, $m'\in M'$,
and $i,j\in \Lambda$, we have 
\begin{multline*}\nabla_{M\otimes M',i}\circ\nabla_{M\otimes M',j}(m\otimes m')
=\nabla_{M,i}\circ\nabla_{M,j}(m)
\otimes\gamma_{M',i}\circ\gamma_{M',j}(m')
+\nabla_{M,j}(m)\otimes\nabla_{M',i}\circ\gamma_{M',j}(m')\\
+\nabla_{M,i}(m)\otimes \gamma_{M',i}\circ\nabla_{M',j}(m')
+m\otimes\nabla_{M',i}\circ\nabla_{M',j}(m').
\end{multline*}
This implies that $\nabla_{M\otimes M'}$ is integrable if $\nabla_M$ 
and $\nabla_{M'}$ are integrable. \par

It is straightforward to show that $\gamma_{M\otimes M',i}$ coincides
with $\gamma_{M,i}\otimes\gamma_{M',i}$ for $i\in \Lambda$.
This property implies that, given another module with connection 
$(M'',\nabla_{M''})$ over $(A,d)$, the isomorphism of 
$A$-modules $(M\otimes_AM')\otimes_AM''\cong M\otimes_A(M'\otimes_AM'')$
is compatible with connections. The map $d\colon A\to \Omega_{A,\ugamma}$
is an integrable connection, and the canonical isomorphism 
$A\otimes_AM\cong M$ is compatible with connections by
$d(1)=0$.

\begin{definition}\label{def:ConnTensorProdDef}
For modules with connection $(M,\nabla_M)$ and $(M',\nabla_{M'})$ over $(A,d)$,
we call $\nabla_{M\otimes M'}$ (resp.~$(M\otimes_AM',\nabla_{M\otimes M'})$)
the {\it tensor product of $\nabla_M$ and $\nabla_{M'}$}
(resp.~$(M,\nabla_M)$ {\it and} $(M',\nabla_{M'})$).
\end{definition}

\begin{remark}\label{rmk:TensorProdDRcpx}
Let $(M,\nabla_M)$ and $(M',\nabla_{M'})$ be modules with integrable connection
over $(A,d)$. For $q\in \N$ and $\bmI=(i_1,\ldots, i_q)\in \Lambda^q$, we define
$\gamma_{M',\bmI}$ to be $\gamma_{M',i_1}\circ\gamma_{M',i_2}\circ\cdots
\circ\gamma_{M',i_q}$. Then we can define a morphism of complexes
\begin{equation}\label{eq:dRcpxProd}
(M\otimes_A\Omega^{\bullet}_{A,\ugamma})\otimes_R
(M'\otimes_A\Omega^{\bullet}_{A,\ugamma})
\longrightarrow
(M\otimes_AM')\otimes_A\Omega^{\bullet}_{A,\ugamma}
\end{equation}
by sending $(m\otimes\omega_{\bmI})\otimes(m'\otimes\omega_{\bmI'})$
to $(m\otimes\gamma_{M',\bmI}(m'))\otimes\omega_{\bmI}\wedge
\omega_{\bmI'}$ for
$m\in M$, $m'\in M'$, $q,q'\in \N$, $\bmI\in \Lambda^q$,
and $\bmI'\in \Lambda^{q'}$; writing $x\wedge_{\ugamma}y$ for the
image of $x\otimes_Ry$, the compatibility with the differential maps
is reduced to the following equalities for $m$, $m'$, $\bmI$ and $\bmI'$ as above.
$$
(\nabla_i(m)\otimes\omega_i\wedge\omega_{\bmI}))\wedge_{\ugamma}(m'\otimes\omega_{\bmI'})
+(-1)^q(m\otimes\omega_{\bmI})\wedge_{\ugamma}(\nabla_i(m')\otimes\omega_i\wedge\omega_{\bmI'})
=\nabla_i(m\otimes\gamma_{M',\bmI}(m'))\otimes\omega_i\wedge\omega_{\bmI}\wedge\omega_{\bmI'}.
$$
If we are given another module with integrable connection $(M'',\nabla_{M''})$ over $(A,d)$, 
we have 
$$(x\wedge_{\ugamma} x')\wedge_{\ugamma} x''=x\wedge_{\ugamma} (x'\wedge_{\ugamma} x'')$$
for $x\in M\otimes_A\Omega_{A,\ugamma}^q$,
$x'\in M'\otimes_A\Omega_{A,\ugamma}^{q'}$, and
$x''\in M''\otimes_A\Omega_{A,\ugamma}^{q''}$.
In the case $(M,\nabla_M)=(M',\nabla_{M'})=(A,d)$, the morphism \eqref{eq:dRcpxProd}
coincides with the product \eqref{eq:DGAProdDef} of the differential graded algebra
$(\Omega^{\bullet}_{A,\ugamma},d^{\bullet})$ over $R$.
\end{remark}

Let $(M,\nabla_M)$ and $(M',\nabla_{M'})$ be modules with connection
over $(A,d)$, and assume that $\gamma_{M,i}$ and $\gamma_i$ are 
automorphisms of $M$ and $A$, respectively, for every $i\in \Lambda$.
Then one can define a connection on $M''=\Hom_A(M,M')$ by 
the following formula.
$$\nabla_{M'',i}(f)=(\nabla_{M',i}\circ f-f\circ\nabla_{M,i})\circ\gamma_{M,i}^{-1},\quad
f\in M''.$$ We use the bijectivity of $\gamma_i$ to show that the right-hand 
map  is $A$-linear. We have 
$\gamma_{M'',i}(f)=\gamma_{M',i}\circ f\circ\gamma_{M,i}^{-1}$ for 
$f\in M''$. We see that $\nabla_{M''}$ is the unique connection such that the
evaluation map 
$M\otimes_A\Hom_A(M,M')\to M'; m\otimes f\mapsto f(m)$ is
compatible with the connections. If $\nabla_M$ and
$\nabla_{M'}$ are integrable, the connection $\nabla_{M''}$ is
also integrable by the following computation for 
$i,j\in \Lambda$, $i\neq j$, $f\in M''$, and $m\in M$. 
\begin{align*}
(\nabla_i\circ\nabla_j)(f)(\gamma_i\circ\gamma_j(m))
=&(\nabla_i\circ\nabla_j(f)-\nabla_j(f)\circ\nabla_i)(\gamma_j(m))\\
=&\nabla_i\circ(\nabla_j\circ f-f\circ\nabla_j)(m)-\nabla_j(f)\circ\gamma_j\circ\nabla_i(m)\\
=&\nabla_i\circ\nabla_j\circ f(m)-
\nabla_i\circ f\circ\nabla_i(m)
-(\nabla_j\circ f\circ\nabla_i(m)-f\circ\nabla_j\circ\nabla_i(m))
\end{align*}

In order to study functoriality of modules with integrable
connection and the associated de Rham complexes
with respect to $(A,\ualpha=(\alpha_i)_{i\in\Lambda}, (\partial_i)_{i\in\Lambda})$, we  extend 
$E^{\alpha}(A)$ \eqref{eq:TwExtIsom} and $E^{\alpha_1}(E^{\alpha_2}(A))
\cong E^{\alpha_2}(E^{\alpha_1}(A))$ \eqref{eq:TwExtComp}
to ``multivariables" and give an interpretation of modules
with integrable connection. For a ring $A$ and a finite family of 
elements $\ualpha=(\alpha_{\sigma})_{\sigma\in \Sigma}$,
we define $E^{\ualpha}(A)$ to be the $A$-algebra
$A[T_{\sigma};\sigma\in\Sigma]/(T_{\sigma}(T_{\sigma}-\alpha_{\sigma}),\sigma\in\Sigma))$ equipped with the augmentation homomorphism
$\pi\colon E^{\ualpha}(A)\to A$ defined by $\pi(T_{\sigma})=0$. 
The $A$-module $E_{\ualpha}(A)$ is free with
basis $T_{\usigma}:=\prod_{\sigma\in\usigma}T_{\sigma}$ 
$(\usigma\subset\Sigma)$. 
If we decompose $\Sigma$ into a disjoint union
$\Sigma=\Sigma_1\sqcup \Sigma_2$ and set 
$\ualpha_{\nu}=(\alpha_{\sigma})_{\sigma\in \Sigma_{\nu}}$ 
$(\nu=1,2)$, then we have an isomorphism of $A$-algebras
$E^{\ualpha}(A)\cong E^{\ualpha_1}(E^{\ualpha_2}(A))$
defined by sending $T_{\sigma}$ to $T_{\sigma}$ for each $\sigma\in 
\Sigma$. Let $A'$ and $\ualpha'=(\alpha'_{\sigma'})_{\sigma'\in\Sigma'}$
be another pair of a ring and a finite family of its elements, and
suppose that we are given a ring homomorphism $f\colon A\to A'$
and a map $\psi\colon \Sigma\to \Sigma'$ such that
$f(\alpha_{\sigma})=\alpha'_{\psi(\sigma)}$ $(\sigma\in\Sigma)$.
Then $f$ extends uniquely to a ring homomorphism 
$E^{\psi}(f)\colon E^{\ualpha}(A)\to E^{\ualpha'}(A')$
sending $T_{\sigma}$ to $T_{\psi(\sigma)}$.  \par

Let $R$, $A$, $\ualpha=(\alpha_i)_{i\in\Lambda}$,
and $\partial_i$ $(i\in \Lambda)$ be as in the beginning of
this section, and let $s_i\colon A\to E^{\alpha_i}(A)$
be the $R$-homomorphism corresponding to $\partial_i$
by Proposition \ref{prop:AlphaDerivInterpret}. 
By \eqref{eq:TwDerivFamilyCond1}, we have $s_i(\alpha_j)=\alpha_j$
for $i,j\in\Lambda$, $i\neq j$. Hence, by choosing a total order $\leq$
on $\Lambda$, one can define an $R$-homomorphism 
\begin{equation}\label{eq:ExtensionRightAlgStr}
s\colon A\to E^{\ualpha}(A)\end{equation}
by composing $R$-homomorphisms
$$E^{\id_{\Sigma_{<i}}}(s_i)\colon 
E^{\ualpha_{<i}}(A)
\to E^{\ualpha_{<i}}(E^{\alpha_i}(A))\cong E^{\ualpha_{\leq i}}(A)
\quad(i\in \Lambda),$$
where $\Lambda_{<i}=\{j\in \Lambda;j<i\}$,
$\Lambda_{\leq i}=\{j\in \Lambda;j\leq i\}$, 
$\ualpha_{<i}=(\alpha_j)_{j\in \Lambda_{<i}}$,
and $\ualpha_{\leq i}=(\alpha_j)_{j\in \Lambda_{\leq i}}$. 
One can prove the following by induction on $\sharp \Lambda$.
\begin{equation}\label{eq:TwDerivMultSect}
s(a)=\sum_{i_1<i_2<\cdots<i_r}
\partial_{i_r}\circ\cdots\partial_{i_2}\circ\partial_{i_1}(a)T_{i_1}
T_{i_2}\cdots T_{i_r}\quad (a\in A)
\end{equation}
By \eqref{eq:TwDerivFamilyCond2}, 
this implies that $s$ is independent of the choice
of an order $\leq$ of $\Lambda$. We regard
$E^{\ualpha}(A)$ as an $A$-bialgebra by viewing the
canonical homomorphism $A\to E^{\ualpha}(A)$
(resp.~$s$) as a left (resp.~right) $A$-algebra structure in the following.

The $A$-bimodule $\Omega_{A,\ugamma}$ and
the map $d\colon A\to \Omega_{A,\ugamma}$ can be 
reconstructed from the $A$-bialgebra $E^{\ualpha}(A)$
with the augmentation $\pi$ similarly to the usual universal derivation
as follows. We define  an $A$-bialgebra
$P^{\ualpha}(A)$ to be the quotient of 
the $A$-bialgebra $E^{\ualpha}(A)$ by the ideal generated by
$T_iT_j$ $(i,j\in\Lambda, i\neq j)$. Let $s_P$ be the composition 
$A\xrightarrow{s} E^{\ualpha}(A)\to P^{\ualpha}(A)$, and
let $\pi_P$ be the augmentation $P^{\ualpha}(A)\to A$
induced by that of $E^{\ualpha}(A)$. 

\begin{lemma}\label{lem:TwOmegafKer}
(1) The left $A$-module $P^{\ualpha}(A)$ is free with
basis $1$, $T_i$ $(i\in \Lambda)$. \par
(2) The isomorphism of left $A$-modules
$\Omega_{A,\ugamma}\xrightarrow{\cong}\Ker (\pi_P)$
defined by sending $\omega_i$ to $T_i$ is also right $A$-linear. 
The composition of $d\colon A\to \Omega_{A,\ugamma}$
with this isomorphism is given by $a\mapsto s_P(a)-a$. 
\end{lemma}

\begin{proof}
(1) Since $T_i(T_iT_j)=\alpha_iT_iT_j$ and $T_j(T_iT_j)=\alpha_jT_iT_j$,
we see that the ideal generated by $T_iT_j$ $(i,j\in\Lambda,i\neq j)$
is the left $A$-module generated by $T_{\ui}$ $(\ui\subset \Lambda,
\sharp \ui\geq 2)$. This implies the claim.\par
(2) We obtain the claim from the following computation in $P^{\ualpha}(A)$. 
For $a\in A$, we have $s_P(a)=a+\sum_{j\in \Lambda}\partial_j(a)T_j$
and $T_is_P(a)=
T_i(a+\sum_{j\in \Lambda}\partial_j(a)T_j)
=T_i(a+\partial_i(a)\alpha_i)=T_i\gamma_i(a)$. 
\end{proof}

For a subset $\ui\subset \Lambda$, we define $\alpha_{\ui}$
to be  $(\alpha_i)_{i\in\ui}$. By applying the construction of $s$ \eqref{eq:ExtensionRightAlgStr}
to $\alpha_{\ui}$ and $\partial_i$ $(i\in \ui)$, we obtain an 
$R$-homomorphism $s_{\ui}\colon A\to E^{\ualpha_i}(A)$.
We have $s_{\ui}(\alpha_j)=\alpha_j$ for $j\in \Lambda\backslash \ui$
by  \eqref{eq:TwDerivFamilyCond1} and \eqref{eq:TwDerivMultSect}.
By construction, the following diagram is commutative 
for $\ui\subset \Lambda$, $i\in \Lambda\backslash \ui$,
and $\ui'=\ui\cup \{i\}$. 
\begin{equation}
\xymatrix{
A\ar[r]^(.46){s_{\ui}}\ar[d]_{s_{\ui'}}
& E^{\alpha_{\ui}}(A)\ar[d]^{E^{\alpha_{\ui}}(s_i)}\\
E^{\alpha_{\ui'}}(A)\ar[r]^(.42){\cong}& E^{\alpha_{\ui}}(E^{\alpha_i}(A))
}
\end{equation}
We have $s_i=s_{\{i\}}$. For $\uj\subset \Lambda$,
let $\partial_{\uj}$ denote the composition of $\partial_j$ $(j\in \uj)$
which does not depend on the choice of an order of composition 
by \eqref{eq:TwDerivFamilyCond2}. Then we have 
\begin{equation}\label{eq:TwDerivPartialMultSect}
s_{\ui}(a)=\sum_{\uj\subset\ui}\partial_{\uj}(a)\otimes T_{\uj}\quad
(a\in A)
\end{equation}
by \eqref{eq:TwDerivMultSect}, where $T_{\uj}=\prod_{j\in \uj}T_j$ for $\uj\subset\Lambda$. 
We regard $E^{\alpha_{\ui}}(A)$
as an $A$-bialgebra by viewing the canonical homomorphism
(resp.~$s_{\ui}$) as a left (resp.~right) $A$-algebra structure. 
The tensor product $E^{\alpha_{\ui}}(A)\otimes_A-$
(resp.~$-\otimes_AE^{\ualpha_i}(A)$) is taken with
respect to the right (resp.~left) $A$-algebra structure in the following. 

\begin{lemma}\label{lem:TwPartialExtTensor}
Let $\ui$ and $\ui'$ be disjoint subsets of $\Lambda$,
put $\ui''=\ui\cup \ui'$, and regard $E^{\alpha_{\ui}}(A)\otimes_AE^{\alpha_{\ui'}}(A)$ as an $A$-bialgebra by giving the left (resp.~right) algebra structure
via the left (resp.~right) algebra structure of $E^{\alpha_{\ui}}(A)$
(resp.~$E^{\alpha_{\ui'}}(A)$). Then the left $A$-algebra homomorphism 
\begin{equation}\label{eq:TwPartialExtTensor}
E^{\ualpha_{\ui''}}(A)\to E^{\alpha_{\ui}}(A)\otimes_AE^{\alpha_{\ui'}}(A)
\end{equation}
sending $T_j$ $(j\in \ui'')$ to $T_j\otimes 1$ if $j\in \ui$
and $1\otimes T_j$ if $j\in \ui'$ is well-defined,
and it is an isomorphism compatible also with the right $A$-algebra
structures. 
\end{lemma}

\begin{proof}
The homomorphism \eqref{eq:TwPartialExtTensor} is well-defined because
$s_{\ui}(\alpha_j)=\alpha_j$ for $j\in \ui'$. 
It  is an isomorphism because 
$E^{\alpha_{\ui}}(A)\otimes_AE^{\alpha_{\ui'}}(A)$
is a free $E^{\alpha_{\ui}}(A)$-module with basis
$1\otimes T_{\uj'}$ $(\uj'\subset \ui')$ and hence
is a free left $A$-module with basis $T_{\uj}\otimes T_{\uj'}$
$(\uj\subset \ui, \uj'\subset \ui')$. Let $a$ be an
element of $A$. By \eqref{eq:TwDerivPartialMultSect}, the image of 
$s_{\ui'}(a)\in E^{\alpha_{\ui'}}(A)$ in 
$E^{\alpha_{\ui}}(A)\otimes_AE^{\alpha_{\ui'}}(A)$
is $$1\otimes s_{\ui'}(a)
=1\otimes\sum_{\uj'\subset\ui'}\partial_{\uj'}(a)T_{\uj'}
=\sum_{\uj'\subset \ui'}(\sum_{\uj\subset \ui}
\partial_{\uj}(\partial_{\uj'}(a))T_{\uj})\otimes T_{\uj'},$$
which coincides with the image of 
$s_{\ui''}(a)=\sum_{\uj''\subset \ui''}\partial_{\uj''}(a)T_{\uj''}$
under \eqref{eq:TwPartialExtTensor}.
\end{proof}

\begin{remark}
Let $\ui_1,\ldots,\ui_r$  be disjoint subsets of $\Lambda$,
and put $\ui=\ui_1\cup\cdots\cup \ui_r\subset \Lambda$. 
Then, by applying Lemma
\ref{lem:TwPartialExtTensor} repeatedly, we obtain an isomorphism of $A$-bialgebras
\begin{equation}\label{eq:TwPartialExtItTensor}
E^{\ualpha}(A)
\xrightarrow{\cong}
E^{\alpha_{\ui_1}}(A)\otimes_A E^{\alpha_{\ui_2}}(A)
\otimes_A\cdots \otimes_AE^{\alpha_{\ui_r}}(A),
\end{equation}
where we regard the codomain as an $A$-bialgebra by
giving  a left (resp.~right) $A$-algebra structure via
that of $E^{\alpha_{\ui_1}}(A)$ (resp.~$E^{\alpha_{\ui_r}}(A)$).
\end{remark}

We define $\overline{E^{\ualpha}(A)\otimes_AE^{\ualpha}(A)}$
to be the quotient of $E^{\ualpha}(A)\otimes_AE^{\ualpha}(A)$
by the ideal generated by $T_i\otimes T_i$ $(i\in \Lambda)$, 
and regard it as an $A$-bialgebra by giving the left (resp.~right)
$A$-algebra structure via the left (resp.~right) $A$-algebra
structure of the left (resp.~right) $E^{\ualpha}(A)$. 

\begin{lemma}\label{lem:MultExtensionDeltaMap}
(1) $\overline{E^{\ualpha}(A)\otimes_AE^{\ualpha}(A)}$
is a free left $A$-module with basis 
$T_{\ui}\otimes T_{\ui'}$ $(\ui,\ui'\subset \Lambda, \ui\cap \ui'=\emptyset)$.\par
(2) There exists a left $A$-algebra homomorphism 
$\delta\colon E^{\ualpha}(A)\to \overline{E^{\ualpha}(A)\otimes_AE^{\ualpha}(A)}$ sending $T_i$ to $T_i\otimes 1+1\otimes T_i$, and it is also compatible
with the right $A$-algebra structures.
\end{lemma}

\begin{proof}
(1) The left $A$-module $E^{\ualpha}(A)\otimes_AE^{\ualpha}(A)$ is
free with basis $T_{\ui}\otimes T_{\ui'}$ $(\ui,\ui'\subset\Lambda)$.
By $(T_i\otimes1)\cdot (T_i\otimes T_i)=\alpha_i T_i\otimes T_i$
and $(1\otimes T_i)\cdot(T_i\otimes T_i)=T_i\otimes \alpha_i T_i
=T_i(\alpha_i+\partial_i(\alpha_i)T_i)\otimes T_i
=\gamma_i(\alpha_i)T_i\otimes T_i$,
we see that the ideal generated by $T_i\otimes T_i$ $(i\in \Lambda)$
is the left $A$-module generated by $T_{\ui}\otimes T_{\ui'}$
$(\ui,\ui'\in \Lambda,\ui\cap\ui'\neq\emptyset)$. This completes the proof.\par
(2) For $i\in \Lambda$, we have the following equalities in 
in $E^{\ualpha}(A)\otimes_AE^{\ualpha}(A)$. 
\begin{align*}
(T_i\otimes 1+1\otimes T_i)^2&=
\alpha_iT_i\otimes 1+2 T_i\otimes T_i +1\otimes \alpha_i T_i\\
&=\alpha_iT_i\otimes 1+2 T_i\otimes T_i+(\alpha_i+\partial_i(\alpha_i) T_i)\otimes T_i\\
&=\alpha_i(T_i\otimes 1+1\otimes T_i)+(2+\partial_i(\alpha_i))T_i\otimes T_i
\end{align*}
This implies the first claim.
Let $a$ be an element of $A$. By \eqref{eq:TwDerivMultSect}, we have 
equalities in $E^{\ualpha}(A)\otimes_AE^{\ualpha}(A)$
$$
1\otimes s(a)=1\otimes \sum_{\ui\subset \Lambda}\partial_{\ui}(a)T_{\ui}
=\sum_{\ui\subset\Lambda}(\sum_{\ui'\subset\Lambda}
\partial_{\ui'}(\partial_{\ui}(a))T_{\ui'})\otimes T_{\ui}
$$
and the equalities in $\overline{E^{\ualpha}(A)\otimes_AE^{\ualpha}(A)}$
$$
\delta(s(a))=\delta(\sum_{\ui\subset\Lambda}\partial_{\ui}(a)T_{\ui})
=\sum_{\ui\subset\Lambda}\partial_{\ui}(a)
\prod_{i\in \ui}(T_i\otimes 1+1\otimes T_i)
=\sum_{\ui\subset \Lambda}\sum_{\ui=\ui_1\sqcup \ui_2}
\partial_{\ui_1}(\partial_{\ui_2}(a))T_{\ui_1}\otimes T_{\ui_2}.
$$
Hence the image of $1\otimes s(a)$ in 
$\overline{E^{\ualpha}(A)\otimes_A E^{\ualpha}(A)}$
coincides with $\delta(s(a))$. 
\end{proof}

Now we are ready to interpret integrable connections over $(A,d)$
in terms of the $A$-bialgebra $E^{\ualpha}(A)$ with the augmentation $\pi$
and the map $\delta$ (Lemma \ref{lem:MultExtensionDeltaMap} (2)). 

Let $(M,\nabla_M)$ be a module with integrable connection over
$(A,d)$. For $\ui\subset \Lambda$, we regard $M\otimes_AE^{\alpha_{\ui}}(A)$
as a right $A$-module via the right $A$-algebra structure of $E^{\alpha_{\ui}}(A)$. For $i\in \Lambda$, we define the additive map
$s_{M,i}\colon M\to M\otimes_AE^{\alpha_i}(A)$ by 
$s_{M,i}(m)=m\otimes 1+\nabla_{M,i}(m)\otimes T_i$,
which is right $A$-linear by 
\begin{align*}
s_{M,i}(am)&=am\otimes 1+((a+\alpha_i\partial_i(a))\nabla_{M,i}(m)+\partial_i(a)m)
\otimes T_i\\
&=(m\otimes 1+\nabla_{M,i}(m)\otimes T_i)\cdot (a+\partial_i(a)T_i).
\end{align*}
By choosing a total order $\leq$ on $\Lambda$, we can define a
right $A$-linear map 
\begin{equation} s_M\colon M\to M\otimes_AE^{\ualpha}(A)
\end{equation}
by composing 
\begin{equation}\label{eq:ConnectionStratConstr}
M\otimes_AE^{\ualpha_{<i}}(A)
\xrightarrow{s_{M,i}\otimes \id}
M\otimes_AE^{\alpha_i}(A)\otimes_A
E^{\ualpha_{<i}}(A)
\xleftarrow[\eqref{eq:TwPartialExtTensor}]{\cong}
M\otimes_AE^{\ualpha_{\leq i}}(A).
\end{equation}
for $i\in\Lambda$, where $\Lambda_{<i}=\{j\in \Lambda; j<i\}$,
$\Lambda_{\leq i}=\{j\in\Lambda;j\leq i\}$,
$\ualpha_{<i}=(\alpha_j)_{j\in\Lambda_{<i}}$, and 
$\ualpha_{\leq i}=(\alpha_j)_{j\in\Lambda_{\leq i}}$.
For $\ui\subset \Lambda$, let $\nabla_{M,\ui}$ denote the composition
of $\nabla_{M,i}$ $(i\in \ui)$ which does not depend on the
choice of an order of composition since $\nabla_M$ is integral. 
By induction on $\sharp \Lambda_{\leq j}$, we see that the
right $A$-linear map $M\to M\otimes_AE^{\ualpha_{\leq j}}(A)$ obtained
by composing the maps \eqref{eq:ConnectionStratConstr}
for $i\in \Lambda_{\leq j}$
is given by
\begin{equation}\label{eq:ConnectoinStratFormula}
m\mapsto \sum_{\uj\subset\Lambda_{\leq j}}
\nabla_{M,\uj}(m)\otimes T_{\uj}\qquad(m\in M).
\end{equation}
This implies that $s_M$ does not depend on the choice of a total order
$\leq$ on $\Lambda$. We see that $s_M$ satisfies the 
following properties.\par
\begin{align}
&\text{The composition }
M\xrightarrow{s_M} M\otimes_AE^{\ualpha}(A)
\xrightarrow{\id_M\otimes\pi} M\text{ is the identity.}\label{eq:StratCond1}\\
&\text{The following diagram is commutative.}\label{eq:StratCond2}\\
&\xymatrix@C=40pt{
M\ar[r]^(.4){s_M}\ar[d]_{s_M}&
M\otimes_AE^{\ualpha}(A)\ar[r]^(.4){s_M\otimes\id}&
M\otimes_AE^{\ualpha}(A)\otimes_AE^{\ualpha}(A)\ar[d]\\
M\otimes_AE^{\ualpha}(A)\ar[rr]^(.45){\id_M\otimes\delta}&&
M\otimes_A\overline{E^{\ualpha}(A)\otimes_AE^{\ualpha}(A)}
}\notag
\end{align}
The first property follows from \eqref{eq:ConnectoinStratFormula}.
We obtain the second one from the following explicit computation
for $m\in M$.
\begin{align*}
(s_M\otimes\id)\circ s_M(m)
&=(s_M\otimes\id)\biggl(\sum_{\ui\subset\Lambda}\nabla_{M,\ui}(m)\otimes
T_{\ui}\biggr)
=\sum_{\ui\subset\Lambda}\sum_{\uj\subset\Lambda}
\nabla_{M,\uj}(\nabla_{M,\ui}(m))\otimes T_{\uj}\otimes T_{\ui}\\
(\id_M\otimes\delta)\circ s_M(m)&=
\sum_{\ui\subset\Lambda}\nabla_{M, \ui}(m)\otimes\prod_{i\in\ui}
(T_i\otimes 1+1\otimes T_i)
=\sum_{\ui\subset\Lambda}\sum_{\ui=\ui_1\sqcup\ui_2}
\nabla_{M,\ui}(m)\otimes T_{\ui_1}\otimes T_{\ui_2}.
\end{align*}
The construction of $s_M$ is functorial in $(M,\nabla_M)$ as follows.
If $f\colon (M,\nabla_M)\to (M',\nabla_{M'})$ is a morphism
of modules with integrable connection over $(A,d)$, then we have
$(f\otimes\id)\circ s_{M,i}=s_{M',i}\circ f$ for each $i\in\Lambda$,
which implies $(f\otimes\id)\circ s_M=s_M\circ f$. 

Conversely suppose that we are given an $A$-module $M$
and a right $A$-linear map 
$s_M\colon M\to M\otimes_AE^{\ualpha}(A)$ satisfying
\eqref{eq:StratCond1} and \eqref{eq:StratCond2}. Let $s_{M,P}$
be the composition $M\xrightarrow{s_M}M\otimes_AE^{\ualpha}(A)
\to M\otimes_AP^{\ualpha}(A)$. Since $(\id_M\otimes{\pi_P})\circ s_{M,P}$
is $\id_M$ by \eqref{eq:StratCond1}, identifying $\Omega_{A,\ugamma}$
with $\Ker\, \pi_P$ by the isomorphism in Lemma \ref{lem:TwOmegafKer} (2), 
we obtain an additive map 
$\nabla_M\colon M\to M\otimes_A\Omega_{A,\ugamma};
m\mapsto s_{M,P}(m)-m\otimes 1$. 
We see that $\nabla_M$ is a connection over $(A,d)$ as follows.
For $m\in M$ and $a\in A$, we have 
\begin{align*}
\nabla_M(am)&=s_{M,P}(m)s_P(a)-am\otimes 1\\
&=(s_{M,P}(m)-m\otimes 1)s_P(a)+m\otimes(s_P(a)-a)\\
&=\nabla_M(m)a+m\otimes d(a).
\end{align*}
This construction is obviously functorial in $(M,\nabla_M)$.
We show that \eqref{eq:StratCond2} implies that $\nabla_M$ is
integrable and $s_M(m)=\sum_{\ui\subset\Lambda}\nabla_{M,\ui}(a)
\otimes T_{\ui}$, where 
$\nabla_{M,\ui}$ denotes the composition of $\nabla_{M,i}$
$(i\in \ui)$ for $\ui\subset\Lambda$. 
We define endomorphisms $s_M^{(\ui)}$ $(\ui\subset\Lambda)$
of $M$ by $s_M(m)=\sum_{\ui\subset\Lambda}s_M^{(\ui)}(m)\otimes T_{\ui}$.
Then, for $m\in M$, we have
\begin{align*}
(s_M\otimes\id)\circ s_M(m)&=
(s_M\otimes \id)\biggl(\sum_{\ui\subset \Lambda} s_M^{(\ui)}(m)\otimes T_{\ui}\biggr)
=\sum_{\ui\subset\Lambda}\sum_{\uj\subset\Lambda}
s_M^{(\uj)}\circ s_M^{(\ui)}(m) \otimes T_{\uj}\otimes T_{\ui},\\
\delta\circ s_M(m) &=
\sum_{\ui\subset\Lambda}s_M^{(\ui)}(m)\otimes\prod_{i\in \ui}(T_i\otimes 1+1\otimes T_i)
=\sum_{\ui\subset\Lambda}\sum_{\ui=\ui_1\sqcup\ui_2}
s_M^{(\ui)}(m)\otimes T_{\ui_1}\otimes T_{\ui_2}.
\end{align*}
By Lemma \ref{lem:MultExtensionDeltaMap} (1), the property 
\eqref{eq:StratCond2} is equivalent to 
$s_M^{(\ui)}\circ s_M^{(\uj)}= s_M^{(\ui+\uj)}$
$(\ui,\uj\subset\Lambda, \ui\cap\uj=\emptyset)$. 
Since $s_M^{(\{i\})}=\nabla_{M,i}$ by the construction of $\nabla_M$,
this implies $\nabla_{M,i}\circ\nabla_{M,j}=
s_M^{(\{i,j\})}=\nabla_{M,j}\circ\nabla_{M,i}$
for $(i,j\in \Lambda, i\neq j)$, and $s_M^{(\ui)}$
is the composition of $\nabla_{M,i}$ $(i\in \ui)$ for $\ui\subset \Lambda$.

Thus we obtain the following proposition.

\begin{proposition}\label{prop:IntConPartialStrat}
The above constructions give an equivalence between the
following two categories.\par
(i) The category of modules with integrable connection over $(A,d)$.\par
(ii) The category of $A$-modules $M$ equipped with a
right $A$-linear map $s_M\colon M\to M\otimes_AE^{\ualpha}(A)$
satisfying \eqref{eq:StratCond1} and \eqref{eq:StratCond2}.
\end{proposition}

\begin{remark}\label{rmk:PratialStratProd}
To simplify the notation, we write $E$ for $E^{\ualpha}(A)$ in this remark.\par
(1) Let $M$ be an $A$-module. Since $\delta\colon E\to \overline{E\otimes_AE}$ is 
an $A$-bialgebra homomorphism, giving a right $A$-linear map 
$s_M\colon M\to M\otimes_AE$ is equivalent to giving an
$E$-linear map $\ts_M\colon E\otimes_AM\to M\otimes_AE$. 
Since the composition $A\xrightarrow{s}E\xrightarrow{\pi} A$ is the identity map,
the condition \eqref{eq:StratCond1} on $s_M$ is equivalent to the scalar
extension $\pi^*(\ts_M)\colon M\to M$ of $\ts_M$ under $\pi$
being the identity map. The condition \eqref{eq:StratCond2} is equivalent to the
commutativity of the following diagram
\begin{equation}\label{eq:StratCond2bis}
\xymatrix@C=40pt{
E\otimes_AE\otimes_AM\ar[r]^{\id_E\otimes\ts_M}\ar[d]&
E\otimes_AM\otimes_AE\ar[r]^{\ts_M\otimes_{\id_E}}&
M\otimes_AE\otimes_AE\ar[d]\\
\overline{E\otimes_AE}\otimes_AM
\ar[rr]^{\delta^*(\ts_M)}&&
M\otimes_{A}\overline{E\otimes_AE},
}
\end{equation}
where $\delta^*(\ts_M)$ denotes the scalar
extension of $\ts_M$ under the $A$-bialgebra homomorphism $\delta$ and 
the vertical maps are induced by the projection
map $E\otimes_AE\to\overline{E\otimes_AE}$.\par
(2) Let $(M,\nabla_M)$ and $(M',\nabla_{M'})$ be modules with
integrable connection over $(A,d)$, and let  
$\ts_M\colon E\otimes_AM\to M\otimes_AE$
and $\ts_{M'}\colon E\otimes_AM'\to M'\otimes_AE$
be the $E$-linear maps corresponding to $\nabla_M$ and $\nabla_{M'}$, respectively,
by the remark (1) above and Proposition \ref{prop:IntConPartialStrat}.
Let $s_{M\otimes M'}$ be the composition of the $E$-linear maps
$E\otimes_AM\otimes_AM'\xrightarrow{\ts_M\otimes\id_{M'}}
M\otimes_AE\otimes_AM'\xrightarrow{\id_M\otimes\ts_{M'}}
M\otimes_AM'\otimes_AE$. Then it is straightforward to 
verify that $\pi^*(\ts_{M\otimes M'})=\id_{M\otimes_AM'}$
and the diagram \eqref{eq:StratCond2bis} for $M\otimes_AM'$ and $\ts_{M\otimes M'}$
commutes. Let $P$ denote $P^{\ualpha}(A)$, and let
$\ts_{M,P}$, $\ts_{M',P}$, and $\ts_{M\otimes M',P}$ denote the scalar extensions
of $\ts_M$, $\ts_{M'}$, and $\ts_{M\otimes M'}$ under the projection 
$E\to P$. Then, for $m\in M$ and $m'\in M'$, we have 
\begin{multline*}
\ts_{M\otimes M,P}(1\otimes m\otimes m')
=(\id_{P}\otimes \ts_{M',P})\circ (\ts_{M,P}\otimes\id_P)(1\otimes m\otimes m')\\
=(\id_{P}\otimes \ts_{M',P})((m\otimes 1+\sum_{i\in \Lambda}\nabla_i(m)\otimes T_i)\otimes m')\\
=m\otimes m'\otimes 1+\sum_{i\in \Lambda} m\otimes\nabla_i(m')\otimes T_i
+\sum_{i\in\Lambda}(\nabla_i(m)\otimes m'+\alpha_i\nabla_i(m)\otimes\nabla_i(m'))\otimes T_i.
\end{multline*}
Note that we have $T_i^2=\alpha_i T_i$ and $T_iT_j=0$ $(i\neq j)$ in $P$.
This shows that the integrable connection on $M\otimes_AM'$ corresponding
to $\ts_{M\otimes M}$ coincides with the tensor product 
$\nabla_{M\otimes M'}$ of $\nabla_M$ and $\nabla_{M'}$
(Definition \ref{def:ConnTensorProdDef}, \eqref{eq:ConnProdSymmForm}).
\end{remark}

\section{Scalar extensions of integrable connections}\label{sec:ScalarExtConnection}
We discuss the functoriality of modules with integrable connection
and the associated de Rham complexes with respect to $(A,\ualpha,(\partial_i))$ by using 
Proposition \ref{prop:IntConPartialStrat}.

Let $A'$ be an  algebra over a ring $R'$, let $\ualpha'=(\alpha'_{i'})_{i'\in\Lambda'}$ be a family of elements
of $A'$, and suppose that we are given an $\alpha'_{i'}$-derivation 
$\partial_{i'}'$ of $A'$ over $R'$ for each $i'\in\Lambda'$
satisfying the same conditions as \eqref{eq:TwDerivFamilyCond1} and \eqref{eq:TwDerivFamilyCond2}, 
i.e., $\partial'_{i'}(\alpha'_{j'})=0$
$(i',j'\in\Lambda', i'\neq j')$ and $\partial'_{i'}\circ\partial'_{j'}
=\partial'_{j'}\circ\partial'_{i'}$ $(i',j'\in\Lambda')$.
Put $\gamma'_{i'}=\id_{A'}+\alpha_{i'}'\partial'_{i'}$ $(i'\in\Lambda')$
and $\ugamma'=(\gamma'_{i'})_{i'\in\Lambda'}$. 
We define $d'\colon A'\to \Omega_{A',\ugamma'}$ and 
$s'\colon A'\to E^{\ualpha'}(A')$ in the same way as
$d\colon A\to \Omega_{A,\ugamma}$ and
$s\colon A\to E^{\ualpha}(A)$ \eqref{eq:ExtensionRightAlgStr}
by using $\alpha'_{i'}$ and 
$\partial'_{i'}$ $(i'\in \Lambda')$. 

Suppose that we are given a ring homomorphism $f\colon R\to R'$,
a ring homomorphism $g\colon A\to A'$ compatible with $f$,
a map $\psi\colon \Lambda\to \Lambda'$, and $\uc=(c_i)_{i\in\Lambda}
\in(A')^{\Lambda}$ such that 
\begin{align}
g(\alpha_i)&=c_i\alpha'_{\psi(i)}\qquad (i\in\Lambda),\label{eq:DerFamMorphCond1}\\
\partial'_{i'}(c_i)&=0\qquad (i\in\Lambda,\;i'\in \Lambda'\backslash\{\psi(i)\}).
\label{eq:DerFamMorphCond2}
\end{align}
The triplet $(g,\psi,\uc)$ defines a homomorphism
\begin{equation}\label{eq:TwExtAlgMap}
E^{\psi,\uc}(g)\colon E^{\ualpha}(A)\to E^{\ualpha'}(A')
\end{equation} 
compatible with $g$ and sending $T_i$ to $c_iT_{\psi(i)}$ for $i\in\Lambda$. Note that
we have $(c_iT_{\psi(i)})^2=c_i^2\alpha'_{\psi(i)}T_{\psi(i)}=g(\alpha_i)c_iT_{\psi(i)}$
in $E^{\ualpha'}(A')$. We abbreviate $E^{\psi,\uc}(g)$ to $E^{\psi}(g)$ when
$c_i=1$ for all $i\in\Lambda$.
 For an $A$-bialgebra $C$ and an $A'$-bialgebra $C'$,
we say that a ring homomorphism $C\to C'$ is a {\it homomorphism 
of bialgebras over} $g$ if it is compatible with $g$ 
for both the left algebra structures and the right algebra structures.
We define a {\it homomorphism of left (resp.~right) algebras over $g$}
similarly. The homomorphism $E^{\psi,\uc}(g)$ is a homomorphism
of left algebras over $g$. We consider the following condition.
\begin{equation}\label{eq:ConnFunctCond}
\text{ The homomorphism }E^{\psi,\uc}(g)\text{ is a homomorphism of
right algebras over }g.
\end{equation}

\begin{lemma}\label{lem:TwDerivSystemFunct}
(1) The condition \eqref{eq:ConnFunctCond} is satisfied if and only if
\begin{equation}\label{eq:TwDerivSystemFunct1}
\partial_{i'}'(g(a))=\sum_{\emptyset\neq \ui\subset \psi^{-1}(i')}
g(\partial_{\ui}(a))\prod_{i\in \ui}c_i\cdot(\alpha_{i'}')^{\sharp \ui-1}\quad\text{for every }
i'\in \Lambda' \text{ and }a\in A.
\end{equation}
If $\psi$ is injective, the right-hand side is $c_ig(\partial_i(a))$ if $\psi^{-1}(i')=\{i\}$
and $0$ if $\psi^{-1}(i')=\emptyset$. \par
(2) The condition \eqref{eq:ConnFunctCond} implies 
\begin{equation}\label{eq:TwDerivSystemFunct2}
\gamma_{i'}'(g(a))=g\biggl(\biggl(\prod_{i\in\psi^{-1}(i')}\gamma_i\biggr)(a)\biggr)
\quad\text{ for every }i'\in \Lambda'\text{ and }a\in A.
\end{equation}
The converse holds if $\alpha'_{i'}\in A'$ is regular for every 
$i'\in \Lambda'$. 
\end{lemma}

\begin{proof}
(1) For $i'\in \Lambda'$, put $\ualpha_{i'}=(\alpha_i)_{i\in\psi^{-1}(i')}$
and $\uc_{i'}=(c_i)_{i\in\psi^{-1}(i')}$, 
and let $\psi_{i'}$ be the map $\psi^{-1}(i')\to \{i'\}$ (induced by $\psi$). 
Then the homomorphism 
$E^{\psi_{i'},\uc_{i'}}(g)\colon E^{\ualpha_{i'}}(A)\to E^{\alpha'_{i'}}(A')$
maps $s_{\psi^{-1}(i')}(a)=
\sum_{\ui\subset\psi^{-1}(i')}\partial_{\ui}(a)T_{\ui}$
\eqref{eq:TwDerivPartialMultSect} for $a\in A$ to 
$\sum_{\ui\subset \psi^{-1}(i')}g(\partial_{\ui}(a))(\prod_{i\in\ui}c_i)T_{i'}^{\sharp \ui}
=g(a)+\sum_{\emptyset\neq\ui\subset \psi^{-1}(i')}
g(\partial_{\ui}(a))(\prod_{i\in\ui}c_i)(\alpha_{i'}')^{\sharp \ui -1}T_{i'}$.
Comparing it with $s'_{i'}(g(a))=g(a)+\partial'_{i'}(g(a))T_{i'}$,
we see that the latter condition in the claim (1) is equivalent
to saying that $E^{\psi_{i'}}(g)$ is a homomorphism 
of right algebras over $g$ for every $i'\in \Lambda'$. \par

First let us prove the necessity. We have a commutative diagram 
\begin{equation*}
\xymatrix{
E^{\ualpha}(A)\ar[r]\ar[d]_{E^{\psi,\uc}(g)}
& E^{\ualpha_{i'}}(A)\ar[d]^{E^{\psi_{i'},\uc_{i'}}(g)}\\
E^{\ualpha'}(A')\ar[r]
& E^{\alpha_{i'}}(A'),
}
\end{equation*}
where the upper (resp.~lower) horizontal map is
a left $A$ (resp.~$A'$)-algebra homomorphism defined by 
$T_i$ $(i\in\Lambda)$ $\mapsto$ $T_i$  if $\psi(i)=i'$, $0$
if $\psi(i)\neq i'$ 
(resp.~$T_{j'}$ $(j'\in \Lambda')$ $\mapsto$ $T_{i'}$
if $j'=i'$, $0$ if $j'\neq i'$), which is also a right $A$
(resp.~$A'$)-algebra homomorphism by \eqref{eq:TwDerivPartialMultSect}.
Hence if $E^{\psi,\uc}(g)$ is a homomorphism of bialgebras over $g$,
so is the homomorphism $E^{\psi_{i'},\uc_{i'}}(g)$ for every $i'\in \Lambda'$.

Conversely suppose that $E^{\psi_{i'},\uc_{i'}}(g)$ is a homomorphism
of bialgebras over $g$ for every $i'\in \Lambda'$. 
Put $r=\sharp \Lambda'$ and identify $\Lambda'$ with 
$\N\cap [1,r]$ by choosing a bijective map between them. 
For $i'\in\Lambda'$ and $j\in \Lambda\backslash \psi^{-1}(i')$, 
we have $s'_{i'}(c_j)=c_j+\partial'_{i'}(c_j)T_{i'}=c_j$ by the assumption 
\eqref{eq:DerFamMorphCond2} on $\uc$. 
Hence $E^{\psi,\uc}(g)$ coincides with the composition
\begin{multline*}
E^{\ualpha}(A)\xrightarrow[\eqref{eq:TwPartialExtItTensor}]{\cong}
E^{\ualpha_1}(A)\otimes_AE^{\ualpha_2}(A)
\otimes_A\cdots\otimes_AE^{\ualpha_r}(A)\\
\xrightarrow{E^{\psi_1,\uc_1}(g)\otimes E^{\psi_2,\uc_2}(g)\otimes\cdots\otimes E^{\psi_r,\uc_r}(g)}
E^{\alpha'_1}(A')\otimes_{A'}E^{\alpha'_2}(A')
\otimes_{A'}\cdots\otimes_{A'}E^{\alpha'_r}(A')
\xleftarrow[\eqref{eq:TwPartialExtItTensor}]{\cong}E^{\ualpha'}(A'),
\end{multline*}
where the second map is well-defined by the assumption that
$E^{\psi_n,\uc_n}(g)$ $(n\in \N\cap [1,r])$ are homomorphisms of bialgebras
over $g$. 
Hence $E^{\psi,\uc}(g)$ is a homomorphism of right algebras over $g$
because so is the homomorphism $E^{\psi_r,\uc_r}(g)$ by assumption. \par
(2) The claim follows from the following computation for $a\in A$
and $i'\in\Lambda'$. Note that we have \eqref{eq:TwDerivFamilyCond2} and
$\alpha_i\id_A\circ \partial_j=\partial_j\circ\alpha_i\id_A$ 
for $i,j\in\Lambda$, $i\neq j$ by \eqref{eq:TwDerivFamilyCond1}. 
\begin{multline*}
g(a)+
{\textstyle 
\alpha_{i'}'
\suml_{\emptyset\neq \ui\subset \psi^{-1}(i')}
g(\partial_{\ui}(a))(\prod_{i\in\ui}c_i)(\alpha'_{i'})^{\sharp \ui-1}
=g\biggl(\suml_{\ui\subset \psi^{-1}(i')}\biggl(\prodl_{i\in \ui}\alpha_i\partial_i\biggr)(a)\biggr)}\\
{\textstyle=g\biggl(\biggl(\prodl_{i\in \psi^{-1}(i')}(\id_A+\alpha_i\partial_i)\biggr)(a)\biggr)
=g\biggl(\biggl(\prodl_{i\in\psi^{-1}(i')}\gamma_i\biggr)(a)\biggr).}
\end{multline*}
\end{proof}

Assume that the condition \eqref{eq:ConnFunctCond} holds in the following. 

Let $(M,\nabla_M)$ be a module with integrable connection over $(A,d)$
(Definitions \ref{def:connection} and \ref{def:intconnection}),
and let $s_M\colon M\to M\otimes_AE^{\ualpha}(A)$ be the right
$A$-linear map associated to $\nabla_M$ by Proposition \ref{prop:IntConPartialStrat}. Put $M'=M\otimes_AA'$. Then 
we obtain right $A'$-linear homomorphisms
\begin{multline}\label{eq:PartialStratBC}
M'=M\otimes_AA'
\xrightarrow{s_M\otimes\id_{A'}}
M\otimes_AE^{\ualpha}(A)\otimes_AA'
\to M\otimes_AA'\otimes_AE^{\ualpha}(A)\otimes_AA'\\
\to M'\otimes_{A'}E^{\ualpha}(A'),
\end{multline}
where the last map is induced by the $A'$-bialgebra homomorphisms
$A'\otimes_AE^{\ualpha}(A)\otimes_AA'\to E^{\ualpha}(A')$
induced by $E^{\psi,\uc}(g)$ which is a bialgebra homomorphism over $g$
by the condition \eqref{eq:ConnFunctCond}.
Let $s_{M'}$ be the composition of \eqref{eq:PartialStratBC}.

\begin{lemma}\label{lem:ConnectionScalarExt}
The right $A'$-linear map $s_{M'}$ satisfies the conditions
\eqref{eq:StratCond1} and \eqref{eq:StratCond2}. 
\end{lemma}

\begin{proof}
Since the relevant maps are all right $A'$-linear, it 
suffices to check the conditions after composing with 
$g_M\colon M\to M\otimes_AA';m\mapsto m\otimes 1$. By the construction of $s_{M'}$,
the diagram
\begin{equation*}
\xymatrix@C=50pt{
M\ar[r]^(.35){s_M}\ar[d]_{g_M}&
M\otimes_AE^{\ualpha}(A)\ar[d]^{g_M\otimes E^{\psi,\uc}(g)}\\
M'\ar[r]^(.35){s_{M'}}&
M'\otimes_{A'}E^{\ualpha'}(A')&
}
\end{equation*}
is commutative. Hence the claim is reduced to the conditions \eqref{eq:StratCond1} and \eqref{eq:StratCond2} for $M$ and $s_M$ by the 
following facts. The homomorphisms $E^{\psi,\uc}(g)$ and $g$ are compatible with the augmentation maps $\pi$ from $E^{\ualpha}(A)$ and $E^{\ualpha'}(A')$
to $A$ and $A'$; the tensor product
$E^{\psi,\uc}(g)\otimes E^{\psi,\uc}(g)
\colon E^{\ualpha}(A)\otimes_A E^{\ualpha}(A)
\to E^{\ualpha'}(A')\otimes_{A'}E^{\ualpha'}(A')$
of the bialgebra homomorphism $E^{\psi,\uc}(g)$ over $g$ induces a bialgebra homomorphism 
$\overline{E^{\psi,\uc}(g)\otimes E^{\psi,\uc}(g)}
\colon \overline{E^{\ualpha}(A)\otimes_A E^{\ualpha}(A)}
\to \overline{E^{\ualpha'}(A')\otimes_{A'}E^{\ualpha'}(A')}$
over $g$,
which is compatible with the
maps $\delta$ (Lemma \ref{lem:MultExtensionDeltaMap} (2))
for $(A,\ualpha,(\partial_i))$
and $(A',\ualpha',(\partial'_{i'}))$ since we have
$1\otimes c_iT_{\psi(i)}=
s'(c_i)\otimes T_{\psi(i)}=(c_i+\partial_{\psi(i)}'(c_i)T_{\psi(i)})\otimes T_{\psi(i)}$ $(i\in\Lambda)$
in $E^{\ualpha'}(A')\otimes_{A'}E^{\ualpha'}(A')$ by the assumption 
\eqref{eq:DerFamMorphCond2} on $c_i$. 
\end{proof}

\begin{definition}\label{def:connectionScalarExt}
Under the notation and assumption as above, 
we define the {\it scalar extension $\nabla_{M'}\colon M'\to M'\otimes_{A'}\Omega_{A',\ugamma'}$
of $\nabla_M$ under $(f,g,\psi,\uc)$ (or simply $g$)} 
to be the integrable connection on
$M'$ 
corresponding to $s_{M'}$ by Proposition \ref{prop:IntConPartialStrat}. 
This construction is functorial in $(M,\nabla_M)$, and defines a functor
\begin{equation}
(f,g,\psi,\uc)^*\colon \MIC(A,d)\longrightarrow \MIC(A',d').
\end{equation}
We abbreviate $(f,g,\psi,\uc)^*$ to $(f,g,\psi)^*$ when $c_i=1$ for all $i\in\Lambda$.
\end{definition}

\begin{remark}\label{rmk:ConnScalarExtStratTensProd}
To simplify the notation,  we write $E$, $E'$ and $E(g)$ for
$E^{\ualpha}(A)$, $E^{\ualpha'}(A')$, and $E^{\psi,\uc}(g)$, respectively, 
in this remark.\par
(1) Under the notation and assumption before 
Lemma \ref{lem:ConnectionScalarExt}, 
let $\ts_M$ (resp.~$\ts_{M'}$) be the $E$-linear (resp.~$E'$-linear)
extension $E\otimes_AM\to M\otimes_AE$
 (resp.~$E'\otimes_{A'}M'\to M'\otimes_{A'}E'$) of $s_M$ (resp.~$s_{M'}$).
(See Remark \ref{rmk:PratialStratProd} (1).) Then $\ts_{M'}$ coincides with the scalar extension of 
$\ts_M$ under the bialgebra homomorphism $E(g)\colon E\to E'$
over $g$.\par
(2) Let $(N,\nabla_N)$ be another module with integrable connection 
over $(A,d)$, let $\nabla_{M\otimes N}$ be the tensor product of 
$\nabla_M$ and $\nabla_N$ (Definition \ref{def:ConnTensorProdDef}). 
Let $\nabla_{M'}$ and $\nabla_{N'}$ be
the scalar extensions of $\nabla_M$ and $\nabla_N$, respectively,
under $(f,g,\psi,\uc)$ (Definition \ref{def:connectionScalarExt}). Then, 
by the remark (1) above and Remark \ref{rmk:PratialStratProd} (2), we see that
the scalar extension of $\nabla_{M\otimes N}$
under $(f,g,\psi,\uc)$ coincides with  the tensor product
of $\nabla_{M'}$ and $\nabla_{N'}$.
\end{remark}

\begin{proposition}\label{prop:ConnPBFormula}
The following equalities hold in $M'=M\otimes_AA'$ for $i'\in \Lambda'$ and $m\in M$.\par
\begin{align}
\nabla_{M',i'}(m\otimes 1)&=
\sum_{\emptyset\neq \ui\subset \psi^{-1}(i')}
\nabla_{M,\ui}(m)\otimes(\prod_{i\in\ui}c_i)(\alpha'_{i'})^{\sharp \ui -1},
\label{eq:ConnPBFormula1}\\
\gamma_{M',i'}(m\otimes 1)&=
\biggl(\prod_{i\in \psi^{-1}(i')}\gamma_{M,i}\biggr)(m)\otimes 1.
\label{eq:ConnPBFormula2}
\end{align}
In the case $\psi$ is injective, the right-hand side of \eqref{eq:ConnPBFormula1}
is $\nabla_{M,i}(m)\otimes c_i$ if $\psi^{-1}(i')=\{i\}$
and $0$ if $\psi^{-1}(i')=\emptyset$.\par

\end{proposition}

\begin{proof} Let $g_M$ be the morphism $M\to M'=M\otimes_AA';m\mapsto m\otimes 1$. 
The first equality follows from the observation that the image of 
$s_M(m)=\sum_{\ui\subset\Lambda}\nabla_{M,\ui}(m)\otimes T_{\ui}$
under $g_M\otimes E^{\psi,\uc}(g)\colon M\otimes_A E^{\ualpha}(A)
\to M'\otimes_{A'}E^{\ualpha'}(A')$ is 
$\sum_{\ui\subset\Lambda}\nabla_{M,\ui}(m)
\otimes \prod_{i\in\ui}c_i\prod_{i'\in \psi(\ui)}(\alpha'_{i'})^{\sharp (\ui\cap \psi^{-1}(i'))-1}
T_{i'}$. The second equality is derived from the first one by the following
computation.
\begin{align*}
(\id_{M'}+\alpha_{i'}'\nabla_{M',i'})(m\otimes 1)
&=m\otimes 1+{\textstyle \suml_{\emptyset\neq \ui\subset\psi^{-1}(i')}
\biggl(\prodl_{i\in\ui}\alpha_i\nabla_{M,i}\biggr)(m)\otimes 1}\\
&={\textstyle\biggl(\prodl_{i\in\psi^{-1}(i')}\gamma_{M,i}\biggr)(m)\otimes 1}
\end{align*}
Note that we have $\alpha_i\id_M\circ \nabla_{M,j}=\nabla_{M,j}\circ\alpha_i\id_M$ 
for $i,j\in \Lambda$, $i\neq j$ by \eqref{eq:TwDerivFamilyCond1}.
\end{proof}

We discuss behavior of de Rham complexes under the scalar extension
by $(f,g,\psi,\uc)$. 
To have compatibility with differential maps, we need to twist 
pullback maps in positive degrees by using the endomorphisms 
$\gamma_i$ and $\gamma_{M,i}$ in a suitable manner 
unless $\psi$ is injective. Indeed, by the
formula \eqref{eq:TwDerivSystemFunct1}, the additive map 
$\Omega^q_{A,\ugamma}\to \Omega^q_{A',\ugamma'}$ $(q\in \N)$
defined by $a\omega_{\bmI}\mapsto g(a)c_{\bmI}\omega_{\psi^q(\bmI)}$
$(a\in A, \bmI\in\Lambda^q)$ is not compatible with the differential
maps when $\psi\colon \Lambda\to \Lambda'$ is not injective.
Here and hereafter $\psi^q$ denotes
the product $\Lambda^q\to \Lambda^{\prime q};
(i_n)_{1\leq n\leq q}\mapsto (\psi(i_n))_{1\leq n\leq q}$ of the map $\psi$,
and $c_{\bmI}$ for $\bmI=(i_n)_{1\leq n\leq q}\in \Lambda^q$
denotes the product $\prod_{n=1}^qc_{i_n}\in A'$.
\par

For $a\in A$ and $i'\in\Lambda'$, if we set $\psi^{-1}(i')=\{i_1,\ldots, i_r\}$, 
we obtain the following equalities from \eqref{eq:TwDerivSystemFunct2}.
\begin{multline*}
\alpha'_{i'}\partial_{i'}'(g(a))
=(\gamma'_{i'}-1)(g(a))
=g((\prod_{1\leq n\leq r}\gamma_{i_n})(a)-a)
=\sum_{1\leq n\leq r}
g((\prod_{1\leq m<n}\gamma_{i_m})(\gamma_{i_n}(a)-a))\\
=\alpha_{i'}'
\sum_{1\leq n\leq r}c_{i_n}\cdot g((\prod_{1\leq m<n}\gamma_{i_m})(\partial_{i_n}(a))).
\end{multline*}
Note that this computation depends on the choice of the order
$i_1,\ldots, i_r$ of elements of $\psi^{-1}(i')$. 
This observation leads us to define twisted additive maps 
$\Omega^q_{A,\ugamma}\to \Omega^q_{A',\ugamma'}$
$(q\in \N)$ as follows. 

We choose and fix a total order $<$ of $\Lambda$ in the following.
For each $i\in \Lambda$, let $\Lambda^{<}_{\psi,i}$ be the subset
$\{j\in \Lambda\,\vert\, j\in \psi^{-1}(\psi(i)), j<i\}$ of $\Lambda$,
and let $\gamma_{\psi,i}^{<}$ denote $\prod_{j\in \Lambda_{\psi,i}^<}\gamma_j$.
For $\bmI=(i_n)_{1\leq n\leq q}\in\Lambda^q$, 
we define $\gamma_{\psi,\bmI}^{<}$ to be
$\prod_{1\leq n\leq q}\gamma_{\psi,i_n}^{<}$.
Then we define additive maps $\Omega^q_{g/f,\psi,\uc}\colon \Omega^q_{A,\ugamma}
\to \Omega_{A',\ugamma'}^q$ by 
\begin{equation}
\Omega^q_{g/f,\psi,\uc}(a\omega_{\bmI})=g(\gamma_{\psi,\bmI}^{<}(a))c_{\bmI}\omega_{\psi^q(\bmI)}
\qquad (a\in A, \;\bmI\in\Lambda^q).
\end{equation}

\begin{proposition}\label{prop:dRCpxFunctConst}
The additive maps $\Omega^q_{g/f,\psi,\uc}$ $(q\in \N)$
define a morphism of complexes
\begin{equation}\label{eq:dRCpxFunctConst}
\Omega^{\bullet}_{g/f,\psi,\uc}\colon
\Omega^{\bullet}_{A,\ugamma}\to \Omega^{\bullet}_{A',\ugamma'}.
\end{equation}
We abbreviate $\Omega^{\bullet}_{g/f,\psi,\uc}$ to 
$\Omega^{\bullet}_{g/f,\psi}$ when $c_i=1$ for all $i\in\Lambda$. 
\end{proposition}

\begin{proof}
We first prove $\Omega^1_{g/f,\psi,\uc}\circ d=d'\circ g$.
It suffices to verify that the right-hand side of the
equality \eqref{eq:TwDerivSystemFunct1} coincides with 
$\sum_{i\in \psi^{-1}(i')}c_i\cdot g(\gamma_{\psi,i}^{<}(\partial_i(a)))$
for $i'\in \Lambda'$ and $a\in A$. 
We have 
$$g(\gamma_{\psi,i}^{<}(\partial_i(a)))
=g\Biggl(\Biggl(\prod_{j\in \Lambda^{<}_{\psi,i}}(\id_A+\alpha_j\partial_j)\Biggr)(\partial_i(a))\Biggr)
=\sum_{\uj\subset\Lambda^{<}_{\psi, i}}(\alpha_{i'}')^{\sharp \uj}\Biggl(\prod_{j\in \uj}c_j\Biggr)
g(\partial_{\uj\cup\{i\}}(a))$$
for each $i\in \psi^{-1}(i')$. This shows the desired claim since $\uj\sqcup \{i\}$
is a  subset  of $\psi^{-1}(i')$ for $i\in \psi^{-1}(i')$ and 
$\uj\subset\Lambda_{\psi,i}^{<}$, 
and every non-empty
subset of $\psi^{-1}(i')$ is uniquely written in this form.\par
We can derive the equalities 
$\Omega^{q+1}_{g/f,\psi,\uc}\circ d^q
=d^{\prime q}\circ\Omega^q_{g/f,\psi,\uc}$ $(q\geq 1)$
from the above claim as follows. For $a\in A$ and
$\bmI\in \Lambda^q$ without common components,
we have
\begin{align*}
\Omega^{q+1}_{g/f,\psi,\uc}
(d^q(a\omega_{\bmI}))
&=\sum_{i\in\Lambda}c_ic_{\bmI}\cdot g(\gamma_{\psi,i}^{<}
\circ\gamma_{\psi,\bmI}^{<}(\partial_i(a)))\otimes
\omega_{\psi(i)}\wedge\omega_{\psi^q(\bmI)},\\
d^{\prime q}(\Omega^q_{g/f,\psi,\uc}(a\omega_{\bmI}))
&=
\sum_{i'\in \Lambda'}\partial'_{i'}(c_{\bmI}\cdot g(\gamma_{\psi,\bmI}^{<}(a)))
\otimes\omega_{i'}\wedge\omega_{\psi^q(\bmI)}.
\end{align*}
Hence we are reduced to showing 
$$\partial'_{i'}(c_{\bmI}\cdot g(\gamma_{\psi,\bmI}^{<}(a)))
=\sum_{i\in \psi^{-1}(i')}
c_ic_{\bmI}\cdot g(\gamma_{\psi,i}^{<}\circ\gamma_{\psi,\bmI}^{<}
(\partial_i(a)))$$
for $i'\in \Lambda'$ not appearing in $\psi^q(\bmI)$.
This follows from the claim shown in the previous paragraph
since $\gamma_{\psi,\bmI}^{<}(\partial_i(a))
=\partial_i(\gamma^{<}_{\psi,\bmI}(a))$ for $i\in \psi^{-1}(i')$ and
$\partial_{i'}'\circ c_{\bmI}\id_{A'}=c_{\bmI}\id_{A'}\circ \partial'_{i'}$
by the assumption on $i'$ and $\psi^q(\bmI)$.
\end{proof}

Similarly one can construct a morphism from the 
de Rham complexe of $(M,\nabla_M)\in \Ob\MIC(A,d)$
to that of its scalar extension $(M',\nabla_{M'})\in \Ob\MIC(A',d')$
under $(f,g,\psi,\uc)$ as follows. We define 
endomorphisms $\gamma_{M,\psi,i}^{<}$
$(i\in \Lambda)$ and $\gamma_{M,\psi,\bmI}^{<}$ $(\bmI\in \Lambda^q)$
of $M$ 
in the same way as the endomorphisms $\gamma_{\psi,i}^{<}$
and $\gamma_{\psi,\bmI}^{<}$ of $A$ by using $\gamma_{M,i}
=\id_M+\alpha_i\nabla_{M, i}$ $(i\in \Lambda)$, which commute
with each other by the integrability of $\nabla_M$ and \eqref{eq:TwDerivFamilyCond1}. 
As $\gamma_{M,\psi,\bmI}^{<}$ is $\gamma_{\psi,\bmI}^{<}$-semilinear,
the biadditive map
$$M\times \Omega^q_{A,\ugamma}
\to M'\otimes_{A'}\Omega^q_{A',\ugamma'};
(m,a\omega_{\bmI})
\mapsto (\gamma_{M,\psi,\bmI}^{<}(m)\otimes 1)
\otimes\Omega^q_{f/g,\psi,\uc}(a\omega_{\bmI})$$
is an $A$-balanced product, and therefore defines 
an additive map 
$$\Omega^q_{g/f,\psi,\uc}(M)\colon 
M\otimes_A\Omega^q_{A,\ugamma}
\to M'\otimes_{A'}\Omega^q_{A',\ugamma'}.
$$

\begin{proposition}\label{prop:dRCpxFunct}
The additive maps $\Omega^q_{g/f,\psi,\uc}(M)$ $(q\in \N)$
define a morphism of complexes
\begin{equation}\label{eq:dRCpxFunct}
\Omega^{\bullet}_{g/f,\psi,\uc}(M)\colon 
M\otimes_A\Omega^{\bullet}_{A,\ugamma}
\to M'\otimes_{A'}\Omega^{\bullet}_{A',\ugamma'}.
\end{equation}
We abbreviate $\Omega^{\bullet}_{g/f,\psi,\uc}(M)$ to
$\Omega^{\bullet}_{g/f,\psi}(M)$ when $c_i=1$ for all $i\in\Lambda$.
\end{proposition}

\begin{proof}
We can prove the proposition exactly in the same way as the proof 
of Proposition \ref{prop:dRCpxFunctConst} by using \eqref{eq:ConnPBFormula1} 
instead of \eqref{eq:TwDerivSystemFunct1}. 
We repeat the argument for the convenience of readers. \par
First the equality $\Omega^1_{g/f,\psi,\uc}(M)\circ\nabla_M
=\nabla_{M'}\circ\Omega^0_{g/f,\psi,\uc}(M)$ is reduced to
showing that the right-hand side of \eqref{eq:ConnPBFormula1} coincides with
$\sum_{i\in \psi^{-1}(i')}
\gamma_{M,\psi,i}^{<}(\nabla_{M,i}(m))\otimes c_i$
for $i'\in\Lambda'$ and $m\in M$.
This  follows from the equality in $M'=M\otimes_AA'$
$$\gamma_{M,\psi,i}^{<}(\nabla_{M,i}(m))\otimes 1
=\Biggl(\prod_{j\in \Lambda_{\psi,i}^{<}}(\id_M+\alpha_j\nabla_{M,j})\Biggr)(\nabla_{M,i}(m))\otimes 1
=\sum_{\uj\subset\Lambda_{\psi,i}^{<}}\nabla_{M,\uj\cup\{i\}}(m)
\otimes\Biggl(\prod_{j\in \uj}c_j\Biggr)(\alpha'_{i'})^{\sharp \uj}$$
for $i\in \psi^{-1}(i')$ and $m\in M$. \par
For $q\geq 1$, $m\in M$, and $\bmI\in \Lambda^q$ without
common components, we have
\begin{align*}
\Omega^{q+1}_{g/f,\psi,\uc}(M)\circ\nabla_M^q(m\otimes\omega_{\bmI})
&=\sum_{i\in\Lambda}
(\gamma_{M,\psi,i}^{<}\circ\gamma_{M,\psi,\bmI}^{<}
(\nabla_{M,i}(m))\otimes c_ic_{\bmI})\otimes \omega_{\psi(i)}\wedge\omega_{\psi^q(\bmI)},\\
\nabla_{M'}^q\circ\Omega^q_{g/f,\psi,\uc}(M)
(m\otimes\omega_{\bmI})
&=\sum_{i'\in\Lambda'}
\nabla_{M',i'}(\gamma_{M,\psi,\bmI}^{<}(m)\otimes c_{\bmI})
\otimes\omega_{i'}\wedge\omega_{\psi^q(\bmI)}.
\end{align*}
Hence the compatibility with the $q$th differential maps
is reduced to 
$$\nabla_{M',i'}(\gamma_{M,\psi,\bmI}^{<}(m)\otimes c_{\bmI})
=\sum_{i\in \psi^{-1}(i')}\gamma_{M,\psi,i}^{<}
\circ\gamma_{M,\psi,\bmI}^{<}(\nabla_{M,i}(m))
\otimes c_ic_{\bmI}$$
for $i'\in \Lambda'$ not appearing in $\psi^q(\bmI)$.
By the claim shown in the previous paragraph, this is reduced to 
$\gamma_{M,\psi,\bmI}^{<}(\nabla_{M,i}(m))=
\nabla_{M,i}(\gamma_{M,\psi,\bmI}^{<}(m))$
for $i\in \psi^{-1}(i')$
and $\nabla_{M',i'}\circ c_{\bmI}\id_{M'}
=c_{\bmI}\id_{M'}\circ\nabla_{M',i'}$, 
which hold by the assumption on $i'$ and 
$\psi^q(\bmI)$. 
\end{proof}

\begin{remark}\label{rmk:dRCpxProdScExtComp}
Suppose that $\psi$ is injective. Then, for $(M,\nabla_M)\in \Ob \MIC(A,d)$
and its scalar extension $(M',\nabla_{M'})$ under $(f,g,\psi,\uc)$
(Definition \ref{def:connectionScalarExt}), we have
\begin{equation}\label{eq:dRCpxScExtPsiInj}
\Omega^q_{g/f,\psi,\uc}(M)(m\otimes\omega_{\bmI})=m\otimes c_{\bmI}\otimes \omega_{\psi^q(\bmI)}
\end{equation}
for $q\in \N$, $m\in M$, and $\bmI\in \Lambda^q$.
This together with \eqref{eq:TwDerivSystemFunct2}
and $\gamma_{i'}'(c_i)=c_i$ for $i\in \Lambda$ and $i'\in \Lambda'\backslash \psi(i)$
implies the following. Let $(M,\nabla_{M})$ and $(N,\nabla_{N})$
be modules with integrable connection over $(A,d)$, and let
$(M',\nabla_{M'})$ and $(N',\nabla_{N'})$ be their scalar extensions
under $(f,g,\psi,\uc)$. Then the product morphisms 
of de Rham complexes \eqref{eq:dRcpxProd} for these
two pairs are compatible with $\Omega^{\bullet}_{g/f,\psi,\uc}$ of
$(M,\nabla_{M})$, $(N,\nabla_{N})$,
and $(M\otimes_AN,\nabla_{M\otimes N})$. Note that
the scalar extension of the last one coincides with
$(M'\otimes_{A'}N',\nabla_{M'\otimes N'})$
by Remark \ref{rmk:ConnScalarExtStratTensProd} (2).
\end{remark}

Finally we prove the compatibility of the morphisms 
\eqref{eq:dRCpxFunctConst} and \eqref{eq:dRCpxFunct} with compositions.
Let $(f',g',\psi')\colon (A'/R',\ualpha',(\partial'_{i'})_{i'\in\Lambda'})
\to (A''/R'',\ualpha'',(\partial''_{i''})_{i''\in\Lambda''})$
be another triplet of maps, and 
let $\uc'=(c'_{i'})_{i'\in\Lambda'}$ be an element of
$(A'')^{\Lambda'}$ such that $(f',g',\psi',\uc')$
satisfies the same conditions as $(f,g,\psi,\uc)$.
Suppose that we are given a total order $<$ on $\Lambda'$
such that the map $\psi\colon \Lambda\to \Lambda'$ 
preserves the orders. We define an element
$\uc''=(c''_i)_{i\in\Lambda}\in (A'')^{\Lambda}$ by $c''_i=c'_{\psi(i)}g'(c_i)$.
Then we have $g'\circ g(\alpha_i)=c''_i\alpha''_{\psi'\circ\psi(i)}$
for $i\in\Lambda$ \eqref{eq:DerFamMorphCond1},
and $\partial_{i''}''(c_i'')=0$ for $i\in \Lambda$ and 
$i''\in \Lambda''\backslash\{\psi'\circ\psi(i)\}$
by Lemma \ref{lem:TwDerivSystemFunct} (1).

\begin{lemma}\label{lem:dRCpxFuncCocyc}
(1) The composition 
$\Omega^{\bullet}_{g'/f',\psi',\uc'}\circ \Omega^{\bullet}_{g/f,\psi,\uc}
\colon \Omega_{A,\ugamma}^{\bullet}
\to \Omega_{A',\ugamma'}^{\bullet}
\to \Omega_{A'',\ugamma''}^{\bullet}$ coincides
with $\Omega^{\bullet}_{g'\circ g/f'\circ f,\psi'\circ\psi,\uc''}$.\par
(2) Let $(M,\nabla_M)$ be a module with integrable connection
over $(A,d)$, let $(M'=M\otimes_AA',\nabla_{M'})$ be the
scalar extension of $(M,\nabla_M)$ under $(f,g,\psi,\uc)$,
and let $(M''=M'\otimes_{A'}A'',\nabla_{M''})$ be the 
scalar extension of $(M',\nabla_{M'})$ under $(f',g',\psi',\uc')$.
Then under the canonical identification $M''=M\otimes_AA''$,
$\nabla_{M''}$ coincides with the scalar extension 
of $\nabla_M$ under $(f'\circ f,g'\circ g, \psi'\circ\psi,\uc'')$,
and the composition 
$\Omega^{\bullet}_{g'/f',\psi',\uc'}(M')\circ
\Omega^{\bullet}_{g/f,\psi,\uc}(M)
\colon M\otimes_A\Omega^{\bullet}_{A,\ugamma}
\to M'\otimes_{A'}\Omega^{\bullet}_{A',\ugamma'}
\to M''\otimes_{A''}\Omega^{\bullet}_{A'',\ugamma''}$
coincides with $\Omega^{\bullet}_{g'\circ g/f'\circ f,\psi'\circ\psi,\uc''}(M)$.
\end{lemma}

\begin{proof}
Put $f''=f'\circ f$, $g''=g'\circ g$, and $\psi''=\psi'\circ \psi$ to
simplify the notation. For $a\in A$,  $i\in \Lambda$, $i'=\psi(i)$,
and $i''=\psi''(i)$, we have 
\begin{align*}\Omega^1_{f'/g',\psi',\uc'}\circ\Omega^1_{f/g,\psi,\uc}(a\omega_i)
&=g'(\gamma^{\prime<}_{\psi',i'}(g(\gamma_{\psi,i}^{<}(a))))c'_{\psi(i)}g'(c_i)\omega_{i''}\\
&=g'\circ g\biggl(\biggl(\prod_{j'\in \psi^{\prime-1}(i''),j'<i'}
\prod_{k\in \psi^{-1}(j')}\gamma_k\circ\prod_{j\in \psi^{-1}(i'),j<i}\gamma_j\biggr)(a)\biggr)c_i''\omega_{i''}
\end{align*} 
by \eqref{eq:TwDerivSystemFunct2}. 
Since $\psi$ preserves orders, $k\in \Lambda$ appearing in the
last term ranges over the set 
$\Lambda_{\psi'',i}^{\psi^{-1}(<)}
=\{k\in \psi^{\prime\prime-1}(i'')\;\vert\;
k<i, \psi(k)<i'\}$. Hence the last term equals to 
$g'\circ g(\gamma_{\psi'',i}^{<}(a))c_i''\omega_{i''}
=\Omega^1_{g''/f'',\psi'',\uc''}(a\omega_i)$. 
Put $\gamma_{\psi'',i}^{\psi^{-1}(<)}
=\prod_{k\in \Lambda^{\psi^{-1}(<)}_{\psi'',i}}\gamma_k$. 
Then, for $a\in A$, $\bmI=(i_n)_{1\leq n\leq q}\in\Lambda^q$, and
$\bmI''=\psi^{\prime\prime q}(\bmI)$, the above computation 
is generalized to 
\begin{align*}
\Omega^q_{g'/f',\psi',\uc'}\circ\Omega^q_{g/f,\psi,\uc}(a\omega_{\bmI})
&=
g'((\prod_{n=1}^q\gamma_{\psi',\psi(i_n)}^{\prime<})
(g((\prod_{n=1}^q\gamma_{\psi,i_n}^{<})(a))))c'_{\psi^q(\bmI)}g'(c_{\bmI})\omega_{\bmI''}\\
&=g'\circ g((\prod_{n=1}^{q}(\gamma_{\psi'',i_n}^{\psi^{-1}(<)}\circ
\gamma_{\psi,i_n}^{<}))(a))c''_{\bmI}\omega_{\bmI''}\\
&=g'\circ g((\prod_{n=1}^q\gamma_{\psi'',i_n}^{<})(a))c''_{\bmI}\omega_{\bmI''}
=\Omega^q_{g''/f'',\psi'',\uc''}(a\omega_{\bmI}).
\end{align*}
(2) The first claim follows from the fact that the
composition 
$E^{\psi',\uc'}(g')\circ E^{\psi,\uc}(g)\colon 
E^{\ualpha}(A)\to E^{\ualpha''}(A'')$
coincides with $E^{\psi'',\uc''}(g'')$. 
Let $g_M$ (resp.~$g'_{M'}$) denote the
homomorphism $M\to M'=M\otimes_AA';m\mapsto m\otimes 1$
(resp.~$M'\to M''=M'\otimes_{A'}A''; m'\mapsto m'\otimes 1$). 
For $m\in M$, $\bmI\in\Lambda^q$, $\bmI'=\psi^q(\bmI)$, and
$\bmI''=\psi^{\prime\prime q}(\bmI)$, 
we see that $\Omega^q_{g'/f',\psi',\uc'}(M)\circ\Omega^q_{g/f,\psi,\uc}(M)
(m\otimes \omega_{\bmI})
=g'_{M'}(\gamma_{M',\psi',\bmI'}^{\prime <}(g_M(\gamma_{M,\psi,\bmI}^{<}(m))))
\otimes c'_{\psi^q(\bmI)}g'(c_{\bmI})\omega_{\bmI''}$
coincides with 
$g'_{M'}\circ g_M(\gamma^{<}_{M,\psi'',\bmI}(m))\otimes c''_{\bmI}\omega_{\bmI''}
=\Omega^q_{g''/f'',\psi'',\uc''}(M)(m\otimes\omega_{\bmI})$
by the same argument as the proof of (1) by using
\eqref{eq:ConnPBFormula2} instead of \eqref{eq:TwDerivSystemFunct2}. 
\end{proof}

The product morphism \eqref{eq:dRcpxProd} is not compatible with
the morphism \eqref{eq:dRCpxFunct} unless $\psi$ is injective
(cf.~Remark \ref{rmk:dRCpxProdScExtComp}). However we
have the following weaker compatibility. 

To state the claim,
we simplify some notation. First we define the category $\DAlg$
as follows: An object is a family of data $\uA=(A/R,\Lambda,\ualpha,\upartial)$
consisting of a ring $R$, an $R$-algebra $A$, a finite
totally ordered set $\Lambda$, $\ualpha=(\alpha_i)_{i\in\Lambda}\in A^{\Lambda}$,
and $\upartial=(\partial_i)_{i\in\Lambda}\in\prod_{i\in\Lambda}\Der^{\alpha_i}_R(A)$
(Definition \ref{def:alphaDerivation} (1)) satisfying \eqref{eq:TwDerivFamilyCond1} 
and \eqref{eq:TwDerivFamilyCond2}. A morphism $\ug=(f,g,\psi,\uc)\colon 
\uA=(A/R,\Lambda,\ualpha,\upartial)\to (A'/R',\Lambda',\ualpha',\upartial')$
consists of a ring homomorphism $f\colon R\to R'$, a ring homomorphism 
$g\colon A\to A'$ over $f$, a map $\psi\colon \Lambda'\to \Lambda$ 
preserving orders, and $\uc=(c_i)_{i\in\Lambda}\in (A')^{\Lambda}$ satisfying
\eqref{eq:DerFamMorphCond1}, \eqref{eq:DerFamMorphCond2}, 
and \eqref{eq:ConnFunctCond}. We define composition of morphisms as
before Lemma \ref{lem:dRCpxFuncCocyc}. For $\uA=(A/R,\Lambda,\ualpha,\upartial)
\in \Ob\DAlg$, let $(\Omega_{\uA}^{\bullet},d_{\uA}^{\bullet})$ denote the differential
graded algebra over $R$ defined as before Definition \ref{def:connection},
and let $\MIC(\uA)$ denote the category of modules with integrable connection 
over $(\Omega_{\uA}^{\bullet},d_{\uA}^{\bullet})$ (Definitions \ref{def:connection} and 
\ref{def:intconnection}).
For $(M,\nabla_M)\in \Ob \MIC(\uA)$, we write $\Omega_{\uA}^{\bullet}(M)$
for the de Rham complex of $(M,\nabla_M)$ defined after Definition \ref{def:intconnection}.
For a morphism $\ug\colon \uA\to \uA'$ in $\DAlg$ and 
$M\in \Ob \MIC(\uA)$, we simply write $\ug_{M,\Omega}^*$ for
the morphism $\Omega^{\bullet}_{\uA}(M)\to\Omega^{\bullet}_{\uA'}(\ug^*M)$
\eqref{eq:dRCpxFunct}.

\begin{proposition}\label{prop:ProddRCpxFunct}
We consider the following commutative diagram in $\DAlg$.
\begin{equation*}
\xymatrix@C=90pt{
\uA^{(l)}=(A^{(l)}/R,\Lambda^{(l)},\ualpha^{(l)},\upartial^{(l)})
\ar[r]^(.58){\uiota^{(l)}=(\id_R,\iota^{(l)},\chi^{(l)},\ue^{(l)})}
\ar[d]_{\ug^{(l)}=(f^{(l)},g^{(l)},\psi^{(l)},\uc^{(l)})}
&
\uA=(A/R,\Lambda,\ualpha,\upartial)
\ar[d]^{\ug=(f,g,\psi,\uc)}
\\
\tuA^{(l)}=(\tA^{(l)}/\tR,\tLambda^{(l)},\tualpha^{(l)},\tupartial^{(l)})
\ar[r]^(.58){\tuiota^{(l)}=(\id_{\tR},\tiota^{(l)},\tchi^{(l)},\tue^{(l)})}&
\tuA=(\tA/\tR,\tLambda,\tualpha,\tupartial)
}
\qquad (l=0,1)
\end{equation*}
such that the maps $\chi\colon \Lambda^{(0)}\sqcup \Lambda^{(1)}\to \Lambda$
and $\tchi\colon \tLambda^{(0)}\sqcup \tLambda^{(1)}\to \tLambda$
induced by $\chi^{(l)}$ and $\tchi^{(l)}$ $(l=0,1)$ are injective. 
For $l\in \{0,1\}$, let $M_{l}$ be an object of $\MIC(\uA^{(l)})$,
and put $\tM_{l}=\ug^{(l)*}M_{l}$, 
$N_{l}=\uiota^{(l)*}M_{l}$, and 
$\tN_{l}=\ug^*N_{l}=\tuiota^{(l)*}\tM_{l}$.
Then the following diagram is commutative.
\begin{equation*}
\xymatrix@C=60pt{
\Omega^{\bullet}_{\uA^{(0)}}(M_0)\otimes_R
\Omega^{\bullet}_{\uA^{(1)}}(M_1)
\ar[r]^(.54){\uiota^{(0)*}_{\Omega,M_0}\otimes \uiota^{(1)*}_{\Omega,M_1}}
\ar[d]_{\ug^{(0)*}_{\Omega,M_0}\otimes \ug^{(1)*}_{\Omega,M_1}}
&
\Omega^{\bullet}_{\uA}(N_0)\otimes_R\Omega^{\bullet}_{\uA}(N_1)
\ar[r]^(.55){\eqref{eq:dRcpxProd}}&
\Omega^{\bullet}_{\uA}(N_0\otimes_AN_1)
\ar[d]^{\ug^*_{\Omega,N_0\otimes_A N_1}}
\\
\Omega^{\bullet}_{\tuA^{(0)}}(\tM_0)\otimes_{\tR}
\Omega^{\bullet}_{\tuA^{(1)}}(\tM_1)
\ar[r]^(.54){\tuiota^{(0)*}_{\Omega,\tM_0}\otimes \tuiota^{(1)*}_{\Omega,\tM_1}}&
\Omega^{\bullet}_{\tuA}(\tN_0)\otimes_{\tR}\Omega^{\bullet}_{\tuA}(\tN_1)
\ar[r]^(.55){\eqref{eq:dRcpxProd}}&
\Omega^{\bullet}_{\tuA}(\tN_0\otimes_{\tA}\tN_1)}
\end{equation*}
\end{proposition}

\begin{proof}
We regard $\Lambda^{(0)}\sqcup \Lambda^{(1)}$
(resp.~$\tLambda^{(0)}\sqcup\tLambda^{(1)}$) as a subset of 
$\Lambda$ (resp.~$\tLambda$) by $\chi$ (resp.~$\tchi$).
For $l\in \{0,1\}$, $m\in M_l$, and $i\in \Lambda^{(l)}$,
we have $\gamma_{N_{l},j}(m\otimes 1)=m\otimes 1$
$(j\in \Lambda\backslash \Lambda^{(l)})$ by \eqref{eq:ConnPBFormula1},
and $\gamma_{A,j}(e_i^{(l)})=e_i^{(l)}$ $(j\in\Lambda\backslash\{i\})$.
This implies that, for $m_{l}\in M_{l}$, $q_{l}\in\N$,
and $\bmI_{l}\in (\Lambda^{(l)})^{q_{l}}$ without common components 
$(l\in \{0,1\})$,
the image of $(m_0\otimes\omega_{\bmI_0})\otimes(m_1\otimes\omega_{\bmI_1})$
under the upper horizontal maps is
$x=(m_0\otimes e^{(0)}_{\bmI_0})\otimes(m_1\otimes e^{(1)}_{\bmI_1})
\otimes\omega_{\bmI_0}\wedge\omega_{\bmI_1}$.
The same claim holds for the lower horizontal maps.
For $l\in \{0,1\}$ and $i\in \Lambda^{(l)}$, we have
$\Lambda_{\psi^{(l)},i}^{(l)<}
=\Lambda_{\psi,i}^{<}\cap (\Lambda^{(0)}\sqcup\Lambda^{(1)})$.
Hence, by the remark at the beginning of the proof
and \eqref{eq:ConnPBFormula2}, we see that
the image of $x$ under $\ug^*_{\Omega,N_0\otimes_AN_1}$
is 
$\gamma_{M_0,\psi^{(0)},\bmI_0}^{<}(m_0)\otimes
\gamma_{M_1,\psi^{(1)},\bmI_1}^{<}(m_1)\otimes g(e^{(0)}_{\bmI_0}
e^{(1)}_{\bmI_1})c_{\bmI_0}c_{\bmI_1}\otimes \omega_{\psi^{q_0}(\bmI_0)}
\wedge \omega_{\psi^{q_1}(\bmI_1)}$.
This completes the proof because
the image of $m_{l}\otimes\omega_{\bmI_{l}}$
under $\ug^{(l)*}_{\Omega,M_{l}}$ is
$\gamma_{M_{l},\psi^{(l)},\bmI_{l}}(m_{l})
\otimes c^{(l)}_{\bmI_{l}}\otimes
\omega_{\psi^{(l)q_{l}}(\bmI_{l})}$
and $g(e_i^{(l)})c_{\psi(i)}=\tiota^{(l)}(c^{(l)}_i)
\te_{\psi^{(l)}(i)}^{(l)}$ for $l\in \{0,1\}$ and 
$i\in \Lambda^{(l)}$. 
\end{proof}

Let $\Z[q]$ be the polynomial ring in one variable $q$ 
equipped with the $\delta$-structure defined by $\delta(q)=0$,
whose associated lifting of Frobenius is given by $f(q)\mapsto f(q^p)$. 
We define elements $\mu$, $\pq$, and $\eta$ of $\Z[q]$
as the beginning of \S\ref{sec:TwDeivDivDeltaEnv}. We have 
$\varphi(\mu)=\pq\mu$ and $\delta(\mu)=\eta\mu$.
We study the scalar extension of a module with integrable connection
under a lifting of Frobenius
for derivations of a certain type including those 
considered in Proposition \ref{prop:PrismEnvDeriv}
and in later sections.

\begin{lemma}\label{lem:qDerivFrobComp}
Let $A$ be a $\delta$-$\Z[q]$-algebra, let $t$
be an element of $A$ satisfying $\delta(t)=0$, and let $\partial$
be a $t\mu$-derivation of $A$ over $\Z[q]$ $\delta$-compatible
with respect to $t^{p-1}\eta$ (Definition \ref{def:alphaDerivDeltaComp}). 
(Note that we have $\delta(t\mu)=t^{p-1}\eta\cdot t\mu$ by the
assumption $\delta(t)=0$.)
Let $\varphi_A$ be the lifting of Frobenius of $A$ associated to its
$\delta$-structure, and let $\gamma$ be the 
endomorphism $\id_A+t\mu\partial$ associated to
$\partial$ (Lemma \ref{lem:alphaDergammaDer} (1)). Then the following equalities hold.\par
(1) $\gamma\circ\varphi_A=\varphi_A\circ\gamma$.\par
(2) $\partial\circ\varphi_A=\pq t^{p-1}\cdot \varphi_A\circ\partial$.
\end{lemma}

\begin{proof}
This immediately follows from Lemma \ref{lem:alphaDerivDeltaEquiv}. 
For (2), putting $\alpha=t\mu$ and $\beta=t^{p-1}\eta$, 
we have $\alpha^{p-1}+p\beta=
t^{p-1}\mu^{p-1}+pt^{p-1}\eta=t^{p-1}(\mu^{p-1}+p\eta)
=t^{p-1}\varphi(\mu)\mu^{-1}=t^{p-1}\pq$.
\end{proof}

Let $R$ be a $\Z[q]$-algebra, let $A$ be an $R$-algebra,
let $\ut=(t_i)_{i\in\Lambda}$ be a finite family of elements of $A$,
and suppose that we are given a $t_i\mu$-derivation $\partial_i$
of $A$ over $R$ for each $i\in\Lambda$ satisfying 
$\partial_i\circ\partial_j=\partial_j\circ\partial_i$ $(i,j\in\Lambda)$
and $\partial_j(t_i)=0$ $(i,j\in\Lambda,i\neq j)$. 
We define $\ualpha=(\alpha_i)_{i\in\Lambda}\in A^{\Lambda}$
by $\alpha_i=t_i\mu$. Then $(A/R,\ualpha,(\partial_i)_{i\in\Lambda})$
satisfies \eqref{eq:TwDerivFamilyCond1} and \eqref{eq:TwDerivFamilyCond2}. 
Suppose that we are given a lifting
of Frobenius $\varphi_R\colon R\to R$ and $\varphi_A\colon A\to A$
compatible with that of $\Z[q]$ such that $\varphi_A$ lies
over $\varphi_R$ and satisfies
\begin{align}
\varphi_A(t_i)&=t_i^p,\\
\partial_i\circ\varphi_A&=\pq t_i^{p-1}\cdot\varphi_A\circ\partial_i
\label{eq:tmuDerivFrobCompCond}
\end{align}
for every $i\in\Lambda$. The former implies 
$$\varphi_A(\alpha_i)=\pq t_i^{p-1}\alpha_i.$$
Hence, defining $\uc=(c_i)_{i\in\Lambda}\in A^{\Lambda}$ by 
$c_i=\pq t_i^{p-1}$, we obtain a homomorphism 
$E^{\id_{\Lambda},\uc}(\varphi_A)\colon E^{\ualpha}(A)\to E^{\ualpha}(A)$
\eqref{eq:TwExtAlgMap},
which is a homomorphism of bialgebras over $\varphi_A$ by
the assumption \eqref{eq:tmuDerivFrobCompCond} 
and Lemma \ref{lem:TwDerivSystemFunct}. Therefore we have the scalar extension functor
preserving tensor products
(Definition \ref{def:connectionScalarExt}, Remark \ref{rmk:ConnScalarExtStratTensProd} (2))
\begin{equation}\label{eq:qDerivConnFobPB}
\varphi_A^*:=(\varphi_R,\varphi_A,\id_{\Lambda},\uc)^*\colon
\MIC(A,d)\longrightarrow \MIC(A,d).
\end{equation}
For $(M,\nabla_M)\in \Ob \MIC(A,d)$, the image 
$(\varphi_A^*M=M\otimes_{A,\varphi_A}A,\nabla_{\varphi_A^*M})$ is characterized
by the following formula (Proposition \ref{prop:ConnPBFormula}).
\begin{equation}\label{eq:qDervConnPBFormula}
\nabla_{\varphi_A^*M,i}(m\otimes 1)=\nabla_{M,i}(m)\otimes \pq t_i^{p-1}\quad
(m\in M,\;i\in\Lambda).
\end{equation}
By Proposition \ref{prop:dRCpxFunct} and \eqref{eq:dRCpxScExtPsiInj}, 
we have a morphism of complexes 
\begin{equation}\label{eq:qDervConnCpxPB}
\varphi_{A,\Omega}^{\bullet}(M)\colon
M\otimes_A\Omega^{\bullet}_{A,\ugamma}
\longrightarrow \varphi_A^*M\otimes_A\Omega^{\bullet}_{A,\ugamma}
\end{equation}
sending $m\otimes\omega_{\bmI}$ to $m\otimes 1\otimes \pq^r\prod_{n=1}^rt_{i_n}^{p-1}\omega_{\bmI}$
for $m\in M$, $r\in \N$, and $\bmI=(i_n)_{1\leq n\leq r}\in \Lambda^r$.
By Remark \ref{rmk:dRCpxProdScExtComp}, the morphism \eqref{eq:qDervConnCpxPB} is
compatible with product morphisms of de Rham complexes \eqref{eq:dRcpxProd}.
\par

Let $R'$, $A'$, $\ut'=(t'_{i'})_{i'\in\Lambda'}$, $\partial'_{i'}$ $(i'\in\Lambda')$,
$\varphi_{R'}$, and $\varphi_{A'}$ be another set of data
satisfying the same conditions as $R$, $A$, $\ut$, $\partial_i$, $\varphi_R$, and $\varphi_A$
above. We define $\ualpha'=(\alpha'_{i'})_{i'\in\Lambda'}$
and $\uc'=(c'_{i'})_{i'\in\Lambda'}$ by $\alpha'_{i'}=t'_{i'}\mu$
and $c'_{i'}=\pq (t'_{i'})^{p-1}$ $(i'\in\Lambda')$. Suppose
that we are given a homomorphism of $\Z[q]$-algebras
$f\colon R\to R'$ compatible with $\varphi_R$ and $\varphi_{R'}$,
a homomorphism $g\colon A\to A'$ over $f$ compatible
with $\varphi_A$ and $\varphi_{A'}$, and a map 
$\psi\colon \Lambda\to \Lambda'$ such that 
$g(t_i)=t'_{\psi(i)}$ $(i\in\Lambda)$, which
implies $g(\alpha_i)=\alpha'_{\psi(i)}$ $(i\in\Lambda)$,
and that the homomorphism $E^{\psi}(g)\colon 
E^{\ualpha}(A)\to E^{\ualpha'}(A')$ \eqref{eq:TwExtAlgMap} satisfies
\eqref{eq:ConnFunctCond}, i.e, it is a homomorphism of bialgebras over $g$.
By $\varphi_{A'}\circ g=g\circ\varphi_A$
and $c'_{\psi(i)}=g(c_i)$ $(i\in\Lambda)$, we see that the diagram 
\begin{equation}
\xymatrix@C=50pt{
E^{\ualpha}(A)\ar[r]^{E^{\psi}(g)}
\ar[d]_{E^{\id_{\Lambda},\uc}(\varphi_A)}
&
E^{\ualpha'}(A')
\ar[d]^{E^{\id_{\Lambda'},\uc'}(\varphi_{A'})}\\
E^{\ualpha}(A)\ar[r]^{E^{\psi}(g)}&
E^{\ualpha'}(A')
}
\end{equation}
is commutative. Hence, by Lemma \ref{lem:dRCpxFuncCocyc} (2), 
the diagram 
\begin{equation}\label{eq:qConnFrobPBFunct}
\xymatrix@C=50pt{
\MIC(A,d)\ar[r]^{(f,g,\psi)^*}\ar[d]_{\varphi_A^*}&
\MIC(A',d')\ar[d]^{\varphi_{A'}^*}\\
\MIC(A,d)\ar[r]^{(f,g,\psi)^*}&
\MIC(A',d')
}
\end{equation}
is commutative up to canonical isomorphism, and
the diagram of complexes
\begin{equation}\label{eq:qdRcpxFrobPBFunct}
\xymatrix@C=60pt{
M\otimes_A\Omega_{A,\ugamma}^{\bullet}
\ar[r]^{\Omega^{\bullet}_{g/f,\psi}(M)}\ar[d]_{\varphi^{\bullet}_{A,\Omega}(M)}&
M'\otimes_{A'}\Omega_{A',\ugamma'}^{\bullet}
\ar[d]^{\varphi^{\bullet}_{A',\Omega}(M')}\\
\varphi_A^*M\otimes_A\Omega_{A,\ugamma}^{\bullet}
\ar[r]^{\Omega^{\bullet}_{g/f,\psi}(\varphi_A^*M)}&
\varphi_{A'}^*M'\otimes_{A'}\Omega_{A',\ugamma'}^{\bullet}
}
\end{equation}
commutes for $(M,\nabla_M)\in \Ob \MIC(A,d)$
and $(M',\nabla_{M'})=(f,g,\psi)^*(M,\nabla_M)$.
Here we choose a total order of the set $\Lambda$
to define $\Omega^{\bullet}_{g/f,\psi}(-)$
(when $\psi$ is not injective).

\section{$q$-prismatic envelopes and $q$-Higgs fields}\label{qprismenv}
We follow the notation introduced at the beginning of \S\ref{sec:TwDeivDivDeltaEnv}, and
equip $\Z_p[[q-1]]$ with the $(p,\mu)$-adic completion of the
$\delta$-structure on $\Z[q]$, i.e., the $\delta$-structure corresponding to
the lifting of Frobenius $\varphi$ on $\Z_p[[q-1]]$ defined by 
$\varphi(q)=q^p$. The $\delta$-pair $(\Z_p[[q-1]], \pq\Z_p[[q-1]])$
is a bounded prism since $\delta(\pq)$ is invertible
(Definition \ref{def:prism} (2)). 

\begin{definition}\label{def:FramedSmoothPrism}
(1) We call a bounded prism $(R,\pq R)$ over the bounded prism \linebreak
$(\Z_p[[q-1]],\pq \Z_p[[q-1]])$ simply a {\it $q$-prism}. 
We often abbreviate $(R,\pq R)$ to $R$ if there is no risk of confusion.
A {\it morphism of $q$-prisms} $(R,\pq R)\to (R',\pq R')$
is a morphism of bounded prisms over $(\Z_p[[q-1],[p]_q \Z_p[[q-1]])$. 
We call a bounded prismatic envelope of a $\delta$-pair $(A,J)$
over $(\Z_p[[q-1]],\pq \Z_p[[q-1]])$ (Definition \ref{def:bddPrsimEnv}) simply a {\it $q$-prismatic envelope}.
(Note that, when $(A,J)$ is a $\delta$-pair over a $q$-prism $(R,\pq R)$, then
it is the same as a bounded prismatic envelope of $(A,J)$
over $(R,\pq R)$ (Lemma  \ref{lem:bddPrismEnvBase} (1))).\par
(2) A {\it framed smooth $q$-prism} is a data $(A/R,\ut=(t_i)_{i\in \Lambda})$
consisting of a $q$-prism $R$, a $(p,\pq)$-adically smooth $R$-algebra
(Definition \ref{def:formallyflat} (1)) $(p,\pq)$-adically complete and separated,
and a $(p,\pq)$-adic coordinates $\ut=(t_i)_{i\in \Lambda}\in A^{\Lambda}$
$(\sharp\Lambda<\infty)$ (Definition \ref{def:formallyflat} (2)) 
with a total order given on the index
set $\Lambda$. We always equip $A$ with the unique $\delta$-$R$-algebra
structure satisfying $\delta(t_i)=0$ $(i\in \Lambda)$ 
(Proposition \ref{prop:DeltaStrExtEtSm} (2)). 
Note that $(A,\pq A)$ is a $q$-prism 
by Corollary \ref{cor:bddPrismFlatMap}.
We also call  
$(A,\ut=(t_i)_{i\in \Lambda})$ a {\it framed smooth $\delta$-$R$-algebra}.
A {\it morphism of framed smooth $q$-prisms} 
$(A/R, (t_i)_{i\in \Lambda})
\to (A'/R',(t_{i'}')_{i'\in \Lambda'})$ is a triplet $(g/f,\psi)$
consisting of a morphism of $q$-prisms
$f\colon R\to R'$, a homomorphism of rings $g\colon A\to A'$ 
compatible with $f$, and a map of ordered sets
$\psi\colon \Lambda\to \Lambda'$ satisfying 
$g(t_i)=t'_{\psi (i)}$ for every $i\in \Lambda$. When $R=R'$ and $f=\id$,
we call $(g,\psi)$ a {\it morphism of framed smooth $\delta$-$R$-algebras}\par
(3) A {\it framed smooth $q$-pair} $((A,J)/R,\ut)$ is a 
framed smooth $q$-prism $(A/R,\ut)$ equipped with an ideal $J$ of $A$ containing 
$\pq A$. We also call $((A,J),\ut)$ a {\it framed smooth $\delta$-pair over $R$}.
A {\it morphism of framed smooth $q$-pairs}
$(g/f,\psi)\colon ((A,J)/R,\ut)\to ((A',J')/R',\ut')$ is 
a morphism of framed smooth $q$-prisms satisfying $g(J)\subset J'$.
When $R=R'$ and $f=\id$, we call $(g,\psi)$ a {\it morphism
of framed smooth $\delta$-pairs over $R$.}
\end{definition}

\begin{lemma}\label{lem:FramedSmoothDeltaMap}
Let $(g/f,\psi)\colon (A/R, (t_i)_{i\in\Lambda})\to 
(A'/R',\ut'=(t_{i'}')_{i'\in \Lambda'})$ be a morphism of 
framed smooth $q$-prisms. Then $g$ is a $\delta$-homomorphism.
\end{lemma}

\begin{proof}
Let $A_0$ be the polynomial ring $R[T_i; i\in\Lambda]$ equipped
with the $\delta$-$R$-algebra structure defined by $\delta(T_i)=0$
for $i\in \Lambda$.
Then the $R$-homomorphism $h\colon A_0\to A;T_i\mapsto t_i$
is a $\delta$-$R$-homomorphism $(p,\pq)$-adically \'etale
(Definition \ref{def:formallyflat} (1)). The composition $g\circ h$ is 
also a $\delta$-homomorphism because $\delta(g\circ h(T_i))
=\delta(t'_{\psi(i)})=0$ and $g\circ h$ is a homomorphism
compatible with the $\delta$-homomorphism $f\colon R\to R'$.
Hence $g$ is a $\delta$-homomorphism by Proposition 
\ref{prop:DeltaCompEtExt}.
\end{proof}

\begin{proposition}\label{prop:qHiggsDeriv}
(1) Let $(A/R,\ut=(t_i)_{i\in \Lambda})$ be a framed smooth $q$-prism
(Definition \ref{def:FramedSmoothPrism} (2)). 
For each $i\in \Lambda$, there exists a unique $t_i\mu$-derivation
$\theta_{A,i}\colon A\to A$ over $R$ $\delta$-compatible with respect to
$t_i^{p-1}\eta$ such that $\theta_{A,i}(t_i)=\pq$ and $\theta_{A,i}(t_j)=0$
$(j\in \Lambda, j\neq i)$. These satisfy $\theta_{A,i}\circ \theta_{A,j}
=\theta_{A,j}\circ\theta_{A,i}$ for every $(i,j)\in \Lambda^2$. \par
(2) Let $((A,J)/R,\ut=(t_i)_{i\in \Lambda})$  be a framed smooth $q$-pair
(Definition \ref{def:FramedSmoothPrism} (3)), and let 
$(D,\pq D)$ be a $q$-prismatic envelope (Definition \ref{def:FramedSmoothPrism} (1)) 
(assuming it exists). Then, for each $i\in \Lambda$,
$\theta_{A,i}$ extends uniquely to a $t_i\mu$-derivation 
$\theta_{D,i}\colon D\to D$ over $R$ $\delta$-compatible with
respect to $t_i^{p-1}\eta$. These satisfy
$\theta_{D,i}\circ\theta_{D,j}=\theta_{D,j}\circ \theta_{D,i}$
for every $(i,j)\in \Lambda^2$.
\end{proposition}

\begin{proof}
(1) 
The element $t_i\mu$ is $(p,\pq)$-adically nilpotent in $A$.
Hence we can apply Corollary \ref{cor:SmoothDeltCompDeriv}
to $R$, $pR+\pq R$, $A$, $\ut$, $t_i\mu$, and $t_i^{p-1}\eta$,
and show the unique existence of $\theta_{A,i}$ by using 
Lemma \ref{lem:TwDerivxiDeltaComp}. 
Since the $R$-homomorphism $R[T_i; i\in \Lambda]\to A;T_i\mapsto t_i$
is $(p,\pq)$-adically \'etale, $A$ is $(p,\pq)$-adically separated, and 
$t_i\mu$ is $(p,\pq)$-adically nilpotent in $A$, the commutativity
of $\theta_{A,i}$ and $\theta_{A,j}$ for $i\neq j$ is 
reduced to $\theta_{A,i}\circ \theta_{A,j}(t_k)
=0=\theta_{A,j}\circ\theta_{A,i}(t_k)$ for
$k\in \Lambda$ by Lemma \ref{lem:EtMapDerivComm}. 
Note that we have $\theta_{A,i}(t_j \mu)=\theta_{A,i}(t_j^{p-1}\eta)=0$
for $i\neq j$ because $\theta_{A,i}$ is $R$-linear and 
$\theta_{A,i}(t_j)=0$. 
\par
(2) We show $\theta_{A,i}(A)\subset [p]_q A$,
which implies the two claims by Propositions \ref{prop:TwDerivPrismExt}
and \ref{lem:TwDervPrismExtComm}. 
By Remark \ref{rmk:alphaDerivBC} (3), the reduction mod $[p]_q$ 
of $\theta_{A,i}$ is a $t_i\mu$-derivation of $A/[p]_q A$ over $R/[p]_q R$,
which vanishes on the images of $t_j$ $(j\in \Lambda)$.
As $A/[p]_q A$ is $p$-adically smooth over 
$R/[p]_q R$ and the images of $t_i$ $(i\in \Lambda)$ in $A/[p]_q A$
are $p$-adic coordinates over $R/[p]_q R$, the $R/[p]_q R$-homomorphism
$R/[p]_q R[T_i; i\in \Lambda]\to A/[p]_q A;T_i\mapsto (t_i \mod [p]_q)$ is
$p$-adically \'etale.
Since $A/[p]_q A$ is $p$-adically complete and separated 
by Proposition \ref{prop:bddFlatInv}, Proposition 
\ref{prop:TwistDerEtaleExt} shows that 
$(\theta_{A,i}\mod [p]_q)$ vanishes.
\end{proof}

Under the notation and assumption in Proposition \ref{prop:qHiggsDeriv} (1),
we have $\theta_{A,i}(t_j\mu)=0$ for $i,j\in \Lambda$, $i\neq j$ 
because $\theta_{A,i}$ is $R$-linear and $\theta_{A,i}(t_j)=0$ 
$(i,j\in \Lambda,i\neq j)$. Hence
$(A,(t_i\mu)_{i\in\Lambda},(\theta_{A,i})_{i\in\Lambda})$ 
satisfies the conditions \eqref{eq:TwDerivFamilyCond1} and \eqref{eq:TwDerivFamilyCond2}
on $(A,(\alpha_i)_{i\in\Lambda},(\partial_i)_{i\in\Lambda})$
assumed at the beginning of \S\ref{sec:connection}. The same applies to 
$(D,(t_i\mu)_{i\in\Lambda},(\theta_{D,i})_{i\in\Lambda})$ in 
Proposition \ref{prop:qHiggsDeriv} (2).  We introduce the following
definitions by applying \S\ref{sec:connection} to these triplets.

\begin{definition}\label{def:qHiggsDifferential}
(1) Let $(A/R,\ut=(t_i)_{i\in\Lambda})$ be a framed smooth $q$-prism
(Definition \ref{def:FramedSmoothPrism} (2)), let 
$\theta_{A,i}$ $(i\in \Lambda)$ be the $t_i\mu$-derivation 
of $A$ over $R$ $\delta$-compatible with $t_i^{p-1}\eta$
constructed in Proposition \ref{prop:qHiggsDeriv} (1), 
and let $\gamma_{A,i}$ $(i\in \Lambda)$ be the 
$\delta$-$R$-endomorphism
$\id_A+t_i\mu \theta_{A,i}$ of $A$ (Lemma \ref{lem:alphaDerivDeltaEquiv} (1)),
which is an automorphism because $A$ is $pR+\pq R$-adically complete
and separated, $\mu$ is $pR+\pq R$-adically nilpotent, and 
$t_i\theta_{A,i}$ is $R$-linear.
We write $q\Omega_{A/R,\ut}$, $dt_i$, $\theta_A\colon A\to q\Omega_{A/R,\ut}$,
and $(q\Omega^{\bullet}_{A/R,\ut},\theta_A^{\bullet})$ for
the $A$-bimodule $\Omega_{A,\ugamma}$, $\omega_i\in \Omega_{A,\ugamma}$, 
the $R$-linear map $d\colon A\to \Omega_{A,\ugamma}$, and 
the differential graded algebra $(\Omega^{\bullet}_{A,\ugamma},d^{\bullet})$
over $R$, respectively, obtained by applying the constructions in 
\S\ref{sec:connection} to 
$(A,(t_i\mu)_{i\in\Lambda},(\theta_{A,i})_{i\in\Lambda})$. We have
$b(adt_i)c=ba\gamma_{A,i}(c)dt_i$ $(a,b,c\in A)$, 
$\theta_A(a)=\sum_{i\in\Lambda}\theta_{A,i}(a)dt_i$ $(a\in A)$, 
and $\theta_{A}(ab)=a\theta_{A}(b)+\theta_A(a)b$ $(a,b\in A)$.
We call $\theta_{A,i}$, $\gamma_{A,i}$, and $\theta_A$
{\it the $q$-Higgs derivations, the $q$-Higgs automorphisms}, and 
{\it the $q$-Higgs differential of $A$ over $R$ with respect to $\ut$}. \par
(2) Let $((A,J)/R,\ut=(t_i)_{i\in\Lambda})$ be a framed smooth $q$-pair
(Definition \ref{def:FramedSmoothPrism} (3)),  let $D$ be a $q$-prismatic envelope of $(A,J)$ (Definition \ref{def:FramedSmoothPrism} (1)) (assuming it exists),
and let $\theta_{D,i}$ $(i\in \Lambda)$ be the $t_{i}\mu$-derivation 
of $D$ over $R$ $\delta$-compatible with $t_i^{p-1}\eta$ constructed
in Proposition \ref{prop:qHiggsDeriv} (2). We define 
$\gamma_{D,i}$, $q\Omega_{D/R,\ut}$, $dt_i$, $\theta_D\colon D\to q\Omega_{D/R,\ut}$, and $(q\Omega^{\bullet}_{D/R,\ut}, \theta_D^{\bullet})$ in the 
same way as (1) by using $(D,(t_i\mu)_{i\in\Lambda},(\theta_{D,i})_{i\in\Lambda})$.
We call $\theta_{D,i}$, $\gamma_{D,i}$, and $\theta_D$
{\it the $q$-Higgs derivations, 
the $q$-Higgs automorphisms}, and {\it the $q$-Higgs differential
of $D$ over $R$ with respect to $\ut$}. 
\end{definition}

\begin{definition}\label{def:qHiggsMod}
(1) Let $(A/R,\ut=(t_i)_{i\in A})$ be a framed smooth $q$-prism,
and let $\theta_A\colon A\to q\Omega_{A/R,\ut}$ be the $q$-Higgs
differential of $A$ over $R$ with respect to $\ut$. We call
a module with integrable connection 
$(M,\theta_M\colon M\to M\otimes_Aq\Omega_{A/R,\ut})$ 
over $(A,\theta_A)$ (Definition \ref{def:intconnection}) {\it a $q$-Higgs module over $(A/R,\ut)$}. We call $\theta_M$ a 
{\it $q$-Higgs field on $M$ over} $(A/R,\ut)$. The de Rham complex of $(M,\theta_M)$
is denoted by
$(M\otimes_Aq\Omega_{A/R,\ut}^{\bullet},\theta_M^{\bullet})$
and is called the {\it $q$-Higgs complex of} $(M,\theta_M)$. We write
$q\HIG(A/R,\ut)$ for the category of $q$-Higgs modules over 
$(A/R,\ut)$. 
\par
(2) Let $((A,J)/R, \ut=(t_i)_{i\in\Lambda})$ be a framed smooth $q$-pair,
and let $D$ be a $q$-prismatic envelope of $(A,J)$ (assuming it exists). 
We define a {\it $q$-Higgs module $(M,\theta_M)$ over $(D/R,\ut)$},
and its {\it $q$-Higgs complex }
$(M\otimes_D q\Omega^{\bullet}_{D/R,\ut},\theta_M^{\bullet})$
by using $(q\Omega_{D/R}^{\bullet},\theta_D^{\bullet})$
in the same way as (1). We write $q\HIG(D/R,\ut)$ for the
category of $q$-Higgs modules over $(D/R,\ut)$. 
\end{definition}
\begin{remark}\label{rmk:FramedSmQPHigFPB}
Let the notation and assumption be the same as Definition 
\ref{def:qHiggsDifferential} (2) and Definition \ref{def:qHiggsMod} (2).
We have $\theta_{D,i}\circ\varphi_D=\pq t_i^{p-1}\cdot \varphi_D\circ \theta_{D,i}$
for $i\in \Lambda$ by Lemma \ref{lem:qDerivFrobComp} (2), where $\varphi_D$ denotes
the lifting of Frobenius of $D$ associated to its $\delta$-structure.
Therefore, for any ideal $J$ of $\Z[q]$ satisfying $\varphi(J)\subset J$
(e.g.~$J=0$, $J=(p,\pq)^{n+1}$ $(n\in\N)$), letting
$\oD$ and $\varphi_{\oD}$ denote the reduction modulo $J$ of 
$D$ and $\varphi_D$, we have the Frobenius pullback functor
preserving tensor products
\eqref{eq:qDerivConnFobPB}
\begin{equation}\label{eq:FramedSmQPHigFPBFunct}
\varphi_{\oD}^*\colon q\HIG(\oD/R,\ut)\to q\HIG(\oD/R,\ut),
\end{equation}
where $q\HIG(\oD/R,\ut)$ denotes the full subcategory
of $q\HIG(D/R,\ut)$ consisting of $(M,\theta_M)$ with $JM=0$.
For $(M,\theta_M)\in \Ob q\HIG(\oD/R,\ut)$, we have a morphism of complexes \eqref{eq:qDervConnCpxPB}
\begin{equation}\label{eq:FramedSmQPHigFPBMap}
\varphi_{\oD,\Omega}^{\bullet}(M)\colon M\otimes_{D}q\Omega^{\bullet}_{D/R}
\to \varphi_{\oD}^*M\otimes_Dq\Omega^{\bullet}_{D/R}.
\end{equation}
compatible with the product morphism 
\eqref{eq:dRcpxProd}.
\end{remark}

We can apply scalar extensions  of modules with 
integrable connections defined and studied in \S9
to $q$-Higgs modules and morphisms of framed smooth $q$-prisms
and morphisms of framed smooth $q$-pairs
as follows.

\begin{proposition}\label{prop:qHiggsDerivFunct}
(1) Let $(g/f,\psi)\colon (A/R,\ut=(t_i)_{i\in\Lambda})
\to (A'/R',\ut'=(t'_{i'})_{i'\in\Lambda})$ be a morphism of
framed smooth $q$-prisms. Put $\ualpha=(t_i\mu)_{i\in\Lambda}$
and $\ualpha'=(t'_{i'}\mu)_{i'\in\Lambda'}$,
and let $s_A$ and $s_{A'}$ be the $R$-homomorphisms
$A\to E^{\ualpha}(A)$ and $A'\to E^{\ualpha'}(A')$ defined by 
the $q$-Higgs derivations $\theta_{A,i}$ $(i\in \Lambda)$ 
and $\theta_{A',i'}$ $(i'\in\Lambda')$
\eqref{eq:ExtensionRightAlgStr}, respectively. Then the homomorphisms
$g$ and $E^{\psi}(g)\colon E^{\ualpha}(A)\to E^{\ualpha}(A')$ \eqref{eq:TwExtAlgMap}
are compatible with $s_A$ and $s_{A'}$.
Hence, by applying Definition \ref{def:connectionScalarExt} 
and Remark \ref{rmk:ConnScalarExtStratTensProd} (2)
to $(f,g,\psi)$, we obtain 
a scalar extension functor preserving tensor products
\begin{equation}
(g/f,\psi)^*\colon q\HIG(A/R,\ut)\longrightarrow q\HIG(A'/R',\ut').
\end{equation}
\par
(2) Let $(g/f,\psi)\colon ((A,J)/R,\ut=(t_i)_{i\in\Lambda})
\to ((A',J')/R',\ut'=(t'_{i'})_{i'\in\Lambda'})$ be a morphism
of framed smooth $q$-pairs such that there exist $q$-prismatic
envelopes $D$ and $D'$ of $(A,J)$ and $(A',J')$, and let
 $g_D\colon D\to D'$ be the unique morphism of $q$-prisms
 extending $g$. For $\ualpha$, $\ualpha'$, $s_D\colon D\to 
 E^{\alpha}(D)$, and $s_{D'}\colon D'\to E^{\ualpha'}(D')$
 defined similarly as (1) by using the $q$-Higgs derivations
 $\theta_{D,i}$ $(i\in\Lambda)$ and $\theta_{D',i'}$ $(i'\in \Lambda')$,
 the homomorphisms $g_D$ and $E^{\psi}(g_D)\colon 
 E^{\ualpha}(D)\to E^{\ualpha'}(D')$ \eqref{eq:TwExtAlgMap} are compatible with $s_D$
 and $s_{D'}$. Hence, by applying Definition \ref{def:connectionScalarExt} 
and Remark \ref{rmk:ConnScalarExtStratTensProd} (2)
to $(f, g_D,\psi)$, we obtain 
a scalar extension functor preserving tensor products
\begin{equation}\label{eq:qPrismEnvHigPB}
(g_D/f,\psi)^*\colon q\HIG(D/R,\ut)\longrightarrow q\HIG(D'/R',\ut').
\end{equation}
\end{proposition}
To prove Proposition \ref{prop:qHiggsDerivFunct} (2), 
we introduce a $\delta$-structure on $E^{\ualpha}(A)$ similarly
to $E^{\alpha}(A)$ (Lemma \ref{lem:DeltaStrTwExt}).
Let $A$ be a $\delta$-algebra, and let 
$\ualpha=(\alpha_{\sigma})_{\sigma\in \Sigma}$ and 
$\ubeta=(\beta_{\sigma})_{\sigma\in\Sigma}$ be a finite family
of elements of $A$ satisfying $\delta(\alpha_{\sigma})=\alpha_{\sigma}
\beta_{\sigma}$  for every $\sigma\in \Sigma$. 
We equip $E^{\ualpha}(A)$ with 
the $\delta$-$A$-algebra structure 
defined by $\delta(T_{\sigma})=\beta_{\sigma}T_{\sigma}$
$(\sigma\in\Sigma)$; one can verify that the $\delta$-structure is well-defined
by the same argument as the proof of Lemma \ref{lem:DeltaStrTwExt}. 
If we decompose $\Sigma$ into a disjoint union $\Sigma=\Sigma_1\sqcup
\Sigma_2$, and set $\ualpha_i=(\alpha_{\sigma})_{\sigma\in \Sigma_i}$
$\ubeta_i=(\beta_{\sigma})_{\sigma\in\Sigma_i}$ $(i=1,2)$, then
the isomorphism of $A$-algebras 
$E^{\ualpha}(A)\cong E^{\ualpha_1}(E^{\ualpha_2}(A))$
defined by $T_{\sigma}\mapsto T_{\sigma}$ $(\sigma\in \Sigma)$
underlies an isomorphism of $\delta$-$A$-algebras
\begin{equation}
E_{\delta}^{\ualpha,\ubeta}(A)\cong E_{\delta}^{\ualpha_1,\ubeta_1}
(E_{\delta}^{\ualpha_2,\ubeta_2}(A))
\end{equation}

Let $A'$, $\ualpha'=(\alpha'_{\sigma'})_{\sigma'\in \Sigma'}$,
and $\ubeta'=(\beta'_{\sigma'})_{\sigma'\in\Sigma'}$ be another
set of data satisfying the same conditions as $A$, $\ualpha$,
and $\ubeta$ above, and suppose that we are given a 
$\delta$-homomorphism $f\colon A\to A'$ and a map 
$\psi\colon \Sigma\to \Sigma'$ such that $f(\alpha_{\sigma})
=\alpha'_{\psi(\sigma)}$ and $f(\beta_{\sigma})=
\beta'_{\psi(\sigma)}$. Then the extension
$E^{\psi}(f)\colon E^{\ualpha}(A)\to E^{\ualpha'}(A')$
of $f$ defined by $T_{\sigma}\mapsto T_{\psi(\sigma)}$
underlies a $\delta$-homomorphism 
\begin{equation}\label{eq:DeltaTwMultExtFunct}
E_{\delta}^{\psi}(f)\colon E_{\delta}^{\ualpha,\ubeta}(A)\longrightarrow
E_{\delta}^{\ualpha',\ubeta'}(A').
\end{equation}

Now we suppose that $A$ is a $\delta$-algebra over a $\delta$-ring $R$, and
we are given $\partial_{\sigma}\in \Der_{R,\delta}^{\alpha_{\sigma},\beta_{\sigma}}(A)$ for each $\sigma\in \Sigma$ satisfying $\partial_{\sigma}(\alpha_{\sigma'})
=\partial_{\sigma}(\beta_{\sigma'})=0$ and $\partial_{\sigma}\circ\partial_{\sigma'}
=\partial_{\sigma'}\circ\partial_{\sigma}$
 for every $\sigma$, $\sigma'\in \Sigma$,
$\sigma\neq\sigma'$. 
We see that the $R$-algebra homomorphism 
$s\colon A\to E^{\ualpha}(A)$ \eqref{eq:ExtensionRightAlgStr}
defined by $\partial_{\sigma}$ $(\sigma\in \Sigma)$
underlies a $\delta$-$R$-homomorphism
\begin{equation}\label{eq:TwDeltaMultExtSection}
s_{\delta}\colon A\longrightarrow E_{\delta}^{\ualpha,\ubeta}(A)
\end{equation}
as follows.\par
Let $s_{\sigma}\colon A\to E_{\delta}^{\alpha_{\sigma},\beta_{\sigma}}(A)$ be the $\delta$-$R$-homomorphism corresponding to 
$\partial_{\sigma}$ by Proposition \ref{prop:DeltaAlphaDerivInt} (2). 
We have $s_{\sigma}(\alpha_{\sigma'})=\alpha_{\sigma'}$
and $s_{\sigma}(\beta_{\sigma'})=\beta_{\sigma'}$ for
$\sigma\neq\sigma'$ by the assumption on $\partial_{\sigma}$.
Hence, by choosing a total order $\leq$ on $\Sigma$, one can define
a $\delta$-$R$-homomorphism 
$s_{(\Sigma,\leq)}\colon A\longrightarrow E_{\delta}^{\ualpha,\ubeta}(A)$
by composing the $\delta$-$R$-homomorphisms
$$E_{\delta}^{\id_{\Sigma_{<\sigma}}}(s_{\sigma})\colon 
E_{\delta}^{\ualpha_{<\sigma},\ubeta_{<\sigma}}(A)
\to E_{\delta}^{\ualpha_{<\sigma},\ubeta_{<\sigma}}
(E_{\delta}^{\alpha_{\sigma},\beta_{\sigma}}(A))\cong 
E_{\delta}^{\ualpha_{\leq\sigma},\ubeta_{\leq\sigma}}(A)
\quad (\sigma\in \Sigma),
$$
where $\Sigma_{<\sigma}=\{\tau\in \Sigma; \tau<\sigma\}$,
$\Sigma_{\leq\sigma}=\{\tau\in\Sigma;\tau\leq \sigma\}$,
$\ualpha_{<\sigma}=(\alpha_{\tau})_{\tau\in\Sigma_{<\sigma}}$,
$\ubeta_{<\sigma}=(\beta_{\tau})_{\tau\in\Sigma_{<\sigma}}$,
$\ualpha_{\leq\sigma}=(\alpha_{\tau})_{\tau\in\Sigma_{\leq \sigma}}$, and
$\ubeta_{\leq\sigma}=(\beta_{\tau})_{\tau\in\Sigma_{\leq \sigma}}$.
Since the underlying $R$-homomorphism of $s_{\sigma}$ corresponds
to the $\alpha_{\sigma}$-derivation $\partial_{\sigma}$ over $R$
by Proposition \ref{prop:AlphaDerivInterpret}, 
we see that $s$ underlies $s_{(\Sigma,\leq)}$.

\begin{proof}[Proof of Proposition \ref{prop:qHiggsDerivFunct}]
(1) The composition of $E^{\psi}(g)\circ s_A$ and $s_{A'}\circ g$ with
the $A$-homomorphism $\pi'\colon E^{\ualpha'}(A')\to 
A'$ defined by $T_{i'}\mapsto 0$ $(i'\in \Lambda')$ both become $g$,
the kernel of $\pi'$ is $(p,\pq)$-adically nilpotent as 
$T_{i'}^{n+1}=\alpha_{i'}^{\prime n}T_{i'}$ $(i'\in \Lambda)$, and
$E^{\ualpha'}(A')$ is $(p,\pq)$-adically separated.
Hence it suffices to prove that $E^{\psi}(g)\circ s_A$ and $s_{A'}\circ g$
coincide after composing with the $(p,\pq)$-adically \'etale homomorphism 
$R[U_i;i\in \Lambda]\to A;U_i\mapsto t_i$ $(i\in \Lambda)$. 
Since $g$ and $E^{\psi}(g)$ are compatible with 
$f\colon R\to R'$ and $s_A$ (resp.~$s_{A'}$) is an $R$ (resp.~$R'$)-homomorphism, 
this is reduced to the following simple computation
for $i\in \Lambda$. 
\begin{align*}
&E^{\psi}(g)\circ s_A(t_i)=
E^{\psi}(g)(t_i+\pq T_i)=t'_{\psi(i)}+\pq T_{\psi(i)},\\
&s_{A'}\circ g(t_i)=s_{A'}(t_{\psi(i)}')=t'_{\psi(i)}+\pq T_{\psi (i)}.
\end{align*}

(2) The ring homomorphisms $A\to D$ and $A'\to D'$, and those 
induced by them $E^{\ualpha}(A)\to E^{\ualpha}(D)$
and $E^{\ualpha'}(A')\to E^{\ualpha'}(D')$ are compatible with 
the pairs of ring homomorphisms
$(g,g_D)$, $(E^{\psi}(g),E^{\psi}(g_D))$, 
$(s_A,s_D)$, and $(s_{A'},s_{D'})$. The compatibility for the latter two
follows from the compatibility of $A\to D$ and $A'\to D'$ with
the pairs $(\theta_{A,i},\theta_{D,i})$
and $(\theta_{A', i'},\theta_{D', i'})$.  Hence, by the claim (1), 
the desired compatibility holds after composing with 
$A\to D$.  
Put $\beta_i=t_i^{p-1}\eta$ $(i\in \Lambda)$, 
$\beta'_{i'}=t'_{i'}\eta$ $(i'\in \Lambda')$,
$\ubeta=(\beta_i)_{i\in\Lambda}$, and $\ubeta'=(\beta'_{i'})_{i'\in\Lambda'}$. 
The $\alpha_i$-derivation $\theta_{D,i}$ $(i\in \Lambda)$ (resp.~
$\alpha'_{i'}$-derivation $\theta_{D',i'}$ $(i'\in\Lambda')$) is 
$\delta$-compatible with respect to $\beta_i$
(resp.~$\beta'_{i'}$), $\theta_{D,i}(\beta_j)=0$ $(i,j\in \Lambda,i\neq j)$,
$\theta_{D',i'}(\beta'_{j'})$ $(i',j'\in\Lambda', i'\neq j')$, and 
$g(\beta_i)=\beta'_{\psi(i)}$ $(i\in \Lambda)$.
Hence the two compositions $E^{\psi}(g_D)\circ s_D$
and $s_{D'}\circ g_D$ in question define 
morphisms of $q$-prisms 
$D\to E_{\delta}^{\ualpha,\beta}(D)\to E_{\delta}^{\ualpha',\ubeta'}(D')$
and $D\to D'\to E_{\delta}^{\ualpha' ,\ubeta'}(D')$
by \eqref{eq:DeltaTwMultExtFunct} and \eqref{eq:TwDeltaMultExtSection}. 
Hence the universality
of the $q$-prismatic envelope $A\to D$ implies that these 
two compositions coincide.
\end{proof}

Let $(g/f,\psi)$ (resp.~$(g/f,\psi)$ and $g_D$) be as in 
Proposition \ref{prop:qHiggsDerivFunct} 
(1) (resp.~(2)). Let $(M,\theta_M)$ be a 
$q$-Higgs module over $(A/R,\ut)$ (resp.~$(D/R,\ut)$)
(Definition \ref{def:qHiggsMod}), 
and let $(M',\theta_{M'})$ be its scalar
extension under $(g/f,\psi)$ (resp.~$(g_D/f,\psi)$) 
(Proposition \ref{prop:qHiggsDerivFunct}). Then, 
by Proposition \ref{prop:qHiggsDerivFunct}, we can
apply Proposition \ref{prop:dRCpxFunct} to $(f,g,\psi)$ (resp.~$(f,g_D,\psi)$)
and the total order of $\Lambda$ assumed to be
given by the definition of a framed smooth $q$-prism
(resp.~$q$-pair), obtaining a morphism between 
$q$-Higgs complexes
\begin{align}
&q\Omega^{\bullet}_{g/f,\psi}(M)\colon
(M\otimes_Aq\Omega^{\bullet}_{A/R,\ut},\theta_M^{\bullet})
\longrightarrow
(M'\otimes_{A'}q\Omega^{\bullet}_{A'/R',\ut'},\theta_{M'}^{\bullet})\\
(\text{resp. }\quad
&q\Omega^{\bullet}_{g_D/f,\psi}(M)\colon
(M\otimes_Dq\Omega^{\bullet}_{D/R,\ut},\theta_M^{\bullet})
\longrightarrow
(M'\otimes_{D'}q\Omega^{\bullet}_{D'/R',\ut'},\theta_{M'}^{\bullet})
\quad).\label{eq:qDolbCpxPB2}
\end{align}
We abbreviate $q\Omega^{\bullet}_{g/f,\psi}$ (resp.~$q\Omega^{\bullet}_{g_D/f,\psi}$)
to $q\Omega^{\bullet}_{g, \psi}$ (resp.~$q\Omega^{\bullet}_{g_D,\psi}$)
when $R=R'$ and $f=\id$. 
Since $\psi\colon \Lambda\to\Lambda'$ is a map of ordered
sets, the morphisms 
$q\Omega^{\bullet}_{g/f,\psi}(M)$
(resp.~$q\Omega^{\bullet}_{g_D/f,\psi}(M)$) are 
compatible with compositions of $(g/f,\psi)$'s by 
Lemma \ref{lem:dRCpxFuncCocyc} (2). 

\begin{remark}\label{rmk:FramedSmQPHigFPBComp}
Let the notation and assumption be the same as Proposition \ref{prop:qHiggsDerivFunct} (2).
Let $J$ be an ideal of $\Z[q]$ satisfying $\varphi(J)\subset J$,
and let $(\oD,\varphi_{\oD})$ and $(\oD',\varphi_{\oD'})$ be
the reduction modulo $J$ of $(D,\varphi_D)$ and $(D',\varphi_{D'})$
as in Remark \ref{rmk:FramedSmQPHigFPB}. Then, by \eqref{eq:qConnFrobPBFunct},
the restriction $q\HIG(\oD/R,\ut)\to q\HIG(\oD'/R)$ of
the functor $(g_D/f,\psi)^*$ \eqref{eq:qPrismEnvHigPB}
is compatible with the Frobenius pullback functors 
$\varphi_{\oD}^*$ and $\varphi_{\oD'}^*$ \eqref{eq:FramedSmQPHigFPBFunct}.
Moreover, for $(M,\theta_M)\in \Ob q\HIG(\oD/R,\ut)$
and $(M',\theta_{M'})=(g_D/f,\psi)^*(M,\theta_M)$, the morphisms
$q\Omega^{\bullet}_{g_D/f,\psi}(M)$ and
$q\Omega^{\bullet}_{g_D/f,\psi}(\varphi_{\oD}^*M)$ \eqref{eq:qDolbCpxPB2}
are compatible with $\varphi_{\oD,\Omega}^{\bullet}(M)$
and $\varphi_{\oD',\Omega}^{\bullet}(M')$ \eqref{eq:FramedSmQPHigFPBMap}
by \eqref{eq:qdRcpxFrobPBFunct}.
\end{remark}

\begin{proposition}\label{prop:qprismEnvPL}
Let $R$ be a $q$-prism, and let 
$(g,\psi)\colon ((A',J'),(t'_{i'})_{i'\in\Lambda'})\to ((A,J),(t_i)_{i\in\Lambda})$
be a morphism of framed smooth $\delta$-pairs over $R$
(Definition \ref{def:FramedSmoothPrism} (3)) such that $\psi$ is injective
and $g$ induces an isomorphism $A'/J'\xrightarrow{\cong}A/J$.
Assume that $(A',J')$ has a $q$-prismatic envelope $(D',[p]_qD')$
(Definition \ref{def:FramedSmoothPrism} (1)).
Put $\Lambda_{\psi}=\Lambda\backslash\psi(\Lambda')$. 
Then the following hold.\par
(1) There exists a $q$-prismatic envelope $(D,[p]_qD)$ of $(A,J)$.\par
(2) For each $i\in \Lambda_{\psi}$, choose 
$a_i\in A'$ such that the images of $g(a_i)$ and $t_i$ in $A/J$ coincide, 
i.e., $t_i-g(a_i)\in J$, let $\tau_i$ be the unique element of $D$
such that $[p]_q\tau_i=t_i-g(a_i)$, and define 
$\tau_i^{\{m\}_{\delta}}$ for $m\in \N$ as \eqref{eq:TauDeltaPower}.
Then, for any ideal $K$ of $D'$ containing a power of $pD'+[p]_qD'$,
$D/KD$ is a free $D'/K$-module with a basis 
$\prod_{i\in\Lambda_{\psi}}\tau_i^{\{m_i\}_{\delta}}$
mod $KD$ $((m_i)\in \N^{\Lambda_{\psi}})$.\par
(3) The $q$-Higgs derivation $\theta_{D,i}$ $(i\in \Lambda)$ 
of $D$ over $R$ with respect to $(t_i)_{i\in\Lambda}$
(Definition \ref{def:qHiggsDifferential} (2)) is $D'$-linear if
$i\in \Lambda_{\psi}$.\par
(4) Let $(q\Omega^{\bullet}_{D/D'},
\theta^{\bullet}_{D/D'})$ be the differential graded algebra over $D$
obtained by applying the construction in \S\ref{sec:connection}
to $(D/D',(t_i\mu)_{i\in \Lambda_{\psi}},
(\theta_{D,i})_{i\in\Lambda_{\psi}})$.
Then, for any $D'$-module $N$ annihilated by a power of 
$(p,[p]_q)$, the complex 
$(N\otimes_{D'}q\Omega^{\bullet}_{D/D'},\id_N\otimes\theta^{\bullet}_{D/D'})$
gives a resolution of $N$.\par
(5) Under the notation in (2), let $\oa_i$ and $\otau_i$ 
$(i\in \Lambda_{\psi})$ be the images of
$a_i$ and $\tau_i$ in $\oD'=D'/\mu D'$ and
$\oD=D/\mu D$, respectively, and assume that we are given
$\ob_i\in \oD'$ $(i\in \Lambda_{\psi})$ such that $\delta(\oa_i)=p\ob_i$. 
Put $\osigma_i=(1-p^{p-1})^{-1}(-\ob_i-\delta(\tau_i)+
\sum_{\nu=1}^{p-1}p^{-1}\binom {p}{\nu}\oa_i^{p-\nu}
p^{\nu-1}\otau_i^{\nu})\in \oD$. Then we have 
$\otau_i^p=p\osigma_i$, and the homomorphism 
\eqref{eq:DeltaRingPDMap2} associated to 
$(\otau_i)$ and $(\osigma_i)$ induces an isomorphism of $\oD'_N$-algebras
$$\oD'_N[X_i;i\in \Lambda_{\psi}]_{\PD}\xrightarrow{\;\cong\;}
\oD_N$$
for $N\in \N$, where $\oD'_N=\oD'/p^{N+1}\oD'$ and $\oD_N=\oD/p^{N+1}\oD$.
Let $\otau_i^{[n]}$ $(i\in \Lambda_{\psi},
n\in \N)$ be the image of $X_i^{[n]}$ under the isomorphism above,
and let $\theta_{\oD_N,i}$ $(i\in \Lambda_{\psi})$
be the $\oD'_N$-linear derivation of $\oD_N$  induced by $\theta_{D,i}$.
(Note that $\theta_{D,i}$ is a $t_i\mu$-derivation over $D'$ by (3)
and $t_i\mu \oD_N=0$). Then we have $\theta_{\oD_N,i}(
\otau_i^{[n+1]})=\otau_i^{[n]}$ $(n\in \N)$
and $\theta_{\oD_N,i}(\otau_j^{[n]})
=0$ $(i,j\in \Lambda_{\psi}, j\neq i, n\in \N)$.
\end{proposition}

\begin{proof}
Since $\psi$ is injective, $g$ is $(p,\pq)$-adically smooth and 
$t_i$ $(i\in\Lambda_{\psi})$ are $(p,\pq)$-adic coordinates
of $A$ over $A'$ (Definition \ref{def:formallyflat}).
Hence, by Corollary \ref{cor:PrismEnvSection} and its proof, 
there exists $f\in 1+J\cdot (A\otimes_{A'}D')$,
such that the $(p,[p]_q)$-adic completion $B$ of 
$(A\otimes_{A'}D')_f$ regarded as a $\delta$-$D'$-algebra,
and the image $T_i$ of $t_i-f(a_i)\in A$ in $B$ for
$i\in \Lambda_{\psi}$ satisfy the 
assumptions of Proposition \ref{prop:DeltaEnvRegSeq} (4), 
and the $(p,[p]_q)$-adic completion of the 
$\delta$-envelope of $B[S_i; i\in \Lambda_{\psi}]/
(\xi S_i-T_i;i\in\Lambda_{\psi})$ over $B$
gives the $q$-prismatic envelope $D$ of $(A,J)$
and the claim (2) holds. By Proposition \ref{prop:qHiggsDerivFunct} (2)
and Lemma \ref{lem:TwDerivSystemFunct} (1), we have
$\theta_{D,i}(D'\cdot 1_D)=0$ for $i\in \Lambda_{\psi}$,
which implies the claim (3) (Remark \ref{rmk:alphaDerivBC} (4)).
This also shows $\theta_{D,i}(a_j)=0$ $(i,j\in \Lambda_{\psi})$.
Hence $\theta_{D,i}(\tau_i)=1$ and $\theta_{D,i}(\tau_j)=0$
for $i,j\in \Lambda_{\psi}$, $i\neq j$ as
$\theta_{D,i}(t_i)=[p]_q$, $\theta_{D,i}(t_j)=0$, and
$\theta_{D,i}$ is $R$-linear. Therefore $\theta_{D,i}$
$(i\in \Lambda_{\psi})$ 
coincides with the $t_i\mu$-derivation obtained by applying
Proposition \ref{prop:PrismEnvDeriv} 
to the $\delta$-$D'$-algebra $B$
and $T_i$ $(i\in \Lambda_{\psi})$,
whence the claim (4) (resp.~(5)) holds by 
Theorem \ref{thm:PoincareLem} (resp.~Proposition \ref{prop:DivDeltaEnvPDEnv} and 
Remark \ref{rmk:DeltaEnvPDEnv}).
\end{proof}

\section{Prismatic crystals, stratifications, and $q$-Higgs fields}
\label{sec:CrystalStratqHiggs}
\begin{definition}[{\cite[Definition 4.1]{BS}}]\label{def:PrismaticSite}
Let $(R,I)$ be a bounded prism (Definition \ref{def:prism}), and let 
$\fX$ be a $p$-adic formal scheme over $\Spf(R/I)$, 
where $R/I$ is equipped with the $p$-adic topology
(Corollary \ref{cor:bddPrismRedpcomp}). \par
(1) {\it The prismatic site $(\fX/(R,I))_{\prism}$ (or $(\fX/R)_{\prism}$ for short)
of $\fX$ over $(R,I)$} is defined as follows. An object of the underlying category is
a pair $((P,IP),v)$ of a bounded prism $(P,IP)$ over $(R,I)$ and a
morphism $v\colon \Spf(P/IP)\to \fX$ over $\Spf(R/I)$. 
A morphism $u\colon ((P',IP'),v')\to ((P,IP),v)$ in the underlying 
category is a morphism of bounded prisms 
$u\colon (P,IP)\to (P',IP')$ over $(R,IR)$
compatible with $v$ and $v'$, i.e., $v\circ \Spf(u\otimes R/I)=v'$.
The fiber product of two morphisms $u_i\colon ((P_i,IP_i),v_i)\to ((P,IP),v)$ 
$(i=1,2)$ is represented by the $pR+I$-adic completion of 
$P_1\otimes_{P}P_2$ if one of $u_i$ is $pR+I$-adically flat
(Definition \ref{def:formallyflat} (1)) 
by Corollary \ref{cor:bddPrismFlatMap} and Lemma \ref{lem:CompFinGenIdeal}.\par
We equip this category with the topology associated to 
the pretopology defined by 
finite families of morphisms
$(u_{\lambda}\colon ((P_{\lambda},IP_{\lambda}),v_{\lambda})
\to ((P,IP),v))_{\lambda\in \Lambda}$, $\sharp \Lambda<\infty$ 
such that $u_{\lambda}$
is $pR+I$-adically flat for every $\lambda\in\Lambda$ and 
$(\Spf(u_{\lambda})\colon \Spf(P_{\lambda})\to \Spf(P))_{\lambda\in \Lambda}$
is a surjective family of morphisms. 
For an object $((P,IP),v)$ of $(\fX/R)_{\prism}$, 
we write $\Cov_{\fpqc}((P,IP),v)$ for the set of families 
of morphisms as above with target $((P,IP),v)$. 
\par
We write $(P,IP)$, $P$, or $(P,v)$ for an object 
$((P,IP),v)$ of $(\fX/R)_{\prism}$ if there is no risk of confusion. 
For an object $P$ of $(\fX/R)_{\prism}$, we write
$\varphi_P$ for the lifting of Frobenius of $P$
associated to the $\delta$-structure of $P$, 
define $P_n$ $(n\in\N)$ to be $P/(pP+IP)^{n+1}$,
and write $\varphi_{P_n}$ for the lifting of Frobenius of $P_n$
induced by $\varphi_P$.
We define presheaves of rings $\CO_{\fX/R}$ and $\CO_{\fX/R,n}$ 
$(n\in \N)$ on 
$(\fX/R)_{\prism}$ by $\CO_{\fX/R}(P)=P$ and $\CO_{\fX/R,n}(P)
=P_n$. These are sheaves by \cite[VIII Corollaire 1.5]{SGA1}.
We define the lifting of Frobenius $\varphi=\varphi_{\fX/R}$
(resp.~$\varphi_n=\varphi_{\fX/R,n}$ $(n\in\N)$) of
$\CO_{\fX/R}$ (resp.~$\CO_{\fX/R,n}$) by $\varphi(P)=\varphi_P$
(resp.~$\varphi_n(P)=\varphi_{P_n}$).
\par
(2) {\it A  crystal of $\CO_{\fX/R,n}$-modules on} $(\fX/R)_{\prism}$ 
is a presheaf of $\CO_{\fX/R,n}$-modules $\CF$ such that,
for any morphism $(u\colon P'\to P)$ in $(\fX/R)_{\prism}$,
the morphism $\CF(u)\colon \CF(P)\to \CF(P')$ induces
an isomorphism $\CF(P)\otimes_{P,u}P'\xrightarrow{\cong}\CF(P')$.
Let $\Crystal_{\prism}(\CO_{\fX/R,n})$ be the category of crystals
of $\CO_{\fX/R,n}$-modules on $(\fX/R)_{\prism}$.
Any crystal of $\CO_{\fX/R,n}$-module is a sheaf by 
\cite[VIII Corollaire 1.5]{SGA1}.\par
{\it A complete crystal (or a crystal, shortly) of $\CO_{\fX/R}$-modules on $(\fX/R)_{\prism}$}
is a presheaf of $\CO_{\fX/R}$-modules $\CF$ such that 
$\CF(P)$ is $pR+I$-adically complete and separated for every 
$P\in (\fX/R)_{\prism}$, and that the presheaf quotient 
$\CF/(pR+I)^{n+1}\CF$ is a crystal of $\CO_{\fX/R,n}$-modules
for every $n\in \N$. By Lemma \ref{lem:CompFinGenIdeal}, the latter condition is
equivalent to saying that, for any morphism $(u\colon (P',v')\to (P,v))$
in $(\fX/R)_{\prism}$, the morphism $\CF(u)\colon \CF(P)\to \CF(P')$
induces an isomorphism $\CF(P)\hotimes_{P,u}P'\xrightarrow{\cong}\CF(P')$,
where the source denotes $\varprojlim_n (\CF(P)\otimes_{P,u}P')/
((pR+I)^{n+1}(\CF(P)\otimes_{P,u}P'))$. 
Let $\hCrystal_{\prism}(\CO_{\fX/R})$ be the
category of complete crystals of $\CO_{\fX/R}$-modules.
A crystal of $\CO_{\fX/R,n}$-modules is a complete crystal of 
$\CO_{\fX/R}$-modules annihilated by $(pR+I)^{n+1}$. 
Any complete crystal of $\CO_{\fX/R}$-module is a
sheaf by \cite[VIII Corollaire 1.5]{SGA1}.
\par
(3) Suppose that we are given another pair of a bounded prism
$(R',I')$ and a $p$-adic formal scheme $\fX'$ over $\Spf(R'/I')$,
and a pair of a morphism of bounded prisms
$f\colon (R,I)\to (R',I')$ and a morphism $g\colon \fX'\to \fX$
compatible with $f$. Then the functor $(\fX'/R)_{\prism}
\to (\fX/R)_{\prism}$ sending $((P,IP),v)$ to
$((P,IP), g\circ v)$ is cocontinuous, and hence
defines a morphism of topos 
$(\fX'/R)_{\prism}^{\sim}
\to (\fX/R)_{\prism}^{\sim}$, which is denoted
by $g_{\prism}$ in the following. We have
 $g_{\prism}^*\CO_{\fX/R}=\CO_{\fX'/R'}$
 and $g_{\prism}^*\CO_{\fX/R,n}=
 \CO_{\fX'/R',n}$, and $g_{\prism}^*$ maps
crystals of $\CO_{\fX/R,n}$-modules
to crystals of $\CO_{\fX'/R',n}$-modules
and complete crystals of $\CO_{\fX/R}$-modules
to complete crystals of $\CO_{\fX'/R'}$-modules.
\end{definition}

\begin{remark}\label{rmk:CompleteCrystalInvSystem}
We keep the notation and assumption in Definition \ref{def:PrismaticSite}.\par
(1) For $n,m\in \N$ with $m<n$, and a crystal of 
$\CO_{\fX/R,n}$-modules $\CF$ on $(\fX/R)_{\prism}$,
the scalar extension $\CF\otimes_{\CO_{\fX/R,n}}\CO_{\fX/R,m}$
as a sheaf on $(\fX/R)_{\prism}$ is a crystal of $\CO_{\fX/R,m}$-modules and
coincides with the scalar extension as a presheaf because the presheaf
scalar extension $P\mapsto \CF(P)\otimes_{P_{n}}P_m$
is a crystal of $\CO_{\fX/R,m}$-modules, which is a sheaf.\par
(2) For $n\in\N$ and a complete crystal of 
$\CO_{\fX/R,n}$-modules $\CF$ on $(\fX/R)_{\prism}$,
the scalar extension $\CF\otimes_{\CO_{\fX/R}}\CO_{\fX/R,m}$
as a sheaf on $(\fX/R)_{\prism}$ is a crystal of $\CO_{\fX/R,n}$-modules and
coincides with the scalar extension as a presheaf because the presheaf
scalar extension $P\mapsto \CF(P)\otimes_{P}P_n$
is, by definition, a crystal of $\CO_{\fX/R,n}$-modules, which is a sheaf.\par
(3) Let $\Crystal_{\prism}(\CO_{\fX/R,\bullet})$ be the category 
of inverse systems of $\CO_{\fX/R,n}$-modules $\CF_{\bullet}=(\CF_n)_{n\geq 1}$
such that $\CF_n$ is a crystal of $\CO_{\fX/R,n}$-modules for every $n\in \N$.
We say that $\CF_{\bullet}$ is {\it adic} if the transition morphism
$\CF_{n+1}\to \CF_n$ induces an isomorphism $\CF_{n+1}\otimes_{\CO_{\fX/R,{n+1}}}
\CO_{\fX/R,n}\xrightarrow{\cong} \CF_n$ for every $n\in \N$.
Let $\Crystal^{\ad}_{\prism}(\CO_{\fX/R,\bullet})$ denote the
full subcategory of $\Crystal_{\prism}(\CO_{\fX/R,\bullet})$ consisting
of adic inverse systems.
\par
(4) For a complete crystal of $\CO_{\fX/R}$-modules $\CF$,
the remarks (1) and (2) above imply that the inverse system
$(\CF\otimes_{\CO_{\fX}}\CO_{\fX/R,n})_{n\geq 1}$
is an object of $\Crystal^{\ad}_{\prism}(\CO_{\fX/R,\bullet})$
and we have an isomorphism of sheaves of $\CO_{\fX/R}$-modules
$\CF\xrightarrow{\cong}\varprojlim_n(\CF\otimes_{\CO_{\fX/R}}\CO_{\fX/R,n})$.
For an inverse system $(\CF_n)_{n\in \N}\in \Ob \Crystal^{\ad}_{\prism}(\CO_{\fX/R,\bullet})$,
we have $\CF_{n+1}(P)\otimes_{P_{n+1}}P_n\xrightarrow{\cong}\CF_n(P)$
for $P\in \Ob (\fX/R)_{\prism}$ and $n\in \N$ by the remark (1) above. 
Hence, by applying Lemma \ref{lem:CompFinGenIdeal}
to $P$, $pP+IP$ and $(\CF_n(P))_{n\in \N}$, we see 
that the sheaf of $\CO_{\fX/R}$-modules $\CF=\varprojlim_n\CF_n$ 
is a complete crystal of $\CO_{\fX/R}$-modules,
and we have isomorphisms $\CF\otimes_{\CO_{\fX/R}}\CO_{\fX/R,n}
\xrightarrow{\cong} \CF_n$ $(n\in \N)$. Thus we obtain equivalences of
categories quasi-inverse to each other
\begin{equation}\label{eq:CompleteCrysInvSys}
\hCrystal_{\prism}(\CO_{\fX/R})\overset{\sim}\leftrightarrows
\Crystal_{\prism}^{\ad}(\CO_{\fX/R,\bullet})
\end{equation}
given by $\CF\mapsto (\CF\otimes_{\CO_{\fX/R}}\CO_{\fX/R,n})_{n\in\N}$
and $\varprojlim_n \CF_n$ \raisebox{0.3ex}{\rotatebox[origin=c]{180}{$\mapsto$}} $(\CF_n)_{n\in \N}$.
\par
\end{remark}
\begin{remark}\label{rmk:crystalFrobPBTensor}
We keep the notation and assumption in Definition \ref{def:PrismaticSite}.\par
(1) For a crystal of $\CO_{\fX/R,n}$-modules $\CF$ on $(\fX/R)_{\prism}$,
the scalar extension $\varphi_n^*\CF=\CF\otimes_{\CO_{\fX/R,n},\varphi_n}\CO_{\fX/R,n}$
by $\varphi_n\colon \CO_{\fX/R,n}\to \CO_{\fX/R,n}$ 
as a sheaf on $(\fX/R)_{\prism}$ is a crystal of $\CO_{\fX/R,n}$-modules and
coincides with the scalar extension as a presheaf because the presheaf
scalar extension $P\mapsto \CF(P)\otimes_{P_n,\varphi_{P_n}}P_n$
is a crystal of $\CO_{\fX/R,n}$-modules, which is a sheaf.
For $(\CF_n)_{n\in \N}\in \Ob \Crystal^{\ad}_{\prism}(\CO_{\fX/R,\bullet})$
(Remark \ref{rmk:CompleteCrystalInvSystem} (3)), $\varphi_{\bullet}^*
(\CF_n)_{n\in\N}:=(\varphi_n^*\CF_n)_{n\in\N}$ is an object
of $\Crystal^{\ad}_{\prism}(\CO_{\fX/R,\bullet})$. By the equivalence
\eqref{eq:CompleteCrysInvSys}, we can define a functor $\hvarphi^*\colon \hCrystal_{\prism}(\CO_{\fX/R})
\to \hCrystal_{\prism}(\CO_{\fX/R})$ by $\hvarphi^*\CF=
\varprojlim_n\varphi^*_n(\CF\otimes_{\CO_{\fX/R}}\CO_{\fX/R,n})$.
\par

(2) For crystals of $\CO_{\fX/R,n}$-modules $\CF$ and $\CG$
on $(\fX/R)_{\prism}$, the tensor product $\CF\otimes_{\CO_{\fX/R,n}}\CG$
as a sheaf on $(\fX/R)_{\prism}$ is a crystal of $\CO_{\fX/R,n}$-modules
and coincides with the tensor product as a presheaf because the
presheaf tensor product  $P\mapsto \CF(P)\otimes_{P_n}\CG(P)$
is a crystal of $\CO_{\fX/R,n}$-modules, which is a sheaf.
For $(\CF_n)_{n\in \N}$, $(\CG_n)_{n\in \N}\in \Ob \Crystal^{\ad}_{\prism}(\CO_{\fX/R,\bullet})$
(Remark \ref{rmk:CompleteCrystalInvSystem} (3)), 
$(\CF_n)_{n\in\N}\otimes_{\CO_{\fX/R,\bullet}}(\CG_n)_{n\in\N}:=
(\CF_n\otimes_{\CO_{\fX/R,n}}\CG_n)_{n\in\N}$ is an object
of $\Crystal^{\ad}_{\prism}(\CO_{\fX/R,\bullet})$. 
For $\CF$, $\CG\in \Ob\hCrystal_{\prism}(\CO_{\fX/R})$, by the equivalence
\eqref{eq:CompleteCrysInvSys}, we can define a tensor product
$\CF\hotimes_{\CO_{\fX/R}}\CG\in \Ob\hCrystal_{\prism}(\CO_{\fX/R})$
by $\varprojlim_n((\CF\otimes_{\CO_{\fX/R}}\CO_{\fX/R,n})\otimes_{\CO_{\fX/R,n}}
(\CG\otimes_{\CO_{\fX/R}}\CO_{\fX/R,n}))$.
\end{remark}
\begin{remark}\label{rmk:CrystalCompQCohShf}
We keep the notation and assumption in Definition \ref{def:PrismaticSite}.
For an object $((P,IP),v_P)$ of $(\fX/R)_{\prism}$,
let us write $P_n$ for $P/(pP+IP)^{n+1}$. 
Let $\CF$ be a crystal of $\CO_{\fX/R,n}$-modules
on $(\fX/R)_{\prism}$. Then, for each object $((P,IP),v_P)$
of $(\fX/R)_{\prism}$, the $P_n$-module
$\CF(P)$ defines a quasi-coherent module
$\CF_P$ on $\Spec(P_n)$, and for each morphism 
$u\colon ((P',IP'),v_{P'})\to ((P,IP),v_P)$ in $(\fX/R)_{\prism}$,
the isomorphism $\CF(P)\otimes_{P,u}P'\xrightarrow{\cong}
\CF(P')$ induced by $\CF(u)$ defines an isomorphism 
$\CF_u\colon \mathtt{u}_n^*\CF_P\xrightarrow{\cong}\CF_{P'}$,
where $\mathtt{u}_n$ denotes the morphism of schemes
$\Spec(P'_n)\to \Spec(P_n)$ induced by 
$u$. If $\Spf(P')$ is an open affine formal subscheme of $\Spf(P)$,
then $P'$ admits a unique $\delta$-$P$-algebra structure
(Proposition \ref{prop:DeltaStrExtEtSm} (1)), 
$(P',IP')$ with $v_{P'}\colon \Spf(P')
\overset{\iota}\hookrightarrow \Spf(P)
\xrightarrow{v_P}\fX$ becomes an object
of $(\fX/R)_{\prism}$ lying over $((P,IP),v_P)$, and 
we have an isomorphism $\CF_{\iota}\colon 
\CF_{P}\vert_{\Spec(P'_n)}\xrightarrow{\cong}
\CF_{P'}$. In particular, we have 
$\CF_P(\Spec(P'_n))=\CF((P',IP'),v_{P'})$.
The construction $\CF\mapsto (\CF_P,\CF_u)$ 
gives an equivalence between the category of 
crystals of $\CO_{\fX/R,n}$-modules on $(\fX/R)_{\prism}$
and that of data $(\CM_P,\CM_u)$ consisting of
a quasi-coherent module $\CM_P$ on 
$\Spec(P_n)$ for each object $P$ of $(\fX/R)_{\prism}$
and an isomorphism $\CM_u\colon 
\mathtt{u}_n^*\CM_{P}\xrightarrow{\cong}\CM_{P'}$
for each morphism $u\colon P'\to P$ in $(\fX/R)_{\prism}$ satisfying
$\CF_{\id_P}=\id_{\CF_P}$ and the cocycle condition
for composition of $u$'s. 
\end{remark}

Under the notation and assumption in Definition \ref{def:PrismaticSite}, 
we assume that $I$ is generated by an element $\xi$ of $R$,
$\fX$ is an affine formal scheme $\Spf(A)$, and we are given 
a smooth adic morphism  $\Spf(B)\to \Spf(R)$ 
and a closed immersion $i\colon \Spf(A)\to \Spf(B)$ 
over $\Spf(R)$ such that $B$ has $pR+I$-adic coordinates over $R$
(Definition \ref{def:formallyflat} (2)). We further 
assume that $B$ is given a $\delta$-$R$-algebra 
structure and satisfies the following, where $J=\Ker (i^*\colon B\to A)$.
\begin{condition}\label{cond:PrismCohDescription}
The $\delta$-pair $(B,J)$ has a 
bounded prismatic envelope over $(R,IR)$ (Definition \ref{def:bddPrsimEnv}), 
denoted by $(D,ID)$ in the following.
\end{condition}
We regard $(D,ID)$ as an 
objet of $(\fX/R)_{\prism}$ by the canonical morphism 
$v_D\colon \Spf(D/ID)\to \Spf(B/J)=\fX$.
\begin{proposition}\label{prop:PrismCofFinObj}
The object $(D,ID)$ is a covering of the final object
of the topos $(\fX/R)_{\prism}^{\sim}$. 
\end{proposition}

\begin{proof}(\cite[Proposition 1.1.2]{Chat}, \cite[Lemma 4.2]{YTian}). 
Let $((P,IP),v)$ be an object of $(\fX/R)_{\prism}$. 
Let $P_B$ be the $pR+I$-adic completion of $P\otimes_RB$,
let $i_{P_B}$ be the closed immersion $\Spf(P/IP)\to \Spf(P_B)$
induced by the closed immersion $\Spf(P/IP)\to \Spf(P)$
and the composition $\Spf(P/IP)\xrightarrow{v}\Spf(A)\to \Spf(B)$.
Since the morphism $P\to B_P$ is $pR+I$-adically smooth 
and has $pR+I$-adic coordinates,
the $\delta$-pair $(P_B,J_{P_B}:=\Ker(i^*_{P_B}))$ has a bounded prismatic
envelope $(D_B,ID_B)$ over $(R,IR)$, and the morphism 
$P\to D_B$ is $pR+I$-adically faithfully flat (Definition \ref{def:formallyflat} (1))
by Corollary \ref{cor:PrismEnvSection}. Hence 
$(D_B,ID_B)$ equipped with the morphism $v_{D_B}\colon \Spf(D_B/ID_B)\to \Spf(P/IP)\to \fX$
is a covering of $((P,IP),v)$ in $(\fX/R)_{\prism}$, and 
we also have a morphism $((D_B,ID_B),v_{D_B})\to ((D,ID),v_D)$
in $(\fX/R)_{\prism}$ induced by the morphism of $\delta$-pairs 
$(B,J)\to (P_B,J_{P_B})$. This completes the proof.
\end{proof}

For an integer $r\geq 0$, we define $\Spf(B(r))$ to be the
fiber product of $r+1$ copies of $\Spf(B)$ over $\Spf(R)$
with $B(r)$ being equipped with the unique $\delta$-structure
such that the $r+1$ maps $B\to B(r)$ are $\delta$-homomorphisms.
Let $i(r)\colon \Spf(A)\to \Spf(B(r))$ be the closed immersion 
induced by $i$, let $J(r)$ be the kernel of $i(r)^*\colon B(r)
\to A$, and let $D(r)$ be the bounded prismatic envelope
of the $\delta$-pair $(B(r), J(r))$ over $(R,IR)$,
which exists by Corollary \ref{cor:PrismEnvSection}. 
We regard $(D(r),I D(r))$ as an object of $(\fX/R)_{\prism}$
by the canonical morphism $v_{D(r)}\colon\Spf(D(r)/ID(r))\to \Spf(B(r)/J(r))=\fX$.
Then $((D(r),ID(r)),v_{D(r)})$ 
represents the product of
$r+1$ copies of $((D,ID),v_D)$ in $(\fX/R)_{\prism}$, and 
forms a simplicial object of $(\fX/R)_{\prism}$. Note that 
we have $D(0)=D$. The $\delta$-pairs $(B(r),J(r))$ $(r\in \N)$ 
form a cosimplicial $\delta$-pair over $(R,I)$. 

Let $[n]$ denote the ordered set
$\N\cap [0,n]$ for $n\in \N$. For a non-decreasing
map $u\colon [r]\to [s]$ $(r,s\in \N)$, 
we write $D(u)$ (resp.~$B(u)$) for
the corresponding morphism $(D(r),ID(r))\to (D(s),ID(s))$
(resp.~$(B(r),J(r))\to (B(s),J(s))$) of bounded prisms 
(resp.~$\delta$-pairs) over $(R,I)$.
Let $\Delta$ denote the unique surjective map
$[1]\to [0]$. Let $p_i$ $(i\in \{0,1\})$ denote
the map $[0]\to [1]$ sending $0$ to $i$, 
let $q_i$ $(i\in \{0,1,2\})$ denote the
map $[0]\to [2]$ sending $0$ to $i$,
and let $p_{ij}$ $((i,j)\in \{(0,1),(1,2),(0,2)\})$
denote the map $[1]\to [2]$ sending $0$ and $1$ to
$i$ and $j$,  respectively.
We have $p_{ij}\circ p_0=p_i$ and $p_{ij}\circ p_1=p_j$. 
We abbreviate $D(\Delta)$, $D(p_i)$, $D(q_i)$, and $D(p_{ij})$ to $\Delta$,
$p_i$, $q_i$ and $p_{ij}$, respectively.

We put $D_n=D/(pD+ID)^{n+1}$ and $D(r)_n=D(r)/(pD(r)+ID(r))^{n+1}$
for $r\in \N$ and $n\in \N$. 
For a homomorphism of rings $f\colon S\to S'$ and an $S$-module $M$
(resp.~a homomorphism of $S$-module $\varphi$),
we write $f^*M$ (resp.~$f^*(\varphi)$) for the scalar extension 
$M\otimes_{S,f}S'$ of $M$ (resp.~$\varphi\otimes_{S,f}S'$ of $\varphi$). 
under $f$.

\begin{definition}\label{def:PrismStrat}
Let $M$ be a $D$-module. {\it A stratification on $M$ with respect to 
$D(\bullet)$} is an isomorphism of $D(1)$-modules
$\varepsilon\colon p_1^*M\xrightarrow{\cong} p_0^*M$
satisfying the following two conditions:\par
(i) $\Delta^*(\varepsilon)=\id_M$.\par
(ii) The composition $p_{01}^*(\varepsilon)\circ p_{12}^*(\varepsilon)
\colon q_2^*M\xrightarrow{\cong}q_1^*M\xrightarrow{\cong} q_0^*M$
coincides with $p_{02}^*(\varepsilon)\colon q_2^*M\xrightarrow{\cong}
q_0^*M$. \par
A morphism $f\colon (M',\varepsilon')\to (M,\varepsilon)$
of $D$-modules with stratification with respect to $D(\bullet)$
is a homomorphism of $D$-modules $f\colon M'\to M$ compatible
with the stratifications, i.e., $p_0^*(f)\circ\varepsilon'=\varepsilon\circ p_1^*(f)$.
We write $\Strat(D(\bullet))$ for the category of $D$-modules with 
stratification with respect to $D(\bullet)$, and
$\Strat(D(\bullet)_n)$ $(n\in \N)$ for the full subcategory 
of $\Strat(D(\bullet))$ consisting of $D_n$-modules with stratification 
with respect to $D(\bullet)$.

For two $D$-modules with stratification $(M,\varepsilon)$ and $(M',\varepsilon')$, 
one can define a stratification on $M\otimes_D M'$ by 
$p_1^*(M\otimes_DM')\cong p_1^*(M)\otimes_{D(1)}p_1^*(M')
\xrightarrow{\varepsilon\otimes_{D(1)}\varepsilon'}
p_0^*M\otimes_{D(1)}p_0^*M'\cong p_0^*(M\otimes_DM')$.
We write $\varepsilon\otimes\varepsilon'$ for this stratification,
and call $\varepsilon\otimes\varepsilon'$
(resp.~$(M\otimes_DM', \varepsilon\otimes\varepsilon')$)
the {\it tensor product of $\varepsilon$ and $\varepsilon'$}
(resp.~$(M,\varepsilon)$ {\it and} $(M',\varepsilon')$).
This construction is functorial in $(M,\varepsilon)$ and $(M',\varepsilon')$.
\end{definition}

\begin{remark}\label{rmk:PrismStrat}
As it is well-known, we see that 
any $D(1)$-linear homomorphism $\varepsilon\colon p_1^*M
\to p_0^*M$ satisfying (i) and (ii) in Definition \ref{def:PrismStrat}
is an isomorphism as follows. For $r\in \N$ and $i\in [r]$,
let $p_i^{(r)}$ denote the map $[0]\to [r]$ sending $0$ to $i$.
Then, for any map $u\colon [r]\to [s]$, not necessarily non-decreasing, 
we have a unique homomorphism of $\delta$-pairs
$B(u)\colon (B(r),J(r))\to (B(s),J(s))$
satisfying $B(u)\circ B(p_i^{(r)})=B(p_{u(i)}^{(s)})$ for
every $i\in [r]$. Let $D(u)$ be the morphism of bounded prisms
$(D(r),ID(r))\to (D(s),ID(s))$ induced by $B(u)$, which satisfies
$D(u)\circ D(p_i^{(r)})=D(p_{u(i)}^{(s)})$. We have $D(u'\circ u)
=D(u')\circ D(u)$ for composable morphisms $u$ and $u'$.
Let $\iota$
be the map $[1]\to [1]$ sending $0$, $1$ to $1$, $0$,
and let $\iota_0$ (resp.~$\iota_1$) be the map 
$[2]\to [1]$ sending $0$, $1$, $2$ to $0$, $1$, $0$
(resp.~$1$, $0$, $1$). Then the composition of 
$p_{01}$, $p_{12}$, $p_{02}\colon [1]\to [2]$
with $\iota_0$ (resp.~$\iota_1$) is $\id$, $\iota$, $p_0\circ \Delta$
(resp.~$\iota$, $\id$, $p_1\circ \Delta$). 
Hence, by taking the scalar extension of the
equality $p_{01}^*(\varepsilon)\circ p_{12}^*(\varepsilon)=p_{02}^*(\varepsilon)$
in the condition (ii) under $D(\iota_0)$ (resp.~$D(\iota_1)$) and using
the condition (i), we obtain $\varepsilon \circ D(\iota)^*(\varepsilon)=\id_{p_0^*M}$
and $D(\iota)^*(\varepsilon)\circ\varepsilon=\id_{p_1^*M}$.
\end{remark}

Let $\CF$ be a crystal of $\CO_{\fX/R,n}$-modules on $(\fX/R)_{\prism}$,
put $M=\CF(D)$, and let $\varepsilon$ be the composition 
$p_1^*M\xrightarrow[\CF(p_1)]{\cong}\CF(D(1))\xleftarrow[\CF(p_0)]{\cong}p_0^*M$. 
Then,  by using $\Delta\circ p_0=\Delta\circ p_1=\id_D$, we see 
$\Delta^*(\varepsilon)=\id_M$. One can verify that $\varepsilon$
satisfies the cocycle condition (ii) in Definition \ref{def:PrismStrat}
by showing that $p_{ij}^*(\varepsilon)\colon q_j^*M\xrightarrow{\cong}
q_i^*M$ $((i,j)\in\{(0,1), (0,2), (1,2)\})$ coincides with the
composition $q_j^*M
\xrightarrow[\CF(q_j)]{\cong}\CF(D(2))\xleftarrow[\CF(q_i)]{\cong}q_i^*M$.
This construction is obviously functorial in $\CF$. 
By Proposition \ref{prop:PrismCofFinObj},
one can prove the following proposition by a standard argument
using the usual faithfully flat descent of quasi-coherent modules on schemes.

\begin{proposition}\label{prop:CrysStratEquiv}
The above construction gives an equivalence of categories
preserving tensor products
\begin{equation}\label{eq:CrysStratEquiv}
\Crystal_{\prism}(\CO_{\fX/R,n})\xrightarrow{\sim}
\Strat(D(\bullet)_n).\end{equation}
\end{proposition}

\begin{remark}\label{rmk:CrysStratEquivJAnn}
Let $n\in \N$, and let $J$ be an ideal of $R$ containing 
$(pR+\pq R)^{n+1}$. Then by construction, the equivalence
\eqref{eq:CrysStratEquiv} induces an equivalence between
the subcategories  consisting of objects annihilated by $J$.
\end{remark}

The equivalence \eqref{eq:CrysStratEquiv} is functorial in 
$(R,I)$ and $i\colon \fX=\Spf(A)\to \fY=\Spf(B)$ as follows.
Let $(R',I')$ be another bounded prism, and let $i'\colon\fX'=\Spf(A')\to \fY'=\Spf(B')$
be a closed immersion of affine formal schemes over $\Spf(R')$
satisfying the same conditions as $i\colon \fX\to \fY$ (see before Proposition 
\ref{prop:PrismCofFinObj}). Suppose that we are given a morphism of bounded prisms
$f\colon (R,I)\to (R',I')$ and morphisms $g\colon \fX'\to \fX$
and $h\colon \fY'\to \fY$ compatible with $f$ and
closed immersions $i$ and $i'$.  We define an object $((D',I'D'),v_{D'})$
of $(\fX'/R')_{\prism}$,  a cosimplicial $\delta$-pair $(B'(r), J'(r))$ $(r\in \N)$
over $(R',I')$,
and a simplicial object $((D'(r),I'D'(r)),v_{D'(r)})$ $(r\in \N)$ of $(\fX'/R')_{\prism}$ 
in the same way as $((D,ID),v_D)$, $(B(r),J(r))$, and
$((D(r),ID(r)),v_{D(r)})$ by using $(R',I')$ and $i'\colon \fX'\to \fY'$.
The morphism $h$ induces a morphism of cosimplicial $\delta$-pairs
$h_{B(\bullet)}\colon (B(\bullet),J(\bullet))\to (B'(\bullet),J'(\bullet))$ 
compatible with $f$, which induces a morphism 
$h_{D(\bullet)}\colon ((D'(\bullet),I'D'(\bullet)),g\circ v_{D'(\bullet)})
\to ((D(\bullet),ID(\bullet)), v_{D(\bullet)})$
of simplicial objects of $(\fX/R)_{\prism}$. Therefore one can 
define a functor 
\begin{equation}\label{eq:StratPullback}
h_{D(\bullet)}^*\colon \Strat(D(\bullet)_n)\longrightarrow \Strat(D'(\bullet)_n)
\end{equation}
by sending $(M,\varepsilon)$ to $h_D^*(M)=D'\otimes_DM$
equipped with the $D'(1)$-linear isomorphism 
$p_1^{\prime*}h_D^*M\cong h_{D(1)}^*p_1^*M
\xrightarrow[h_{D(1)}^*(\varepsilon)]{\cong}h_{D(1)}^*p_0^*M\cong 
p_0^{\prime*}h_D^*M$,
where $p_{\nu}^{\prime}$ $(\nu=0,1)$ denotes the morphism $D'(1)\to D'$
corresponding to $[0]\to [1];0\mapsto \nu$. By the construction
of the equivalence \eqref{eq:CrysStratEquiv},
it is straightforward to verify that the following diagram is 
commutative up to canonical isomorphism.\par
\begin{equation}\label{eq:FunctCrystalStrat}
\xymatrix@C=60pt{
\Crystal_{\prism}(\CO_{\fX/R,n})
\ar[r]^{\sim}_{\eqref{eq:CrysStratEquiv}}\ar[d]^{g_{\prism}^*}&
\Strat(D(\bullet)_n)\ar[d]^{h_{D(\bullet)}^*}\\
\Crystal_{\prism}(\CO_{\fX'/R',n})
\ar[r]^{\sim}_{\eqref{eq:CrysStratEquiv}}&
\Strat(D'(\bullet)_n)
}
\end{equation}

\par

In the rest of this section, we assume that $R$ and $B$ satisfy the following
conditions
\begin{condition}\label{cond:qPrism}
(i) The bounded prism $(R,I)$ is a $q$-prism (Definition \ref{def:FramedSmoothPrism} (1)).
\par
(ii) We are given  $pR+I$-adic coordinates $t_1,\ldots, t_d$ of $B$ over $R$ 
(Definition \ref{def:formallyflat} (2)) such that
$\delta(t_i)=0$ for every $i\in \N\cap [1,d]$. 
\end{condition}

We first introduce the $q$-Higgs differential and $q$-Higgs modules on $D(r)$ 
by applying results in \S\ref{qprismenv}.
For $r\in \N$ and $l\in [r]$, let $p_l^{(r)}$ denote the map $[0]\to [r]$ 
sending $0$ to $l$.
Put $\Lambda=\N\cap [1,d]$. For $r\in \N$, let $\Lambda(r)$
be the set $[r]\times \Lambda$ equipped with the lexicographic order.
For $(l,i)\in \Lambda(r)$, let $t_{l;i}^{(r)}$ denote the image of 
$t_i$ under $B(p_l^{(r)})\colon B\to B(r)$ and also its image
in $D(r)$. Put $\ut^{(r)}=(t_{l;i}^{(r)})_{(l,i)\in\Lambda(r)}$. 
Then $((B(r),J(r))/R, \ut^{(r)})$
is a framed smooth $q$-pair (Definition \ref{def:FramedSmoothPrism} (3))
and we have the associated $q$-Higgs derivations,
$q$-Higgs automorphisms, and $q$-Higgs differential of
$D(r)$ over $R$ with respect to $\ut^{(r)}$ (Definition \ref{def:qHiggsDifferential} (2)), which are denoted by 
$$\theta_{D(r),l;i},\;\; \gamma_{D(r),l;i} \;\;((l,i)\in\Lambda(r)), 
\;\;\text{ and }\;\;\theta_{D(r)}\colon D(r)\to q\Omega_{D(r)/R}
=\bigoplus_{(l,i)\in \Lambda(r)}D(r)\omega_{l;i}^{(r)}$$ in the following.
Note that the $\delta$-structure on $B(r)$ coincides with 
that defined by the frame, i.e., $\delta(t_{l;i}^{(r)})=0$.
For $(l,i)$, $(m,j)\in \Lambda(r)$, 
we have $\theta_{D(r), l;i}(t_{m;j}^{(r)})=0$ if $(l,i)\neq(m,j)$ and
$=[p]_q$ if $(l,i)=(m,j)$. We have
$\theta_{D(r)}(x)=\sum_{(l,i)\in\Lambda(r)}\theta_{D(r), l;i}(x)
\omega_{l;i}^{(r)}$. We write $q\HIG(D(r)/R)$ for the 
category of $q$-Higgs modules associated to the framed smooth $q$-pair
$((B(r),J(r))/R,  \ut^{(r)})$ (Definition 
\ref{def:qHiggsMod} (2)) and define $q\HIG(D(r)_n/R)$ $(n\in\N)$ to be
its full subcategory consisting of objects whose underlying
$D(r)$-modules are annihilated by $(pR+[p]_qR)^{n+1}$. 
For an object $(M,\theta_M)$ of $q\HIG(D(r)/R)$, 
we define endomorphisms $\theta_{M,l;i}$ $((l,i)\in\Lambda(r))$
by $\theta_M(x)=\sum_{(l,i)\in \Lambda(r)}\theta_{M,l;i}(x)\otimes\omega^{(r)}_{l;i}$
$(x\in M)$. In the case $r=0$, we identify $\Lambda(0)$ with $\Lambda$,
write $D(r)$ for $D$, and omit the superscript $(r)$ and the
subscript $l;$ $(l\in [r])$ as $\ut=(t_i)_{i\in\Lambda}$, $\theta_{D,i}$, $\gamma_{D,i}$,
$\theta_D$, $q\Omega_{D/R}$, 
$\omega_i$, and $\theta_{M,i}$.

Let $w\colon [s]\to [r]$
$(r,s\in \N)$ be an injective non-decreasing map,
and define a subset $\Lambda(w)$ of $\Lambda(r)$
to be $([r]\backslash w([s]))\times \Lambda$.
Then $(B(r),J(r))$ regarded as a $\delta$-pair over $B(s)$ via 
$B(w)$ and $(t_{l;i}^{(r)})_{(l,i)\in\Lambda(w)}$
form a smooth framed $q$-pair, and we obtain the
associated $q$-Higgs derivations, endomorphisms,
and differential; the first two coincide with 
$\theta_{D(r),l;i}$ and $\gamma_{D(r),l;i}$ for $(l,i)\in \Lambda(w)$
by the uniqueness of $\theta_{A,i}$ and $\theta_{D,i}$
in Proposition \ref{prop:qHiggsDeriv} (1) and (2).
We write the differential map as
$\theta_{D(w)}\colon D(r)\to q\Omega_{D(w)}
=\oplus_{(l,i)\in \Lambda(w)}D(r)\omega^{(r)}_{l;i}$,
which coincides with the composition of $\theta_{D(r)}$
with the natural projection map 
$q\Omega_{D(r)}\to q\Omega_{D(w)}$. 
We define $q\HIG(D(w))$ to be the category of $q$-Higgs
modules with respect the above framed smooth $q$-pair,
and $q\HIG(D(w)_n)$ to be its full subcategory 
consisting of objects whose underlying modules are
annihilated by $(pR+[p]_qR)^{n+1}$. The composition with
the projection $q\Omega_{D(r)/R}\to q\Omega_{D(w)}$
induces functors  $q\HIG(D(r)/R)\to q\HIG(D(w))$
and $q\HIG(D(r)_n/R)\to q\HIG(D(w)_n)$ $(n\in \N)$. 

Let $u\colon[r]\to [r']$ $(r, r'\in \N)$, be a non-decreasing
map (not necessarily injective). Then the homomorphisms
$B(u)\colon B(r)\to B(r')$, $\id_R$, and 
$u\times \id_{\Lambda}\colon \Lambda(r)\to \Lambda(r')$
define a morphism of framed smooth $q$-pairs
(Definition \ref{def:FramedSmoothPrism} (3)) 
$((B(r),J(r))/R,\ut^{(r)})
\to ((B(r'), J(r'))/R,\ut^{(r')})$.
Hence,  by applying Proposition \ref{prop:qHiggsDerivFunct} (2)
We obtain a functor 
\begin{equation}\label{eq:SimpPrisEnvHIGPB}
u^*\colon q\HIG(D(r)/R)\longrightarrow q\HIG(D(r')/R),
\end{equation}
which is restricted to functors
$q\HIG(D(r)_n/R)\longrightarrow q\HIG(D(r')_n/R)$ $(n\in\N)$. 
We have 
\begin{equation}\label{eq:DsimplicialDerivFormula}
\theta_{D(r'),l';i}(D(u)(a))
=\sum_{\emptyset\neq\ul\subset u^{-1}(l')}
D(u)\biggl(\biggl(\prod_{l\in \ul}\partial_{D(r),l;i}\biggr)(a)\biggr)\cdot
(\mu t_{l';i}^{(r')})^{\sharp \ul-1}
\quad((l',i)\in \Lambda(r'), a\in D(r))
\end{equation}
by Lemma \ref{lem:TwDerivSystemFunct} (1). 
For $(M,\theta_M)\in q\HIG(D(r)/R)$ and 
$(M',\theta_{M'})=u^*(M,\theta_M)\in q\HIG(D(r')/R)$,
we have 
\begin{equation}\label{eq:DSimplqHIGPBFormula}
\theta_{M',l';i}(x\otimes 1)=
\sum_{\emptyset\neq\ul\subset u^{-1}(l')}
\biggl(\prod_{l\in \ul}\theta_{M,l;i}\biggr)(x)\otimes(\mu t_{l';i}^{(r')})^{\sharp l-1}
\quad ((l',i)\in \Lambda(r'), x\in M)
\end{equation}
by Proposition \ref{prop:ConnPBFormula}.
The functors $u^*$ \eqref{eq:SimpPrisEnvHIGPB} for various $u$'s are compatible with compositions
by the remark after \eqref{eq:qDolbCpxPB2}. 

\begin{definition}\label{def:qNilpHiggsField}
For $n\in \N$, we say that an object $(M,\theta_M)$ of $q\HIG(D_n/R)$
is {\it quasi-nilpotent}, if it satisfies the following property: 
For any $x\in M$, there exists a positive integer $N$ such that
$(\theta_{M,i})^N(x)=0$ for every $i\in \Lambda$.
Let $q\HIG_{\qnilp}(D_n/R)$ denote the full subcategory of 
$q\HIG(D_n/R)$ consisting of quasi-nilpotent objects.
\end{definition}

\begin{theorem}\label{thm:StratHiggsEquiv}
There exists a canonical equivalence of categories
\begin{equation}\label{eq:StratHiggsEquiv}
\Strat(D(\bullet)_n)\xrightarrow{\sim}
q\HIG_{\qnilp}(D_n/R).\end{equation}
See Remark \ref{rmk:CrysFrobPBHiggs} 
for compatibility with scalar extensions by Frobenius
and with tensor products.
\end{theorem}

For $r\in \N$ and $l,m\in [r]$, $l\neq m$, we have
$t^{(r)}_{l;i}-t_{m;i}^{(r)}\in J(r)\subset B(r)$, which 
implies that its image in $D(r)$ is divisible by $[p]_q$, which
is regular on $D(r)$. We define $\tau^{(r)}_{l,m;i}\in D(r)$
to be the unique element satisfying
$$[p]_q\tau^{(r)}_{l,m;i}=t^{(r)}_{l;i}-t_{m;i}^{(r)}\quad\text{in }D(r).$$
For $(k,j)\in \Lambda(r)$, we have 
\begin{equation}
\theta_{D(r),k;j}(\tau_{l,m;i}^{(r)})
=\begin{cases}
0 &(\text{if }k\not\in\{l,m\} \text{ or }j\neq i),\\
1 &(\text{if }k=l\text{ and } j=i), \\
-1& (\text{if }k=m\text{ and }j=i). 
\end{cases}
\end{equation}
For $n\in \N$, we define $(\tau^{(r)}_{l,m;i})^{\{n\}_{\delta}}$
in the same way as \eqref{eq:TauDeltaPower}.

\begin{proposition}\label{prop:PrismEnvNervStr}
Let $u\colon [s]\to [r]$ $(s,r\in \N)$ be an injective non-decreasing
map,  and regard 
$D(r)$ as a $D(s)$-algebra by the homomorphism $D(u)$. 
Let $m$ be an element of $u([s])$.
Then, for each $N\in \N$, 
$D(r)_N$ is a free $D(s)_N$-module with basis
$\prod_{(l,i)\in \Lambda(u)}(\tau^{(r)}_{l,m;i})^{\{N_{l,i}\}_{\delta}}$,
$(N_{l,i})\in \N^{\Lambda(u)}$. In particular,
the homomorphism $D(u)$ is $pR+I$-adically faithfully flat
(Definition \ref{def:formallyflat} (1)). 
\end{proposition}

\begin{proof}
We obtain the claim simply applying Proposition \ref{prop:qprismEnvPL} (2)
to the morphism of framed smooth $q$-pairs
$((B(s),J(s)),\ut^{(s)})\to ((B(r),J(r)), \ut^{(r)})$ over $R$
defined by $B(u)\colon B(s)\to B(r)$ and 
$u\times\id_{\Lambda}\colon \Lambda(s)\to \Lambda(r)$. 
\end{proof}

For $r\in \N$ and $l,m\in [r]$, $l\neq m$, we define $\sigma_{l,m;i}^{(r)}\in D(r)$
by 
\begin{equation}
\sigma_{l,m;i}^{(r)}=
(1-p^{p-1})^{-1}\left(-\delta(\tau_{l,m;i}^{(r)})
+\sum_{\nu=1}^{p-1}p^{-1}\binom {p}{\nu} (t^{(r)}_{m;i})^{p-\nu}p^{\nu-1}
(\tau_{l,m;i}^{(r)})^{\nu}\right).
\end{equation}

We define $\oD(r)$ $(r\in \N)$ to be $D(r)/\mu D(r)$, and
$\oD(r)_N$ $(r\in \N, N\in \N)$ to be $\oD(r)/p^{N+1}\oD(r)$. 
We write $\otau^{(r)}_{l,m;i}$ and $\osigma^{(r)}_{l,m;i}$ for
the images of $\tau^{(r)}_{l,m;i}$ and $\sigma^{(r)}_{l,m;i}$ in 
$\oD(r)$ and in $\oD(r)_N$. For a non-decreasing map $u\colon [r]\to [r']$,
we write $\oD(u)$ (resp.~$\oD(u)_N$)  for the reduction mod $\mu$
(resp.~$\mu R+p^{N+1}R$) of the homomorphism $D(u)\colon D(r)\to D(r')$.

\begin{propdef}\label{propdef:DSimpPDStr}
Let $r\in \N$ and $m\in [r]$, and put $\Lambda(r)_m=\Lambda(p^{(r)}_m)
=[r]\backslash\{m\}\times\Lambda$. Let $N$ be a positive integer. 
We regard
$\oD(r)_N$ as a $\oD_N$-algebra by the homomorphism $\oD(p^{(r)}_m)_N$.
\par
(1) We have $(\otau^{(r)}_{l,m;i})^p=p\osigma^{(r)}_{l,m;i}$ $((l,i)\in \Lambda(r)_m)$,
and the homomorphism \eqref{eq:DeltaRingPDMap2} associated to the 
pair $((\otau^{(r)}_{l,m;i})_{(l,i)\in \Lambda(r)_m},
(\osigma^{(r)}_{l,m;i})_{(l,i)\in \Lambda(r)_m})$ of elements of
$\oD(r)^{\Lambda(r)_m}$ induces an isomorphism 
of $\oD_N$-algebras
\begin{equation}\label{eq:DSimpPDIsom}
 \pi_m^{(r)}\colon 
\oD_N[X_{l;i}; (l,i)\in \Lambda(r)_m]_{\PD}\xrightarrow{\;\;\cong\;\;} \oD(r)_N.\end{equation}

(2) Let $\partial_{X_{l;i}}$ ($(l,i)\in \Lambda(r)_m$) be the
$\oD_N$-linear derivation on $\oD_N[X_{l;i}; (l,i)\in \Lambda(r)_m]_{\PD}$
defined by $\partial_{X_{l;i}}(X_{l';i'}^{[n]})=X_{l';i'}^{[n-1]}$ if $(l',i')=(l,i)$ and
$n\geq 1$, and $=0$ otherwise. Then we have
$\theta_{D(r), l;i}\circ \pi_m^{(r)}=\pi_m^{(r)}\circ\partial_{X_{l;i}}$ for
$(l,i)\in \Lambda(r)_m$. 

(3) Let $\Delta^{(r)}$  denote the unique map $[r]\to [0]$. Then 
the kernel of $\oD(\Delta^{(r)})_N\colon \oD(r)_N\to 
\oD_N$
coincides with the image of the PD-ideal of $\oD_N[X_{l;i};(l,i)\in \Lambda(r)_m]_{\PD}$ under
$\pi_m^{(r)}$.\par
(4) The PD-structure on $\oD(r)_N$ defined by the
transport of structure via $\pi_m^{(r)}$ does not depend on $m$. We equip
$\oD(r)_N$ with this PD-structure.\par
(5) For any non-decreasing map $u\colon [r]\to [r']$ $(r,r'\in \N)$, 
the homomorphism $\oD(u)_N\colon \oD(r)_N
\to \oD(r')_N$ is a PD-homomorphism. Thus we obtain 
an inverse system of simplicial PD-rings
$(\oD(\bullet)_N)_{N\in \N}$. 
\end{propdef}

\begin{proof}
We obtain the claims (1) and (2) just by applying Proposition 
\ref{prop:qprismEnvPL} (5)
to the morphism of framed smooth $q$-pairs
$((B,J),\ut)\to ((B(r),J(r)),\ut^{(r)})$ over $R$ defined by 
$B(p_m^{(r)})\colon B\to B(r)$ and $p_m^{(r)}\times
\id_{\Lambda}\colon \Lambda=\Lambda(0)\to \Lambda(r)$. 
For the claim (3), it suffices to prove $\oD(\Delta^{(r)})_N\circ\pi_m^{(r)}(X_{l;i}^{[n]})=0$
for every $(l,i)\in \Lambda(r)_m$ and $n\in \N$, $n>0$ because the composition
$D(\Delta^{(r)})\circ D(p_m^{(r)})\colon D\to D$ is $D(\id_{[0]})=\id_D$. 
The $\delta$-homomorphism $D(\Delta^{(r)})\colon D(r)\to D$ maps
$\tau_{l,m;i}^{(r)}$ to $\pq^{-1}(t_i-t_i)=0$, and hence $\sigma^{(r)}_{l,m;i}$ to $0$ as well.
Therefore the claim follows from Lemma \ref{lem:DeltaRingPD} (4). Note that the morphism
\eqref{eq:DeltaRingPDMap} factors through the quotient $\Z_{(p)}$ by the $\delta$-ideal
generated by $x_i$ and $y_i$ $(i\in \N\cap [1,d])$ when $s_i=t_i=0$ for all $i$.\par

Let us prove the claim (4). 
Let $C_m$ be the PD-ring $\Z_{(p)}[X_{l,i};(l,i)\in \Lambda(r)_m]_{\PD}$,
and let $\varpi_m\colon C_m\to \oD(r)$ be the homomorphism 
\eqref{eq:DeltaRingPDMap2} used in the construction of
\eqref{eq:DSimpPDIsom}. Let $m$ and $m'$ be two distinct elements of $[r]$.
Since the coincidence of two PD-structures on an ideal follows from that on 
generators of the ideal, it suffices to show that the isomorphism 
$C_m\xrightarrow{\cong}C_{m'}$
of PD-rings defined by 
$X_{l;i}\mapsto X_{l;i}-X_{m;i}$ $(l\neq m')$
and $X_{m',i}\mapsto -X_{m,i}$ is compatible 
with 
$\varpi_m$ and $\varpi_{m'}$. 
We have $\tau_{l,m;i}^{(r)}=\tau^{(r)}_{l,m';i}-\tau^{(r)}_{m,m';i}$
$(l\neq m')$ and $\tau_{m',m;i}^{(r)}=-\tau_{m,m';i}^{(r)}$
by definition. By Proposition \ref{prop:DeltaPDDiff}, we obtain
$\sigma_{l,m;i}^{(r)}
=\sigma_{l,m';i}^{(r)}+(-1)^p\sigma_{m,m';i}^{(r)}
+\sum_{\nu=1}^{p-1}p^{-1}\binom p\nu
(\tau_{l,m';i}^{(r)})^{\nu}(-\tau_{m,m';i}^{(r)})^{p-\nu}$
for $l\neq m'$, and 
$\sigma_{m',m;i}^{(r)}=(-1)^p\sigma_{m,m';i}^{(r)}$.
Hence the compatibility holds by Lemma \ref{lem:DeltaRingPD} (2) and (3). 

For the last claim (5), choose $m\in [r]$ and put $m'=u(m)\in [r']$. 
Then we have $D(u)\circ D(p_m^{(r)})=D(u\circ p_m^{(r)})=D(p_{m'}^{(r)})$
and the $\delta$-homomorphism $D(u)\colon D(r)\to D(r')$ maps
$\tau_{l,m,;i}^{(r)}$ $((l,i)\in \Lambda(r)_m)$ to 
$\tau_{u(l),m';i}^{(r')}$, and hence $\sigma_{l,m;i}^{(r)}$ to 
$\sigma_{u(l),m';i}^{(r')}$, where $\tau^{(r')}_{m',m';i}=0$
and $\sigma^{(r')}_{m',m';i}=0$. By Lemma \ref{lem:DeltaRingPD} (4), 
we see that, under the isomorphisms $\pi_m^{(r)}$
and $\pi_{m'}^{(r')}$ in the claim (1),
the homomorphism $\oD(u)_N\colon \oD(r)_N\to \oD(r')_{N}$ is
compatible with the PD-homomorphism
$\oD_N[X_{l,i};(l,i)\in \Lambda(r)_m]_{\PD}
\to \oD_N[X_{l';i};(l',i)\in \Lambda(r')_{m'}]_{\PD}$
sending $X_{l;i}$ to $X_{u(l);i}$, where we put
$X_{m';i}=0$. This completes the proof.
\end{proof}

Now we are ready to prove Theorem \ref{thm:StratHiggsEquiv}.
We start by constructing a quasi-nilpotent 
$q$-Higgs field from a stratification. 
Let $(M,\varepsilon)$ be an object of $\Strat(D(\bullet)_n)$. 
For $i\in \Lambda$, we define a $(t_{1;i}^{(1)}\mu,\theta_{D(1), 1;i})$-connection $\theta_{p_0^*M,1;i}$ 
(resp.~a $(t_{0;i}^{(1)}\mu, \theta_{D(1), 0;i})$-connection 
$\theta_{p_1^*M,0;i}$) on $p_0^*M$ (resp.~$p_1^*M$)
to be $\id_M\otimes\theta_{D(1),1;i}$
(resp.~$\id_M\otimes\theta_{D(1),0;i}$).
Note that $\theta_{D(1),1;i}$ (resp.~$\theta_{D(1),0;i}$)
is $D$-linear with respect to the $D$-algebra structure of $D(1)$
via $p_0$ (resp.~$p_1$).
Let  $\theta_{p_1^*M,1;i}$ be the $(t_{1;i}^{(1)}\mu, \theta_{D(1),1;i})$-connection on $p_1^*M$
defined by $\theta_{p_0^*M,1;i}$ by the transport of structure 
via the isomorphism $\varepsilon\colon p_1^*M\xrightarrow{\cong}
p_0^*M$. 

\begin{lemma}\label{lem:StratHiggsComp}
We have $\theta_{p_1^*M,0;i}\circ \theta_{p_1^*M,1;j}
=\theta_{p_1^*M,1;j}\circ\theta_{p_1^*M,0;i}$ for
every $i,j\in\Lambda$. 
\end{lemma}

\begin{proof}
For $l,m\in \{0,1,2\}$, $l\neq m$, we define
$\theta_{q_{l}^*M,m;i}$ $(i\in \Lambda)$
to be $\id_M\otimes \theta_{D(2),m;i}$. Note that
$\theta_{D(2),m;i}$ is $D$-linear for the $D$-algebra
structure of $D(2)$ via $q_{l}$. 
The isomorphism 
$p_{12}^*(\varepsilon)\colon q_2^*M\xrightarrow{\cong}q_1^*M$
is compatible with $\theta_{q_{l}^*M,0;i}$
$(l=1,2)$ for each $i\in \Lambda$ because
$\theta_{D(2),0;i}$ is $D(1)$-linear via $p_{12}\colon D(1)\to D(2)$.
Similarly the isomorphism $p_{01}^*(\varepsilon)
\colon q_1^*M\to q_0^*M$ is compatible with
$\theta_{q_{l}^*M,2;i}$ $(l=0,1)$ for
each $i\in \Lambda$. Since 
$\theta_{q_1^*M,0;i}$ and $\theta_{q_1^*M,2;j}$
commute with each other by their definition, we see that 
$\theta_{q_2^*M,0;i}$ and $\theta_{q_0^*M,2;j}$
commute with each other via the isomorphism 
$p_{01}^*(\varepsilon)\circ
p_{12}^*(\varepsilon)=p_{02}^*(\varepsilon)$. 
By Proposition \ref{prop:qprismEnvPL} (4), 
the homomorphisms $p_1^*M\to q_2^*M$
and $p_0^*M\to q_0^*M$ induced by $p_{02}\colon D(1)\to D(2)$
are injective. The restriction of 
$\theta_{q_2^*M,0;i}$ (resp.~$\theta_{q_0^*M, 2;i}$)
under the former map
 (resp.~the latter map) 
coincides with $\theta_{p_1^*M,0;i}$ (resp.~$\theta_{p_0^*M,1;i}$)
by \eqref{eq:DsimplicialDerivFormula}. 
Hence $\theta_{p_1^*M,0;i}$ and $\theta_{p_0^*M,1;j}$ commute
with each other via $\varepsilon$.
\end{proof}

We define a $q$-Higgs field $\theta_{p_1^*M,p_1}
\colon p_1^*M\to p_1^*M\otimes_{D(1)}q\Omega_{D(p_1)}$
by $\id_M\otimes_{D,p_1}\theta_{D(p_1)}$.
Then, by Proposition \ref{prop:qprismEnvPL} (4), 
we have 
\begin{equation}\label{eq:StratToHiggs}
M\xrightarrow{\;\;\cong\;\;} (p_1^*M)^{\theta_{p_1^*M,p_1}=0}.\end{equation}
Since $\theta_{p_1^*M,p_1}(x)=
\sum_{i\in\Lambda}\theta_{p_1^*M,0;i}(x)\otimes\omega_{0;i}^{(1)}$, 
Lemma \ref{lem:StratHiggsComp} implies that the image of $M$ in 
$p_1^*M$ is stable under $\theta_{p_1^*M,1;i}$ 
for every $i\in \Lambda$. Since $\theta_{p_1^*M,1;i}$ $(i\in \Lambda)$
are $(t_i\mu,\theta_{D,i})$-connections if we regard $p_1^*M$ as a $D$-module
via $p_1$ \eqref{eq:DsimplicialDerivFormula} and they mutually commute,
$\theta_{p_1^*M,1,i}$ is restricted to $(t_i\mu,\theta_{D,i})$-connections 
$\theta_{M,i}$ on $M$ via the injection $M\hookrightarrow p_1^*M$,
and they define a $q$-Higgs field 
$\theta_M\colon M\to M\otimes_D q\Omega_{D/R}$
sending $x$ to $\sum_{i\in \Lambda}\theta_{M,i}(x)\otimes \omega_i$.
This construction of $(M,\theta_M)$ from $(M,\varepsilon)$ 
is obviously functorial. 

\begin{lemma}\label{lem:StratToConnection}
The $q$-Higgs field $\theta_M$ on $M$ constructed above is quasi-nilpotent.
\end{lemma}

\begin{proof}
Since the construction of $\theta_M$ is functorial, we may assume
$(pR+\mu R)M=0$ by taking the reduction mod $pR+\mu R$ of
$(M,\varepsilon)$. Let $\otau_i$ denote the image of 
$\tau^{(1)}_{1,0;i}\in D(1)$ in $\oD(1)_0=D(1)/(pD(1)+\mu D(1))$,
and define $\otau^{[\un]}$ for $\un=(n_i)_{i\in\Lambda}\in \N^{\Lambda}$
to be $\prod_{i\in\Lambda}\otau_i^{[n_i]}$. 
Then, by Proposition and Definition \ref{propdef:DSimpPDStr} (1), we have
$p_0^*M=\oplus_{\un\in \N^\Lambda}\utau^{[\un]}\otimes M$.
For $x\in M$, put $\varepsilon(1\otimes x)
=\sum_{\un\in\N^\Lambda}\otau^{[\un]}\otimes\theta_{M,\un}(x)$.
We prove $\theta_{M,i}=\theta_{M,\uone_i}$ and 
$\theta_{M,\un}=\prod_{i\in\Lambda}\theta_{M,\uone_i}^{n_i}$
$(\un=(n_i)\in\N^\Lambda)$,
which imply the claim because $\theta_{M,\un}(x)=0$ for $\vert\un\vert\gg0$.
Since $\theta_{\oD(1)_0,1;i}(\otau_i^{[n+1]})=\otau_i^{[n]}$ $(n\in \N)$
and $\theta_{\oD(1)_0,1;i}(\otau_j^{[n]})=0$ $(j\neq i, n\in \N)$
by Proposition and Definition \ref{propdef:DSimpPDStr} (2), we have
$\varepsilon(1\otimes\theta_{M,i}(x))=
\theta_{p_0^*M,1;i}(\varepsilon(1\otimes x))
=\sum_{\un\in \N^{\Lambda}, n_i>0}
\otau^{[\un-\uone_i]}\otimes\theta_{M,\un}(x)$.
Since the composition  $M\to p_1^*M
\xrightarrow{\varepsilon} p_0^*M\xrightarrow{\Delta\otimes\id_M}
M$ is the identity map by $\Delta^*(\varepsilon)=\id_M$
and $\Delta\circ p_1=\id_D$,
we obtain $\theta_{M,i}(x)=(\Delta\otimes\id_M)(\theta_{p_0^*M,1;i}(\varepsilon(1\otimes x)))
=\theta_{M,\uone_i}(x)$. For $(l,m)\in\{(0,1),(1,2),(0,2)\}$,
we define $\otau_{m,l}^{(2)[\un]}\in \oD(2)_0$ ($\un\in \N^d$) to be
$\prod_{i\in\Lambda}(\otau_{m,l;i}^{(2)})^{[n_i]}$. Then, since
$p_{lm}$ are PD-homomorphisms
by Proposition and Definition \ref{propdef:DSimpPDStr} (5) and
$p_{lm}(\otau_i)=\otau_{m,l;i}^{(2)}$, we have
\begin{align*}
p_{01}^*(\varepsilon)\circ p_{12}^*(\varepsilon)(1\otimes x)
=&p_{01}^*(\varepsilon)
(\sum_{\un\in \N^\Lambda}\otau_{2,1}^{(2)[\un]}\otimes\theta_{M,\un}(x))
=\sum_{\un\in\N^\Lambda}
\sum_{\um\in \N^\Lambda}\otau_{2,1}^{(2)[\un]}
\otau_{1,0}^{(2)[\um]}\otimes\theta_{M,\um}\circ\theta_{M,\un}(x),\\
p_{02}^*(\varepsilon)(1\otimes x)
=&\sum_{\ul\in \N^\Lambda}\otau_{2,0}^{(2)[\ul]}\otimes\theta_{M,\ul}(x)
\end{align*}
for $x\in M$. The equalities $\otau_{2,0;i}^{(2)}=
\otau_{2,1;i}^{(2)}+\otau_{1,0;i}^{(2)}$ $(i\in\Lambda)$ imply
$\otau_{2,0}^{(2)[\ul]}=\sum_{\un,\um\in \N^\Lambda,\ul=\un+\um}
\otau_{2,1}^{(2)[\un]}\otau_{1,0}^{(2)[\um]}$.
Since $\oD(2)_0$ regarded as an $\oD_0$-module via $q_0$ is 
free with basis $\otau^{(2)[\un]}_{2,1}\otau_{1,0}^{(2)[\um]}$
by Proposition and Definition \ref{propdef:DSimpPDStr} (1),
the above computation shows that 
the equality $p_{01}^*(\varepsilon)\circ p_{12}^*(\varepsilon)
=p_{02}^*(\varepsilon)$ is equivalent to
$\theta_{M,\um+\un}=\theta_{M,\um}\circ\theta_{M,\un}$
for all $\un,\um\in \N^\Lambda$. This implies 
$\theta_{M,\un}=\prod_{i\in\Lambda}\theta_{M,\uone_i}^{n_i}$.
\end{proof}

Next let us construct a stratification from a quasi-nilpotent $q$-Higgs module.
Let $(M,\theta_M)$ be an object of $q\HIG_{\qnilp}(D_n/R)$. 
For $l,m\in \{0,1\}$, let $(p_l^*M,\theta_{p_l^*M})$ be the
image of $(M,\theta_M)$ under the functor
$p_l^*\colon q\HIG(D_n/R)\to q\HIG(D(1)_n/R)$ 
\eqref{eq:SimpPrisEnvHIGPB}, and let
$(p_l^*M,\theta_{p_l^*M,p_m})$ be its image under
the functor $q\HIG(D(1)_n/R)
\to q\HIG(D(p_m)_n)$ defined by the composition with the projection
$q\Omega_{D(1)/R}\to q\Omega_{D(p_m)/R}$. 
We define $(q_l^*M,\theta_{q_l^*M})\in q\HIG(D_n(2)/R)$
and $(q_l^*M,\theta_{q_l^*M,q_m})\in q\HIG(D_n(q_m))$
for $l,m\in \{0,1,2\}$ similarly. 

\begin{proposition}\label{prop:HIGToStratKeyProp}
(1) The homomorphism 
$\Delta_M\colon p_0^*M\to \Delta^*p_0^*M\cong M$
defined by $\Delta_M(a\otimes x)=\Delta(a)x$ $(a\in D(1), x\in M)$
induces an isomorphism $c_M\colon (p_0^*M)^{\theta_{p_0^*M,p_1}=0}
\xrightarrow{\cong}M$, and we have
$H^m(p_0^*M\otimes_{D(1)}q\Omega^{\bullet}_{D(p_1)},
\theta^{\bullet}_{p_0^*M,p_1})=0$ for every $m\in \N, m>0$.\par
(2) Let $\Delta^{(2)}\colon D(2)\to D$ denote the $\delta$-$R$-homomorphism
corresponding to the unique map $[2]\to [0]$.  Then 
the homomorphism 
$\Delta^{(2)}_M\colon q_0^*M\to \Delta^*q_0^*M\cong M$
defined by $\Delta^{(2)}_M(a\otimes x)=\Delta^{(2)}(a)x$ $(a\in D(2), x\in M)$
induces an isomorphism $c_M^{(2)}\colon (q_0^*M)^{\theta_{q_0^*M,q_2}=0}
\xrightarrow{\cong}M$, and we have
$H^m(q_0^*M\otimes_{D(2)}q\Omega^{\bullet}_{D(q_2)},
\theta^{\bullet}_{q_0^*M,q_2})=0$ for every $m\in \N, m>0$.\par
\end{proposition}

\begin{proof} 
(cf.~{\cite[Propositions 3.17, 3.20]{MT}})
Put $M_1=(\pq R+pR)M$ and $M_2=M/(\pq R+pR)M$. Then $\theta_M$
induces quasi-nilpotent $q$-Higgs fields $\theta_{M_1}$
and $\theta_{M_2}$ on $M_1$ and $M_2$. Since
$p_0\colon D\to D(1)$ and $q_0\colon D\to D(2)$
is $\pq R+pR$-adically flat by Proposition \ref{prop:PrismEnvNervStr}, 
the exact sequence $0\to M_1\to M \to M_2\to 0$
in $q\HIG(D_n/R)$ remains exact after taking its scalar extensions
in $q\HIG(D(1)_n/R)$ and $q\HIG(D(2)_n/R)$ under
$p_0$ and $q_0$, respectively. Hence we are reduced
to the case $n=1$ by induction on $n$. Assume $n=1$ in the following.
For $r\in \N$ and $(l,i)\in \Lambda(r)$, let 
$\theta_{\oD(r)_1,l;i}$ denote the reduction modulo $pR+\mu R$
of $\theta_{D(r), l;i}$, which is an $R$-linear derivation 
since $\mu t_{l;i}^{(r)}$ vanishes in $\oD(r)_1$.\par
(1) We prove the first isomorphism by explicit computation. 
By Proposition and Definition \ref{propdef:DSimpPDStr} (1) for $r=1$ and $m=0$, 
we have 
$p_0^*M=M\otimes_{\oD_1,p_0}\oD(1)_1
=\oplus_{\un\in \N^\Lambda}M\otimes \outau_{1,0}^{(1)[\un]}$,
where $\outau_{1,0}^{(1)[\un]}=\prod_{i\in\Lambda}(\otau_{1,0;i}^{(1)})^{[n_i]}$
for $\un=(n_i)_{i\in\Lambda}\in \N^{\Lambda}$. 
For an element $x=\sum_{\un\in \N^\Lambda} x_{\un}\otimes \outau_{1,0}^{(1)[\un]}$
of $p_0^*M$, we have 
$$\theta_{p_0^*M,0;i}(x)
=-\sum_{\un\in \N^\Lambda, \un\geq \uone_i}
\outau_{1,0}^{(1)[\un-\uone_i]}\otimes x_{\un}
+\sum_{\un\in \N^\Lambda}\outau_{1,0}^{(1)[\un]}\otimes \theta_{M,i}(x_{\un})\quad (i\in \Lambda)$$
by Proposition and Definition \ref{propdef:DSimpPDStr} (2) for $r=1$ and $m=1$. 
Hence $\theta_{p_0^*M,p_1}(x)=0$
if and only if $x_{\un+\uone_i}=\theta_{M,i}(x_{\un})$ for every $i\in\Lambda$ and
$\un\in \N^\Lambda$. 
Since $\theta_M$ is quasi-nilpotent, this shows that we have an isomorphism 
$c'_M\colon M\xrightarrow{\cong} (p_0^*M)^{\theta_{p_0^*M,p_1}=0}$
sending $x$ to $\sum_{\un\in \N^\Lambda}\outau_{1,0}^{(1)[\un]}\otimes(\prod_{i\in\Lambda}
\theta_{M,i}^{n_i})(x)$, whose composition with $\Delta_M$ is the identity map
as $\Delta(\otau^{(1)[n]}_{1,0;i})=0$ for $i\in \Lambda$ and a positive integer $n$
by Proposition and Definition \ref{propdef:DSimpPDStr} (3). \par
For the second claim, we show that $c'_M$ induces an isomorphism 
$\varepsilon\colon p_1^*M\xrightarrow{\cong} p_0^*M$ compatible
with $\theta_{p_l^*M}$ $(l=1,2)$; this implies the second vanishing thanks to
Proposition and Definition \ref{propdef:DSimpPDStr} (2) for $r=1$ and $m=1$. 
The submodule $(p_0^*M)^{\theta_{p_0^*M,p_1}=0}$ 
of $p_0^*M$ is stable under the action of $D$ via
$p_1\colon D\to D(1)$, and the isomorphism $c_M$ becomes  $D$-linear.
Hence the isomorphism $c_{M}'=c_M^{-1}$ extends to a $D(1)$-linear
homomorphism $\varepsilon \colon p_1^*M\to p_0^*M$, which is 
compatible with $\theta_{p_1^*M,m;i}$ and $\theta_{p_0^*M,m;i}$ 
for $m\in \{0,1\}$ and $i\in \Lambda$ by construction
for $m=0$ and by the following computation for $m=1$. 
$$\theta_{p_0^*M,1;i}(c'_M(x))=
\sum_{\un\in \N^\Lambda, \un\geq \uone_i}\outau_{1,0}^{(1)[\un-\uone_i]}\otimes(\prod_{j\in\Lambda}
\theta_{M,j}^{n_j})(x)=c'_M(\theta_{M,i}(x)).$$
By the same argument as the previous paragraph, we obtain an isomorphism 
$M\xrightarrow{\cong} p_1^*M^{\theta_{p_1^*M,p_0}=0}$ sending 
$x$ to $\sum_{\un\in \N^\Lambda}\outau_{0,1}^{(1)[\un]}\otimes(\prod_{i\in\Lambda}
\theta_{M,i}^{n_i})(x)$. It extends to a  $D(1)$-linear homomorphism 
$\varepsilon'\colon p_0^*M\to p_1^*M$ by the same argument as above. 
It is straightforward to verify
$\varepsilon\circ\varepsilon'=\id$ and $\varepsilon'\circ \varepsilon=\id$
by using $\tau_{0,1;i}^{(1)}=-\tau_{1,0;i}^{(1)}$. \par
(2) By taking the image of the isomorphism $\varepsilon$ constructed above
under the functor $q\HIG(D(1)_n/R)\to q\HIG(D(2)_n/R)$ \eqref{eq:SimpPrisEnvHIGPB} associated
to the map $[1]\to [2]; 0,1\mapsto 0,2$, we obtain an isomorphism 
$q_2^*M\xrightarrow{\cong}q_0^*M$ compatible with 
$\theta_{q_2^*M}$ and $\theta_{q_0^*M}$ such that its pull-back
under $\Delta^{(2)}$ is the identity map. Hence the claim follows from
Proposition and Definition \ref{propdef:DSimpPDStr} (1) and (2) 
for $r=2$ and $m=2$.
\end{proof}

We regard a $D(1)$-module as a $D$-bimodule 
by giving the left (resp.~right) action of $D$ via $p_0$ (resp.~$p_1$)
$D\to D(1)$. Since $\theta_{p_0^*M,p_1}$ is right
$D$-linear, $(p_0^*M)^{\theta_{p_0^*M,p_1}=0}$ is a
right $D$-submodule of $p_0^*M$. Since $\Delta\circ p_1=\id_D$,
$\Delta_M\colon p_0^*M\to \Delta^*p_0^*M\cong M$ is $D$-linear for the right $D$-module structure on 
$p_0^*M$. Hence  the composition $M\to p_0^*M$ 
of $c_M^{-1}$ with the inclusion $(p_0^*M)^{\theta_{p_0^*M,p_1}=0}
\hookrightarrow p_0^*M$ is $D$-linear for the right $D$-module
structure on $p_0^*M$, whence it uniquely extends to a $D(1)$-linear
map 
$$\varepsilon\colon p_1^*M\to p_0^*M$$
sending $1\otimes x$ to $c_M^{-1}(x)$ for $x\in M$.

\begin{proposition}\label{prop:HigToStrat}
The map $\varepsilon$ is a stratification on $M$ with respect
to $D(\bullet)$ (Definition \ref{def:PrismStrat}). 
\end{proposition}

\begin{lemma}\label{lem:qHIGStratHiggsFldComp}
(1) The $D(1)$-linear map $\varepsilon$ defines a morphism 
$(p_1^*M,\theta_{p_1^*M})\to (p_0^*M,\theta_{p_0^*M})$
in $q\HIG(D(1)/R)$.\par
(2) For $(l,m)\in \{(0,1),(1,2),(0,2)\}$, the $D(2)$-linear map
$p_{lm}^*(\varepsilon)$ defines a morphism \\
$(q_m^*M,\theta_{q_m^*M})\to (q_l^*M,\theta_{q_l^*M})$
in $q\HIG(D(2)/R)$.\par
\end{lemma}

\begin{proof}
By the remark after \eqref{eq:DSimplqHIGPBFormula}, 
$(q_m^*M,\theta_{q_m^*M})$ and
$(q_l^*M,\theta_{q_l^*M})$ are the images of 
$(p_1^*M,\theta_{p_1^*M})$ and $(p_0^*M,\theta_{p_0^*M})$ under
the functor $p_{lm}^*\colon q\HIG(D(1)/R)\to q\HIG(D(2)/R)$. 
Hence the claim (2) follows from the claim (1). Let us prove (1).
The image of $M\to p_1^*M$ (resp.~$M\to p_1^*M\xrightarrow{\varepsilon}
p_0^*M$) is contained in the kernel of $\theta_{p_1^*M,p_1}$
(resp.~$\theta_{p_0^*M,p_1}$) by \eqref{eq:DSimplqHIGPBFormula}
(resp.~the definition of $\varepsilon$). Since $p_1^*M$ is generated
by the image of $M$ as a $D(1)$-module, this shows that
$\varepsilon$ is compatible with $\theta_{p_l^*M,p_1}$ $(l=0,1)$. 
It remains to prove the compatibility with $\theta_{p_l^*M,p_0}$ $(l=0,1)$. 
The map $M\to p_1^*M$ is compatible with $\theta_{M,i}$ and 
$\theta_{p_1^*M,1;i}$ by \eqref{eq:DSimplqHIGPBFormula},
$\theta_{p_l^*M,p_0}=\sum_{i\in\Lambda}\theta_{p_l^*M,1;i}\otimes\omega_{1;i}^{(1)}$,
and $p_1^*M$ is generated by the image of $M$.
Hence we are reduced to showing that $c_M\colon (p_0^*M)^{\theta_{p_0^*M,p_1}=0}
\xrightarrow{\cong} M$ induced by $\Delta_M\colon p_0^*M
\to M;a\otimes x\mapsto \Delta(a)x$ is compatible with $\theta_{p_0^*M,1;i}$
and $\theta_{M,i}$. Since the composition of functors
$q\HIG(D/R)\xrightarrow{p_0^*}q\HIG(D(1)/R)\xrightarrow{\Delta^*}
q\HIG(D/R)$ is the identity functor by the remark after
\eqref{eq:DSimplqHIGPBFormula}, the formula \eqref{eq:DSimplqHIGPBFormula} gives
$$\theta_{M,i}(c_M(x))
=\Delta_M(\theta_{p_0^*M,0;i}(x))
+\Delta_M(\theta_{p_0^*M,1;i}(x))+\mu t_i\Delta_M(\theta_{p_0^*M,1,i}
\circ \theta_{p_0^*M,0,i}(x))
=c_M(\theta_{p_0^*M,1;i}(x))$$
for $x\in p_0^*M^{\theta_{p_0^*M,p_1}=0}$. 
\end{proof}

\begin{proof}[Proof of Proposition \ref{prop:HigToStrat}]
The composition $M\to p_1^*M\xrightarrow[\varepsilon]{\cong}p_0^*M\xrightarrow{\Delta}M$
is the identity map by the definition of $\varepsilon$. Since the 
composition $M\to p_1^*M\xrightarrow{\Delta\otimes\id_M}M$ is the
identity map and $p_1^*M$ is generated by the image of $M$
as an $D(1)$-module, we see $\Delta^*(\varepsilon)=\id_M$.
Now we have the following commutative diagram
\begin{equation*}
\xymatrix@C=30pt{
M\ar[r]\ar[dr]_{\id_M}&
q_2^*M\ar[r]^{p_{12}^*(\varepsilon)}\ar[d]^{\Delta^{(2)}\otimes\id_M}&
q_1^*M\ar[r]^{p_{01}^*(\varepsilon)}\ar[d]^{\Delta^{(2)}\otimes\id_M}&
q_0^*M\ar[d]^{\Delta^{(2)}\otimes\id_M}\\
&M\ar@{=}[r]&M\ar@{=}[r]&M
}\quad
\xymatrix@C=30pt{
M\ar[r]\ar[dr]_{\id_M}&
q_2^*M\ar[r]^{p_{02}^*(\varepsilon)}\ar[d]^{\Delta^{(2)}\otimes\id_M}&
q_0^*M\ar[d]^{\Delta^{(2)}\otimes\id_M}\\
&M\ar@{=}[r]&M
}
\end{equation*}
and isomorphisms 
$M\xrightarrow{\cong}(q_2^*M)^{\theta_{q_2^*M,q_2}=0}$
(Proposition \ref{prop:qprismEnvPL} (4)) and
$c_M^{(2)}\colon (q_0^*M)^{\theta_{q_0^*M,q_2}=0}\xrightarrow{\cong}M$
(Proposition \ref{prop:HIGToStratKeyProp} (2)) induced by 
$\Delta^{(2)}\otimes\id_M$. 
Hence, by Lemma \ref{lem:qHIGStratHiggsFldComp} (2), 
we see that the compositions of $M\to q_2^*M$
with $p_{01}^*(\varepsilon)\circ p_{12}^*(\varepsilon)$
and $p_{02}^*(\varepsilon)$ both coincide with 
$(c_M^{(2)})^{-1}$. This implies
$p_{01}^*(\varepsilon)\circ p_{12}^*(\varepsilon)=p_{02}^*(\varepsilon)$ because
$q_2^*M$ is generated by the image of $M$ as a $D(2)$-module.
This completes the proof by Remark \ref{rmk:PrismStrat}.
\end{proof}

\begin{proof}[Proof of Theorem \ref{thm:StratHiggsEquiv}]
It remains to show that the two constructions in Lemma \ref{lem:StratToConnection}
and Proposition \ref{prop:HigToStrat} are the inverses of each other. \par

Let $(M,\theta_M)$ be an object of $q\HIG_{\qnilp}(D_n/R)$,
and let $\varepsilon\colon p_1^*M\xrightarrow{\cong} p_0^*M$
be the stratification on $M$ with respect to $D(\bullet)$
associated to $\theta_M$ by  Proposition \ref{prop:HigToStrat}. Then, 
by Lemma \ref{lem:qHIGStratHiggsFldComp} (1), 
the composition $M\to p_1^*M\xrightarrow{\varepsilon}p_0^*M$
is compatible with $\theta_{M,i}$ and $\theta_{D(1),1;i}\otimes\id_M$ for $i\in \Lambda$,
which means that $\theta_M$ is the $q$-Higgs field on $M$ associated to $\varepsilon$
by Lemma \ref{lem:StratToConnection}. 
\par
Let $(M,\varepsilon)$
be an object of $\Strat(D(\bullet)_n)$, 
and let $\theta_M$ be the quasi-nilpotent
$q$-Higgs field $M\to M\otimes_Dq\Omega_{D/R}$ associated to $\varepsilon$
by Lemma \ref{lem:StratToConnection}.
Let $\theta_{p_1^*M,p_1}$ be the 
$q$-Higgs field $\id_M\otimes_{D,p_1}\theta_{D(p_1)}
\colon p_1^*M\to p_1^*M\otimes_{D(1)}q\Omega_{D(p_1)}$
on $p_1^*M$, and let $\theta_{p_0^*M,p_1}$ be the
$q$-Higgs field $p_0^*M\to p_0^*M\otimes_{D(1)}q\Omega_{D(p_1)}$
on $p_0^*M$ obtained from $\theta_{p_1^*M,p_1}$
by the transport of structure via the
isomorphism $\varepsilon\colon p_1^*M \xrightarrow{\cong}p_0^*M$.
Then the composition $M\xrightarrow[\eqref{eq:StratToHiggs}]{\cong}
(p_1^*M)^{\theta_{p_1^*M,p_1}=0}
\xrightarrow[\varepsilon]{\cong}
(p_0^*M)^{\theta_{p_0^*M,p_1}=0}
\xrightarrow{\Delta_M} M$ is the identity map
by the condition (i) in Definition \ref{def:PrismStrat} on $\varepsilon$.
Hence, by the construction of the stratification associated to 
a quasi-nilpotent $q$-Higgs field given before Proposition \ref{prop:HigToStrat}, it suffices to prove 
that $\theta_{p_0^*M,p_1}$ is induced from $\theta_M$
by the scalar extension $p_0\colon D\to D(1)$. 
This is equivalent to the compatibility 
of the composition $M\to p_0^*M\xrightarrow[\varepsilon^{-1}]{\cong}
p_1^*M$ with $\theta_{M,i}$ and $\id_M\otimes_{D,p_1}\theta_{D(1),0;i}$.
The composition $M\to p_1^*M\xrightarrow{\varepsilon}p_0^*M$
is compatible with $\theta_{M,i}$ and $\id_M\otimes_{D,p_0}\theta_{D(1),1;i}$
by the construction of $\theta_{M,i}$ from $\varepsilon$.
By Remark \ref{rmk:PrismStrat}, we obtain the desired compatibility by 
taking the scalar extension under the 
homomorphism $D(\iota)\colon D(1)\xrightarrow{\cong}D(1)$,
which satisfies $\theta_{D(1),0;i}\circ D(\iota)=D(\iota)\circ\theta_{D(1),1;i}$ and 
$D(\iota)\circ p_l=p_{1-l}$ $(l=0,1)$. 
\end{proof}

Since the automorphism $\gamma_{D(r),i}=\id_{D(r)}+t_i\mu\theta_{D(r),i}$ 
of the $q$-prism $(D(r),\pq D(r))$ over the $q$-prism $(R,\pq R)$ 
induces an automorphism of the $A$-algebra
$D(r)/\pq D(r)$ by the proof of Proposition \ref{prop:qHiggsDeriv} (2),
it defines an automorphism of the object $((D(r),\pq D(r)),v_{D(r)})$ 
of $(\fX/R)_{\prism}$. 

\begin{proposition}\label{prop:HiggsFieldAndCrysAuto}
 Let $n\in \N$, let 
$\CF$ be an object of $\Crystal_{\prism}(\CO_{\fX/R,n})$,
and let $(M=\CF(D),\theta_M)$ be the object of 
$q\HIG_{\qnilp}(D_n/R)$ corresponding to $\CF$
by Proposition \ref{prop:CrysStratEquiv} and Theorem \ref{thm:StratHiggsEquiv}.
Then, for each $i\in \Lambda$, the $\gamma_{D,i}$-semilinear
automorphism $\gamma_{M,i}=\id_M+t_i\mu\theta_{M,i}$
of $M$ (Definition \ref{def:connection}) 
coincides with the $\gamma_{D,i}$-semilinear
automorphism $\CF(\gamma_{D,i})$ of $\CF(D)$.
Note that $\gamma_{M,i}$ is an automorphism since 
$\mu$ is nilpotent on $D_n$ and $t_i\theta_{M,i}$ is 
$R$-linear.
\end{proposition}

\begin{proof}
Let $\varepsilon$ be the stratification on $M$ associated to $\CF$
as before Proposition \ref{prop:CrysStratEquiv}. Then, by the construction 
of $\theta_M$ from $\varepsilon$ given before Lemma \ref{lem:StratToConnection}, the composition 
$M\hookrightarrow p_1^*M\xrightarrow[\varepsilon]{\cong}p_0^*M$
is compatible with $\theta_{M,i}$ and $\id_M\otimes\theta_{D(1),1;i}$,
whence with $\gamma_{M,i}$ and $\id_M\otimes\gamma_{D(1),1;i}$.
By Proposition \ref{prop:qHiggsDerivFunct} (2) applied
to the morphism of framed smooth $\delta$-pairs $p_l\colon 
(B,J)\to (B(1),J(1))$ over $R$ $(l\in \{0,1\})$ and Lemma 
\ref{lem:TwDerivSystemFunct} (2), we see that $p_0\colon D\to D(1)$
is stable under $\gamma_{D(1),1;i}$ and
$p_1\colon D\to D(1)$ is compatible with 
$\gamma_{D,i}$ and $\gamma_{D(1);1,i}$. 
Hence the composition $\CF(D)\xrightarrow[\CF(p_1)]{}
\CF(D(1))\xleftarrow[\CF(p_0)]{\cong}
p_0^*(\CF(D))$ is compatible with $\CF(\gamma_{D,i})$
and $\id_{\CF(D)}\otimes\gamma_{D(1),1;i}$.
This completes the proof because the two 
compositions considered above are the same
by the construction of $\varepsilon$ from $\CF$. 
\end{proof}
\begin{remark}\label{rmk:CrysFrobPBHiggs}
(1) Let $\CF$ be an object of $\Crystal_{\prism}(\CO_{\fX/R,n})$,
and let $\varphi_n^*\CF$ be its scalar extension by 
$\varphi_n\colon \CO_{\fX/R,n}\to \CO_{\fX/R,n}$ 
(Definition \ref{def:PrismaticSite} (1)), which is a crystal
of $\CO_{\fX/R,n}$-modules (Remark \ref{rmk:crystalFrobPBTensor} (1)).
Let $(M=\CF(D),\theta_M)$ and $(\varphi_{D_n}^*M=M\otimes_{D_n,\varphi_{D_n}}D_n,
\theta_{\varphi_{D_n}^*M}')$ be the objects
of $q\HIG_{\qnilp}(D_n/R)$ corresponding to $\CF$ and $\varphi_n^*\CF$, respectively,
by Proposition \ref{prop:CrysStratEquiv} and Theorem \ref{thm:StratHiggsEquiv}.
Then we see that $(\varphi^*_{D_n}M,\theta'_{\varphi^*_{D_n}M})$ coincides with
the image $(\varphi^*_{D_n}M,\theta_{\varphi^*_{D_n}M})$ of $(M,\theta_M)$
under the Frobenius pullback functor $\varphi_{D_n}^*$ \eqref{eq:FramedSmQPHigFPBFunct}
as follows. By \eqref{eq:qDervConnPBFormula}, it suffices to verify 
\begin{equation}\label{eq:CrystalHiggsFrobPullback}
\theta'_{\varphi_{D_n}^*M,i}(m\otimes 1)=\theta_{M,i}(m)\otimes \pq t_i^{p-1}
\quad (m\in M,\, i\in \Lambda).
\end{equation}
The stratification on $\varphi_{D_n}^*M$ associated to $\varphi_n^*\CF$
by Proposition \ref{prop:CrysStratEquiv} is the scalar
extension of that on $M$ associated to $\CF$
under $\varphi_{D(1)_n}$. Hence the claim is reduced to
$\theta_{D(1),1;i}\circ \varphi_{D(1)}=
\pq  t_{1;i}^{p-1}\cdot \varphi_{D(1)}\circ \theta_{D(1),1;i}$
which is obtained from  Lemma \ref{lem:qDerivFrobComp} (2). \par
(2) Let $\CF$ and $\CF'$ be crystals of $\CO_{\fX/R,n}$-modules
on $(\fX/R)_{\prism}$, and let $(M,\theta_M)$ and $(M',\theta_{M'})$
be the objects of $q\HIG_{\qnilp}(D_n/R)$ corresponding to
$\CF$ and $\CF'$, respectively, by Proposition \ref{prop:CrysStratEquiv}
and Theorem \ref{thm:StratHiggsEquiv}.
Then the $q$-Higgs field on $M\otimes_DM'=(\CF\otimes_{\CO_{\fX/R,n}}\CF')(D)$
corresponding to the crystal of $\CO_{\fX/R,n}$-modules
$\CF\otimes_{\CO_{\fX/R,n}}\CF'$ (Remark \ref{rmk:crystalFrobPBTensor} (2))
coincides with the tensor product of $\theta_M$ and $\theta_{M'}$
(Definition \ref{def:ConnTensorProdDef}). This is an immediate
consequence of the following observation and \eqref{eq:ConnProdSymmForm}: 
For $i\in \Lambda$, let $\theta_{p_0^*M\otimes p_0^*M', 1;i}$
be the $(t_{1;i}^{(1)}\mu,\theta_{D(1),1;i})$-connection on
$p_0^*M\otimes_{D(1)}p_0^*M'$ defined by 
$\id_{M\otimes_DM'}\otimes_{D,p_0}\theta_{D(1),1;i}$
by the transport of structure via the
isomorphism $p_0^*M\otimes_{D(1)}p_0^*M'\cong p_0^*(M\otimes_DM')$.
Then, for $x\in p_0^*M$ and  $x'\in p_0^*M'$, we have 
$$
\theta_{p_0^*M\otimes p_0^*M',1;i}(x\otimes x')
=\theta_{p_0^*M,1;i}(x)\otimes x'
+x\otimes \theta_{p_0^*M',1;i}(x')
+t_{1;i}^{(1)}\mu\cdot 
\theta_{p_0^*M,1;i}(x)\otimes\theta_{p_0^*M',1;i}(x').$$
This follows from the fact that $\theta_{D(1),1;i}$ is
a $t_{1;i}^{(1)}\mu$-derivation of $D(1)$ over $D$ via $p_0$
(Definition \ref{def:alphaDerivation} (1)).
\end{remark}

The equivalence \eqref{eq:StratHiggsEquiv} is 
functorial in $R$, $i\colon \fX=\Spf(A)\hookrightarrow\Spf(B)$
and $\ut=(t_i)_{i\in\Lambda}$ as follows. Let 
$(R',I')$, $i'\colon \fX'=\Spf(A')\to \Spf(B')$, $f\colon 
(R,I)\to (R',I')$, $g\colon \fX'\to \fX$, 
$h\colon \Spf(B')\to\Spf(B)$, $(B'(r),J'(r))$, 
$(D',v_{D'})$, $(D'(r),v_{D'(r)})$, $h_{B(\bullet)}$,
and $h_{D(\bullet)}$ be the same as before 
\eqref{eq:StratPullback}. Note that $R'$ becomes a 
$q$-prism via the morphism $f$. We further assume
that we are given  $pR'+[p]_qR'$-adic coordinates
$\ut'=(t'_{i'})_{i'\in \Lambda'}$, $\Lambda'=\N\cap [1,d']$ of $B'$ over $R'$ 
and a map $\psi\colon \Lambda\to \Lambda'$ of ordered sets such that $\delta(t'_{i'})=0$
$(i'\in\Lambda')$ and $h^*(t_i)=t'_{\psi(i)}$ $(i\in\Lambda)$. We define $\Lambda'(r)$, 
$\ut^{\prime (r)}=(t^{\prime(r)}_{l;i'})_{(l,i')\in\Lambda'(r)}$, 
$\theta_{D'(r),l;i'}$, and $\theta_{D'(r)}\colon D'(r)\to q\Omega_{D'(r)/R'}$
in the same way as $\Lambda(r)$, $\ut^{(r)}$, 
$\theta_{D(r),l;i}$, and $\theta_{D(r)}$ defined after 
Condition \ref{cond:qPrism} by using $R'$, $\fX'\to \Spf(B')$ and $\ut'$. 
The homomorphism $h_{B(r)}$ and $\id_{[r]}\times\psi\colon \Lambda(r)
\to \Lambda'(r)$ define a morphism of framed smooth $q$-pairs
$((B(r),J(r))/R,\ut^{(r)})\to ((B'(r),J'(r))/R',\ut'^{(r)}).$ Therefore 
$h_{D(r)}\colon D(r)\to D'(r)$ satisfies
\begin{equation}\label{eq:StratDerivFunct}
\theta_{D'(r),l;i'}(h_{D(r)}(a))=
\sum_{\emptyset\neq\ui\subset \psi^{-1}(i')} (\mu t_{i'}')^{\sharp\ui-1}
h_{D(r)}(\theta_{D(r),l;\ui}(a)), \quad (l,i')\in\Lambda'(r), \;a\in D(r)
\end{equation}
and gives a functor
$h^*_{D(r)}\colon q\HIG(D(r)/R)\to q\HIG(D'(r)/R')$
by Proposition \ref{prop:qHiggsDerivFunct} (2) and Lemma \ref{lem:TwDerivSystemFunct} (1),
where $\theta_{D(r),l;\ui}=\prod_{i\in \ui}\theta_{D(r),l;i}$. 

\begin{proposition}\label{prop:FunctStratConn}
Under the notation and assumption as above, the following diagram is commutative
up to canonical isomorphisms.
\begin{equation}
\xymatrix@C=50pt{
\Strat(D(\bullet)_n)\ar[r]_(.45){\eqref{eq:StratHiggsEquiv}}^(.45){\sim}
\ar[d]_{\eqref{eq:StratPullback}}^{h_{D(\bullet)}^*}&
q\HIG_{\qnilp}(D_n/R)\ar@{^{(}->}[r]&
q\HIG(D_n/R)\ar[d]^{h_D^*}\\
\Strat(D'(\bullet)_n)\ar[r]_(.45){\eqref{eq:StratHiggsEquiv}}^(.45){\sim}&
q\HIG_{\qnilp}(D'_n/R')\ar@{^{(}->}[r]&
q\HIG(D'_n/R')
}
\end{equation}
\end{proposition}

\begin{proof}
Let $(M,\varepsilon)$ be an object of $\Strat(D(\bullet)_n)$, 
let $(M',\varepsilon')$ be the object $h_{D(\bullet)}^*(M,\varepsilon)$
of $\Strat(D'(\bullet)_n)$, and let $(M,\theta_M)$
(resp.~$(M',\theta_{M'})$) be the object of 
$q\HIG(D_n/R)$ (resp.~$q\HIG(D'_n/R')$)
corresponding to $(M,\varepsilon)$
(resp.~$(M',\varepsilon')$) by \eqref{eq:StratHiggsEquiv}. 
Let $h_M$ denote the homomorphism $M\to M'=D'\otimes_DM;
x\mapsto 1\otimes x$. Then we have the following commutative
diagram.
\begin{equation*}
\xymatrix@C=70pt{
p_0^*M\ar[d]_{h_{D(1)}\otimes h_M}
\ar[r]^{\cong}_{\varepsilon}& 
p_1^*M\ar[d]^{h_{D(1)}\otimes h_M} 
& M\ar@{_{(}->}[l]\ar[d]^{h_M}\\
p_0^{\prime*}M'\ar[r]^{\cong}_{\varepsilon'}& p_1^{\prime*}M'&
M'\ar@{_{(}->}[l]
}
\end{equation*}
By definition, the endomorphism $\theta_{M,i}$ $(i\in\Lambda)$ of
$M$ (resp.~$\theta_{M',i'}$ $(i'\in\Lambda')$ of $M'$) is induced
by $\theta_{D(1),1;i}\otimes \id_M$ on $p_0^*M$
(resp.~$\theta_{D'(1),1;i'}\otimes \id_{M'}$ on $p_0^{\prime*}M'$) via the 
upper (resp.~lower) horizontal maps. Hence
the formula \eqref{eq:StratDerivFunct} for $r=1$, $l=1$ implies that we have
$$\theta_{M',i'}(h_M(x))
=\sum_{\emptyset\neq \ui\subset \psi^{-1}(i')} (\mu t'_{i'})^{\sharp \ui-1}
h_M(\theta_{M,\ui}(x))$$
for $i'\in \Lambda'$ and $x\in M$, where
$\theta_{M,\ui}=\prod_{i\in\ui}\theta_{M,i}$. 
By Proposition \ref{prop:ConnPBFormula}, 
we have $(M',\theta_{M'})=h_D^*(M,\theta_M)$. 
\end{proof}

\section{Linearization and \v{C}ech-Alexander complex}\label{sec:LinCAcpx}
As in Definition \ref{def:PrismaticSite}, 
let $(R,I)$ be a bounded prism (Definition \ref{def:prism} (2)), and
let $\fX$ be a $p$-adic formal scheme over $\Spf(R/I)$.

In \S\ref{sec:PrismCohqDolb}, we will give a Zariski local description of the cohomology
(i.e., a description of the derived direct image under the projection to the Zariski topos
of $\fX$) of a crystal on $(\fX/(R,I))_{\prism}$ in terms of 
the $q$-Higgs complex associated to the crystal 
(Theorem \ref{th:CrystalCohqHiggs}) 
when $(R,I)$ is a $q$-prism
(Definition \ref{def:FramedSmoothPrism} (1)). 
We follow an analogue of the proof of the Zariski local description
of the cohomology of a crystal on a crystalline site in terms
of the de Rham complex of the associated module with
integrable connection given in \cite[V]{BerthelotCrisCoh} and \cite[\S7]{BO}. More concretely, 
we proceed as follows. We can work with any $(R,I)$ for (i) and (ii), which we discuss in this section.\par\noindent
(i) A Zariski local description of the cohomology of a crystal
in terms of the associated \v{C}ech-Alexander complex (Proposition \ref{prop:CechAlexander}). \par\noindent
(ii) Linearization \eqref{eq:LinearizationFunctor}, 
which is a construction of a crystal associated to a module, 
and computation of its cohomology (Proposition \ref{prop:LinearizationLocCoh}).\par
\noindent
(iii) A construction of a resolution of a crystal by a linearization of the associated
$q$-Higgs complex \eqref{eq:qdRResol}.\par

Let us start by introducing the projection morphism from the prismatic topos of $\fX$ over
$(R,I)$ to the Zariski topos of $\fX$. 
Let $\fX_{\ZAR}$ denote the category of $p$-adic formal schemes
over $\fX$ equipped with the Zariski topology.
Then the functor $(\fX/(R,I))_{\prism}\to \fX_{\ZAR}$
sending $((P,IP),v)$ to $v\colon \Spf(P/IP)\to\fX$ is 
cocontinuous (\cite[III D\'efinition 2.1]{SGA4}) 
since $\Spf(P/IP)\to \Spf(P)$ is a
homeomorphism on the underlying topological spaces.
Let 
$$U_{\fX/(R,I)}\colon (\fX/(R,I))_{\prism}^{\sim}\longrightarrow \fX_{\ZAR}^{\sim}$$
be the induced morphism of topos (\cite[III Proposition 2.3]{SGA4}).
We have 
\begin{align*}
&(U_{\fX/(R,I)*}\CF)(\fY\to \fX)=\Gamma((\fY/(R,I))_{\prism},
\CF\vert_{(\fY/(R,I))_{\prism}}),\\
&(U_{\fX/(R,I)}^*\CG)((P,IP),v)=\CG(v\colon \Spf(P/I)\to \fX).
\end{align*}
Let $\fX_{\Zar}$ be the category of open formal subschemes of $\fX$
equipped with the Zariski topology. Then the inclusion functor 
$\fX_{\Zar}\to \fX_{\ZAR}$ is continuous and preserves finite inverse limits which 
exist in $\fX_{\Zar}$. Hence it defines a morphism 
of topos $\varepsilon_{\fX,\Zar}\colon 
\fX_{\ZAR}^{\sim}\to \fX_{\Zar}^{\sim}$. We define a morphism of topos 
\begin{equation}\label{eq:PrismToposZarProj}
u_{\fX/(R,I)}\colon (\fX/(R,I))_{\prism}^{\sim}\longrightarrow \fX_{\Zar}^{\sim}
\end{equation}
to be the composition $\varepsilon_{\fX,\Zar}\circ U_{\fX/(R,I)}$. 
We abbreviate $U_{\fX/(R,I)}$ and $u_{\fX/(R,I)}$ to $U_{\fX/R}$ and
$u_{\fX/R}$ if there is no risk of confusion. 

Assume that $I$ is generated by an element $\xi$ of $R$, $\fX$
is an affine formal scheme $\Spf(A)$, and we are given  a smooth
adic morphism $\fY=\Spf(B)\to \Spf(R)$ and a closed
immersion $i\colon \fX=\Spf(A)\to \fY=\Spf(B)$ over $\Spf(R)$ such that $R\to B$
admits $pR+I$-adic coordinates
(Definition \ref{def:formallyflat} (2)) and that 
the pair of $B$ and $J:=\Ker(B\to A)$ satisfies 
Condition \ref{cond:PrismCohDescription}. 
Put $D_n=D/(pD+ID)^{n+1}$ for $n\in \N$,
and let $\Mod(D_n)$ denote the category of $D_n$-modules.

We  construct a functor
\begin{equation}\label{eq:LinearizationFunctor}
\CL_{A,B}\colon \Mod(D_n)\longrightarrow \Crystal_{\prism}(\CO_{\fX/R,n}),
\end{equation}
which we call the linearization following
\cite[\S6]{Groth}, \cite[IV.3]{BerthelotCrisCoh}, and \cite[6.9]{BO}.

For an object $((P,IP),v)$ of $(\fX/(R,I))_{\prism}$, 
let $B_P$ be the $pR+I$-adic completion of $P\otimes_RB$ 
equipped with the unique $\delta$-structure compatible with
that of $P$ and $B$, let $i_{B_P}$ be the closed immersion 
$\Spf(P/IP)\to \Spf(B_P)$ induced by 
the closed immersion $\Spf(P/IP)\to \Spf(P)$ and
the morphism $i\circ v\colon \Spf(P/IP)\to \Spf(B)$.
Let $J_{B_P}$ be the kernel of $i_{B_P}^*\colon B_P\to P/IP$.
Then, by Corollary \ref{cor:PrismEnvSection} and Lemma \ref{lem:CompFinGenIdeal},  
the $\delta$-pair $(B_P, J_{B_P})$ has 
a bounded prismatic envelope over $(R,I)$, which is denoted by 
$(D_P,ID_P)$ in the following. The morphism 
of $\delta$-pairs $(B,J)\to (B_P,J_{B_P})$ induces
a morphism 
of bounded prisms $D\to D_P$ over $(R,I)$.
This construction is obviously 
functorial in $((P,IP),v)$, i.e., a morphism 
$u\colon ((P',IP'),v')\to ((P,IP),v)$ in $(\fX/R)_{\prism}$
induces a morphism of $\delta$-pairs
$u_B\colon (B_P,J_{B_P})\to (B_{P'},J_{B_{P'}})$
over $(B,J)$ 
compatible with the $\delta$-homomorphism 
$u\colon P\to P'$, and then a morphism of 
bounded prisms $u_D\colon (D_P,ID_P)\to (D_{P'},ID_{P'})$
over the bounded prism $(D,ID)$ compatible 
with $u\colon P\to P'$. 

\begin{lemma}\label{lem:LinearizationEnvBC}
Under the notation above, the homomorphism 
$u_D$ induces an isomorphism 
$D_P/KD_P\otimes_{P}P'\xrightarrow{\cong}
D_{P'}/KD_{P'}$
for any ideal $K$ of $R$ containing a power of $pR+I$. 
\end{lemma}

\begin{proof} Let $t_1,\ldots, t_d$ be 
$pR+I$-adic coordinates of $B$ over $R$,
and choose a lifting $a_i\in P$ $(i\in \N\cap [1,d])$
of the  image of $t_i$ in $P/IP$ under the homomorphism
$B\to A\xrightarrow{v^*} P/IP$. Then we obtain 
the claim by applying Corollary \ref{cor:PrismEnvSection} to 
the morphism of $\delta$-pairs $(P,IP)\to (B_P, J_{B_P})$
over $(R,I)$, the image of $t_i$ in $B_P$, and 
$a_i\in P$, and to 
the morphism of $\delta$-pairs $(P',IP')\to (B_{P'},J_{B_{P'}})$
over $(R,I)$, the image of $t_i$ in $B_{P'}$,
and the image of $a_i$ in $P'$ under $u$. 
 \end{proof}

Let $M$ be a $D_n$-module.
Then one can define a presheaf $\CL_{A,B}(M)$ of $\CO_{\fX/R,n}$-modules
on $(\fX/R)_{\prism}$ by 
$$\CL_{A,B}(M)((P,IP),v)=M\otimes_D D_P$$
for $((P,IP),v)\in \Ob(\fX/R)_{\prism}$, 
and $\CL_{A,B}(M)(u)=\id_M\otimes u_D\colon 
M\otimes_DD_P\to M\otimes_{D}D_{P'}$ for
$u\colon ((P',IP'),v')\to ((P,IP),v)\in \Mor(\fX/R)_{\prism}$. 
By Lemma \ref{lem:LinearizationEnvBC},
$\CL_{A,B}(M)$ is a crystal of $\CO_{\fX/R,n}$-modules
(Definition \ref{def:PrismaticSite} (2)). 
This construction is obviously functorial in $M$, and defines
the desired functor  \eqref{eq:LinearizationFunctor}.
The ring homomorphism 
$\CO_{\fX/R,n}((P,IP),v)=P/(pP+IP)^{n+1}\to D_n\otimes_DD_P$ is functorial in $((P,IP),v)$
and defines an $\CO_{\fX/R,n}$-algebra structure on 
$\CL_{A,B}(D_n)$. The $D_P$-module
structure on $M\otimes_DD_P$ for each $((P,IP),v)$ defines an 
$\CL_{A,B}(D_n)$-module structure
on $\CL_{A,B}(M)$ which is functorial in $M$ and is an extension
of the given $\CO_{\fX/R,n}$-module structure. 

Put $\fD_n=\Spec(D_n)$. Let $M$ be a $D_n$-module,
and let $\CM$ be the quasi-coherent $\CO_{\fD_n}$-module
defined by $M$. 
We can construct a canonical morphism 
\begin{equation}\label{eq:LinearizationZariskiProjMap}
v_{D*}\CM\longrightarrow u_{\fX/R*}\CL_{A,B}(M)
\end{equation}
as follows. Here $v_D$ denotes the canonical morphism
$\Spf(D/ID)\to \fX=\Spf(A)$, and $\CM$ is regarded
as a sheaf on $\Spf(D/ID)_{\Zar}$.
Let $j\colon \fX'=\Spf(A')\hookrightarrow 
\fX=\Spf(A)$ be an open affine formal 
subscheme of $\fX$, and let $\Spf(D')$ be
the open affine formal subscheme of $\Spf(D)$
whose underlying topological space is given by 
$v_D^{-1}(\fX')\subset \Spf(D/ID)$.
For an object $((P,IP),w)$ of $(\fX'/R)_{\prism}$,
the morphism $\Spf(D_P)\to \Spf(D)$ factors
through $\Spf(D')$ because the composition
$\Spf(D_P/ID_P)\to \Spf(D/ID)\to \fX=\Spf(A)$
coincides with the composition
$\Spf(D_P/I D_P)\to \Spf(P/IP)\xrightarrow{w}\fX'
\xrightarrow{j} \fX$. The homomorphism 
thus obtained $D'\to D_P$ induces a homomorphism 
$M\otimes_DD'\to M\otimes_DD_P=\CL_{A,B}(M)((P,IP),j\circ w)$.
This is compatible with the homomorphism 
$\CL_{A,B}(M)(u)$ for any morphism $u\colon ((P',IP'),w')
\to ((P,IP),w)$ in $(\fX'/R)_{\prism}$ as $u_D\colon D_P
\to D_{P'}$ is a $D'$-homomorphism. Varying 
$((P,IP),w)$, we obtain
a homomorphism $M\otimes_DD'\to
\Gamma((\fX'/R)_{\prism},\CL_{A,B}(M))$. 
This construction is functorial in $\fX'$, and 
gives the desired homomorphism 
\eqref{eq:LinearizationZariskiProjMap}. 

\begin{proposition}\label{prop:LinearizationLocCoh}
Under the notation and assumption as above, 
the morphism \eqref{eq:LinearizationZariskiProjMap}
induces an isomorphism 
$$v_{D*}\CM\overset{\cong}\longrightarrow Ru_{\fX/R*}\CL_{A,B}(M).$$
\end{proposition}

To prove Proposition \ref{prop:LinearizationLocCoh}, we use the \v{C}ech-Alexander
complex with respect the closed immersion
$\Spf(A)\to \Spf(B)$ over $R$. Recall that 
$((D,ID),v_D) \in (\fX/R)_{\prism}$ is a covering of the final object
of $(\fX/R)_{\prism}^{\sim}$ and that 
the product of $r+1$ copies of 
$((D,ID),v_D)$ in $(\fX/R)_{\prism}$ is represented
by $((D(r), ID(r)),v_{D(r)})$ for $r\in\N$.
(See Proposition \ref{prop:PrismCofFinObj} and the paragraph following it.) 
For a sheaf of abelian groups $\CF$ 
on $(\fX/R)_{\prism}$ and an object 
$((P,IP),v)$ of $(\fX/R)_{\prism}$, we define
a sheaf $\CF_{P}$ on $\Spf(P)_{\Zar}$ by 
$\CF_{P}(\Spf(P'))=\CF((P',IP'),v')$
for an open affine $\Spf(P')\subset\Spf(P)$,
where $P'$ is equipped with the unique 
$\delta$-$P$-algebra structure, and $v'$
denotes the composition $\Spf(P'/IP')\subset\Spf(P/IP)
\xrightarrow{v} \Spf(A)$. For a sheaf of 
abelian groups $\CF$ on $(\fX/R)_{\prism}$,
we define the cosimplicial sheaves  of abelian groups
$\CechA_{A,B}(\CF)^{\bullet}_{\cs}$ 
on $\fX_{\Zar}$ by 
$\CechA_{A,B}(\CF)^{r}_{\cs}
=v_{D(r)*}\CF_{D(r)}$ $(r\in \N)$ 
with the cosimplicial structure naturally 
induced by the simplicial structure of 
of the objects $((D(r),ID(r)),v_{D(r)})$ $(r\in\N)$
of $(\fX/R)_{\prism}$. We define
the \v{C}ech-Alexander complex 
$\CechA_{A,B}(\CF)^{\bullet}$ to be the
cochain complex associated to the
cosimplicial sheaves of abelian groups
$\CechA_{A,B}(\CF)^{\bullet}_{\cs}$.

We can construct a morphism 
\begin{equation}\label{eq:CechAlexCompMap}
u_{\fX/R*}\CF\longrightarrow
\CH^0(\CechA_{A,B}(\CF)^{\bullet})
\end{equation}
as follows. Let $\fX'=\Spf(A')$ be an open affine formal subscheme of $\fX=\Spf(A)$,
let $\Spf(D(r)')$ $(r\in \N)$ be the open affine formal subscheme of $\Spf(D(r))$
whose underlying topological space is given by $v_{D(r)}^{-1}(\fX')
\subset \Spf(D(r)/ID(r))$, and  let $v_{D(r)'}$ denote the morphism 
$\Spf(D(r)'/ID(r)')\to \Spf(A')$ induced by $v_{D(r)}$. Then 
$((D(r)',ID(r)'), v_{D(r)'})$ $(r\in\N)$ form a simplicial object of $(\fX'/R)_{\prism}$. 
Hence we have a natural morphism 
\begin{equation}\label{eq:CechAlexCompSect}
\Gamma((\fX'/R)_{\prism},\CF\vert_{(\fX'/R)_{\prism}})
\to \Ker(p_1^*-p_0^*\colon
\Gamma(D(0)',\CF)\to
\Gamma(D(1)',\CF))
=H^0(\Gamma(\fX',\CechA_{A,B}(\CF)^{\bullet})),
\end{equation}
where we abbreviate $((D(r)',ID(r)'),v_{D(r)'})$
$(r=0,1)$ to $D(r)'$ in the second term. This construction
is obviously functorial in $\fX'$ and gives the desired
morphism \eqref{eq:CechAlexCompMap}. 

\begin{proposition}\label{prop:CechAlexander}
Let $\CF$ be a sheaf of $\CO_{\fX/R,n}$-modules on $(\fX/R)_{\prism}$. \par
(1) If $\CF$ is injective, then the morphism \eqref{eq:CechAlexCompMap}
is an isomorphism and 
$\CH^q(\CechA_{A,B}(\CF)^{\bullet})=0$ for every $q>0$.\par
(2) If $\CF$ is a crystal, then we have a canonical isomorphism 
$Ru_{\fX/R*}\CF\cong \CechA_{A,B}(\CF)$
in  $D^+(\fX, u_{\fX/R*}\CO_{X/R,n})$ inducing \eqref{eq:CechAlexCompMap} 
on $\CH^0$.
\end{proposition}

\begin{proof}
(1) Let $j\colon \fX'\hookrightarrow \fX$ be an open affine formal subscheme, 
and let $D(r)'$ and $v_{D(r)'}$ be as in the construction of 
\eqref{eq:CechAlexCompMap}. Then, for an object 
$((P,IP),w)$ of $(\fX'/R)_{\prism}$ and $r\in \N$, any morphism 
$((P,IP),j\circ w)\to (D(r),ID(r)),v_{D(r)})$ in $(\fX/R)_{\prism}$
is uniquely induced by a morphism 
$((P,IP),w)\to ((D(r)',ID(r)'),v_{D(r)'})$ in $(\fX'/R)_{\prism}$.
This implies that $((D(r)',ID(r)'),v_{D(r)'})$ $(r\in\N)$ represents 
the fiber product of $r+1$ copies of $((D',ID'),v_{D'})$ in $(\fX'/R)_{\prism}$
and that $((D',ID'),v_{D'})$ is a covering of the final object of 
$(\fX'/R)_{\prism}^{\sim}$ by Proposition \ref{prop:PrismCofFinObj}.
The $\CO_{\fX'/R,n}$-module 
$\CF':=\CF\vert_{(\fX'/R)_{\prism}}$ is injective by \cite[V Proposition 4.11 1)]{SGA4}.
See also \cite[Lemma IV.3.1.10]{AGT}.
Hence \eqref{eq:CechAlexCompSect} is an isomorphism 
and $H^m(\Gamma(\fX',\CechA_{A,B}(\CF')^{\bullet}))
=H^m(\Gamma(D(\bullet)',\CF'))$ vanishes if $m>0$
(\cite[V Corollaire 3.3]{SGA4}). Varying $\fX'$, we obtain the claim.\par
(2) Let $\CF\to \CI^{\bullet}$ be an injective resolution of $\CF$
in the category of $\CO_{\fX/R,n}$-modules. Then 
we have a quasi-isomorphism $u_{X/S*}\CI^{\bullet}
\xrightarrow{\sim} \Tot(\CechA_{A,B}(\CI^{\bullet})^{\bullet})$
by (1), and a morphism $\CechA_{A,B}(\CF)
\to \Tot(\CechA_{A,B}(\CI^{\bullet}))^{\bullet}$. 
Hence it suffices to show that 
$\CechA_{A,B}(\CF)^r(\fX')\to \CechA_{A,B}(\CI^{\bullet})^r(\fX')$
is a resolution for each $r\in\N$
and open formal subscheme $j\colon \fX'\hookrightarrow \fX$.
Let $D(r)'$ and $v_{D(r)'}$ be as in the construction of 
\eqref{eq:CechAlexCompMap}.
Then we have $\Gamma(\fX',\CechA_{A,B}(\CG)^r)
=\Gamma(((D(r)',ID(r)'),j\circ v_{D(r)'}),\CG)$ for a sheaf $\CG$ on $(\fX/R)_{\prism}$.
Hence it suffices to show $H^m(((P,IP),v),\CG)=0$ for every integer $m>0$
for an object $((P,IP),v)$ of $(\fX/R)_{\prism}$ and a crystal of 
$\CO_{\fX/R,n}$-modules $\CG$. For a covering
$\fU=(u_{\lambda}\colon
(P_{\lambda},IP_{\lambda}), v_{\lambda})\to ((P,IP),v))_{\lambda\in\Lambda}
\in \Cov_{\fpqc}((P,IP),v)$ (Definition \ref{def:PrismaticSite} (1)),
the \v{C}ech cohomology $H^m(\fU,\CG)$ vanishes for $m>0$.
Hence the claim holds by \cite[V Proposition 4.3]{SGA4}.
\end{proof}

\begin{proof}[Proof of Proposition \ref{prop:LinearizationLocCoh}]
For $r\in \N$, $B_{D(r)}$ is the $(pR+I)$-adic
completion of $D(r)\otimes_RB=D(r)\otimes_{B(r)}(B(r)\otimes_RB)$, 
the kernel of $D(r)\otimes_RB\to D(r)/ID(r)$ is generated by 
the kernel of $B(r)\otimes_RB\to A$ (because the image
of the kernel of $B(r)\to A$ in $D(r)$ is $ID(r)$), and
the $(pR+I)$-adic completion of $B(r)\otimes_RB$ is isomorphic to
$B(r+1)$. Hence, the $R$-homomorphism 
$B(r)\otimes_RB\to D(r)\otimes_RB$ induces a morphism of
$\delta$-pairs $(B(r+1),J(r+1))\to (B_{D(r)},J_{B_{D(r)}})$ over $(R,I)$,
and Lemma \ref{lem:bddPrismEnvBase} (2) and (3) imply that the latter morphism
gives an isomorphism $(D(r+1),ID(r+1))\xrightarrow{\cong}(D_{D(r)},ID_{D(r)})$
between the bounded prismatic envelopes over $(R,I)$.
Put $\fD(r)=\Spf(D(r))$ and $\fD_{D(r)}=\Spf(D_{D(r)})$ for $r\in\N$.
We omit $(r)$ if $r=0$.
Then $\fD(r)$ (resp.~$\fD_{D(r)}$) $(r\in \N)$ form a simplicial formal scheme 
$\fD(\bullet)\text{(resp.~}\fD_{D(\bullet)})\colon \Delta^{\circ}\to \FSch_R$ over $R$.
For $s\in \N$, let $p_s$ denote the morphism $\fD(s)\to \fD(0)=\fD$
corresponding to the map $[0]\to [s]$ sending $0$ to $s$.
Then $p_{r+1}\colon \fD(r+1)\to \fD$ $(r\in \N)$ form 
a simplicial formal scheme over $\fD$ by the 
composition of $\fD(\bullet)$ with the functor $\Delta^{\circ}\to \Delta^{\circ}$
defined by adding $n+1$ to $[n]$, i.e., by 
sending $[n]$ to $[n+1]$ and 
$f\colon [n]\to [m]$ to $g\colon [n+1]\to[m+1]$ 
such that $g\vert_{[n]}=f$ and $g(n+1)=m+1$.
With this structure given on $\fD(r+1)$ $(r\in \N)$, the isomorphisms 
$\fD_{D(r)}\xrightarrow{\cong}\fD(r+1)$ $(r\in \N)$
define an isomorphism of simplicial formal schemes over $\fD$.

Put $\fD(r)_n=\Spec(D(r)_n)$ for $r\in \N$. We omit $(r)$ if $r=0$.
For an object $((P,IP),v)$ of $(\fX/R)_{\prism}$, 
put  $\fD_{P,n}=\Spec(D_P/(pD_P+ID_P)^{n+1})$, and 
let $p_0^P$ and $p_1^P$
denote the morphism $\fD_{P,n}\to \Spf(P/(pP+IP)^{n+1})$ 
and $\fD_{P,n}\to \fD_n$.
Then, by 
Lemma \ref{lem:LinearizationEnvBC}, we have $\CL_{A,B}(M)_P\cong p_{0*}^P
p_1^{P*}\CM$,
which implies 
$$v_*\CL_{A,B}(M)_P
\cong v_*p_{0*}^Pp_1^{P*}\CM\cong v_{D*}p_{1*}^Pp_1^{P*}\CM
\cong v_{D*}(\CM\otimes_{\CO_{\fD_n}}p_{1*}^P\CO_{\fD_{P,n}}).$$ 
Hence the isomorphism $\fD(\bullet+1)\cong \fD_{D(\bullet)}$
of simplicial formal schemes over $D$ induces an isomorphism of 
cosimplicial $R$-modules on $\fX_{\Zar}$
\begin{equation*}
\CechA_{A,B}(\CL_{A,B}(\CM))_{\cs}
=v_{D(\bullet)*}\CL_{A,B}(M)_{D(\bullet)}
\cong v_{D*}(\CM\otimes_{\CO_{\fD_n}}p_{1*}^{D(\bullet)}\CO_{\fD_{D(\bullet),n}})
\cong v_{D*}(\CM\otimes_{\CO_{\fD_n}}p_{\bullet+1*}\CO_{\fD(\bullet+1)_n}).
\end{equation*}
By \cite[V Lemme 2.2.1]{BerthelotCrisCoh}, the complex 
$p_{\bullet+1*}\CO_{\fD(\bullet+1)_n}$ with $\CO_{\fD_n}\to p_{1*}\CO_{\fD(1)_n}$
added becomes a complex of $\CO_{\fD_n}$-modules homotopically equivalent to zero.
By taking $v_{D*}(\CM\otimes_{\CO_{\fD_n}}-)$, we obtain a complex
$v_{D*}\CM\to v_{D*}(\CM\otimes_{\CO_{\fD_n}}p_{\bullet+1*}\CO_{\fD(\bullet+1)_n})$
homotopically equivalent to zero. By Proposition \ref{prop:CechAlexander} (2), 
it remains to show that the
following two compositions are the same:
\begin{align*}
&v_{D*}\CM\to v_{D*}(\CM\otimes_{\CO_{\fD_n}}p_{1*}^D\CO_{\fD_{D,n}})
\cong \CechA_{A,B}(\CL_{A,B}(\CM))^0,\\
&v_{D*}\CM\xrightarrow{\eqref{eq:LinearizationZariskiProjMap}} u_{\fX/R*}\CL_{A,B}(\CM)
\xrightarrow{\eqref{eq:CechAlexCompMap}} \CechA_{A,B}(\CL_{A,B}(\CM))^0.
\end{align*}
Let $\fX'$ be an open affine formal subscheme of $\fX$,
let $\Spf(D')$ be the open formal subscheme of $\Spf(D)$
whose underlying topological space is given by $v_D^{-1}(\fX')
\subset \Spf(D/ID)$, and regard $(D',D'/ID')$ as an object
of $(\fX/R)_{\prism}$ by the morphism $\Spf(D'/ID')
\subset \Spf(D/ID)\xrightarrow{v_D}\fX$. Then
the morphism $p_1^{D'}\colon \Spf(D_{D'})\to \Spf(D)$ factors
through $\Spf(D')$, and we see that $\Gamma(\fX',-)$
of the above two compositions are
$\id_M\otimes p_1^{D'*}\colon M\otimes_DD'\to M\otimes_{D,p_1^{D'*}}D_{D'}$
by tracing the construction of each map. 
\end{proof}

We have the following analogue of Proposition \ref{prop:LinearizationLocCoh} for an inverse system $(M_n)_{n\in \N}$ of $D_n$-modules
such that the homomorphism 
$M_{n+1}\otimes_{D_{n+1}}D_n\to  M_{n}$
is an isomorphism for every $n\in \N$.
Let $(\CM_{n})_{n\in \N}$ be the inverse system of quasi-coherent $\CO_{\fD_n}$-modules
associated $(M_n)_{n\in \N}$. Then, by taking the inverse limit of \eqref{eq:LinearizationZariskiProjMap} for 
$M_n$ $(n\in \N)$ and using Proposition \ref{prop:LinearizationLocCoh}, we obtain isomorphisms
\begin{equation}\label{eq:LinZarProjMapProjlim}
v_{D*}\bigl(\varprojlim_n \CM_n\bigr)
\cong\varprojlim_n \left(v_{D*}\CM_n\right)\overset{\cong}\longrightarrow
\varprojlim_n \left(u_{\fX/R*}\CL_{A,B}(M_n)\right)\cong u_{\fX/R*}\bigl(\varprojlim_n \CL_{A,B}(M_n)\bigr).
\end{equation}

\begin{proposition}\label{prop:LinLocCohProjlim}
Under the notation and assumption as above, the isomorphisms \eqref{eq:LinZarProjMapProjlim}
induce an isomorphism 
$$v_{D*}\bigl(\varprojlim_n\CM_n\bigr)\xrightarrow{\;\;\cong\;\;} Ru_{\fX/R*}\bigl(\varprojlim_n\CL_{A,B}(M_n)\bigr).$$
\end{proposition}

\begin{lemma}\label{lem:crystalDProjlim}
Let $(\CF_n)_{n\in \N}\in \Ob (\Crystal^{\ad}_{\prism}(\CO_{\fX/R,\bullet}))$ 
be an adic inverse system of crystals of $\CO_{\fX/R,n}$-modules
on $(\fX/R)_{\prism}$ (Remark \ref{rmk:CompleteCrystalInvSystem} (3)).\par
(1) The morphism $\varprojlim_n\CF_n\to R\varprojlim_n\CF_n$ is an isomorphism.\par
(2) The morphism $\varprojlim_n \CechA_{A,B}(\CF_n)^r\to R\varprojlim_n\CechA_{A,B}(\CF_n)^r$
is an isomorphism for every $r\in \N$. 
\end{lemma}

\begin{proof}
(1) For every object $((P,IP),v)$ of $(\fX/R)_{\prism}$, 
we have $H^r(((P,IP),v),\CF_n)=0$ for every $n\in \N$ and $r>0$ 
by the proof of Proposition \ref{prop:CechAlexander} (2),
and the morphism $\Gamma(((P,IP),v),\CF_{n+1})
\to \Gamma(((P,IP),v),\CF_n)$ is surjective by the assumption on $(\CF_n)_{n\in \N}$
and Remark \ref{rmk:CompleteCrystalInvSystem} (1).
Hence we have $\varprojlim_n \CF_n\xrightarrow{\cong} R\varprojlim_n \CF_n$
(\cite[Lemma IV.4.2.3]{AGT}). \par
(2) By definition, we have $\CechA_{A,B}(\CF_n)^r=v_{D(r)*}\CF_{n,D(r)}$.
It suffices to prove that, for any affine open formal subscheme
$\fX'$ of $\fX$, the morphism 
$\Gamma(\fX',v_{D(r)*}\CF_{n+1,D(r)})\to
\Gamma(\fX',v_{D(r)*}\CF_{n,D(r)})$ is surjective
and $H^s(\fX',v_{D(r)*}\CF_{n,D(r)})=0$ for $n\in\N$ and $s>0$
(\cite[Lemma IV.4.2.3]{AGT}).
Let $\Spf(D(r)')$ be the affine open formal subscheme of $\Spf(D(r))$
whose underlying topological space is $v_{D(r)}^{-1}(\fX')\subset \Spf(D(r)/ID(r))$,
and let $v_{D(r)'}$ denote the restriction of $v_{D(r)}$ to $\Spf(D(r)'/ID'(r))$.
Then we have $\Gamma(\fX',v_{D(r)*}\CF_{m,D(r)})=
\Gamma(D(r)',\CF_{m,D(r)})=\Gamma(((D(r)', ID(r)'),v_{D(r)'}), \CF_m)$
for $m=n$, $n+1$. Hence the assumption on $(\CF_n)_{n\in \N}$ and 
Remark \ref{rmk:CompleteCrystalInvSystem} (1) imply the first claim. 
Put $\fD(r)_n=\Spec(D(r)/(ID(r)+pD(r))^{n+1})$. 
Since $v_{D(r)}$ is an affine morphism,
we see $Rv_{D(r)*}\CG=\CG$ for a quasi-coherent $\CO_{\fD(r)_n}$-module
by induction on $n$.  Hence we have 
$R\Gamma(\fX',v_{D(r)*}\CF_{n,D(r)})=R\Gamma(\fX', Rv_{D(r)*}\CF_{n,D(r)})
=R\Gamma(\fD(r)'_n,\CF_{n,D(r)})=\Gamma(\fD(r)'_n,\CF_{n,D(r)})$,
where $\fD(r)'_n=\Spec(D(r)'/(ID(r)'+pD(r)')^{n+1})$.
This completes the proof.
\end{proof}

\begin{proof}[Proof of Proposition \ref{prop:LinLocCohProjlim}]
By the proof of Proposition \ref{prop:LinearizationLocCoh},
the inverse system of complexes
$(v_{D*}\CM_n
\xrightarrow{\eqref{eq:CechAlexCompMap}\circ\eqref{eq:LinearizationZariskiProjMap}}
\CechA_{A,B}(\CL_{A,B}(M_n))^{\bullet})_{n\in \N}$ is 
homotopic to $0$. Hence it also holds for its inverse limit over $n\in \N$.
Therefore the claim follows from 
\begin{multline*}
Ru_{\fX/R*}\bigl(\varprojlim_n\CL_{A,B}(M_n)\bigr)
\xrightarrow[\text{Lem.~\ref{lem:crystalDProjlim} (1)}]{\cong}
Ru_{\fX/R*}\bigl(R\varprojlim_n\CL_{A,B}(M_n)\bigr)
\cong R\varprojlim_n \bigl(Ru_{\fX/R*}\CL_{A,B}(M_n)\bigr)\\
\xleftarrow[\text{Prop.~\ref{prop:CechAlexander} (2)}]{\cong}  R\varprojlim_n \CechA_{A,B}(\CL_{A,B}(M_n))^{\bullet}
\xleftarrow[\text{Lem.~\ref{lem:crystalDProjlim} (2)}]{\cong}
\varprojlim_n \CechA_{A,B}(\CL_{A,B}(M_n))^{\bullet}.
\end{multline*}
Note that $\CL_{A,B}(M_n)$ is a crystal of $\CO_{\fX/R,n}$-modules
and the morphism $\CL_{A,B}(M_{n+1})\otimes_{\CO_{\fX/R,n+1}}\CO_{\fX/R,n}
\to \CL_{A,B}(M_n)$ is an isomorphism by $M_{n+1}\otimes_{D_{n+1}}D_n\cong M_n$
and the construction of $\CL_{A,B}(-)$. 
\end{proof}

\section{Prismatic cohomology and $q$-Higgs complex: Local case}\label{sec:PrismCohqDolb}
Let $(R,I)$, $\fX=\Spf(A)$, $i\colon \fX\to \fY=\Spf(B)$ be the same as
before \eqref{eq:LinearizationFunctor}. We keep the notation introduced in \S\ref{sec:LinCAcpx}, 
assume further that 
that $R$ and $B$ satisfy 
Condition \ref{cond:qPrism}, and follow the notation introduced in 
\S\ref{sec:CrystalStratqHiggs} under this condition. 
As it is mentioned at the beginning of \S\ref{sec:LinCAcpx}, 
we give a description of the derived direct image under the projection 
$u_{\fX/R}$ \eqref{eq:PrismToposZarProj} to the Zariski topos
$\fX_{\Zar}$ of a crystal on $(\fX/R)_{\prism}$ in terms of 
the $q$-Higgs complex associated to the crystal 
(Theorems \ref{th:CrystalCohqHiggs} and \ref{thm:CrystalCohqHigProj}).
We discuss compatibility
with scalar extension under the liftings of Frobenius and tensor products
as remarks every time we introduce a new complex or a new morphism 
of complexes.

Let $\CF$ be a crystal of $\CO_{\fX/R,n}$-modules
on $(\fX/R)_{\prism}$, and let $(M,\theta_{M})$ be the object of 
$q\HIG_{\qnilp}(D_n/R)$ corresponding to $\CF$ by the equivalences
of categories in Proposition \ref{prop:CrysStratEquiv}
and Theorem \ref{thm:StratHiggsEquiv}. 
Put $\fD=\Spf(D)$ and $\fD_n=\Spec(D_n)$ $(n\in \N)$. 
Let $\CM$ be the quasi-coherent $\CO_{\fD_n}$-module 
on $\fD_n$ associated to the $D_n$-module $M$. 
Then one can construct a complex 
\begin{equation}\label{eq:qHiggsShfCpx}
(\CM\otimes_{\CO_{\fD}}q\Omega^{\bullet}_{\fD/R},\theta^{\bullet}_{\CM})
\end{equation}
on $(\fD_n)_{\Zar}=\fD_{\Zar}$ as follows. 
We call it {\it the $q$-Higgs complex on $\fD_{\Zar}$} associated to
$(M,\theta_M)$ or $\CF$.

Let $\Spf(D')$ be an open affine formal subscheme of $\fD=\Spf(D)$.
Then the $t_{i}\mu$-derivation $\theta_{D,i}$ over $R$ uniquely
extends to a $t_i\mu$-derivation $\theta_{D',i}$ over $R$
by Proposition \ref{prop:TwistDerEtaleExt}, and 
$\theta_{D',i}$ $(i\in \Lambda)$ commute with each other
by Lemma \ref{lem:EtMapDerivComm} and 
$\theta_{D',i}(t_j\mu)=0$ $(i\neq j)$. 
By applying the construction after Lemma \ref{lem:ConnectionTwDeriv}
 to $D'$ and $(t_i\mu,\theta_{D',i})$ $(i\in \Lambda)$,
we obtain a differential graded algebra $(q\Omega^{\bullet}_{D'/R},\theta_{D'}^{\bullet})$ over
$(q\Omega^{\bullet}_{D/R},\theta_{D}^{\bullet})$. 
By Lemma \ref{lem:TwDerivSystemFunct} (1) and  Definition \ref{def:connectionScalarExt},
$(M,\theta_M)$ extends to an integrable  connection 
$(D'\otimes_DM, \theta_{D'\otimes_DM})$
over $(q\Omega^{\bullet}_{D'/R},\theta_{D'}^{\bullet})$.
For another open affine formal
subscheme $\Spf(D'')$ of $\fD=\Spf(D)$ contained in $\Spf(D')$, 
we see that the homomorphism $D'\to D''$ is compatible with 
$\theta_{D',i}$ and $\theta_{D'',i}$ by Lemma \ref{lem:TwDerivEtExtComp}.
By Proposition \ref{prop:dRCpxFunct} and Lemma \ref{lem:dRCpxFuncCocyc},
this implies that $(q\Omega^{\bullet}_{D'/R},\theta_{D'}^{\bullet})$
and the complex $((D'\otimes_DM)\otimes_{D'} q\Omega^{\bullet}_{D'/R}, 
\theta^{\bullet}_{D'\otimes_DM})$ associated to $(D'\otimes_DM,\theta_{D'\otimes_DM})$
are functorial in $D'$. Thus, varying $D'$,  we obtain a sheaf of differential graded
algebras $(q\Omega^{\bullet}_{\fD/R},\theta_{\fD}^{\bullet})$ over $R$
and a complex $(\CM\otimes_{\CO_{\fD}}q\Omega^{\bullet}_{\fD/R},\theta_{\CM}^{\bullet})$
on $\fD_{\Zar}$. 

\begin{remark}\label{rmk:qHiggsGammaAction}
Under the notation in Remark \ref{rmk:CrystalCompQCohShf}, we have $\CM=\CF_D$
because $M=\CF(D)$ by definition. 
Let $\theta_{\fD,i}\colon \CO_{\fD}\to \CO_{\fD}$ and
$\theta_{\CM,i}\colon \CM\to \CM$ $(i\in \Lambda)$
be the $R$-linear endomorphisms defined by 
$\theta_{\fD}(x)=\sum_{i\in \Lambda}\theta_{\fD,i}(x)\otimes\omega_i$
$(x\in \CO_{\fD})$, and 
$\theta_{\CM}(m)=\sum_{i\in \Lambda}\theta_{\CM,i}(m)\otimes\omega_i$
$(m\in \CM)$. Since $\gamma_{D,i}$ is trivial modulo $\mu$,
$\Spf(\gamma_{D,i})\colon \fD\to \fD$ is the identity map on the
underlying topological space. Hence it induces an automorphism
$\gamma_{\fD,i}\colon \CO_{\fD}\xrightarrow{\cong}\CO_{\fD}$
over $R$, and then the $\gamma_{D,i}$-semilinear
automorphism $\CF(\gamma_{D,i})\colon \CF(D)\xrightarrow{\cong}\CF(D)$
(see before Proposition \ref{prop:HiggsFieldAndCrysAuto}) 
induces a $\gamma_{\fD,i}$-semilinear
automorphism $\gamma_{\CF_D,i}\colon \CF_D\xrightarrow{\cong}\CF_D$.
We see that these endomorphisms and automorphisms of 
$\CO_{\fD}$ and $\CM=\CF_D$ satisfy
\begin{align}
\gamma_{\fD,i}&=1+t_i\mu\theta_{\fD,i},\label{eq:HigAutoShfComp}\\
\gamma_{\CF_D,i}&=1+t_i\mu\theta_{\CM,i}\label{eq:HigAutoShfComp2}
\end{align}
as follows. Let $\fD'=\Spf(D')$ be an open affine formal subscheme 
of $\fD=\Spf(D)$. Both $\gamma_{\fD,i}(\fD')$
and $(1+t_i\mu\theta_{\fD,i})(\fD')=1+t_i\mu\theta_{D',i}$
are endomorphisms of the $R$-algebra $D'$ compatible with $\gamma_{D,i}
=1+t_i\mu\theta_{D,i}$. Hence they coincide. This implies \eqref{eq:HigAutoShfComp}.
We have $\gamma_{\CF_D,i}(\fD')=\gamma_{\fD,i}(\fD')
\otimes\CF(\gamma_{D,i})$ by definition
and $(1+t_i\mu\theta_{\CM,i})(\fD')=1+t_i\mu\theta_{D'\otimes_DM,i}$
is a $(1+t_i\mu\theta_{D',i})$-semilinear endomorphism 
compatible with $1+t_i\mu\theta_{M,i}$. Hence
the equality \eqref{eq:HigAutoShfComp2} holds for sections 
on $\fD'$ by Proposition 
\ref{prop:HiggsFieldAndCrysAuto} and \eqref{eq:HigAutoShfComp}.
\end{remark}
\begin{remark}\label{rmk:qDolSheafCpxFrob}
(1) We consider the Frobenius pullbacks 
$\varphi_n^*\CF=\CF\otimes_{\CO_{\fX/R,n},\varphi_n}\CO_{\fX/R,n}$
and 
$(\varphi_{D_n}^*M,\theta_{\varphi_{D_n}^*M})$ of $\CF$ and $(M,\theta_M)$,
respectively, as in Remark \ref{rmk:CrysFrobPBHiggs}. 
They correspond to each other by Proposition \ref{prop:CrysStratEquiv} and 
Theorem \ref{thm:StratHiggsEquiv} as observed in loc.~cit.
The lifting of Frobenius 
$\varphi_{D_n}$ induces a lifting of Frobenius $\varphi_{\fD_n}\colon 
\fD_n\to \fD_n$, and the quasi-coherent $\CO_{\fD_n}$-module
associated to the $D_n$-module $\varphi_{D_n}^*M$ may be identified
with $\varphi_{\fD_n}^*\CM=\CM\otimes_{\CO_{\fD_n},\varphi_{\fD_n}^*}\CO_{\fD_n}$.
Let $((\varphi_{\fD_n}^*\CM)\otimes_{\CO_{\fD}}q\Omega_{\fD/R}^{\bullet},
\theta_{\varphi_{\fD_n}^*\CM}^{\bullet})$
be the complex associated to $(\varphi_{D_n}^*M,\theta_{\varphi_{D_n}^*M})$
similarly to $(M,\theta_M)$. Then we can construct a morphism of complexes
\begin{equation}\label{eq:qdRSheafCpxFrobPB}
\varphi_{\fD_n,\Omega}^{\bullet}(\CM)\colon \CM\otimes_{\CO_{\fD}}
q\Omega^{\bullet}_{\fD/R}\longrightarrow (\varphi_{\fD_n}^*\CM)\otimes_{\CO_{\fD}}q\Omega^{\bullet}_{\fD/R}
\end{equation}
as follows. 
For each affine open $\Spf(D')\subset\Spf(D)$, 
$\theta_{D',i}$ is $\delta$-compatible with respect to $t_i^{p-1}\eta$
by Proposition \ref{prop:TwistDerDeltaEtaleExt}. Therefore
we have $\theta_{D',i}\circ\varphi_{D'}=[p]_qt_i^{p-1}\cdot\varphi_{D'}\circ \theta_{D',i}$ by 
Lemma \ref{lem:qDerivFrobComp} (2), and obtain a Frobenius pullback functor 
$\varphi_{D'_n}^*$ \eqref{eq:qDerivConnFobPB} for modules with integrable connection
over $D'_n\otimes_{D'} q\Omega^{\bullet}_{D'/R}$,
which is functorial in $D'$ by \eqref{eq:qConnFrobPBFunct}. 
By \eqref{eq:qDervConnCpxPB}, we have a morphism of complexes
\begin{equation}\label{eq:qdRSheafCpxSecFrobPB}
\varphi_{D'_n,\Omega}^{\bullet}(D'\otimes_DM)\colon
(D'\otimes_DM)\otimes_{D'}q\Omega^{\bullet}_{D'/R}
\to (D'\otimes_D\varphi_{D_n}^*M)\otimes_{D'}q\Omega^{\bullet}_{D'/R},\end{equation}
which is functorial in $D'$ by \eqref{eq:qdRcpxFrobPBFunct};
in each degree $r\in \N$, this is simply given by the $\varphi_{D'_n}$-semilinear
extension of the $\varphi_{D_n}$-semilinear map
$\varphi^r_{D_n,\Omega}(M)\colon M\otimes_Dq\Omega^r_{D/R}
\to \varphi_{D_n}^*M\otimes_Dq\Omega^r_{D/R}$ under $D_n\to D_n'$. 
Varying $D'$, we obtain \eqref{eq:qdRSheafCpxFrobPB} from \eqref{eq:qdRSheafCpxSecFrobPB}.
\par
(2) Let $\CF$ and $\CG$ be crystals of $\CO_{\fX/R,n}$-modules
on $(\fX/R)_{\prism}$, and let $(M,\theta_M)$, $(N,\theta_N)$,
and $(L=M\otimes_DN,\theta_L)$ be the objects of 
$q\HIG_{\qnilp}(D_n/R)$ corresponding to the crystals
$\CF$, $\CG$, and $\CF\otimes_{\CO_{\fX/R,n}}\CG$
(Remark \ref{rmk:crystalFrobPBTensor} (2)) by 
Proposition \ref{prop:CrysStratEquiv} and Theorem \ref{thm:StratHiggsEquiv}.
By Remark \ref{rmk:CrysFrobPBHiggs} (2), $\theta_L$ is the tensor product of $\theta_M$
and $\theta_N$ (Definition \ref{def:ConnTensorProdDef}).
For an affine open $\Spf(D')\subset\Spf(D)$, the scalar extension
$\theta_{D'\otimes_DL}$ of $\theta_{L}$ under $D\to D'$ is the tensor
product of the scalar extensions of $\theta_{D'\otimes_DM}$ and $\theta_{D'\otimes_DN}$
of $\theta_M$ and $\theta_N$ by Remark \ref{rmk:ConnScalarExtStratTensProd} (2), and the product morphism
\eqref{eq:dRcpxProd} for $(D'\otimes_DM,\theta_{D'\otimes_DM})$ and 
$(D'\otimes_DN,\theta_{D'\otimes_DN})$ is functorial in $D'$ by 
Remark \ref{rmk:dRCpxProdScExtComp}. Thus, writing 
$\CM\otimes_{\CO_{\fD}}q\Omega^{\bullet}_{\fD/R}$,
$\CN\otimes_{\CO_{\fD}}q\Omega^{\bullet}_{\fD/R}$,
and $(\CM\otimes_{\CO_{\fD}}\CN)\otimes_{\CO_{\fD}}q\Omega^{\bullet}_{\fD/R}$
for the complexes of sheaves on $\fD_{\Zar}$ associated
to $(M,\theta_M)$, $(N,\theta_N)$, and $(M\otimes_DN,\theta_L)$, we obtain 
a product morphism 
\begin{equation}\label{eq:qDolShfCpxProd}
(\CM\otimes_{\CO_{\fD}}q\Omega^{\bullet}_{\fD/R})\otimes_R
(\CN\otimes_{\CO_{\fD}}q\Omega^{\bullet}_{\fD/R})
\longrightarrow
(\CM\otimes_{\CO_{\fD}}\CN)\otimes_{\CO_{\fD}}q\Omega^{\bullet}_{\fD/R}.
\end{equation}
In the case $\CF=\CG=\CO_{\fX/R,n}$, we have $(M,\theta_M)=(N,\theta_N)
=(D_n,\theta_{D}\bmod (p,\pq)^{n+1})$, and the product 
\eqref{eq:qDolShfCpxProd} coincides with the product structure of the
scalar extension of the sheaf of differential graded algebras
$(q\Omega^{\bullet}_{\fD/R},\theta^{\bullet}_{\fD})$ over
$R$ under $R\to R/(pR+\pq R)^{n+1}$. By the remark
after \eqref{eq:qDervConnCpxPB}, we see that the product
morphisms \eqref{eq:qDolShfCpxProd} for the pairs
$(\CF,\CG)$ an $(\varphi_n^*\CF,\varphi_n^*\CG)$ are
compatible with the morphisms \eqref{eq:qdRSheafCpxFrobPB}
for $\CF$, $\CG$, and $\CF\otimes_{\CO_{\fX/R,n}}\CG$.
\end{remark}

We prove the following theorem.

\begin{theorem}\label{th:CrystalCohqHiggs}
Under Condition \ref{cond:qPrism} and the notation  before Remark \ref{rmk:qHiggsGammaAction}, 
we have the following 
canonical isomorphism in $D^+(\fX_{\Zar},R)$ for a crystal $\CF$ of $\CO_{\fX/R,n}$-modules
on $(\fX/R)_{\prism}$ functorial in $\CF$. 
See Remark \ref{rmk:CrysZarProjCompFrobProd} 
for  compatibility
with scalar extensions by Frobenius and with products. 
\begin{equation}\label{eq:CrystalCohqHiggs}
Ru_{\fX/R*}\CF\cong v_{D*}(\CM\otimes_{\CO_{\fD}}q\Omega^{\bullet}_{\fD/R},
\theta_{\CM}^{\bullet})\end{equation}
\end{theorem}

We start by constructing a complex of $\CO_{\fX/R,n}$-modules
\begin{equation}\label{eq:qdRCpxLinearization}
(\CL_{A,B}(M\otimes_Dq\Omega^{\bullet}_{D/R}),\theta^{\bullet}_{\CL(M)}),
\end{equation}
in a manner functorial in $\CF$,  such that the following diagram 
is commutative for every $r\in \N$.
\begin{equation}\label{eq:qdRCpxLinZarProj}
\xymatrix@C=70pt{
u_{\fX/R*}(\CL_{A,B}(M\otimes_Dq\Omega^r_{D/R}))
\ar[r]^{u_{\fX/R*}(\theta^r_{\CL(M)})}&u_{\fX/R*}(\CL_{A,B}(M\otimes_Dq\Omega^{r+1}_{D/R}))
\\
v_{D*}(\CM\otimes_{\CO_{\fD}}q\Omega^r_{\fD/R})
\ar[r]^{v_{D*}(\theta_{\CM}^r)}
\ar[u]^{\eqref{eq:LinearizationZariskiProjMap}}&
v_{D*}(\CM\otimes_{\CO_{\fD}}q\Omega^{r+1}_{\fD/R})
\ar[u]^{\eqref{eq:LinearizationZariskiProjMap}}
}
\end{equation}

Let $(P,v)$ be an object of $(\fX/R)_{\prism}$,
let $(B_P,J_{B_P})$ and $D_P$ be as defined before Lemma \ref{lem:LinearizationEnvBC}, 
and let $\ut_P=(t_{P,i})_{i\in \Lambda}\in (B_P)^{\Lambda}$ be the image
of $\ut=(t_i)_{i\in \Lambda}\in B^{\Lambda}$. Then 
$((B_P,J_{B_P})/P,\ut_P)$ is a framed smooth $q$-pair 
(Definition \ref{def:FramedSmoothPrism} (3)).
Therefore we have the $q$-Higgs derivations, automorphisms, and differential
$\theta_{D_P,i}$, $\gamma_{D_P,i}$ and $\theta_{D_P}\colon D_P
\to q\Omega_{D_P/P}=\oplus_{i\in\Lambda}D_P\omega_{P,i}; x\mapsto 
\sum_{i\in \Lambda}\theta_{D_P,i}(x)\omega_{P,i}$, respectively,
of $D_P$ over $P$ with respect to $\ut_P$ (Definition \ref{def:qHiggsDifferential} (2)). 
Since the $\delta$-homomorphism $B\to B_P$ and $\id_{\Lambda}$ define a morphism
of framed smooth $q$-pairs (Definition \ref{def:FramedSmoothPrism} (3)) 
$((B,J)/R,\ut)\to ((B_P,J_{B_P})/P,\ut_P)$, we have a scalar extension functor
$q\HIG(D/R,\ut)\to q\HIG(D_P/P,\ut_P)$ (Proposition \ref{prop:qHiggsDerivFunct} (2)). By applying it 
to $(M,\theta_M)$, we obtain a $q$-Higgs field $\theta_{M\otimes_DD_P}$ on 
the scalar extension $M\otimes_DD_P$,
and its $q$-Higgs complex $(M\otimes_D{D_P})\otimes_{D_P}q\Omega^{\bullet}_{D_P/P}$,
whose differential maps are linear over $P=\CO_{\fX/R}(P,v)$. 
For a morphism $u\colon (P',v')\to (P,v)$ in $(\fX/R)_{\prism}$, 
the $\delta$-homomorphism $u_P\colon B_{P}\to B_{P'}$ induced by 
$u$, and $\id_{\Lambda}$ define a morphism of 
of framed smooth $q$-pairs $((B_P,J_{B_P})/P,\ut_P)
\to ((B_{P'},J_{B_{P'}})/P',\ut_{P'})$ compatible with the morphisms
from $((B,J)/R,\ut)$. Hence \eqref{eq:qDolbCpxPB2} and its 
compatibility with compositions of $(g/f,\psi)$'s imply that the complex
$(M\otimes_D{D_P})\otimes_{D_P}q\Omega^{\bullet}_{D_P/P}$
is functorial in $(P,v)$. Since we have a $D_P$-linear isomorphism
$q\Omega^r_{D/R}\otimes_DD_P\xrightarrow{\cong}q\Omega^r_{D_P/P}$
$(r\in \N)$ functorial in $(P,v)$, the degree $r$ term
of the complex above is isomorphic to $\Gamma((P,v),\CL_{A,B}(M\otimes_{D}q\Omega_{D/R}^r))$
in a manner functorial in $(P,v)$. Thus we obtain the complex 
\eqref{eq:qdRCpxLinearization} by the transport of structures via this isomorphism.
In the case $\CF=\CO_{\fX/R,n}$, we have $(M,\theta_M)=(D_n,\theta_D\bmod(p,\pq)^{n+1})$,
and $(D_n\otimes_DD_P)\otimes_{D_P}q\Omega^{\bullet}_{D_P/P}
=q\Omega^{\bullet}_{D_P/P}\otimes_PP/(pP+\pq P)^{n+1}$
is equipped with a structure of differential graded algebra over
$P/(pP+\pq P)^{n+1}$ in a manner functorial in $(P,v)$.
This structure for each $(P,v)$ makes $\CL_{A,B}(D_n\otimes_Dq\Omega^{\bullet}_{D/R})$
a sheaf of differential graded algebras over $\CO_{\fX/R,n}$.

\begin{remark}\label{rmk:LinqdRCpxFrob} 
(1) By applying the construction of $\CL_{A,B}(M\otimes_D\Omega_{D/R}^{\bullet})$
to the Frobenius pullback $(\varphi_n^*(M),\theta_{\varphi_n^*(M)})$ \eqref{eq:FramedSmQPHigFPBFunct}
 of $(M,\theta_M)$,
which corresponds to $\varphi_n^*\CF$ (Remark \ref{rmk:CrysFrobPBHiggs}), we obtain a complex
$\CL_{A,B}(\varphi_{D_n}^*M\otimes_Dq\Omega^{\bullet}_{D/R})$. 
Let $(P,v)$ be an object of $(\fX/R)_{\prism}$. 
By Remark \ref{rmk:FramedSmQPHigFPBComp}, the pullback of $(\varphi_{D_n}^*M,\theta_{\varphi_{D_n}^*M})$
to $q\HIG(D_P/P,\ut_P)$ is canonically isomorphic to 
the Frobenius pullback $\varphi_{D_{P,n}}^*(M\otimes_DD_P,\theta_{M\otimes_DD_P})$, 
$D_{P,n}=D_P/(pD_P+\pq D_P)^{n+1}$, and we have a morphism of complexes
\begin{equation}\label{eq:LinqDelCpxSecFrobMap}
\varphi_{D_{P,n},\Omega}^{\bullet}(M\otimes_DD_P)\colon
(M\otimes_DD_P)\otimes_{D_P}q\Omega^{\bullet}_{D_P/P}
\longrightarrow 
(\varphi_{D_n}^*M\otimes_{D}D_P)\otimes_{D_P}q\Omega^{\bullet}_{D_P/P}
\end{equation}
functorial in $(P,v)$; in each degree $r\in \N$, 
this is simply given by the $\varphi_{D_{P,n}}$-semilinear extension of the 
$\varphi_{D_n}$-semilinear
map 
$\varphi_{D_n,\Omega}^r(M)\colon
M\otimes_Dq\Omega^r_{D/R}
\longrightarrow \varphi_{D_n}^*M\otimes_Dq\Omega^r_{D/R}$
\eqref{eq:qDervConnCpxPB}
under $D_n\to D_{P,n}$. Varying $(P,v)$, we obtain a morphism of complexes
\begin{equation}\label{eq:LinqDelCpxFrobMap}
\CL_{A,B}(\varphi_{D_n,\Omega}^{\bullet}(M))\colon
\CL_{A,B}(M\otimes_Dq\Omega^{\bullet}_{D/R})
\longrightarrow
\CL_{A,B}(\varphi^*_{D_n}(M)\otimes_Dq\Omega^{\bullet}_{D/R}).
\end{equation}
\par
(2) Let $\CF$, $\CG$, $(M,\theta_M)$, $(N,\theta_N)$, and $(L=M\otimes_DN,\theta_L)$
be the same as in Remark \ref{rmk:qDolSheafCpxFrob} (2). Recall
that  $\theta_L$ is the tensor product of $\theta_M$ and $\theta_N$ 
(Remark \ref{rmk:CrysFrobPBHiggs} (2)). For 
an object $(P,v)$ of $(\fX/R)_{\prism}$, the scalar extension
$\theta_{L\otimes_DD_P}$ of $\theta_{L}$ under $D\to D_P$ is the tensor
product of the scalar extensions $\theta_{M\otimes_DD_P}$ and $\theta_{N\otimes_DD_P}$
of $\theta_M$ and $\theta_N$ by Remark \ref{rmk:ConnScalarExtStratTensProd} (2), and the product morphism
\eqref{eq:dRcpxProd} for $(M\otimes_DD_P,\theta_{M\otimes_DD_P})$ and 
$(N\otimes_DD_P,\theta_{N\otimes_DD_P})$ is functorial in $(P,v)$ by 
Remark \ref{rmk:dRCpxProdScExtComp}. Thus we obtain a product morphism
\begin{equation}\label{eq:LinqDolCpxProd}
\CL_{A,B}(M\otimes_Dq\Omega^{\bullet}_{D/R})
\otimes_{\CO_{\fX/R,n}}\CL_{A,B}(N\otimes_D q\Omega^{\bullet}_{D/R})
\longrightarrow
\CL_{A,B}((M\otimes_DN)\otimes_Dq\Omega^{\bullet}_{D/R}).
\end{equation}
In the case $\CF=\CG=\CO_{\fX/R,n}$, the product \eqref{eq:LinqDolCpxProd}
coincides with the product structure of the sheaf of 
differential graded algebras $\CL_{A,B}(D_n\otimes_Dq\Omega^{\bullet}_{D/R})$
over $\CO_{\fX/R,n}$.
\end{remark}

The commutativity of \eqref{eq:qdRCpxLinZarProj} is verified as follows. Let $j\colon \fX'\hookrightarrow\fX$
be an open affine formal subscheme of $\fX$, and
let $\fD'=\Spf(D')$ be the open affine formal subscheme of $\fD=\Spf(D)$ 
whose underlying topological space is given by 
$v_{D}^{-1}(\fX')\subset \Spf(D/\pq D)$. Let $(P,w)$ be an object of 
$(\fX'/R)_{\prism}$. Then, as we saw in the construction of 
\eqref{eq:LinearizationZariskiProjMap}, the morphism $\Spf(D_P)\to \Spf(D)$
associated to $(P,j\circ w)\in \Ob (\fX/R)_{\prism}$ 
factors through $\Spf(D')$. As in the construction of the complex
$\CM\otimes_{\CO_{\fD}}q\Omega^{\bullet}_{\fD/R}$ \eqref{eq:qHiggsShfCpx} before Remark
\ref{rmk:qHiggsGammaAction}, the $t_i\mu$-derivations $\theta_{D,i}$ uniquely extend
to $t_i\mu$-derivations $\theta_{D',i}$ commuting with each other, which define
a differential graded algebra $q\Omega^\bullet_{D'/R}$. 
By Lemma \ref{lem:TwDerivEtExtComp}, we see that 
the homomorphism $D'\to D_P$ is compatible with $\theta_{D',i}$
and $\theta_{D_P,i}$ and induces a morphism of 
differential graded  algebras $q\Omega^{\bullet}_{D'/R}
\to q\Omega^{\bullet}_{D_P/P}$ over $q\Omega^{\bullet}_{D/R}$.
By the construction of
the complex $(M\otimes_DD_P)\otimes_{D_P}q\Omega^{\bullet}_{D_P/P}$
in the paragraph before Remark \ref{rmk:LinqdRCpxFrob}, we see that the morphisms
\begin{equation}\label{eq:qDolqDolLinSecMap}
(M\otimes_{D}D')\otimes_{D'}q\Omega^{r}_{D'/R}
\to (M\otimes_DD_P)\otimes_{D_P}q\Omega^{r}_{D_P/P}\quad
(r\in \N)\end{equation} 
are compatible with the differential maps. Since this holds
for all $(P,w)$, the diagram \eqref{eq:qdRCpxLinZarProj} is commutative after 
taking $\Gamma(\fX',-)$ by the construction of \eqref{eq:LinearizationZariskiProjMap}.
Varying $\fX'$, we see that \eqref{eq:qdRCpxLinZarProj} is commutative.

\begin{remark}\label{rmk:LinZarProjFrobComp}
(1)
We can verify that the morphisms \eqref{eq:qdRSheafCpxFrobPB}
and \eqref{eq:LinqDelCpxFrobMap}  are compatible with the morphisms 
\eqref{eq:LinearizationZariskiProjMap} for 
$M\otimes_Dq\Omega^r_{D/R}$ and $\varphi_{D_n}^*M\otimes_D\Omega^r_{D/R}$
$(r\in\Z)$ as follows. Let $\fX'$ be an open formal subscheme of $\fX$. 
Then under the notation in the paragraph above, the homomorphism $D'\to D_P$
is a $\delta$-homomorphism by Proposition \ref{prop:DeltaCompEtExt}. Since
$\varphi_{D_n',\Omega}^r(M\otimes_DD')$ \eqref{eq:qdRSheafCpxSecFrobPB} and
$\varphi_{D_{P,n},\Omega}^r(M\otimes_DD_P)$ \eqref{eq:LinqDelCpxSecFrobMap}
are the semilinear extensions of $\varphi_{D_n,\Omega}^r(M)$ 
under $D_n\to D_n'$, $D_{P,n}$ as mentioned in 
Remarks \ref{rmk:qDolSheafCpxFrob} and \ref{rmk:LinqdRCpxFrob},
the homomorphisms \eqref{eq:qDolqDolLinSecMap} are compatible with 
$\varphi_{D_n',\Omega}^r(M\otimes_DD')$ and
$\varphi_{D_{P,n},\Omega}^r(M\otimes_DD_P)$.
Varying $(P,w)\in (\fX'/R)_{\prism}$ and then $\fX'$,
we obtain the desired compatibility.\par
(2) Let $\CF$, $\CG$, $(M,\theta_M)$, $(N,\theta_N)$, and $(L=M\otimes_DN,\theta_L)$
be the same as in Remark \ref{rmk:qDolSheafCpxFrob} (2). 
We see that the morphisms of complexes \eqref{eq:qdRCpxLinZarProj} 
for $\CF$, $\CG$, and $\CF\otimes_{\CO_{\fX/R,n}}\CG$ are
compatible with the product morphisms \eqref{eq:qDolShfCpxProd} and 
\eqref{eq:LinqDolCpxProd} as follows. Let $\fX'$ be an affine open of $\fX$.
Then, for an object $(P,w)$ of $(\fX'/R)_{\prism}$,
we have the product morphism \eqref{eq:dRcpxProd} for
$(M\otimes_DD',\theta_{M\otimes_DD'})$ and 
$(N\otimes_DD',\theta_{N\otimes_DD'})$ considered 
in Remark \ref{rmk:qDolSheafCpxFrob} (2), and
that for $(M\otimes_DD_P,\theta_{M\otimes_DD_P})$
and $(N\otimes_DD_P,\theta_{N\otimes_DD_P})$ 
considered in Remark \ref{rmk:LinqdRCpxFrob} (2);
they are compatible with the morphism \eqref{eq:qDolqDolLinSecMap}
by Remark \ref{rmk:dRCpxProdScExtComp}.
Therefore, varying $(P,w)$ over objects of $(\fX'/R)_{\prism}$,
and then $\fX'$ over affine opens of $\fX$, we
obtain the desired compatibility.
\end{remark}

Next we show that the complex \eqref{eq:qdRCpxLinearization} gives
an $\CO_{\fX/R,n}$-linear resolution of $\CF$ via a certain map 
$\CF\to \CL_{A,B}(M)$ in a manner functorial in $\CF$. 

The object of $q\HIG_{\qnilp}(D_n/R)$ associated to the crystal
$\CO_{\fX/R,n}$ is 
$(\theta_{D}\text{ mod }(p,[p]_q)^{n+1})\colon D_n\to D_n\otimes_Dq\Omega_{D/R}$
by the construction given before Lemma \ref{lem:StratToConnection}. 
By applying 
\eqref{eq:qdRCpxLinearization} to $\CF=\CO_{\fX/R,n}$, we obtain a complex
of $\CO_{\fX/R,n}$-modules 
$(\CL_{A,B}(D_n\otimes_Dq\Omega^{\bullet}_{D/R}),\theta^{\bullet}_{\CL(D_n)})$. 

\begin{lemma}\label{lem:CrystalPoincareLemma}
Via the $\CO_{\fX/R,n}$-linear map 
\begin{equation}\label{eq:CrysPLAugMap}
\CF\longrightarrow \CF\otimes_{\CO_{\fX/R,n}}\CL_{A,B}(D_n)
\end{equation}
induced by the $\CO_{\fX/R,n}$-algebra structure of 
$\CL_{A,B}(D_n)$, the complex 
$\CF\otimes_{\CO_{\fX/R,n}}(\CL_{A,B}(D_n\otimes_Dq\Omega^{\bullet}_{D/R}),\theta^{\bullet}_{\CL(D_n)})$ 
gives a resolution of $\CF$.
\end{lemma}

\begin{proof}
Let $(P,v)$ be an object of $(\fX/R)_{\prism}$.
Then, by applying Proposition \ref{prop:qprismEnvPL} (4) to the morphism 
$((P,[p]_qP)/P,\emptyset)\to ((B_P,J_{B_P})/P,\ut_P)$ of framed smooth $q$-pairs, 
we obtain a resolution 
$\CF(P,v)\longrightarrow (\CF(P,v)\otimes_Pq\Omega^{\bullet}_{D_P/P},\id\otimes\theta
^{\bullet}_{D_P}),$ which is nothing but the section over $(P,v)$ of 
$\CF\to \CF\otimes_{\CO_{\fX/R,n}}\CL_{A,B}(D_n\otimes_Dq\Omega^{\bullet}_{D/R})$
in question as $\CL_{A,B}(D_n\otimes_Dq\Omega^r_{D/R})$ $(r\in \N)$
and $\CF$ are crystals of $\CO_{\fX/R,n}$-modules. 
\end{proof}

Let $(P,v)$ be an object of $(\fX/R)_{\prism}$, and 
let $v_{D_P}$ be the $R$-morphism $\Spf(D_P/[p]_qD_P)
\to \Spf(P/[p]_qP)\to \fX=\Spf(A)$. Then we have morphisms
\begin{equation}\label{eq:ProductDPmorphi}
(P,v)\longleftarrow (D_P,v_{D_P})\longrightarrow (D,v_D)
\end{equation}
in $(\fX/R)_{\prism}$, which induce isomorphisms of $D_P$-modules 
\begin{equation}\label{eq:ResolLinCompMap}
\CF(P,v)\otimes_PD_P\xrightarrow{\;\cong\;}\CF(D_P,v_{D_P})
\xleftarrow{\;\cong\;} \CF(D,v_D)\otimes_D D_P=M\otimes_DD_P.
\end{equation}
Since the morphisms \eqref{eq:ProductDPmorphi} are functorial in $(P,v)$,
the isomorphisms \eqref{eq:ResolLinCompMap} are functorial in $(P,v)$.
Therefore, by taking $-\otimes_{D}q\Omega^r_{D/R}$ of \eqref{eq:ResolLinCompMap} 
and varying $(P,v)$
we obtain the following isomorphism of $\CL_{A,B}(D_n)$-modules for $r\in \N$.
\begin{equation}\label{eq:CrystalResCpxComp}
\CF\otimes_{\CO_{\fX/R,n}}\CL_{A,B}(D_n\otimes_Dq\Omega^r_{D/R})
\xrightarrow{\quad\cong\quad}\CL_{A,B}(M\otimes_Dq\Omega^r_{D/R})
\end{equation}
By construction, this is obviously functorial in $\CF$.

\begin{proposition}\label{prop:CrystalResolCpxComp}
The isomorphisms \eqref{eq:CrystalResCpxComp} are compatible
with the differential maps $\id_{\CF}\otimes\theta_{\CL(D_n)}^{\bullet}$
and $\theta_{\CL(M)}^{\bullet}$.
\end{proposition}

\begin{proof}
Let $(P,v)$ be an object of $(\fX/R)_{\prism}$. By the construction of 
the complex \eqref{eq:qdRCpxLinearization}, it suffices to
prove that the isomorphism \eqref{eq:ResolLinCompMap}
is compatible with $\id_{\CF(P,v)}\otimes\theta_{D_P,i}$
and $\theta_{M\otimes_DD_P,i}$ for each $i\in \Lambda$. 
Since both endomorphisms are $(t_i\mu,\theta_{D_P,i})$
connections on $D_P$-modules (Definition \ref{def:ConnectionTwDeriv}) and the latter 
one is compatible
with $\theta_{M,i}$ on $M$ by its definition, we are reduced
to showing that the composition $M\to M\otimes_DD_P
\xrightarrow[\eqref{eq:ResolLinCompMap}]{\cong} \CF(P,v)\otimes_PD_P$
is compatible with $\theta_{M,i}$ and $\id_{\CF(P,v)}\otimes\theta_{D_P,i}$.

We have a commutative diagram
\begin{equation}\label{eq:ResolLinComp}
\xymatrix{
\Spf(P/[p]_qP)\ar@{^(->}[r]\ar[d]^{v}&\Spf(B_P) \ar[r]\ar[d]& \Spf(P)\ar[d]\\
\fX=\Spf(A)\ar@{^(->}[r]&\Spf(B)\ar[r]&\Spf (R),
}
\end{equation}
where the left two horizontal maps are closed immersions,
the right two are smooth, and the right two vertical maps are
$\delta$-morphisms. We define a cosimplicial $\delta$-pair
$(B_P(\bullet),J_{B_P(\bullet)})$ over $(P,\pq P)$, a simplicial object
$(D_P(\bullet),w_{D_P(\bullet)})$ of $(\Spf(P/[p]_qP)/P)_{\prism}$, 
$\ut_P^{(r)}=(t_{P,l;i}^{(r)})_{(l,i)\in\Lambda(r)}$ $(r\in\N)$, 
and $q$-Higgs derivations $\theta_{D_P(r),l;i}$ $((l,i)\in \Lambda(r))$
of $D_P(r)$ over $P$ by applying the constructions
after Proposition \ref{prop:PrismCofFinObj} and 
Condition \ref{cond:qPrism}
to the upper line of \eqref{eq:ResolLinComp} and the $pP+[p]_qP$-adic 
coordinates $\ut_P$ of $B_P$ over $P$. 
The commutative diagram \eqref{eq:ResolLinComp} induces
a morphism of cosimplicial $\delta$-pairs
$(B(\bullet),J_{B(\bullet)})\to (B_{P}(\bullet),J_{B_P(\bullet)})$
over $(R,\pq R)$, and 
a morphism of simplicial objects
$(D_P(\bullet), v_{D_P(\bullet)})\to (D(\bullet),v_{D(\bullet)})$
of $(\fX/R)_{\prism}$, where we put $v_{D_P(\bullet)}=v\circ w_{D_P(\bullet)}$. 
Since $(D_P(\bullet),v_{D_P(\bullet)})$ is a simplicial object over
$(P,v)$ in $(\fX/R)_{\prism}$, we obtain the following commutative  diagram 
$$\xymatrix{
&p_{P}^*(\CF(P,v))\ar[dl]_{\cong}\ar[d]^{\cong}\ar[dr]^{\cong}&&
D_P\otimes_P\CF(P,v)\ar[ll]_{p_{P,1}\otimes\id_{\CF(P,v)}}\ar[d]^{\cong}
\\
p_{P,0}^*(\CF(D_P, v_{D_P}))\ar[r]^{\cong}&
\CF(D_P(1), v_{D_P(1)})&
p_{P,1}^*(\CF(D_P,v_{D_P}))\ar[l]_{\cong}&
\CF(D_P,v_{D_P})\ar[l]
\\
p_0^*(\CF(D,v_D))\ar[r]^{\cong}\ar[u]&
\CF(D(1),v_{D(1)})\ar[u]&
p_1^*(\CF(D,v_D))\ar[l]_(.47){\cong}\ar[u]&
\CF(D,v_D),\ar[l]\ar[u]
}$$
where $p_{P,\nu}$ $(\nu=0,1)$ denotes the
morphism $D_P=D_P(0)\to D_P(1)$ corresponding to the map $[0]\to [1]$
sending $0$ to $\nu$, and $p_P$ denotes the structure morphism 
$P\to D_P(1)$. The homomorphism 
$B(1)\to B_P(1)$ and $\id_{\Lambda(1)}$
(resp.~$p_{P,1}\colon B_P\to B_P(1)$ and 
$\Lambda\to \Lambda(1); i\mapsto (1,i)$) define
a morphism of framed smooth $q$-pairs
$((B(1),J_{B(1)})/R,\ut^{(1)})\to 
((B_P(1), J_{B_P(1)})/P,\ut_P^{(1)})$
(resp.~$((B_P,J_{B_P})/P,\ut_P)\to ((B_P(1),J_{B_P(1)})/P,\ut_P^{(1)})$).
Hence, by Proposition \ref{prop:qHiggsDerivFunct} (2) and 
Lemma \ref{lem:TwDerivSystemFunct} (1), the homomorphism $D(1)\to D_P(1)$
(resp.~$p_{P,1}\colon D_P\to D_P(1)$) is compatible
with $\theta_{D(1),1;i}$ (resp.~$\theta_{D_P,i}$)
and $\theta_{D_P(1),1;i}$ for every $i\in \Lambda$. 
By the definition of $\theta_{M,i}$ 
before Lemma \ref{lem:StratToConnection}, 
$\theta_{M,i}$ on $M=\CF(D,v_D)$ is compatible 
with $\theta_{D(1),1;i}\otimes \id$ on 
$p_0^*(\CF(D,v_D))$ under the composition of the bottom 
homomorphisms. Since the top horizontal map is injective
by Proposition \ref{prop:PrismEnvNervStr}, 
the above compatibility of $\theta_{D(1),1;i}$, $\theta_{D_P(1),1;i}$,
and $\theta_{D_P,i}$ implies that the right vertical composition 
$M=\CF(D,v_D)\to \CF(D_P,v_{D_P})\xleftarrow{\cong}D_P\otimes_P\CF(P,v)$
is compatible with $\theta_{M,i}$ and $\theta_{D_P,i}\otimes \id_{\CF(P,v)}$
for every $i\in \Lambda$.
\end{proof}
\begin{remark}\label{rmk:LinResolTransFrobComp}
(1) For $\varphi_n^*\CF$ as in Remark \ref{rmk:CrysFrobPBHiggs}, let 
$\sigma_{n,\CF}\colon \CF\to \varphi_n^*\CF$ be the
morphism defined by $\CF(P)\to \varphi_n^*\CF(P)
=\CF(P)\otimes_{P_n,\varphi_{P_n}}P_n;
x\mapsto x\otimes1$, $P\in \Ob(\fX/R)_{\prism}$. Then
the isomorphisms \eqref{eq:ResolLinCompMap} for $\CF$ and $\varphi_n^*\CF$
are compatible with
$\sigma_{n,\CF}(P,v)\otimes\varphi_{D_P}$,
$\sigma_{n,\CF}(D_P,v_{D_P})$, and
$\sigma_{n,\CF}(D,v_D)\otimes\varphi_{D_P}$.
Therefore the isomorphism \eqref{eq:CrystalResCpxComp} for $(\CF,M)$ and 
$(\varphi_n^*\CF,\varphi_{D_n}^*M)$ is compatible with
$\sigma_{n,\CF}\otimes\CL_{A,B}(\varphi^{r}_{D_n,\Omega}(D_n))$
and $\CL_{A,B}(\varphi^{r}_{D_n,\Omega}(M))$ \eqref{eq:LinqDelCpxFrobMap} which are constructed
from $\varphi^r_{D_n,\Omega}(D_n)$
$\varphi_{D_n,\Omega}^r(M)$ as mentioned before \eqref{eq:LinqDelCpxFrobMap}.\par
(2) Let $\CF$, $\CG$, $(M,\theta_M)$, $(N,\theta_N)$, and $(L=M\otimes_DN,\theta_L)$
be the same as in Remark \ref{rmk:qDolSheafCpxFrob} (2). Then,
for an object $(P,v)$ of $(\fX/R)_{\prism}$, the isomorphism \eqref{eq:ResolLinCompMap}
for $\CF\otimes_{\CO_{\fX/R,n}}\CG$ is given by the tensor product of 
those for $\CF$ and $\CG$. Since the isomorphism \eqref{eq:ResolLinCompMap}
for $\CG$ is compatible with $\id_{\CG(P,v)}\otimes\theta_{D_P,i}$
and $\theta_{N\otimes_DD_P,i}$ for each $i\in \Lambda$
by the proof of Proposition \ref{prop:CrystalResolCpxComp},
we obtain the following compatibility by going back to
the construction of \eqref{eq:LinqDolCpxProd} via \eqref{eq:dRcpxProd}
for $(M\otimes_DD_P,\theta_{M\otimes_DD_P})$ and
$(N\otimes_DD_P,\theta_{N\otimes_DD_P})$ given in Remark 
\ref{rmk:LinqdRCpxFrob} (2): The isomorphisms \eqref{eq:CrystalResCpxComp} for 
$\CF$, $\CG$, and $\CF\otimes_{\CO_{\fX/R,n}}\CG$ are
compatible with the product \eqref{eq:LinqDolCpxProd} and the morphism
\begin{multline*}
(\CF\otimes_{\CO_{\fX/R,n}}\CL_{A,B}(D_n\otimes_Dq\Omega^{\bullet}_{D/R}))
\otimes_{\CO_{\fX/R,n}}(\CG\otimes_{\CO_{\fX/R,n}}\CL_{A,B}(D_n\otimes_Dq\Omega^{\bullet}_{D/R}))\\
\longrightarrow
(\CF\otimes_{\CO_{\fX/R,n}}\CG)\otimes_{\CO_{\fX/R,n}}\CL_{A,B}(D_n\otimes_Dq\Omega^{\bullet}_{D/R})
\end{multline*}
defined by the product structure of the sheaf of 
differential graded algebra $\CL_{A,B}(D_n\otimes_D q\Omega^{\bullet}_{D/R})$
over $\CO_{\fX/R,n}$.
\end{remark}

By combining Lemma \ref{lem:CrystalPoincareLemma} with 
Proposition \ref{prop:CrystalResolCpxComp}, we obtain the desired
$\CO_{\fX/R,n}$-linear resolution 
\begin{equation}\label{eq:qdRResol}
\CF\longrightarrow (\CL_{A,B}(M\otimes_Dq\Omega^{\bullet}_{D/R}),\theta^{\bullet}_{\CL(M)})
\end{equation}
functorial  in $\CF$.

\begin{proof}[Proof of Theorem \ref{th:CrystalCohqHiggs}]
We have the following isomorphisms in $D^+(\fX_{\Zar},R)$
$$Ru_{\fX/R*}\CF\xrightarrow[\eqref{eq:qdRResol}]{\cong}
Ru_{\fX/R*}(\CL_{A,B}(M\otimes_Dq\Omega^{\bullet}_{D/R}),\theta^{\bullet}_{\CL(M)})
\xleftarrow{\cong} 
v_{D*}(\CM\otimes_{\CO_{\fD}}q\Omega^{\bullet}_{\fD/R},\theta_{\CM}^{\bullet}),$$
where the second isomorphism follows from \eqref{eq:qdRCpxLinZarProj}
and Proposition \ref{prop:LinearizationLocCoh}.
\end{proof}

\begin{remark}\label{rmk:CrysZarProjCompFrobProd}
(1) We follow the notation in Theorem \ref{th:CrystalCohqHiggs}.
Let $\varphi_n^*\CF$ denote the pullback of $\CF$ by the
lifting of Frobenius $\varphi_n$ of $\CO_{\fX/R,n}$ as in Remark \ref{rmk:CrysFrobPBHiggs},
and let $\sigma_{n,\CF}\colon \CF\to \varphi_n^*\CF$ be the
natural $\varphi_n$-semilinear morphism defined in Remark \ref{rmk:LinResolTransFrobComp}.
Then, by Remark \ref{rmk:LinResolTransFrobComp} (1) 
(resp.~Remark \ref{rmk:LinZarProjFrobComp} (1)),
the resolutions \eqref{eq:qdRResol} for $(\CF,(M,\theta_M))$ and $(\varphi_n^*\CF,(\varphi_{D_n}^*M,
\theta_{\varphi_{D_n}^*M}))$ (resp.~the morphisms of complexes defined by 
\eqref{eq:LinearizationZariskiProjMap} and \eqref{eq:qdRCpxLinZarProj} 
for $(M,\theta_M)$ and $(\varphi_{D_n}^*M,\theta_{\varphi_{D_n}^*M})$)
are compatible with $\sigma_{\CF,n}$ (resp.~$\varphi^{\bullet}_{\fD_n,\Omega}(\CM)$ 
\eqref{eq:qdRSheafCpxFrobPB}) and $\CL_{A,B}(\varphi_{D_n,\Omega}^{\bullet}(M))$
\eqref{eq:LinqDelCpxFrobMap}. Therefore, by the proof of Theorem \ref{th:CrystalCohqHiggs},
we see that the following diagram is commutative.
\begin{equation}
\xymatrix@C=50pt{
Ru_{\fX/R*}\CF\ar[r]^(.4){\cong}_(.4){\eqref{eq:CrystalCohqHiggs}}
\ar[d]_{Ru_{\fX/R*}(\sigma_{\CF,n})}&
v_{D*}(\CM\otimes_{\CO_{\fD}}q\Omega^{\bullet}_{\fD/R})
\ar[d]^{v_{D*}(\varphi^{\bullet}_{\fD_n,\Omega}(\CM))}\\
Ru_{\fX/R*}\varphi_n^*\CF\ar[r]^(.4){\cong}_(.4){\eqref{eq:CrystalCohqHiggs}}&
v_{D*}((\varphi_{\fD_n}^*\CM)\otimes_{\CO_{\fD}}q\Omega^{\bullet}_{\fD/R})
}
\end{equation}
(2) 
Let $\CF$ and $\CG$ be crystals of $\CO_{\fX/R,n}$-modules on 
$(\fX/R)_{\prism}$, and let $\CM\otimes_{\CO_{\fD}}q\Omega^{\bullet}_{\fD/R}$,
$\CN\otimes_{\CO_{\fD}}q\Omega^{\bullet}_{\fD/R}$, and 
$(\CM\otimes_{\CO_{\fD}}\CN)\otimes_{\CO_{\fD}}q\Omega^{\bullet}_{\fD/R}$ be as
in Remark \ref{rmk:qDolSheafCpxFrob} (2). We assert that the following
diagram is commutative.
\begin{equation}
\xymatrix@R=15pt{
Ru_{\fX/R*}\CF\otimes^L_RRu_{\fX/R*}\CG\ar[r]&
Ru_{\fX/R}(\CF\otimes_{\CO_{\fX/R,n}}\CG)\\
v_{D*}(\CM\otimes_{\CO_{\fD}}q\Omega^{\bullet}_{\fD/R})
\otimes^L_Rv_{D*}(\CN\otimes_{\CO_{\fD}}q\Omega^{\bullet}_{\fD/R})
\ar[u]_{\cong}^{\eqref{eq:CrystalCohqHiggs}\otimes^L_R\eqref{eq:CrystalCohqHiggs}}
\ar[r]&
v_{D*}((\CM\otimes_{\CO_{\fD}}\CN)
\otimes_{\CO_{\fD}}q\Omega^{\bullet}_{\fD/R}),
\ar[u]_{\cong}^{\eqref{eq:CrystalCohqHiggs}}
}
\end{equation}
where the bottom horizontal map is induced by \eqref{eq:qDolShfCpxProd}.
Let $(M,\theta_M)$ and $(N,\theta_N)$ be as in Remark \ref{rmk:qDolSheafCpxFrob} (2).
Then, by Remark \ref{rmk:LinResolTransFrobComp} (2), 
we see that the composition of the morphism
$\CF\otimes_{\CO_{\fX/R,n}}\CG
\to \CL_{A,B}(M\otimes_Dq\Omega^{\bullet}_{D/R})\otimes_{\CO_{\fX/R,n}}
\CL_{A,B}(N\otimes_Dq\Omega^{\bullet}_{D/R})$ 
defined by \eqref{eq:qdRResol} with the product \eqref{eq:LinqDolCpxProd}
coincides with the resolution \eqref{eq:qdRResol} for
the crystal $\CF\otimes_{\CO_{\fX/R,n}}\CG$. Therefore the compatibility
of \eqref{eq:qDolShfCpxProd} and \eqref{eq:LinqDolCpxProd} observed
in Remark \ref{rmk:LinZarProjFrobComp} (2)
implies the assertion above.
\end{remark}

We have the following analogue of Theorem \ref{th:CrystalCohqHiggs} for a 
complete crystal of $\CO_{\fX/R}$-modules (Definition \ref{def:PrismaticSite} (2)).

\begin{theorem}\label{thm:CrystalCohqHigProj}
Let $\CF$ be a complete crystal of $\CO_{X/R}$-modules
on $(\fX/R)_{\prism}$ (Definition \ref{def:PrismaticSite} (2)), 
and let $(\CF_n)_{n\in\N}$ be the corresponding
adic inverse system of crystals of $\CO_{X/R,n}$-modules
on $(\fX/R)_{\prism}$ (Remark \ref{rmk:CompleteCrystalInvSystem} (3), (4)).
Let $(\CM_{n}\otimes_{\CO_{\fD}}q\Omega^{\bullet}_{\fD/R})_{n\in\N}$
be the inverse system consisting of the $q$-Higgs complexes on $\fD_{\Zar}$
associated to $\CF_n$ $(n\in\N)$ \eqref{eq:qHiggsShfCpx}. 
Then, we have the following canonical isomorphism 
in $D^+(\fX_{\Zar},R)$ functorial in $\CF$.
See Remark \ref{rmk:PrisCohqdRFrobProd} below
for  compatibility with scalar extensions by Frobenius and with products.
\begin{equation}\label{eq:CrystalCohqHigProj}
Ru_{\fX/R*}\CF\cong v_{D*}
\bigl(\varprojlim_n \CM_n\otimes_{\CO_{\fD}}q\Omega_{\fD/R}^{\bullet}\bigr)
\end{equation}
\end{theorem}

\begin{proof} Let $(M_{n},\theta_{M_n})_{n\in\N}$ be the inverse
system of objects of $q\HIG_{\qnilp}(D_n/R)$ associated to $(\CF_n)_{n\in \N}$
by Proposition \ref{prop:CrysStratEquiv} and Theorem \ref{thm:StratHiggsEquiv}. 
The homomorphism $M_{n+1}\otimes_{D_{n+1}}D_n\to M_n$ is an 
isomorphism for every $n\in \N$ by Remark \ref{rmk:CompleteCrystalInvSystem} (1). 
The transition map of the inverse system
$(\CF_n)_{n\in \N}$ (resp.~$(\CL_{A,B}(M_{n}\otimes_{D}q\Omega^r_{D/R}))_{n\in\N}$)
are epimorphisms in the category of presheaves by 
Remark \ref{rmk:CompleteCrystalInvSystem} (1) (resp.~the construction of $\CL_{A,B}(-)$ and the above remark
on $(M_n)_{n\in\N})$. Hence the resolution of inverse systems of $\CO_{\fX/R,n}$-modules
$(\CF_n)_{n\in\N}\to (\CL_{A,B}(M_{n}\otimes_{D}q\Omega^{\bullet}_{D/R}))_{n\in\N}$
defined by \eqref{eq:qdRResol} induces a resolution 
$\CF\to \varprojlim _n\CL_{A,B}(M_{n}\otimes_{D}q\Omega^{\bullet}_{D/R})$.
Hence the claim follows from Proposition \ref{prop:LinLocCohProjlim} and \eqref{eq:qdRCpxLinZarProj}.
\end{proof}

\begin{remark}\label{rmk:PrisCohqdRFrobProd}
(1) We follow the notation and assumption in Theorem \ref{thm:CrystalCohqHigProj}.
Then, the morphisms $\sigma_{\CF_n,n}$, $\varphi^{\bullet}_{\fD_n,\Omega}(\CM_n)$
and $\CL_{A,B}(\varphi_{D_n,\Omega}^{\bullet}(M_n))$ for $n$ define morphisms
of inverse systems by \eqref{eq:qdRcpxFrobPBFunct}. 
Therefore, by the proof of Theorem \ref{thm:CrystalCohqHigProj},
the compatibility on these morphisms mentioned in Remark 
\ref{rmk:CrysZarProjCompFrobProd} (1) implies that the following 
diagram is commutative. Here $\hvarphi^*\CF$ is defined as in 
Remark \ref{rmk:crystalFrobPBTensor} (1) and is canonically isomorphic to $\varprojlim_n\varphi_n^*\CF_n$.
\begin{equation}
\xymatrix@C=50pt{
Ru_{\fX/R*}(\CF)\ar[r]^(.4){\cong}_(.4){\eqref{eq:CrystalCohqHigProj}}
\ar[d]_{Ru_{\fX/R*}(\varprojlim_n\sigma_{\CF_n,n})}&
v_{D*}(\varprojlim_n\CM_n\otimes_{\CO_{\fD}}q\Omega^{\bullet}_{\fD/R})
\ar[d]^{v_{D*}(\varprojlim_n\varphi^{\bullet}_{\fD_n,\Omega}(\CM_n))}\\
Ru_{\fX/R*}(\hvarphi^*\CF)\ar[r]^(.35){\cong}_(.35){\eqref{eq:CrystalCohqHigProj}}&
v_{D*}(\varprojlim_n(\varphi_{\fD_n}^*\CM_n)\otimes_{\CO_{\fD}}q\Omega^{\bullet}_{\fD/R})
}
\end{equation}
\par
(2) Let $\CF$ and $\CG$ be complete crystals of $\CO_{\fX/R}$-modules,
let  $\uCF=(\CF_n)_{n\in\N}$ and $\uCG=(\CG_n)_{n\in\N}$ be the 
adic inverse systems of crystals of $\CO_{\fX/R,n}$-modules on $(\fX/R)_{\prism}$ 
corresponding to $\CF$ and $\CG$, respectively, by the equivalence \eqref{eq:CompleteCrysInvSys}.
We abbreviate $\CF\hotimes_{\CO_{\fX/R}}\CG$ and $\uCF\otimes_{\CO_{\fX/R,\bullet}}\uCG$
(Remark \ref{rmk:crystalFrobPBTensor} (2)) to 
$\CF\hotimes\CG$ and $\uCF\otimes\uCG$, respectively.
For an inverse system of crystals of $\CO_{\fX/R,n}$-modules 
$\uCH=(\CH_n)_{n\in\N}$, we write 
$q\Omega(\uCH)$  (resp.~$\CL\Omega(\uCH)$) for the
inverse system of complexes on $\fD_{\Zar}$ (resp.~$(\fX/R)_{\prism}$)
associated to $\uCH$ as \eqref{eq:qHiggsShfCpx} 
(resp.~\eqref{eq:qdRCpxLinearization}). 
To simplify the notation we abbreviate $u_{\fX/R}$, $v_D$, $\CO_{\fX/R}$,
and $\varprojlim_n$ to $u$, $v$, $\CO$, and $\varprojlim$, respectively.
Then we see that the diagram 
\begin{equation}\label{eq:CrysCohqHigProjProd}
\xymatrix@R=20pt{
Ru_*\CF\otimes_R^LRu_*\CG\ar[r]
\ar[d]^{\cong}_{\eqref{eq:CrystalCohqHigProj}
\otimes^L_R\eqref{eq:CrystalCohqHigProj}}&
Ru_*(\CF\otimes_{\CO}\CG)\ar[r]&
Ru_*(\CF\hotimes\CG)
\ar[d]^{\eqref{eq:CrystalCohqHigProj}}_{\cong}\\
v_*\varprojlim q\Omega(\uCF)\otimes^L_R v_*\varprojlim q\Omega(\uCG)\ar[r]&
v_*\varprojlim(q\Omega(\uCF)\otimes_Rq\Omega(\uCG))\ar[r]^(.58){\eqref{eq:qDolShfCpxProd}}&
v_*\varprojlim q\Omega(\uCF\otimes\uCG)
}
\end{equation}
is commutative as follows. By Remark 
\ref{rmk:LinResolTransFrobComp} (2) and Remark \ref{rmk:LinZarProjFrobComp} (2),
we have the following commutative diagrams, where the vertical morphisms
in the second one are defined by \eqref{eq:LinearizationZariskiProjMap} and 
\eqref{eq:qdRCpxLinZarProj}.
\begin{equation}
\xymatrix@R=20pt{
\CF\otimes_{\CO}\CG\ar[r]
\ar[d]_{\varprojlim\eqref{eq:qdRResol}\otimes_{\CO}\varprojlim\eqref{eq:qdRResol}}&
\CF\hotimes \CG
\ar[r]^(.45){\varprojlim \eqref{eq:qdRResol}}
\ar[d]_{\varprojlim\eqref{eq:qdRResol}\otimes_{\CO}\eqref{eq:qdRResol}}
&
\varprojlim \CL\Omega(\uCF\otimes\uCG)\\
\varprojlim \CL\Omega(\uCF)\otimes_{\CO}
\varprojlim \CL\Omega(\uCG)\ar[r]&
\varprojlim (\CL\Omega(\uCF)\otimes_{\CO}\CL\Omega(\uCG))
\ar[ur]_(.55){\varprojlim\eqref{eq:LinqDolCpxProd}}&
}
\end{equation}
\begin{equation}
\xymatrix@R=20pt@C=20pt{
\varprojlim u_*\CL\Omega(\uCF)\otimes^L_R\varprojlim u_*\CL\Omega(\uCG)
\ar[r]&
\varprojlim (u_*\CL\Omega(\uCF)\otimes_Ru_*\CL\Omega(\uCG))
\ar[r]^(.57){\eqref{eq:LinqDolCpxProd}}&
\varprojlim u_*\CL\Omega(\uCF\otimes\uCG)\\
\varprojlim v_*q\Omega(\uCF)\otimes^L_{R} \varprojlim v_*q\Omega(\uCG)
\ar[r]\ar[u]_{\cong}&
\varprojlim(v_*q\Omega(\uCF)\otimes_Rv_*q\Omega(\uCG))
\ar[r]^(.57){\eqref{eq:qDolShfCpxProd}}\ar[u]&
\varprojlim v_*q\Omega(\uCF\otimes\uCG)
\ar[u]_{\cong}
}
\end{equation}
We obtain the commutative diagram \eqref{eq:CrysCohqHigProjProd}
by combining $Ru_{*}$ of the upper diagram 
and the the lower one  with the horizontal compositions
replaced by the alternative ones given by Lemma \ref{lem:ToposDirectImVarProjProd}
below.
\end{remark}

\begin{lemma}\label{lem:ToposDirectImVarProjProd}
Let $(f^*,f_*)\colon (E',R')\to (E,R)$ be a morphism of ringed topos,
and let $(K_n^{\bullet})_{n\in\N}$ and $(L_n^{\bullet})_{n\in\N}$
be inverse systems of  complexes of $R'$-modules on $E'$.
Then the following diagram is commutative.
\begin{equation*}
\xymatrix@R=15pt{
\varprojlim_nf_*K^{\bullet}_n\otimes_R\varprojlim_nf_*L_n^{\bullet}\ar[r]\ar@{=}[d]&
\varprojlim_n(f_*K^{\bullet}_n\otimes_Rf_*L^{\bullet})\ar[r]&
\varprojlim_nf_*(K^{\bullet}_n\otimes_{R'}L^{\bullet}_n)\ar@{=}[d]\\
f_*\varprojlim_nK^{\bullet}_n\otimes_R f_*\varprojlim_nL_n^{\bullet}\ar[r]&
f_*(\varprojlim_nK_n^{\bullet}\otimes_{R'}\varprojlim_nL_n^{\bullet})\ar[r]&
f_*\varprojlim_n(K_n^{\bullet}\otimes_{R'}L_n^{\bullet})
}
\end{equation*}
\end{lemma}

\begin{proof}
For both upper and lower horizontal maps, we see
that their compositions with  the projection 
$\varprojlim_nf_*(K_n^{\bullet}\otimes_{R'} L_n^{\bullet})
\to f_*(K_n^{\bullet}\otimes_{R'}L_n^{\bullet})$ coincide
with the composition of the projection 
to $f_*K_n^{\bullet}\otimes_Rf_*L_n^{\bullet}$
with the morphism 
$f_*K_n^{\bullet}\otimes_Rf_*L_n^{\bullet}
\to f_*(K_n^{\bullet}\otimes_{R'}L_n^{\bullet})$.
\end{proof}

\section{Prismatic cohomology and $q$-Higgs complex: 
Functoriality}\label{sec:PrismCohqDolbFunct}

We prove the functoriality of the comparison isomorphisms \eqref{eq:CrystalCohqHiggs} and 
\eqref{eq:CrystalCohqHigProj} (Propositions \ref{prop:PrismZarProjqDRFunct} and 
\ref{prop:CrystalCohqHigProjFunct}) with respect to 
$R$,  $\fX\hookrightarrow \Spf(B)$,
and $\ut=(t_i)_{i\in \Lambda}$
by showing the functoriality \eqref{eq:LineDRcpxPB} and \eqref{eq:qDolShfCpxFunct} 
of \eqref{eq:qdRCpxLinearization} and \eqref{eq:qHiggsShfCpx}, respectively,
and its compatibility with 
\eqref{eq:qdRCpxLinZarProj}, \eqref{eq:CrysPLAugMap}, and \eqref{eq:CrystalResCpxComp}
(Proposition \ref{eq:LindRCpxZarProjFunct}, \eqref{eq:CrystaldRResolFunct}, and 
Lemma \ref{lem:LinZarqDRProjFunct}). 
After each step in our arguments, we discuss relevant compatibility 
with scalar extension under the lifting of Frobenius and with tensor products
as remarks.

Let $(R',I')$, $i'\colon \fX'=\Spf(A')\hookrightarrow \fY'=\Spf(B')$, $f\colon (R,I)\to (R',I')$, $g\colon \fX'\to \fX$,
$h\colon \fY'\to \fY$, $(D',v_{D'})$, and $h_D\colon (D',g\circ v_{D'})\to (D,v_D)$
be the same as before \eqref{eq:StratPullback}. Note that $(R', I')$ is also a $q$-prism. 
As in the paragraph before Proposition \ref{prop:FunctStratConn}, we assume 
that we are given $pR'+[p]_qR'$-adic coordinates $\ut'=(t'_{i'})_{i'\in\Lambda'}$, 
$\Lambda'=\N\cap [1,d']$ of $B'$ over $R'$ and a map $\psi\colon \Lambda\to\Lambda'$
of ordered sets such that $\delta(t'_{i'})=0$ $(i'\in \Lambda')$ and $h^*(t_i)=t'_{\psi(i)}$ $(i\in \Lambda)$. 

We first discuss the functoriality of the complex \eqref{eq:qdRCpxLinearization}.
Let $\CF$ be a crystal of $\CO_{\fX/R,n}$-modules, and let $\CF'$ be the inverse image
$g^{-1}_{\prism}\CF$ of $\CF$ under the morphism of topos $g_{\prism}\colon 
(\fX'/R')^{\sim}_{\prism}\to (\fX/R)^{\sim}_{\prism}$ induced by $g$ and $f$
(Definition \ref{def:PrismaticSite} (3)). Let $(M,\theta_M)$ (resp. $(M',\theta_{M'})$)
be the object of $q\HIG_{\qnilp}(D_n/R)$ (resp.~$q\HIG_{\qnilp}(D'_n/R'))$
associated to the crystal $\CF$ (resp.~$\CF'$) by Proposition \ref{prop:CrysStratEquiv} 
and Theorem \ref{thm:StratHiggsEquiv}. By \eqref{eq:FunctCrystalStrat}
and Proposition \ref{prop:FunctStratConn}, $(M',\theta_{M'})$
is canonical isomorphic to the image of $(M,\theta_M)$ under
the scalar extension functor $h_D^*\colon 
q\HIG(D_n/R)\to q\HIG(D_n'/R')$ constructed before Proposition 
\ref{prop:FunctStratConn}. Then we can construct a morphism of 
complexes of crystals of $\CO_{\fX'/R',n}$-modules functorial in $\CF$
and compatible with compositions of $(f,g,h)$ and $\psi$
\begin{equation}\label{eq:LineDRcpxPB}
g_{\prism}^{-1}(\CL_{A,B}(M\otimes_Dq\Omega^{\bullet}_{D/R}))
\longrightarrow \CL_{A',B'}(M'\otimes_{D'}q\Omega^{\bullet}_{D'/R'})
\end{equation}
as follows. \par

Let $(P,w)$ be an object of $(\fX'/R')_{\prism}$,
and define the $\delta$-pair $(B'_P,J_{B'_P})$ (resp.~$(B_P,J_{B_P})$)
and $q$-prisms $D'_P$   (resp.~$D_P$) by applying the construction
in the paragraph before Lemma \ref{lem:LinearizationEnvBC} 
to $(P,w)$ and $\fX'\to \Spf(B')$ over $R'$
(resp.~$(P,g\circ w)$ and $\fX\to \Spf(B)$ over $R$). Let
$\ut_{P}=(t_{P,i})_{i\in \Lambda}\in B_P^{\Lambda}$ 
(resp.~$\ut_{P'}=(t_{P',i'}')_{i'\in\Lambda'}\in B_{P}^{\prime\Lambda'}$)
be the image of $\ut=(t_i)_{i\in\Lambda}$ (resp.~$\ut'=(t'_{i'})_{i'\in\Lambda'}$).
Then we have a commutative diagram of framed smooth $q$-pairs
whose horizontal maps are defined by using $(f, g,h)$ and $\psi$
\begin{equation}\label{eq:LinFramedqPairMaps}
\xymatrix{
((B_P,J_{B_P})/P,\ut_P)\ar[r]&
((B'_P,J_{B'_P})/P,\ut'_{P'})\\
((B,J)/R,\ut)\ar[u]\ar[r]&
((B',J')/R',\ut').\ar[u]
}
\end{equation}
Hence,  by \eqref{eq:qPrismEnvHigPB} and \eqref{eq:qDolbCpxPB2}, and 
their compatibility with compositions: Lemma \ref{lem:dRCpxFuncCocyc} (2), 
we obtain a commutative diagram of complexes
\begin{equation}\label{eq:LineDRcpxFunctDiag}
\xymatrix{
(M\otimes_D{D_P})\otimes_{D_P}q\Omega^{\bullet}_{D_P/P}
\ar[r]&
(M'\otimes_{D'} D'_{P})\otimes_{D'_P}q\Omega^{\bullet}_{D'_P/P}\\
M\otimes_Dq\Omega^{\bullet}_{D/R}\ar[r]\ar[u]&
M'\otimes_{D'}q\Omega^{\bullet}_{D'/R'}.\ar[u]
}
\end{equation}
The diagram \eqref{eq:LinFramedqPairMaps} is obviously 
functorial in $(P,w)$. Hence we see that the upper horizontal map in 
\eqref{eq:LineDRcpxFunctDiag} is also functorial in $(P,w)$
and defines the desired morphism \eqref{eq:LineDRcpxPB}, whose
functoriality in $\CF$ follows from that of \eqref{eq:LineDRcpxFunctDiag}.
The diagram \eqref{eq:LinFramedqPairMaps} satisfies the obvious cocycle 
condition for compositions 
of $(f,g,h)$ and $\psi$, which implies that of \eqref{eq:LineDRcpxFunctDiag} 
and then of \eqref{eq:LineDRcpxPB}. 

\begin{remark}\label{rmk:LinPBFrobComp}
(1) 
By Remark \ref{rmk:FramedSmQPHigFPBComp}, the
Frobenius pullback $(\varphi_{D_n'}^*M',\theta_{\varphi_{D_n'}^*M'})$
\eqref{eq:FramedSmQPHigFPBFunct} of $(M',\theta_{M'})$
is canonically isomorphic to the scalar extension of the Frobenius pullback
$(\varphi_{D_n}^*M,\theta_{\varphi_{D_n}^*M})$ of $(M,\theta_M)$
by the homomorphism $h_D$. By the commutative diagram
\eqref{eq:LinFramedqPairMaps} and Remark \ref{rmk:FramedSmQPHigFPBComp},
we see that the upper horizontal morphisms in \eqref{eq:LineDRcpxFunctDiag}
for $(M,\theta_M)$ and $(\varphi_{D_n}^*M,\theta_{\varphi_{D_n}^*M})$ is
compatible with $\varphi_{D_{P,n},\Omega}^{\bullet}(M\otimes_DD_P)$
and $\varphi^{\bullet}_{D'_{P,n},\Omega}(M'\otimes_{D'}D'_P)$ \eqref{eq:LinqDelCpxSecFrobMap}.
Varying $(P,w)\in \Ob (\fX'/R')_{\prism}$, we see that the morphisms
\eqref{eq:LineDRcpxPB} for $(M,\theta_M)$ and $(\varphi_{D_n}^*M,\theta_{\varphi_{D_n}^*M})$
is compatible with $g_{\prism}^{-1}(\CL_{A,B}(\varphi_{D_n,\Omega}^{\bullet}(M)))$
and $\CL_{A,B}(\varphi^{\bullet}_{D_n',\Omega}(M'))$ \eqref{eq:LinqDelCpxFrobMap}.\par
(2) Suppose that $\psi$ is injective. Let $\CF$ and $\CG$ be
crystals of $\CO_{\fX/R,n}$-modules on $(\fX/R)_{\prism}$. 
Then we see that the morphisms \eqref{eq:LineDRcpxPB} for 
$\CF$, $\CG$, and $\CF\otimes_{\CO_{\fX/R,n}}\CG$ are
compatible with the products \eqref{eq:LinqDolCpxProd}
for $\CF$, $\CG$ and for $g_{\prism}^{-1}(\CF)$ and $g_{\prism}^{-1}(\CG)$
as follows. Let $(M,\theta_M)$ and $(N,\theta_N)$
(resp.~$(M',\theta_{M'})$ and $(N',\theta_{N'})$) be the objects
 of $q\HIG_{\qnilp}(D_n/R)$ (resp.~$q\HIG_{\qnilp}(D_n'/R')$)
 associated to $\CF$ and $\CG$ (resp.~$g_{\prism}^{-1}(\CF)$
 and $g_{\prism}^{-1}(\CG)$). We follow the notation in the
 construction of \eqref{eq:LineDRcpxPB}. Then,
 as $\psi$ is assumed to be injective, 
Remark \ref{rmk:dRCpxProdScExtComp} implies
that  the product  morphism \eqref{eq:dRcpxProd} for the scalar extensions of
 $(M,\theta_M)$ and $(N,\theta_N)$ under 
 $D\to D_P$ and that of the scalar extensions
 of $(M',\theta_{M'})$ and $(N',\theta_{N'})$ under
 $D'\to D'_{P}$ are compatible with the upper horizontal maps
 in \eqref{eq:LineDRcpxFunctDiag} for $\CF$, $\CG$,
 and $\CF\otimes_{\CO_{\fX/R,n}}\CG$.
 Varying $(P,w)\in \Ob(\fX'/R')_{\prism}$, we
 obtain the desired compatibility.
 \end{remark}
\begin{remark}\label{rmk:LinPBProdComp2}
As in Remark \ref{rmk:LinPBFrobComp} (2),
the  morphism \eqref{eq:LineDRcpxPB} is compatible
with the product \eqref{eq:LinqDolCpxProd} only when 
$\psi$ is injective. However, thanks to Proposition \ref{prop:ProddRCpxFunct},
one can show compatibility for any $\psi$ by considering 
a product whose codomain is the linearization of the
$q$-Higgs complex associated to $\CF\otimes_{\CO_{\fX/R,n}}\CG$
with respect to the diagonal closed immersion $\fX\to \fY\times_{\Spf(R)}\fY$
as follows.\par
(1) We define $\fY(1)=\Spf(B(1))=\fY\times_{\Spf(R)}\fY$, 
the $\delta$-structure on $B(1)$, the closed immersion 
$i(1)\colon \fX\to \fY(1)$, the ideal $J_{B(1)}=\Ker(i(1)^*)$ of $B(1)$,
and the bounded prismatic envelope $(D(1),\pq D(1))$ of
$(B(1),J_{B(1)})$ over $(R,\pq R)$ as in the paragraph after
Proposition \ref{prop:PrismCofFinObj}. Then we follow the notation
$\Lambda(1)=\{0,1\}\times \Lambda$, $\ut^{(1)}=(t^{(1)}_{l;i})_{(l,i)\in \Lambda(1)}
\in B(1)^{\Lambda(1)}$, $\theta_{D(1),l;i}$ $((l,i)\in\Lambda(1))$,
$\theta_{D(1)}\colon D(1)\to q\Omega_{D(1)/R}$, and $q\HIG(D(1)_n/R)$ 
introduced in the paragraph after Condition \ref{cond:qPrism}. For $l\in \{0,1\}$,
let $p_{l,\fY}\colon \fY(1)\to \fY$ be the projection to the
$(l+1)$th component, which satisfies $p_{l,\fY}\circ i(1)=i$,
write $p_{l,B}$ for the morphism of $\delta$-pairs
$p_{l,\fY}^*\colon (B,J)\to (B(1),J_{B(1)})$ induced by $p_{l,\fY}$,
and let $p_{l,D}$ denote the $\delta$-homomorphism 
$D\to D(1)$ induced by $p_{l,B}$. We define a map of ordered
sets $\chi_{l}\colon \Lambda\to \Lambda(1)$
for $l\in\{0,1\}$ by $\chi_{l}(i)=(l,i)$, which satisfies
$p_{l,B}(t_i)=t^{(1)}_{l;i}$ for $i\in \Lambda$ by
the definition of $\ut^{(1)}$. Then $(\id_R,\id_{\fX},p_{l,\fY})$ and
$\chi_{l}$ induce a scalar extension functor
$p_{l,D}^*\colon q\HIG(D_n/R)\to q\HIG(D(1)_n/R)$
(Proposition \ref{prop:qHiggsDerivFunct} (2)).\par

Let $\CF_{l}$ $(l=0,1)$ be crystals of $\CO_{\fX/R,n}$-modules
on $(\fX/R)_{\prism}$, and let $(M_{l},\theta_{M_{l}})$ 
(resp.~$(M_{l}(1),\theta_{M_{l}(1)})$) be the object
of $q\HIG_{\qnilp}(D_n/R)$ (resp.~$q\HIG_{\qnilp}(D(1)_n/R)$)
associated to $\CF_{l}$ by Proposition \ref{prop:CrysStratEquiv} 
and Theorem \ref{thm:StratHiggsEquiv} applied to 
$(i\colon \fX\to \fY,\ut)$ (resp.~$(i(1)\colon \fX\to \fY(1)$).  Then, by applying 
\eqref{eq:LineDRcpxPB} to $\CF_{l}$, 
$(\id_R,\id_{\fX},p_{l,\fY})$ and $\chi_{l}$, we obtain an
$\CO_{\fX/R,n}$-linear morphism
\begin{equation}\label{eq:LinProdPB}
\CL\Omegab(p_{l,D}^*)\colon
\CL_{A,B}(M_{l}\otimes_Dq\Omega^{\bullet}_{D/R})
\longrightarrow 
\CL_{A,B(1)}(M_{l}(1)\otimes_{D(1)}q\Omega^{\bullet}_{D(1)/R}).
\end{equation}

Let $(M(1),\theta_{M(1)})$ be the object of $q\HIG_{\qnilp}(D(1)_n/R)$
associated to the crystal $\CF=\CF_0\otimes_{\CO_{\fX/R,n}}\CF_1$,
which is the tensor product of $(M_{l}(1),\theta_{M_{l}(1)})$
$(l=0,1)$ (Remark \ref{rmk:CrysFrobPBHiggs} (2)).
Then writing $\CL\Omega^{\bullet}(M_{l})$
(resp.~$\CL\Omega^{\bullet}(M_{l}(1))$,
resp.~$\CL\Omega^{\bullet}(M(1))$)
for the domain of \eqref{eq:LinProdPB},
(resp.~the codomain of \eqref{eq:LinProdPB}, 
resp.~$\CL_{A,B(1)}(M(1)\otimes_{D(1)}q\Omega^{\bullet}_{D(1)/R})$),
we consider the following composition.
\begin{multline}\label{eq:LinqdRCpxMdfProd}
\CL\Omega^{\bullet}(M_0)\otimes_{\CO_{\fX/R,n}}
\CL\Omega^{\bullet}(M_1)
\xrightarrow{\CL\Omegab(p_{0,D}^*)\otimes\CL\Omegab(p_{1,D}^*)}
\CL\Omega^{\bullet}(M_0(1))\otimes_{\CO_{\fX/R,n}}
\CL\Omega^{\bullet}(M_1(1))\\
\xrightarrow{\eqref{eq:LinqDolCpxProd}}
\CL\Omega^{\bullet}(M(1))
\end{multline}
(2) We keep the notation in (1). We define
$\fY'(1)$, $i'(1)$, $B'(1)$, $J_{B'(1)}$, $\theta_{D'(1),l;i'}$,
$D'(1)$, $p_{\fY',l}$, $p_{B',l}$, $p_{D',l}$, and $\chi'_{l}$
in the same way as (1) by using $i'\colon \fX'\to \fY'/R'$ and $\ut'$.
We define a morphism $h(1)\colon \fY'(1)\to \fY(1)$ over $f$ to be the product
of $h$, and a morphism of ordered sets $\psi(1)\colon \Lambda(1)
\to \Lambda'(1)$ by $(l,i)\mapsto (l,\psi(i))$.
Put $\CF'_{l}=g_{\prism}^{-1}\CF_{l}$ and
$\CF'=\CF_{0}'\otimes_{\CO_{\fX'/R',n}}\CF_1'=g_{\prism}^{-1}(\CF)$,
and let $(M'_{l},\theta_{M_{l}'})\in \Ob q\HIG_{\qnilp}(D'_n/R)$
and $(M'_{l}(1),\theta_{M'_{l}(1)})$, $(M'(1),\theta_{M'(1)})\in 
\Ob q\HIG_{\qnilp}(D'(1)_n/R)$ be the objects associated to 
$\CF'_{l}$ and $\CF'$ similarly  to (1). Then the 
following diagram is commutative.
\begin{equation}\label{eq:LinqDolMdfProdFunct}
\xymatrix@C=50pt@R=20pt{
g_{\prism}^{-1}(\CL\Omegab(M_0)\otimes_{\CO_{\fX/R,n}}
\CL\Omegab(M_1))\ar[r]^(.6){g_{\prism}^{-1}\eqref{eq:LinqdRCpxMdfProd}}\ar[d]&
g_{\prism}^{-1}(\CL\Omegab(M(1)))\ar[d]\\
\CL\Omegab(M_0')\otimes_{\CO_{\fX'/R',n}}\CL\Omegab(M'_1)
\ar[r]^(.6){\eqref{eq:LinqdRCpxMdfProd}}&
\CL\Omegab(M'(1)).
}
\end{equation}
Here the left (resp.~right) vertical morphism is induced by 
\eqref{eq:LineDRcpxPB} for $\CF_{l}$, $(f,g,h)$ and $\psi$
(resp.~$\CF$, $(f,g,h(1))$, and $\psi(1)$). \par

We can verify the commutativity
above as follows. Let $(P,w)$ be an object of $(\fX'/R')_{\prism}$. Then
we obtain commutative diagrams of framed smooth $q$-pairs
with the left one lying over the right one
\begin{equation*}
\xymatrix@C=15pt@R=20pt
{((B(1)_P,J_{B(1)_P})/P,\ut_P^{(1)})\ar[d]
&((B_P,J_{B_P})/P,\ut_P)\ar[d]\ar@<0.5ex>[l]\ar@<-0.5ex>[l]\\
((B'(1)_P,J_{B'(1)_P})/P,\ut_P^{\prime(1)})
&((B'_P,J_{B'_P})/P, \ut'_P)\ar@<0.5ex>[l]\ar@<-0.5ex>[l]
}
\xymatrix@C=15pt@R=20pt
{((B(1),J_{B(1)})/R,\ut^{(1)})\ar[d]
&((B,J)/R,\ut)\ar[d]\ar@<0.5ex>[l]\ar@<-0.5ex>[l]\\
((B'(1),J_{B'(1)})/R',\ut^{\prime(1)})
&((B',J')/R', \ut')\ar@<0.5ex>[l]\ar@<-0.5ex>[l]
}
\end{equation*}
by applying the construction of \eqref{eq:LinFramedqPairMaps} 
to the following data.
\begin{equation}\label{eq:FramedSmEmbProdMorph}
\xymatrix@C=100pt@R=20pt{
(i(1)\colon \fX\to \fY(1)/R,\ut^{(1)})
\ar@<0.5ex>[r]^{(\id_{R},\id_{\fX},p_{\fY,l}),\chi_{l}}
\ar@<-0.5ex>[r]
& (i\colon\fX\to \fY/R,\ut)\\
(i'(1)\colon\fX'\to \fY'(1)/R',\ut^{\prime (1)})
\ar[u]^{(f,g,h(1)),\psi(1)}
\ar@<0.5ex>[r]^{(\id_{R'},\id_{\fX'},p_{\fY',l}),\chi'_{l}}
\ar@<-0.5ex>[r]
&(i'\colon\fX'\to\fY'/R',\ut')
\ar[u]_{(f,g,h),\psi}
}
\end{equation}
By applying Proposition \ref{prop:qHiggsDerivFunct} (2) to the above  diagrams of
framed smooth $q$-pairs, 
we obtain the following commutative diagrams in the category
$\DAlg$ introduced before Proposition \ref{prop:ProddRCpxFunct}
with the upper one lying over the lower one.
\begin{gather}
\xymatrix@R=10pt{
\uD_P=(D_P/P,\Lambda,\mu\ut_P,\utheta_{D_P})\ar[d]
\ar@<0.5ex>[r]\ar@<-0.5ex>[r]&
\underline{D(1)}_P=(D(1)_P/P,\Lambda(1),\mu\ut_P^{(1)},\utheta_{D(1)_P})\ar[d]\\
\uD'_P=(D'_P/P,\Lambda',\mu\ut'_P,\utheta_{D'_P})
\ar@<0.5ex>[r]\ar@<-0.5ex>[r]&
\underline{D'(1)}_P=(D'(1)_P/P,\Lambda'(1),\mu\ut_P^{\prime(1)},\utheta_{D'(1)_P})
}\\
\xymatrix@R=10pt{
\uD=(D/R,\Lambda,\mu\ut,\utheta_{D})\ar[d]
\ar@<0.5ex>[r]\ar@<-0.5ex>[r]
&
\underline{D(1)}=(D(1)/R,\Lambda(1),\mu\ut^{(1)},\utheta_{D(1)})\ar[d]\\
\uD'=(D'/R',\Lambda',\mu\ut',\utheta_{D'})
\ar@<0.5ex>[r]\ar@<-0.5ex>[r]
&
\underline{D'(1)}=(D'(1)/R',\Lambda'(1),\mu\ut^{\prime(1)},\utheta_{D'(1)})
}\label{eq:FrSmEmbEnvProdMorph}
\end{gather}
By applying Proposition \ref{prop:ProddRCpxFunct} to the pullbacks of 
$(M_{l},\theta_{M_{l}})\in \Ob\MIC(\uD)$ to $\MIC(\uD_P)$ and
the upper diagram, we see that the section of \eqref{eq:LinqDolMdfProdFunct} over
$(P,w)$ is commutative.
\end{remark}

One can construct a homomorphism 
\begin{equation}\label{eq:qDolShfCpxFunct}
\fh_D^{-1}(\CM\otimes_{\CO_{\fD}}q\Omega^{\bullet}_{\fD/R})
\longrightarrow \CM'\otimes_{\CO_{\fD'}}q\Omega^{\bullet}_{\fD'/R'}
\end{equation}
functorial in $\CF$ and satisfying the cocycle condition for compositions
of $(f,g,h)$ and $\psi$ as follows. 
Here $\fD'=\Spf(D')$, $\fD'_n=\Spec(D'_n)$,  $\fh_D$
is the morphism $\fD'\to \fD$ defined by $h_D\colon D\to D'$, 
$\CM'$ denotes 
the quasi-coherent $\CO_{\fD'_n}$-module on $\fD'_n$ associated to $M'$,
and the target is the complex on $\fD'_{\Zar}$ associated to $(M', \theta_{M'})$. 

By taking the adjoint, it suffices to construct 
\begin{equation}\label{eq:qDolShfCpxFunct2}
\CM\otimes_{\CO_{\fD}}q\Omega^{\bullet}_{\fD/R}
\to \fh_{D*}(\CM'\otimes_{\CO_{\fD'}}q\Omega^{\bullet}_{\fD'/R'}),
\end{equation}
i.e., a morphism 
\begin{equation}\label{eq:qDolShfCpxFunctLoc}
(M\otimes_D\tD)\otimes_{\tD}q\Omega^{\bullet}_{\tD/R}
\longrightarrow (M'\otimes_{D'}\tD')\otimes_{\tD'}q\Omega^{\bullet}_{\tD'/R'}
\end{equation}
for an open affine formal subscheme $\Spf(\tD)$ of $\fD$ and its inverse image
$\Spf(\tD')$ in $\fD'$, in a way functorial in $\tD$. 
Put $\ualpha=(t_i\mu)_{i\in \Lambda}$ and $\ualpha'=(t'_{i'}\mu)_{i'\in\Lambda'}$.
Let $h_{\tD}$ be the homomorphism $\tD\to \tD'$ induced by $h_D$. 
As we have seen in  the construction of the complex $\CM\otimes_{\CO_{\fD}}q\Omega^{\bullet}_{\fD/R}$, 
the $t_i\mu$-derivation $\theta_{D,i}$ 
$(i\in \Lambda$) and the $t'_{i'}\mu$-derivation $\theta_{D',i'}$ $(i'\in \Lambda')$
have unique extensions  $\theta_{\tD,i}$ on $\tD$ and $\theta_{\tD',i'}$ on $\tD'$, 
respectively,   and they are functorial in $\tD$. 
Therefore,  by Definition \ref{def:connectionScalarExt} and Proposition \ref{prop:dRCpxFunct},
it suffices to show that the homomorphism 
$E^{\psi}(h_{\tD})\colon E^{\ualpha}(\tD)\to E^{\ualpha'}(\tD')$
\eqref{eq:TwExtAlgMap}
is a right algebra homomorphism over $h_{\tD}$ 
\eqref{eq:ConnFunctCond}, i.e., the following diagram is commutative.
\begin{equation}\label{eq:FunctZarLocDerivComp}
\xymatrix{
E^{\ualpha}(\tD)\ar[r]& E^{\ualpha'}(\tD')\\
\tD\ar[u]^{\ts}\ar[r]& \tD'\ar[u]^{\ts'},
}
\end{equation}
where $\ts$ and $\ts'$ denote the right algebra structures
\eqref{eq:ExtensionRightAlgStr} defined by $(t_i\mu,\theta_{\tD,i})$ $(i\in \Lambda)$ and
$(t'_{i'}\mu,\theta_{\tD',i'})$ $(i'\in \Lambda')$, respectively.
The composition with $D\to \tD$ which is $pR+[p]_qR$-adically \'etale 
is commutative because the homomorphism $E^{\psi}(h_{D})\colon E^{\ualpha}(D)\to E^{\ualpha'}(D')$ 
is a right algebra homomorphism
over $h_{D}$ by Proposition \ref{prop:qHiggsDerivFunct} (2)
applied to the morphism of framed smooth $q$-pairs
$((B,J)/R,\ut)\to ((B',J')/R',\ut')$ defined by $f$, $h$ and $\psi$.  
Since the composition with $\pi\colon E^{\ualpha'}(\tD')\to \tD'$
whose kernel is $pR+[p]_qR$-adically nilpotent is also commutative,
we see that the diagram \eqref{eq:FunctZarLocDerivComp}  itself is commutative. 

The functoriality of \eqref{eq:qDolShfCpxFunct} and \eqref{eq:qDolShfCpxFunct2} in $\CF$ follows from that
of \eqref{eq:qDolShfCpxFunctLoc}. By Lemma \ref{lem:dRCpxFuncCocyc} (2), 
the morphism \eqref{eq:qDolShfCpxFunctLoc} satisfies the
cocycle condition for compositions of $(f,g,h)$ and $\psi$. Varying $\Spf(\tD)\subset\fD$,
we see that it also holds for \eqref{eq:qDolShfCpxFunct} and \eqref{eq:qDolShfCpxFunct2}.

\begin{remark}\label{rmk:qDolCpxShfFbFunct}
(1) As mentioned in Remark \ref{rmk:LinPBFrobComp}, $(\varphi^*_{D_n'}M',\theta_{\varphi^*_{D_n'}M'})$
is canonically isomorphic to the scalar extension of $(\varphi^*_{D_n}M,\theta_{\varphi^*_{D_n}(M)})$
by $h_D$. 
By \eqref{eq:FunctZarLocDerivComp}, \eqref{eq:qConnFrobPBFunct},
and \eqref{eq:qdRcpxFrobPBFunct}, we see that the morphisms
\eqref{eq:qDolShfCpxFunctLoc} for $(M,\theta_M)$ and $(\varphi_{D_n}^*M,\theta_{\varphi_{D_n}^*M})$
are compatible with $\varphi^{\bullet}_{\tD_n,\Omega}(M\otimes_D\tD)$
and $\varphi^{\bullet}_{\tD_n',\Omega}(M'\otimes_{D'}\tD')$ \eqref{eq:qdRSheafCpxSecFrobPB}.
Varying $\Spf(\tD)\subset \Spf(D)$, we see that \eqref{eq:qDolShfCpxFunct2}
is compatible with $\fh_{D*}(\varphi_{\fD_n,\Omega}^{\bullet}(\CM))$
and $\varphi_{\fD'_n,\Omega}^{\bullet}(\CM')$ \eqref{eq:qdRSheafCpxFrobPB}.\par
(2) Suppose that $\psi$ is injective. Let $\CF$ and $\CG$ be crystals of $\CO_{\fX/R,n}$-modules
on $(\fX/R)_{\prism}$. Then we see that the morphisms \eqref{eq:qDolShfCpxFunct2}
for $\CF$, $\CG$, and $\CF\otimes_{\CO_{\fX/R,n}}\CG$ are compatible
with the product morphisms \eqref{eq:qDolShfCpxProd} for $\CF$ and $\CG$
and for $g_{\prism}^{-1}(\CF)$ and $g_{\prism}^{-1}(\CG)$ in the same way
as Remark \ref{rmk:LinPBFrobComp} (2); one can verify the compatibility of 
of the morphisms \eqref{eq:qDolShfCpxFunctLoc} for $\CF$, $\CG$, and $\CF\otimes_{\CO_{\fX/R,n}}\CG$
with the products by using Remark \ref{rmk:dRCpxProdScExtComp}.
\end{remark}

\begin{remark}\label{rmk:qdRModfProdFunct}
Similarly to Remark \ref{rmk:LinPBProdComp2}, one can prove
compatibility of the morphism \eqref{eq:qDolShfCpxFunct2} with
a product whose codomain is the $q$-Higgs complex with 
respect to the diagonal immersion $\fX\to\fY\times_{\Spf(R)}\fY$.
We follow the notation in Remark \ref{rmk:LinPBProdComp2}.\par
(1) Put $\fD(1)=\Spf(D(1))$, $\fD(1)_n=\Spec(D(1)/(p,\pq)^{n+1})$,
and $\overline{\fD(1)}=\Spf(D(1)/\pq D(1))$. Let
$p_{l,\fD}$ $(l=0,1)$ be the morphism $\Spf(p_{l,D})\colon \fD(1)\to \fD$,
and let $v_{D(1)}$ be the canonical morphism $\overline{\fD(1)}\to \Spf(B(1)/J_{B(1)})\cong\fX$.
Let $\CM_{l}$ $(l=0,1)$ (resp.~$\CM_{l}(1)$ $(l=0,1)$ and $\CM(1)$)
denote the quasi-coherent modules on $\fD_n$ (resp.~$\fD(1)_n$) associated
to $M_{l}$ (resp.~$M_{l}(1)$ and $M(1)$). By applying \eqref{eq:qDolShfCpxFunct2}
to $(M_{l},\theta_{M_{l}})$, $(\id_R,\id_{\fX},p_{l,\fY})$, and $\chi_{l}$,
we obtain a morphism 
\begin{equation}\label{eq:qDolProdPB}
q\Omegab(p_{l,D}^*)\colon \CM_{l}\otimes_{\CO_{\fD}}q\Omega^{\bullet}_{\fD/R}
\longrightarrow p_{l,\fD*}(\CM_{l}(1)\otimes_{\CO_{\fD(1)}}q\Omega^{\bullet}_{\fD(1)/R}).
\end{equation}
Then writing $q\Omega^{\bullet}(\CM_{l})$
(resp.~$q\Omega^{\bullet}(\CM_{l}(1))$, 
resp.~$q\Omega^{\bullet}(\CM(1))$) for the domain of \eqref{eq:qDolProdPB}
(resp.~the codomain of \eqref{eq:qDolProdPB}, 
resp.~$\CM(1)\otimes_{\CO_{\fD(1)}}q\Omega^{\bullet}_{\fD(1)/R}$),
we consider the following composition
\begin{multline}\label{eq:qDolMdfProd}
v_{D*}q\Omega^{\bullet}(\CM_0)\otimes_Rv_{D*}q\Omega^{\bullet}(\CM_1)
\xrightarrow{v_{D*}q\Omega^{\bullet}(p_{0,D}^*)\otimes v_{D*}q\Omega^{\bullet}(p_{1,D}^*)}
v_{D*}p_{0,\fD*}q\Omega^{\bullet}(\CM_0(1))\otimes_Rv_{D*}p_{1,\fD*}q\Omega^{\bullet}(\CM_1(1))\\
\longrightarrow v_{D(1)*}(q\Omega^{\bullet}(\CM_0(1))\otimes_Rq\Omega^{\bullet}(\CM_1(1)))
\xrightarrow{v_{D(1)*}\eqref{eq:qDolShfCpxProd}}v_{D(1)*}(q\Omega^{\bullet}(\CM(1)))
\end{multline}
Note  that we have $v_{D*}p_{l,\fD*}\cong v_{D(1)*}$ $(l=0,1)$. 
The morphism \eqref{eq:qDolMdfProd} is constructed explicitly as follows.
Let $\widetilde{\fX}$ be an affine open formal subscheme of $\fX$, 
let $\Spf(\tD)\subset \fD$ (resp.~$\Spf(\tD(1))\subset \fD(1)$) be the affine open
formal subscheme whose underlying set is $v_{D}^{-1}(\widetilde{\fX})$
(resp.~$v_{D(1)}^{-1}(\widetilde{\fX})$). Put $\tM_{l}=M_{l}\otimes_D\tD$,
$\tM_{l}(1)=M_{l}(1)\otimes_{D(1)}\tD(1)$, and
$\tM(1)=M(1)\otimes_{D(1)}\tD(1)$. Then we have the following morphisms of complexes
\begin{multline}\label{eq:qdRcpxMdfProd}
(\tM_0\otimes_{\tD}q\Omegab_{\tD/R})\otimes_R
(\tM_1\otimes_{\tD}q\Omegab_{\tD/R})\\
\longrightarrow
(\tM_0\otimes_{\tD(1)}\Omegab_{\tD(1)/R})
\otimes_R
(\tM_1(1)\otimes_{\tD(1)}\Omegab_{\tD(1)/R})
\longrightarrow
\tM(1)\otimes_{\tD(1)}\Omegab_{\tD(1)/R},
\end{multline}
where the first morphism is defined by \eqref{eq:qDolShfCpxFunctLoc}
for $(M_{l},\theta_{M_{l}})$, $(\id_R,\id_{\fX},p_{l,\fY})$,
and $\chi_{l}$, and the second one is the product \eqref{eq:dRcpxProd}.
We obtain \eqref{eq:qDolMdfProd} by varying $\widetilde{\fX}$ and taking
the associated sheaf.
\par
(2) We define $\fD'(1)$, $\overline{\fD'(1)}$, $p_{l,\fD'}$, and $v_{D'(1)}$ similarly to 
$\fD(1)$, $\overline{\fD(1)}$, $p_{l,\fD}$, and $v_{D(1)}$ by using $D'(1)$ and $p_{l,D'}$.
We define $\fh_{D(1)}$ to be the morphism $\fD'(1)\to\fD(1)$ induced by the
morphism of framed smooth $q$-pairs $(h(1)^*,\psi(1))\colon ((B(1),J_{B(1)})/R,\ut^{(1)})
\to ((B'(1),J_{B'(1)})/R',\ut^{\prime(1)})$. We have $p_{l,\fD}\circ \fh_{D(1)}
=\fh_{D}\circ p_{l,\fD'}$ $(l=0,1)$. We define $q\Omegab(\CM_l')$ and
$q\Omegab(\CM'(1))$ similarly to $q\Omegab(\CM_l)$ and $q\Omegab(\CM(1))$
in (1) by using $(M'_l,\theta_{M'_l})$ and $(M'(1),\theta_{M'(1)})$.
We assert that the following diagram is
commutative.
\begin{equation}\label{eq:qdRCpxMdfProdFunct}
\xymatrix@R=15pt@C=50pt{
v_{D*}q\Omegab(\CM_0)\otimes_Rv_{D*}q\Omegab(\CM_1)
\ar[r]^(.55){\eqref{eq:qDolMdfProd}}\ar[d]_{\eqref{eq:qDolShfCpxFunct2}}&
v_{D(1)*}q\Omegab(\CM(1))\ar[d]_{\eqref{eq:qDolShfCpxFunct2}}\\
v_{D*}\fh_{D*}q\Omegab(\CM_0')\otimes_R
v_{D*}\fh_{D*}q\Omegab(\CM_1 ')\ar[d]^{\cong}&
v_{D(1)*}\fh_{D(1)*}q\Omega^{\bullet}(\CM'(1))\ar[d]^{\cong}\\
g_*v_{D'*}q\Omegab(\CM_0')
\otimes_Rg_*v_{D'*}q\Omegab(\CM_1')
\ar[r]^(.58){\eqref{eq:qDolMdfProd}}&
g_*v_{D'(1)*}q\Omegab(\CM'(1))
}
\end{equation}
Let $\widetilde{\fX}\subset \fX$ be an affine open as in (1), 
and let $\widetilde{\fX'}$ be $\widetilde{\fX}\times_{\fX}\fX'$, which
is an affine open of $\fX'$. We define $\Spf(\tD')\subset \fD'$,
$\Spf(\tD'(1))\subset \fD'(1)$, $\tM_{l}'$, and $\tM'(1)$
similarly to $\tD$, $\tD(1)$, $\tM_{l}$ and $\tM(1)$ defined in (1)
by using $\widetilde{\fX'}\subset \fX'$, $M'_{l}$, and $M'(1)$.
By applying \eqref{eq:FunctZarLocDerivComp} to the 
diagram \eqref{eq:FramedSmEmbProdMorph}, $\widetilde{\fX}\subset \fX$, 
and $\widetilde{\fX}'\subset \fX'$, we obtain a commutative diagram
in the category $\DAlg$ introduced before Proposition \ref{prop:ProddRCpxFunct}
\begin{equation*}
\xymatrix@R=10pt{
\widetilde{\uD}=(\tD/R,\Lambda,\mu\ut,\utheta_{\tD})\ar[d]
\ar@<0.5ex>[r]\ar@<-0.5ex>[r]
&
\underline{\tD(1)}=(\tD(1)/R,\Lambda(1),\mu\ut^{(1)},\utheta_{\tD(1)})\ar[d]\\
\widetilde{\uD}'=(\tD'/R',\Lambda',\mu\ut',\utheta_{\tD'})
\ar@<0.5ex>[r]\ar@<-0.5ex>[r]
&
\underline{\tD'(1)}=(\tD'(1)/R',\Lambda'(1),\mu\ut^{\prime(1)},\utheta_{\tD'(1)})
}
\end{equation*}
lying over \eqref{eq:FrSmEmbEnvProdMorph}. Hence, by Proposition \ref{prop:ProddRCpxFunct},
we see that the morphism \eqref{eq:qdRcpxMdfProd} and
the corresponding one for $(\CF'_{l}, i'\colon \fX'\to \fY', \ut',
\widetilde{\fX}'\subset \fX')$ are compatible with the morphisms 
\eqref{eq:qDolShfCpxFunctLoc} 
$\tM_{l}\otimes_{\tD}q\Omegab_{\tD/R}
\to \tM'_{l}\otimes_{\tD'}q\Omegab_{\tD'/R'}$
and $\tM(1)\otimes_{\tD(1)}q\Omegab_{\tD(1)/R}
\to \tM'(1)\otimes_{\tD'(1)}q\Omegab_{\tD'(1)/R'}$
for $(\CF_{l},(f,g,h),\psi)$ and
$(\CF,(f,g,h(1)),\psi(1))$. Varying $\widetilde{\fX}$,
we see that the diagram \eqref{eq:qdRCpxMdfProdFunct} is commutative.
\end{remark}

We show that the morphisms \eqref{eq:LineDRcpxPB} and \eqref{eq:qDolShfCpxFunct} are compatible
with \eqref{eq:LinearizationZariskiProjMap}. 
We have the following diagrams of topos commutative up to canonical isomorphisms.
\begin{equation}\label{eq:PrismZarProjFunct}
\xymatrix@C=40pt{
(\fX'/R')^{\sim}_{\prism}\ar[r]^{U_{\fX'/R}}\ar[d]_{g_{\prism}}&
\fX^{\prime\sim}_{\ZAR}\ar[r]^{\varepsilon_{\fX',\Zar}}\ar[d]^{g_{\ZAR}}&
\fX^{\prime\sim}_{\Zar}\ar[d]^{g_{\Zar}}&
\fD^{\prime\sim}_{\Zar}\ar[l]_{v_{D',\Zar}}\ar[d]^{\fh_{D,\Zar}}\\
(\fX/R)^{\sim}_{\prism}\ar[r]^{U_{\fX/R}}&
\fX^{\sim}_{\ZAR}\ar[r]^{\varepsilon_{\fX,\Zar}}&
\fX^{\sim}_{\Zar}&
\fD^{\sim}_{\Zar}\ar[l]_{v_{D,\Zar}}
}
\end{equation}

\begin{proposition}\label{eq:LindRCpxZarProjFunct}
The following diagram is commutative for $r\in \N$.
\begin{equation*}
\xymatrix
{
g^{-1}u_{\fX/R*}(\CL_{A,B}(M\otimes_Dq\Omega^r_{D/R}))\ar[r]
&u_{\fX'/R'*}g_{\prism}^{-1}(\CL_{A,B}(M\otimes_Dq\Omega^r_{D/R}))
\ar[r]^{\eqref{eq:LineDRcpxPB}}
& u_{\fX'/R'*}(\CL_{A',B'}(M'\otimes_{D'}q\Omega^r_{D'/R'}))\\
g^{-1}(v_{D*}(\CM\otimes_{\CO_{\fD}}q\Omega^r_{\fD/R}))\ar[r]\ar[u]^{\eqref{eq:LinearizationZariskiProjMap}}
&v_{D'*}(\fh_D^{-1}(\CM\otimes_{\CO_{\fD}}q\Omega^r_{\fD/R}))\ar[r]^{\eqref{eq:qDolShfCpxFunct}}& 
v_{D'*}(\CM'\otimes_{\CO_{\fD'}}q\Omega^r_{\fD'/R'}),\ar[u]_{\eqref{eq:LinearizationZariskiProjMap}}
}
\end{equation*}
Here the upper (resp.~lower) left horizontal morphism is 
the base change morphism with respect to the 
left and middle squares (resp.~the right square) of the
diagram \eqref{eq:PrismZarProjFunct}. 
\end{proposition}

\begin{proof}
Put $\CL=\CL_{A,B}(M\otimes_D q\Omega^r_{D/R})$,  
$\CL'=\CL_{A',B'}(M'\otimes_{D'}q\Omega^r_{D'/R'})$, 
$\CN=\CM\otimes_{\CO_{\fD}}q\Omega^r_{\fD/R}$,
and $\CN'=\CM'\otimes_{\CO_{\fD'}}q\Omega^r_{\fD'/R'}$
to simplify the notation. 
The adjoint of the upper (resp.~lower) left horizontal map with respect to 
the morphism of topos $g_{\Zar}\colon \fX_{\Zar}^{\prime\sim}\to \fX_{\Zar}^{\sim}$
is given by 
$u_{\fX/R*}\CL
\to u_{\fX/R*}g_{\prism*}g_{\prism}^{-1}\CL
\cong g_*u_{\fX'/R'*}g_{\prism}^{-1}\CL$
(resp.~$v_{D*}\CN\to v_{D*}\fh_{D*}\fh_D^{-1}\CN\cong
g_*v_{D'*}\fh_D^{-1}\CN$). Hence, by taking the adjoint with respect to $g_{\Zar}$, 
the commutativity in question is equivalent to that of the diagram
\begin{equation}\label{eq:LinedRZARProjFunct}
\xymatrix@C=40pt{
u_{\fX/R*}\CL\ar[r]&
u_{\fX/R*}g_{\prism*}g_{\prism}^{-1}\CL\cong
g_*u_{\fX'/R'*}g_{\prism}^{-1}\CL\ar[r]^(.67){\eqref{eq:LineDRcpxPB}}
&
g_*u_{\fX'/R'*}\CL'\\
v_{D*}\CN\ar[r]\ar[u]_{\eqref{eq:LinearizationZariskiProjMap}}&
v_{D*}\fh_{D*}\fh_D^{-1}\CN\cong
g_*v_{D'*}\fh_D^{-1}\CN\ar[r]^(.65){\eqref{eq:qDolShfCpxFunct}}&
g_*v_{D'*}\CN'.\ar[u]_{\eqref{eq:LinearizationZariskiProjMap}}
}
\end{equation}

Let $j\colon \fU\to \fX$ be an open immersion, let 
$j'\colon \fU'\to \fX'$ be the base change of $j$ under $g$,
and let $g_{\fU}\colon \fU'\to \fU$ be the morphism 
induced by $g\colon \fX'\to \fX$. 
Let $\fD_{\fU}=\Spf(D_{\fU})$
(resp.~$\fD'_{\fU'}=\Spf(D'_{\fU'})$) be the open formal subscheme of 
$\fD=\Spf(D)$ (resp.~$\fD'=\Spf(D')$) whose underlying topological space
is $v_D^{-1}(\fU)\subset \Spf(D/\pq D)$ (resp.~$v_{D'}^{-1}(\fU')\subset \Spf(D'/\pq D')$).
We have $\fh_D(\fD'_{\fU'})\subset \fD_{\fU}$.
By taking
$\Gamma(\fU,-)$ of the diagram \eqref{eq:LinedRZARProjFunct} 
and then evaluating the upper line at $(P, w')\in (\fU'/R')_{\prism}$
and $(P,g_{\fU}w')\in (\fU/R)_{\prism}$, we obtain a diagram 
\begin{equation*}
\xymatrix{
\txt{$(M\otimes_DD_P)\otimes_{D_P}q\Omega^r_{D_P/P}$\\
$=\Gamma((P,jg_{\fU}w'),\CL)$
}\ar@<1ex>@{=}[r]&
\txt{$(M\otimes_DD_P)\otimes_{D_P}q\Omega^r_{D_P/P}$\\
$=\Gamma((P, j' w'),g_{\prism}^{-1}\CL)$}\ar@<1ex>[r]^(.49){\eqref{eq:LineDRcpxFunctDiag}}
&
\txt{$(M'\otimes_{D'}D'_{P})\otimes_{D'_P}q\Omega^r_{D'_P/P}$\\
$=\Gamma((P,j'w'),\CL')$}\\
\Gamma((\fU/R)_{\prism},\CL\vert_{(\fU/R)_{\prism}})\ar[u]\ar[r]&
\Gamma((\fU'/R')_{\prism},g_{\prism}^{-1}(\CL)\vert_{(\fU'/R')_{\prism}})\ar[u]\ar[r]^(.53){\eqref{eq:LineDRcpxPB}}&
\Gamma((\fU'/R')_{\prism},\CL'\vert_{(\fU'/R')_{\prism}})\ar[u]
\\
\txt{$\Gamma(\fD_{\fU},\CN)=$\\
$(M\otimes_DD_{\fU})\otimes_{D_{\fU}}q\Omega_{D_{\fU}/R}^r$}\ar[u]\ar@<-2ex>[rr]^{\eqref{eq:qDolShfCpxFunctLoc}}&&
\txt{$\Gamma(\fD'_{\fU'},\CN')=$\\
$(M'\otimes_{D'}D'_{\fU'})\otimes_{D'_{\fU'}}q\Omega_{D'_{\fU'}/R'}$.}\ar[u]
}
\end{equation*}
The upper left (resp.~right) square is commutative by Lemma \ref{lem:PrismZarProjPB}
below (resp.~the definition of the morphism \eqref{eq:LineDRcpxPB}),
and the outer big rectangle is commutative by the commutativity of the diagram below
and Lemma \ref{lem:dRCpxFuncCocyc}.
\begin{equation*}\xymatrix{
E^{\ualpha}(D_P)\ar[r]& E^{\ualpha'}(D'_P)\\
E^{\ualpha}(D_{\fU})\ar[u]\ar[r]&E^{\ualpha'}(D'_{\fU'})\ar[u]
}
\end{equation*}
By considering all objects $(P,w')$ of $(\fU'/R')_{\prism}$, 
we see that the lower rectangle is commutative. As it holds for any 
$\fU$,  the diagram \eqref{eq:LinedRZARProjFunct} is also commutative.
This completes the proof.
\end{proof}

\begin{lemma}\label{lem:PrismZarProjPB}
Let $j\colon \fU\to \fX$ be an object of $\fX_{\ZAR}$, 
let $j'\colon \fU'=\fU\times_{\fX,g}\fX'\to \fX'$ be its base change under
$g\colon \fX'\to \fX$, and let $g_{\fU}\colon \fU'\to \fU$ denote the projection.
Then the following
diagram is commutative for $\CF\in \Ob (\fX/R)^{\sim}_{\prism}$. 
\begin{equation*}\xymatrix@C=10pt{
\Gamma((\fU/R)_{\prism},\CF\vert_{(\fU/R)_{\prism}})\ar[r]^(.43){g_{\fU,\prism}^*}&
\Gamma((\fU'/R')_{\prism},g_{\fU,\prism}^*(\CF\vert_{(\fU/R)_{\prism}}))\ar@{=}[r]&
\Gamma((\fU'/R')_{\prism},(g_{\prism}^*\CF)\vert_{(\fU'/R')_{\prism}})\\
\Gamma(\fU,U_{\fX/R*}\CF)\ar@{=}[u]\ar[r]&
\Gamma(\fU,U_{\fX/R*}g_{\prism*}g_{\prism}^*\CF)\cong
\Gamma(\fU,g_{\ZAR*}U_{\fX'/R'*}g_{\prism}^*\CF)\ar@{=}[r]&
\Gamma(\fU',U_{\fX'/R'*}g_{\prism}^*\CF)\ar@{=}[u]
}
\end{equation*}
\end{lemma}

\begin{proof}
Let $\fU^{\sim}$
(resp.~$\fU^{\prime\sim}$) be the sheaf of sets on
$\fX_{\ZAR}$ (resp.~$\fX'_{\ZAR}$) represented by 
$\fU$ (resp.~$\fU'$). Then $\Gamma(\fU,-)$ of the morphism $U_{\fX/R*}\CF
\to U_{\fX/R*}g_{\prism*}g_{\prism}^*\CF
\cong g_{\ZAR*}U_{\fX'/R'*}g_{\prism}^*\CF$
is given by
\begin{multline}\label{eq:PrismZARProjPB}
\Gamma((\fU/R)_{\prism},\CF\vert_{(\fU/R)_{\prism}})
=\Hom_{(\fX/R)^{\sim}_{\prism}}(U_{\fX/R}^*\fU^{\sim},\CF)
\xrightarrow{g_{\prism}^*}
\Hom_{(\fX'/R')^{\sim}_{\prism}}(g_{\prism}^*U_{\fX/R}^*\fU^{\sim}, g_{\prism}^*\CF)\\
=\Hom_{(\fX'/R')^{\sim}_{\prism}}(U_{\fX'/R'}^*g_{\ZAR}^*\fU^{\sim}, g_{\prism}^*\CF)
=\Hom_{(\fX'/R')^{\sim}_{\prism}}(U_{\fX'/R'}^*\fU^{\prime\sim},g_{\prism}^*\CF)
=\Gamma((\fU'/R')_{\prism},(g^*_{\prism}\CF)\vert_{(\fU'/R')_{\prism}}).
\end{multline}
Let $(P,w')$ be an object of $(\fU'/R')_{\prism}$, which gives 
an object $(P,g_{\fU} w')$ of $(\fU/R)_{\prism}$. 
Then via the first (resp.~the last) equality above
the evaluation $\Gamma((\fU/R)_{\prism},\CF\vert_{(\fU/R)_{\prism}})\to\Gamma((P,j g_{\fU} w'),\CF)$
at $(P,g_{\fU} w')$ (resp.~$\Gamma((\fU'/R')_{\prism},(g_{\prism}^*\CF)\vert_{(\fU'/R')_{\prism}})
\to\Gamma((P,j' w'),g_{\prism}^*\CF)$ at $(P,w')$) is obtained by 
taking the value at $g_{\fU} w'\in \Gamma(\Spf(P/\pq P),\fU^{\sim})
=\Gamma((P, jg_{\fU} w'),U_{\fX/R}^*\fU^{\sim})$
(resp.~$w'\in \Gamma(\Spf(P/\pq P),\fU^{\prime\sim})
=\Gamma((P,j' w'),U_{\fX'/R'}^*\fU^{\prime\sim})$).
Since the image of $g_{\fU} w'$ under
$\Gamma((P,jg_{\fU}w'),U_{\fX/R}^*\fU^{\sim})
=\Gamma((P,j'w'),g_{\prism}^*U_{\fX/R}^*\fU^{\sim})
=\Gamma((P,j' w'), U_{\fX'/R'}^*\fU^{\prime\sim})$
is $w'$, the composition \eqref{eq:PrismZARProjPB}
is compatible with the evaluations at $(P,g_{\fU} w')$ and $(P,w')$
above. By varying $(P,w')$, we see that the diagram in the lemma
is commutative.
\end{proof}

\begin{remark}\label{rmk:LinqDolqDolMdfProdComp}
We see that the morphisms \eqref{eq:LinqdRCpxMdfProd} and \eqref{eq:qDolMdfProd} are compatible
with the morphisms of complexes \eqref{eq:qdRCpxLinZarProj} 
for $((i,\ut),\CF_l)$ $(l=0,1)$ and $((i(1),\ut^{(1)}),\CF)$
as follows.
Under the notation in Remarks \ref{rmk:LinPBProdComp2} (1) and 
\ref{rmk:qdRModfProdFunct} (1), we have the following
diagram of complexes of $R$-modules, where we abbreviate
$u_{\fX/R}$, $v_D$, $v_{D(1)}$,
$u_{\fX/R*}\CL\Omegab(p_{l,D}^*)$, and $v_{D*}q\Omegab(p_{l,D}^*)$
to $u$, $v$, $v^{(1)}$, $p_{l}^*$, and $p_{l}^*$.
\begin{equation*}
\xymatrix@R=20pt{
u_*\CL\Omegab(M_0)\otimes_Ru_*\CL\Omegab(M_1)
\ar[r]^(.45){p_0^*\otimes p_1^*}&
u_*\CL\Omegab(M_0(1))\otimes_Ru_*\CL\Omegab(M_1(1))
\ar[r]^(.65){\eqref{eq:LinqDolCpxProd}}&
u_*\CL\Omegab(M(1))\\
v_*q\Omegab(\CM_0)\otimes_R v_*q\Omegab(\CM_1)
\ar[r]^(.45){p_0^*\otimes p_1^*}\ar[u]^{\eqref{eq:qdRCpxLinZarProj}}&
v^{(1)}_*q\Omegab(\CM_0(1))\otimes_R
v^{(1)}_*\Omegab(\CM_1(1))
\ar[r]^(.65){\eqref{eq:qDolShfCpxProd}}\ar[u]^{\eqref{eq:qdRCpxLinZarProj}}&
v^{(1)}_*q\Omegab(\CM(1))\ar[u]^{\eqref{eq:qdRCpxLinZarProj}}
}
\end{equation*}
The left (resp.~right) square is commutative by \eqref{eq:LinedRZARProjFunct} 
(resp.~Remark \ref{rmk:LinZarProjFrobComp} (2)).
\end{remark}

By construction, the degree $0$-term
$g_{\prism}^{-1}(\CL_{A,B}(D_n))\to \CL_{A',B'}(D_n')$
of the morphism \eqref{eq:LineDRcpxPB} for $\CF=\CO_{\fX/R,n}$ is a homomorphism of  
algebras over $g_{\prism}^{-1}(\CO_{\fX/R,n})=\CO_{\fX'/R',n}$.
Hence the following diagram is commutative.
\begin{equation}\label{eq:CrystaldRResolFunct}
\xymatrix{
g_{\prism}^{-1}(\CF)\ar[r]^(.25){\eqref{eq:CrysPLAugMap}}\ar@{=}[d]
&
g_{\prism}^{-1}(\CF\otimes_{\CO_{\fX/R,n}}\CL_{A,B}(D_n\otimes_Dq\Omega^{\bullet}_{D/R}))
&
g_{\prism}^{-1}(\CF)\otimes_{\CO_{\fX'/R',n}}
g_{\prism}^{-1}(\CL_{A,B}(D_n\otimes_Dq\Omega^{\bullet}_{D/R}))
\ar[d]^{\id\otimes\eqref{eq:LineDRcpxPB}}
\ar[l]_(.52){\cong}
\\
g_{\prism}^{-1}(\CF)\ar[rr]^(.4){\eqref{eq:CrysPLAugMap}}
&&
g_{\prism}^{-1}(\CF)\otimes_{\CO_{\fX'/R',n}}\CL_{A',B'}(D_n'\otimes_{D'}q\Omega^{\bullet}_{D'/R'})
}
\end{equation}

\begin{lemma}\label{lem:LinZarqDRProjFunct}
The following diagram is commutative, where
the upper horizontal map is the composition of the right horizontal map
with the right vertical one in the diagram \eqref{eq:CrystaldRResolFunct}. 
\begin{equation*}
\xymatrix@C=60pt{
g_{\prism}^{-1}(\CF\otimes_{\CO_{\fX/R,n}}\CL_{A,B}(D_n\otimes_Dq\Omega^{\bullet}_{D/R}))
\ar[r]^(.47){\eqref{eq:LineDRcpxPB}}
\ar[d]^{\cong}_{\eqref{eq:CrystalResCpxComp}}
&
g_{\prism}^{-1}(\CF)\otimes_{\CO_{\fX'/R',n}}\CL_{A',B'}(D_n'\otimes_{D'}q\Omega^{\bullet}_{D'/R'})
\ar[d]^{\cong}_{\eqref{eq:CrystalResCpxComp}}
\\
g_{\prism}^{-1}(\CL_{A,B}(M\otimes_Dq\Omega^{\bullet}_{D/R}))
\ar[r]^{\eqref{eq:LineDRcpxPB}}
&
\CL_{A',B'}(M'\otimes_{D'}q\Omega^{\bullet}_{D'/R'})
}
\end{equation*}
\end{lemma}

\begin{proof}
By the construction of \eqref{eq:LineDRcpxPB} from 
the upper map in \eqref{eq:LineDRcpxFunctDiag} which goes
back to Proposition \ref{prop:dRCpxFunct} via \eqref{eq:qDolbCpxPB2}
applied to the upper horizontal map of \eqref{eq:LinFramedqPairMaps},
and that of \eqref{eq:CrystalResCpxComp} from \eqref{eq:ResolLinCompMap},
 we are reduced to showing the following two claims for
 each object $(P,w)$ of $(\fX'/R')_{\prism}$. The isomorphism
 \eqref{eq:ResolLinCompMap} for $\CF$ and $(P,g\circ w)$:
 $\CF(P,g\circ w)\otimes_DD_P\cong M\otimes_DD_P$
 is compatible with $\id_{\CF(P,g\circ w)}\otimes\gamma_{D_P,i}$
 and $\gamma_{M\otimes_DD_P,i}$ for $i\in \Lambda$, and the diagram
 \begin{equation*}
 \xymatrix@C=50pt{
 \CF(P,g\circ w)\otimes_PD_P
 \ar[r]^(.57){\cong}_(.57){\eqref{eq:ResolLinCompMap}}
 \ar[d]
&
 M\otimes_DD_P
 \ar[d]
 \\
 \CF(P,g\circ w)\otimes_PD'_P
 \ar[r]^(.57){\cong}_(.57){\eqref{eq:ResolLinCompMap}}
 &
 M'\otimes_{D'}D'_P
 }
 \end{equation*}
is commutative. Here $\gamma_{D_P,i}$
is the $q$-Higgs automorphism of $D_P$ associated 
to the framed smooth $q$-pair $((B_P,J_P)/P,\ut_P)$
(Definition \ref{def:qHiggsDifferential} (2)), 
$\gamma_{M\otimes_DD_P,i}$ 
is the $\gamma_{D_P,i}$-semilinear endomorphism 
$\id+t_{P,i}\mu\nabla_{M\otimes_DD_P,i}$
of $M\otimes_DD_P$ defined by the $(t_{i,P}\mu,\theta_{D_P,i})$-connection
 $\theta_{M\otimes_DD_P,i}$ associated to the 
$q$-Higgs field $\theta_{M\otimes_DD_P}$ on $M\otimes_DD_P$
over $(q\Omega^{\bullet}_{D_P/P},\theta_{D_P}^{\bullet})$
(Definition \ref{def:connection}),
and the vertical maps of the diagram are 
defined by the morphism $D_P\to D'_P$ of $q$-prisms
induced by the morphism of $\delta$-pairs $(B_P,J_{B_P})\to (B'_P,J_{B'_P})$
\eqref{eq:LinFramedqPairMaps} over $f\colon (R,I)\to (R',I')$. The first claim is reduced to
the compatibility with $\id_{\CF(P,g\circ w)}\otimes \theta_{D_P,i}$
and $\theta_{M\otimes_DD_P,i}$, which was verified in the
proof of Proposition \ref{prop:CrystalResolCpxComp}. 
The second claim follows from the commutative diagram in $(\fX/R)_{\prism}$
\begin{equation*}
\xymatrix{
(P,g\circ w)\ar@{=}[d]
&
\ar[l]
(D_P,v_{D_P})
\ar[r]
&
(D,v_D)\\
(P,g\circ w)
&
\ar[l]
\ar[u]
(D'_P,g\circ v_{D'_P})
\ar[r]
&
\ar[u]
(D',g\circ v_{D'}),
}
\end{equation*}
where the upper (resp.~lower) line is \eqref{eq:ProductDPmorphi} 
for $(P,g\circ w)$ and $B\to A$ over $R$ (resp.~for $(P,w)$ and $B'\to A'$ over $R'$
composed with $g\colon \fX'\to \fX$), and the right square
is induced by \eqref{eq:LinFramedqPairMaps}.
\end{proof}

By combining Lemma \ref{lem:LinZarqDRProjFunct}
with \eqref{eq:CrystaldRResolFunct}, we obtain the following commutative
diagram, a functoriality of \eqref{eq:qdRResol}.
\begin{equation}\label{eq:CrysqdRResolFunct}
\xymatrix@C=40pt{
g_{\prism}^{-1}(\CF)\ar[r]^(.3){\eqref{eq:qdRResol}}\ar@{=}[d]
&
g_{\prism}^{-1}(\CL_{A,B}(M\otimes_Dq\Omega^{\bullet}_{D/R}))
\ar[d]^{\eqref{eq:LineDRcpxPB}}
\\
\CF'\ar[r]^(.3){\eqref{eq:qdRResol}}
&
\CL_{A',B'}(M'\otimes_{D'}q\Omega^{\bullet}_{D'/R'})
}
\end{equation}
\begin{remark}\label{rmk:ResolMdfProdComp}
The morphism \eqref{eq:LinqdRCpxMdfProd} is compatible with
the resolution \eqref{eq:qdRResol} as follows. Under the notation 
in Remark \ref{rmk:LinPBProdComp2} (1), we have the following diagram, whose
left (resp.~right square) is commutative by \eqref{eq:CrysqdRResolFunct}
(resp.~Remark \ref{rmk:LinResolTransFrobComp} (2)).
We abbreviate $\CL\Omegab(p_{l,D}^*)$ to $p_{l}^*$.
\begin{equation*}\xymatrix@C=35pt@R=20pt{
\CF_0\otimes_{\CO_{\fX/R,n}}\CF_1\ar@{=}[r]\ar[d]^{\eqref{eq:qdRResol}}&
\CF_0\otimes_{\CO_{\fX/R,n}}\CF_1\ar@{=}[r]\ar[d]^{\eqref{eq:qdRResol}}&
\CF\ar[d]^{\eqref{eq:qdRResol}}\\
\CL\Omega^{\bullet}(M_0)\otimes_{\CO_{\fX/R,n}}
\CL\Omega^{\bullet}(M_1)
\ar[r]^(.45){p_0^*\otimes p_1^*}&
\CL\Omega^{\bullet}(M_0(1))\otimes_{\CO_{\fX/R,n}}
\CL\Omega^{\bullet}(M_1(1))
\ar[r]^(.67){\eqref{eq:LinqDolCpxProd}}&
\CL\Omega^{\bullet}(M(1))
}
\end{equation*}
\end{remark}

By using Proposition \ref{eq:LindRCpxZarProjFunct}
and \eqref{eq:CrysqdRResolFunct} above, we obtain the following functoriality of \eqref{eq:CrystalCohqHiggs}.

\begin{proposition}\label{prop:PrismZarProjqDRFunct}
Under the notation and assumption introduced before
\eqref{eq:LineDRcpxPB}, the following diagram is commutative.
\begin{equation*}
\xymatrix@C=50pt{
g^{-1}Ru_{\fX/R*}\CF\ar[r]&
Ru_{\fX'/R'*}g_{\prism}^{-1}(\CF)\ar@{=}[r]&
Ru_{\fX'/R'*}\CF'
\\
g^{-1}v_{D_*}(\CM\otimes_{\CO_{\fD}}q\Omega^{\bullet}_{\fD/R})
\ar[u]^{\eqref{eq:CrystalCohqHiggs}}_{\cong}
\ar[r]&
v_{D'*}(\fh_{D}^{-1}(\CM\otimes_{\CO_{\fD}}q\Omega^{\bullet}_{\fD/R}))
\ar[r]^{\eqref{eq:qDolShfCpxFunct}}&
v_{D'*}(\CM'\otimes_{\CO_{\fD'}}q\Omega^{\bullet}_{\fD'/R'})
\ar[u]^{\cong}_{\eqref{eq:CrystalCohqHiggs}}
}
\end{equation*}
\end{proposition}

\begin{proof}
By \eqref{eq:CrysqdRResolFunct}, we have  the following
commutative diagram, where
$\CL\Omega(M)$ and $\CL\Omega(M')$ denote
$\CL_{A,B}(M\otimes_Dq\Omega^{\bullet}_{D/R})$
and $\CL_{A',B'}(M'\otimes_{D'}q\Omega^{\bullet}_{D'/R'})$,
respectively.
\begin{equation*}
\xymatrix@C=40pt{
g^{-1}Ru_{\fX/R*}\CF\ar[r]\ar[d]_{\cong}^{\eqref{eq:qdRResol}}&
Ru_{\fX'/R'*}g_{\prism}^{-1}(\CF)\ar@{=}[r]\ar[d]_{\cong}^{\eqref{eq:qdRResol}}&
Ru_{\fX'/R'*}\CF'\ar[d]_{\cong}^{\eqref{eq:qdRResol}}
\\
g^{-1}Ru_{\fX/R*}\CL\Omega(M)\ar[r]&
Ru_{\fX'/R'*}g_{\prism}^{-1}(\CL\Omega(M))\ar[r]^(.53){\eqref{eq:LineDRcpxPB}}&
Ru_{\fX'/R'*}\CL\Omega(M')
}
\end{equation*}
We obtain the claim by combining the commutative diagram in 
Proposition \ref{eq:LindRCpxZarProjFunct} with the above one via the morphisms
$u_{\fX/R*}\to Ru_{\fX/R*}$ and $u_{\fX'/R'*}\to Ru_{\fX'/R'*}$.
\end{proof}

We have the following analogue of Proposition \ref{prop:PrismZarProjqDRFunct} for the 
isomorphism \eqref{eq:CrystalCohqHigProj}. 

\begin{proposition}\label{prop:CrystalCohqHigProjFunct}
Let $\CF$ be a complete crystal of $\CO_{\fX/R}$-modules
on $(\fX/R)_{\prism}$ (Definition \ref{def:PrismaticSite} (2)),
and let $(\CF_n)_{n\in \N}$ be the corresponding adic inverse system of crystals
of $\CO_{\fX/R,n}$-modules on $(\fX/R)_{\prism}$ 
(Remark \ref{rmk:CompleteCrystalInvSystem} (2), (3)).
Let $\CF'$ denote the complete crystal of $\CO_{\fX/R}$-modules
$g_{\prism}^{-1}\CF$, and let $(\CF'_n)_{n\in \N}$ denote the
corresponding adic inverse system of crystals of $\CO_{\fX/R,n}$-modules,
which coincides with $(g_{\prism}^{-1}\CF_n)_{n\in \N}$ (Definition \ref{def:PrismaticSite}(3)). 
Let $(\CM_n\otimes_{\CO_{\fD}}q\Omega_{\fD/R}^{\bullet})_{n\in \N}$
(resp.~$(\CM_n'\otimes_{\CO_{\fD'}}q\Omega_{\fD'/R}^{\bullet})_{n\in \N}$)
be the inverse system of $q$-Higgs complexes on $\fD_{\Zar}$ (resp.~$\fD'_{\Zar}$)
associated to $(\CF_n)_{n\in \N}$ (resp.~$(\CF_n')_{n\in \N}$) \eqref{eq:qHiggsShfCpx}.
Then the following diagram is commutative.
\begin{equation*}
\xymatrix{
Ru_{\fX/R*}\CF\ar[r]&
Ru_{\fX/R*}Rg_{\prism*}\CF'\ar[r]^{\cong}&
Rg_*Ru_{\fX'/R'*}\CF'\\
v_{D*}\varprojlim_n(\CM_n\otimes_{\CO_{\fD}}q\Omega^{\bullet}_{\fD/R})
\ar[rr]\ar[u]^{\cong}_{\eqref{eq:CrystalCohqHigProj}}
&&
Rg_*(v_{D'*}\varprojlim_n(\CM'_n\otimes_{\CO_{\fD'}}q\Omega^{\bullet}_{\fD'/R}))
\ar[u]^{\cong}_{\eqref{eq:CrystalCohqHigProj}},
}
\end{equation*}
where the bottom horizontal morphism is induced by the inverse
limit of $(v_{D*}(\CM_n\otimes_{\CO_{\fD}}q\Omega^{\bullet}_{\fD/R})
\to h_{\fD*}(\CM'_n\otimes_{\CO_{\fD'}}q\Omega^{\bullet}_{\fD'/R'}))_{n\in \N}$
\eqref{eq:qDolShfCpxFunct2}.
\end{proposition}

\begin{proof} 
From \eqref{eq:LinedRZARProjFunct} and the adjoint commutative diagram of \eqref{eq:CrysqdRResolFunct} 
with respect to $g_{\prism}$, we obtain a commutative diagram
\begin{equation*}
\xymatrix@C=40pt{
Ru_{\fX/R*}\CF\ar[r]^(.45){\cong}_(.45){\eqref{eq:qdRResol}}
\ar[d]&
Ru_{\fX/R*}\varprojlim_n\CL_n^{\bullet}\ar[d]&
\ar[l]_(.45){\cong}^(.45){\eqref{eq:LinearizationZariskiProjMap}}
v_{D*}\varprojlim_n\CN_n^{\bullet}\ar[d]\\
Rg_*Ru_{\fX'/R'*}\CF'
\ar[r]^(.45){\cong}_(.45){\eqref{eq:qdRResol}}&
Rg_*Ru_{\fX'/R'*}\varprojlim_n\CL_n^{\prime\bullet}&
\ar[l]_(.45){\cong}^(.45){\eqref{eq:LinearizationZariskiProjMap}}
Rg_*(v_{D'*}\varprojlim_n\CN_n^{\prime\bullet}),
}
\end{equation*}
where the left (resp.~right) vertical morphism is
the upper (resp.~lower) horizontal  composition (resp.~morphism) in the claim,
and $\CL_n^{\bullet}$, $\CL_n^{\prime\bullet}$, $\CN_n^{\bullet}$ and 
$\CN_n^{\prime\bullet}$ denote
$\CL_{A,B}(M_n\otimes_Dq\Omega^{\bullet}_{D/R})$,
$\CL_{A',B'}(M'_n\otimes_{D'}q\Omega^{\bullet}_{D'/R'})$,
$\CM_n\otimes_{\CO_{\fD}}q\Omega^{\bullet}_{\fD/R}$,
and $\CM'_n\otimes_{\CO_{\fD'}}q\Omega^{\bullet}_{\fD'/R}$,
respectively.
\end{proof}

\section{Prismatic cohomology and $q$-Higgs complex: 
Global case}\label{sec:PrismCohqDolbGlobal}

Let $R$ be a $q$-prism (Definition \ref{def:FramedSmoothPrism} (1)), 
and let $\fX$ be a separated $p$-adic formal scheme over 
$\Spf(R/[p]_qR)$. Suppose that there exist a  Zariski hypercovering
$\fX_{\star}=([\nu]\mapsto \fX_{[\nu]}=\Spf(A_{[\nu]}))_{\nu\in\N}$ of $\fX$
by affine formal schemes and an
$R$-closed immersion $i_{\star}\colon \fX_{\star}\hookrightarrow \fY_{\star}
=([\nu]\mapsto \fY_{[\nu]}=\Spf(B_{[\nu]}))_{\nu\in\N}$
into a simplicial smooth $pR+[p]_q R$-adic affine formal scheme over $R$ 
such that the cosimplicial $R$-algebras with ideals $([\nu]\mapsto (B_{[\nu]},J_{[\nu]}))_{\nu\in \N}$, 
where $J_{[\nu]}=\Ker(i_{[\nu]}^*\colon B_{[\nu]}\to A_{[\nu]})$ for $\nu\in \N$,
underlies a cosimplicial framed smooth $q$-pairs 
$([\nu]\mapsto ((B_{[\nu]},J_{[\nu]})/R, \ut_{[\nu]}=(t_{[\nu], i})_{i\in \Lambda_{[\nu]}})_{\nu\in \N}$
(Definition \ref{def:FramedSmoothPrism} (3))
and that each $(B_{[\nu]},J_{[\nu]})$ has a $q$-prismatic envelope $D_{[\nu]}$
(Definition \ref{def:FramedSmoothPrism} (1)).

Choose and fix such data in the following. Under this assumption, we will give
a description of $Ru_{\fX/R*}\CF$ for a crystal of $\CO_{\fX/R,n}$-modules
$\CF$ on $(\fX/R)_{\prism}$ in terms of the $q$-Higgs complex 
on $\fD_{[\nu]}=\Spf(D_{[\nu]})$ associated to the pullback of $\CF$ on 
$(\fX_{[\nu]}/R)_{\prism}$ by using Zariski cohomological descent
(Theorem \ref{th:CrystalCohqHiggsSimpl}). 
We prove also a variant for inverse systems of crystals 
(Theorem \ref{thm:CrystalCohqHiggsSimplProj}). Then we discuss
compatibility with scalar extension under the lifting of Frobenius
and with tensor products in Remarks \ref{rmk:CrystZarProjqDolSimpFrob} and 
\ref{rmk:CrystZarProjqDolProd}, 
and the functoriality with respect to $i_{\star}\colon \fX_{\star}\to\fY_{\star}$
and $\ut_{\star}$ (Propositions \ref{eq:FunctCrysProjQDRHP} and \ref{prop:PrismZarProjqdRFunctHC}).

\begin{remark}\label{rmk:ExistSimpEmbedPrEnv}
If $\fX$ is quasi-compact and smooth over $\Spf(R/\pq R)$, we can
construct $\fX_{\star}\to \fX$, $i_{\star}\colon \fX_{\star}\hookrightarrow\fY_{\star}$,
and $((B_{\star},J_{\star})/R,\ut_{\star})$ as above in the following way.
Since $\fX$ is quasi-compact, there exist 
an affine open covering $\fU_{\alpha}$ $(\alpha=1,\ldots, N)$ of $\fX$,
an $R$-closed immersion $i_{\alpha}\colon \fU_{\alpha}\hookrightarrow \fV_{\alpha}=\Spf(B_{\alpha})$
into a smooth $pR+\pq R$-adic affine formal scheme over $R$ for each $\alpha\in \N\cap [1,N]$,
and $pR+\pq R$-adic coordinates $\ut_{\alpha}=(t_{\alpha,i})_{i\in\Lambda}$ of $B_{\alpha}$
over $R$ (Definition \ref{def:formallyflat} (2)) for each $\alpha\in \N\cap [1,N]$
with a common totally ordered set $\Lambda$.
We may assume that the conormal sheaf of the reduction modulo $pR+\pq R$ of 
the closed immersion $i_{\alpha}$ 
is free for every $\alpha\in \N\cap [1,N]$.
Let $\fX_{[0]}=\Spf(A_{[0]})$ (resp.~$\fY_{[0]}=\Spf(B_{[0]})$)
be the disjoint union of $\fU_{\alpha}$ (resp.~$\fV_{\alpha}$) $(\alpha\in \N\cap [1,N])$,
and let $\ut_{[0]}=(t_{[0],i})_{i\in\Lambda}$ be the $pR+\pq R$-adic
coordinates of $B_{[0]}$ over $R$
defined by $t_{[0],i}\vert_{\fV_{\alpha}}=t_{\alpha,i}$. 
The \v{C}ech nerve $\fX_{\star}=([\nu]\mapsto \fX_{[\nu]}=\Spf(A_{[\nu]}))$ of 
$\fX_{[0]}$ over $\fX$
is a Zariski hypercovering of $\fX$.
Since $\fX$ is separated, $\fX_{[\nu]}$ $(\nu\in \N)$ are affine,
and the closed immersions $i_{\alpha}$ $(\alpha\in\N\cap [1,N])$ induce
a closed immersion of simplicial formal schemes $i_{\star}$ of
$\fX_{\star}$ into the \v{C}ech nerve $\fY_{\star}=([\nu]\mapsto \fY_{[\nu]}=\Spf(B_{[\nu]}))$ 
of $\fY_{[0]}$ over $R$. For each $\nu\in \N$, the $R$-algebra $B_{[\nu]}$
with the ideal $J_{[\nu]}=\Ker(i_{[\nu]}^*\colon B_{[\nu]}\to A_{[\nu]})$ has a $q$-prismatic
envelope by Proposition \ref{prop:bddPrismEnvSmSmEmb} and the assumption on $i_{\alpha}$.
For $\nu\in \N$, let $\Lambda_{[\nu]}$  be the set $[\nu]\times\Lambda$ equipped
with the lexicographic order. 
Then $\Lambda_{[\nu]}$ $(\nu\in\N)$
naturally form a cosimplicial totally ordered set $\Lambda_{\star}$. 
We define the $pR+\pq R$-adic coordinates $\ut_{[\nu]}=(t_{[\nu],(l,i)})_{(l,i)\in \Lambda_{[\nu]}}$
of $B_{[\nu]}$ over $R$ by letting $t_{[\nu],(l,i)}$ 
be the inverse image of $t_{[0],i}$ by the projection
to the $(\ell+1)$th component $\fY_{[\nu]}\to \fY_{[0]}$.
Then the cosimplicial $R$-algebras $([\nu]\mapsto B_{[\nu]})_{\nu\in\N}$
and the sets of coordinates $\ut_{[\nu]}$ $(\nu\in \N)$ with the simplicial index set
$\Lambda_{\star}$ form a cosimplicial framed smooth 
$\delta$-$R$-algebra (Definition \ref{def:FramedSmoothPrism} (2)).
\end{remark}

Let $((\fX_{\star}/R)^{\sim}_{\prism}, \CO_{\fX_{\star}/R,n})$ and 
$(\fX_{\star,\Zar}^{\sim}, \CO_{\fX_{\star}})$ be the ringed topos associated 
to the simplicial ringed topos 
$([\nu]\mapsto ((\fX_{[\nu]}/R)^{\sim}_{\prism},\CO_{\fX_{[\nu]}/R,n}))_{\nu\in\N}$ and 
the simplicial Zariski topos $([\nu]\mapsto (\fX^{\sim}_{[\nu],\Zar},\CO_{\fX_{[\nu]}}))_{\nu\in\N}$, respectively,  
defined by the simplicial $p$-adic formal scheme $\fX_{\star}$ over $\Spf(R)$,  
and let $u_{\fX_{\star}/R}\colon (\fX_{\star}/R)_{\prism}^{\sim}
\to \fX^{\sim}_{\star,\Zar}$ be the morphism of topos induced
by the morphism of simplicial topos
$([\nu]\mapsto u_{\fX_{[\nu]}/R}\colon (\fX_{[\nu]}/R)^{\sim}_{\prism}\to \fX^{\sim}_{[\nu],\Zar})$. 
Let $\theta_{\prism}$ and $\theta$ denote the morphisms of 
ringed topos $((\fX_{\star}/R)^{\sim}_{\prism},\CO_{\fX_{\star}/R})\to 
((\fX/R)^{\sim}_{\prism},\CO_{\fX/R})$
and $(\fX_{\star,\Zar}^{\sim},\CO_{\fX_{\star}})\to (\fX_{\Zar}^{\sim},\CO_{\fX})$
(\cite[V\bis (2.2.1)]{SGA4}). 
We have $\theta_{\prism}^{-1}(\CO_{\fX/R,n})=\CO_{\fX_{\star}/R,n}$
and $\theta^{-1}(\CO_{\fX})=\CO_{\fX_{\star}}$. 

The $q$-prismatic envelopes $D_{[\nu]}$ of $(B_{[\nu]},J_{[\nu]})$ $(\nu\in \N)$
form a cosimplicial $q$-prism $D_{\star}=([\nu]\mapsto D_{[\nu]})_{\nu\in \N}$ over
$R$ and define simplicial $R$-formal schemes 
$\fD_{\star}=([\nu]\mapsto \fD_{[\nu]}=\Spf(D_{[\nu]}))_{\nu\in \N}$, 
$\ofD_{\star}=([\nu]\mapsto \ofD_{[\nu]}=\Spf(D_{[\nu]}/[p]_qD_{[\nu]}))_{\nu\in \N}$,
and simplicial $R$-schemes
$\fD_{\star,n}=([\nu]\mapsto \fD_{[\nu],n}=\Spec(D_{[\nu]}/(\pq,p)^{n+1}))_{\nu\in \N}$
$(n\in\N)$. 
Let $\ofD_{\star,\Zar}^{\sim}$ be the topos associated the simplicial 
topos $([\nu]\mapsto \ofD_{[\nu],\Zar})_{\nu\in \N}$.
The $R$-homomorphisms $A_{[\nu]}\cong B_{[\nu]}/J_{[\nu]} \to D_{[\nu]}/[p]_qD_{[\nu]}$
define a morphism of simplicial formal schemes
$v_{D_{\star}}=([\nu]\mapsto v_{D_{[\nu]}}\colon \ofD_{[\nu]}\to \fX_{[\nu]})_{\nu\in \N}
\colon \ofD_{\star}\to \fX_{\star}$
over $R$, which induces a morphism of topos 
$v_{D_{\star},\Zar}\colon \ofD_{\star, \Zar}^{\sim}\to \fX_{\star,\Zar}^{\sim}$. 

Let $\CF$ be a crystal  of $\CO_{\fX/R,n}$-modules on $(\fX/R)_{\prism}$,
and let $\CF_{\star}$ be the $\CO_{\fX_{\star}/R,n}$-module
$\theta_{\prism}^{-1}\CF$, which consists of the pull-back $\CF_{[\nu]}=g_{[\nu], \prism}^{-1}(\CF)$ of
$\CF$ under the morphism $g_{[\nu]}\colon \fX_{[\nu]}\to \fX$ for each $\nu\in \N$,
and the canonical isomorphism $g_{s,\prism}^{-1}(\CF_{[\nu]})
\xrightarrow{\cong} (g_{[\nu]}\circ g_s)^{-1}_{\prism}(\CF)
=g_{[\lambda],\prism}^{-1}(\CF)=\CF_{[\lambda]}$
for the morphism $g_s\colon \fX_{[\lambda]}\to \fX_{[\nu]}$
corresponding to  each non-decreasing map $s\colon [\nu]\to [\lambda]$.
By applying the resolution \eqref{eq:qdRResol}, the morphism of complexes 
\eqref{eq:qdRCpxLinZarProj}, and 
Proposition \ref{prop:LinearizationLocCoh} 
to $i_{[\nu]}\colon \fX_{[\nu]}\hookrightarrow \fY_{[\nu]}=\Spf(B_{[\nu]})$,
$((B_{[\nu]},J_{[\nu]})/R,\ut_{[\nu]})$, and $\CF_{[\nu]}$ for each $\nu\in \N$, 
and then using the functoriality of \eqref{eq:qdRResol} and \eqref{eq:qdRCpxLinZarProj}:
\eqref{eq:CrysqdRResolFunct} and Proposition \ref{eq:LindRCpxZarProjFunct},
we obtain a resolution 
of $\CO_{\fX_{\star}/R,n}$-modules
on $(\fX_{\star}/R)_{\prism}^{\sim}$
\begin{equation}\label{eq:CrystalqDRResolHC}
\CF_{\star}\longrightarrow 
\CL_{A_{\star},B_{\star}}(M_{\star}\otimes_{D_{\star}}q\Omega^{\bullet}_{D_{\star}/R}),
\end{equation}
an isomorphism in $C^+(\fX_{\star,\Zar}^{\sim},R)$
\begin{equation}\label{eq:LinqdRCpxProjHC}
v_{D_{\star*}}(\CM_{\star}\otimes_{\CO_{\fD_{\star}}}q\Omega_{\fD_{\star}/R}^{\bullet})
\xrightarrow{\quad\cong\quad}
u_{\fX_{\star}/R*}(\CL_{A_{\star},B_{\star}}(M_{\star}\otimes_{D_{\star}}q\Omega_{D_{\star}/R}^{\bullet})),
\end{equation}
and an isomorphism in  $D^+(\fX_{\star,\Zar}^{\sim},R)$
\begin{equation}\label{eq:LinqdRCpxProjHC2}
u_{\fX_{\star}/R*}(\CL_{A_{\star},B_{\star}}(M_{\star}\otimes_{D_{\star}}q\Omega_{D_{\star}/R}^{\bullet}))
\xrightarrow{\quad\cong\quad}
Ru_{\fX_{\star}/R*}(\CL_{A_{\star},B_{\star}}(M_{\star}\otimes_{D_{\star}}q\Omega_{D_{\star}/R}^{\bullet})).
\end{equation}
Note that $Ru_{\fX_{\star}/R*}$ can be computed on each component $\fX_{[\nu]}$ by
\cite[V\bis Corollaire (1.3.12)]{SGA4}. 
The composition of \eqref{eq:LinqdRCpxProjHC},
\eqref{eq:LinqdRCpxProjHC2}, and $Ru_{\fX_{\star}/R*}$\eqref{eq:CrystalqDRResolHC}
yields an isomorphism in $D^+(\fX_{\star,\Zar}^{\sim},R)$
\begin{equation}\label{eq:CrysZarProjDolCpxHC}
v_{D_{\star*}}(\CM_{\star}\otimes_{\CO_{\fD_{\star}}}q\Omega_{\fD_{\star}/R}^{\bullet})
\xrightarrow{\quad\cong\quad}
Ru_{\fX_{\star}/R*}\CF_{\star}
\end{equation}
This allows us to glue \eqref{eq:CrystalCohqHiggs} as follows.

\begin{theorem}\label{th:CrystalCohqHiggsSimpl}
Under the notation and assumption above, we have the following canonical isomorphism
in $D^+(\fX_{\Zar},R)$ for a crystal of $\CO_{\fX/R,n}$-modules $\CF$ on $(\fX/R)_{\prism}$
functorial in $\CF$. 
\begin{equation}\label{eq:CrystalCohqHiggsSimpl}
Ru_{\fX/R*}\CF\cong (g_{\star*}(v_{D_{\star}*}(\CM_{\star}\otimes_{\CO_{\fD_{\star}}}
q\Omega^{\bullet}_{\fD_{\star}/R})))_s
\end{equation}
Here $[\nu]\mapsto 
g_{[\nu]*}(v_{D_{[\nu]}*}(\CM_{[\nu]}\otimes_{\CO_{\fD_{[\nu]}}}
q\Omega^{\bullet}_{D_{[\nu]}/R}))$ forms a cosimplicial complex,
which defines a complex of complexes, and $(-)_s$ in the right-hand
side denotes its associated simple complex.
\end{theorem}

\begin{proof}
Since $\fX_{\star}$ is a Zariski hypercovering of $\fX$ and $Ru_{\fX_{\star}/R*}$
can be computed on each $\fX_{[\nu]}$,  the base change morphism 
$L\theta^{-1}Ru_{\fX/R*}\CF\to Ru_{\fX_{\star}/R*}\CF_{\star}$
is an isomorphism in $D^+(X_{\star,\Zar}^{\sim},R)$. 
By cohomological descent \cite[V\bis Proposition (3.2.4), Proposition (3.3.1) a), Th\'eor\`eme (3.3.3)]{SGA4}, 
the morphism $Ru_{\fX/R*}\CF\to R\theta_*L\theta^{-1}(Ru_{\fX/R*}\CF)$
is an isomorphism in $D^+(\fX_{\Zar},R)$. By combining these with 
$R\theta_*$\eqref{eq:CrysZarProjDolCpxHC}, we obtain an isomorphism 
\begin{equation}
Ru_{\fX/R*}\CF\cong 
R\theta_*(v_{D_{\star*}}(\CM_{\star}\otimes_{\CO_{\fD_{\star}}}q\Omega_{\fD_{\star}/R}^{\bullet}))
\end{equation}
in $D^+(\fX_{\Zar},R)$. 
Let $c(\fX)_{\star}$ denote the constant simplicial formal scheme $([\nu]\mapsto \fX)_{\nu\in\N}$.
Then the morphism of topos $\theta\colon \fX_{\star,\Zar}^{\sim}\to \fX_{\Zar}^{\sim}$
decomposes as 
$\fX_{\star,\Zar}^{\sim}\xrightarrow{g_{\star,\Zar}} c(\fX)^{\sim}_{\star,\Zar}
\xrightarrow{\varepsilon} \fX^{\sim}_{\Zar}$. The derived functor
$R\varepsilon_*\colon D^+(c(\fX)_{\star,\Zar}^{\sim},R)\to D^+(\fX^{\sim}_{\Zar},R)$ is known
to be canonically isomorphic to the functor induced by the 
exact functor $K^+(c(\fX)_{\star,\Zar}^{\sim},R)
\to K^+(\fX^{\sim}_{\Zar},R)$ sending a cosimplicial complex 
$\CF_{\star}^{\bullet}=([\nu]\mapsto \CF_{[\nu]}^{\bullet})_{\nu\in\N}$ of
$R$-modules on $\fX_{\Zar}$ to the simple complex associated
to the complex of complexes 
$\CF_{[0]}^{\bullet}\to \CF_{[1]}^{\bullet}\to\cdots
\to \CF_{[\nu]}^{\bullet}\to \cdots$ defined by 
$\CF_{\star}^{\bullet}$ (\cite[V\bis Proposition (2.3.5), (2.3.9)]{SGA4}). 
Since the higher direct
image under $g_{\star,\Zar}$ can be computed for each $g_{[\nu],\Zar}$, 
it now remains to show the vanishing 
$R^rg_{[\nu],\Zar*}(v_{D_{[\nu]}, \Zar*}(\CN))=0$ $(r>0)$
for a quasi-coherent $\CO_{\fD_{[\nu],n}}$-module $\CN$.
This is reduced to the case $n=0$, where the claim
holds because the morphisms $v_{D_{[\nu]}}\colon 
\ofD_{[\nu]}\to \fX_{[\nu]}$ and $g_{[\nu]}\colon \fX_{[\nu]}\to \fX$ are
 affine as 
$\fX$ is assumed to be separated.
\end{proof}

Similarly to Theorem \ref{thm:CrystalCohqHigProj}, the following theorem holds for an
inverse system of crystals.

\begin{theorem}\label{thm:CrystalCohqHiggsSimplProj}
Let $\CF$ be a complete crystal of $\CO_{\fX/R}$-modules
(Definition \ref{def:PrismaticSite} (2)), and let
$(\CF_n)_{n\in \N}$ be the adic inverse system of
crystals of $\CO_{\fX/R,n}$-modules on $(\fX/R)_{\prism}$
corresponding to $\CF$ (Remark \ref{rmk:CompleteCrystalInvSystem} (3), (4)).
Let $\CF_{\star}$ be 
the complete crystal of $\CO_{\fX_{\star}/R}$-modules
$\theta^{-1}_{\prism}\CF$, let $(\CF_{\star,n})_{n\in\N}$ be 
the corresponding adic inverse system of crystals of
$\CO_{\fX_{\star}/R,n}$-modules, which coincides with 
$(\theta^{-1}_{\prism}\CF_n)_{n\in\N}$,
and let $(\CM_{\star,n}\otimes_{\CO_{\fD_{\star}}}
q\Omega^{\bullet}_{\fD_{\star}/R})_{n\in\N}$ be the inverse
system of complexes of $R$-modules on $\ofD_{\star,\Zar}$
associated to $(\CF_{\star,n})_{n\in \N}$ by 
\eqref{eq:qHiggsShfCpx} and \eqref{eq:qDolShfCpxFunct}. 
Then, we have the 
following isomorphism in $D^+(\fX_{\Zar},R)$ functorial
in $\CF$.
\begin{equation}\label{eq:CrystalCohqHiggsSimplProj}
Ru_{\fX/R*}\CF
\cong(g_{\star*}(v_{D_{\star}*}
(\varprojlim_n\CM_{\star,n}\otimes_{\CO_{\fD_{\star}}}q\Omega^{\bullet}_{\fD_{\star}/R})))_s
\end{equation}
Here the notation $(-)_s$ is the same as Theorem \ref{th:CrystalCohqHiggsSimpl}.
\end{theorem} 

\begin{proof}
By applying the argument in the proof of Theorem \ref{thm:CrystalCohqHigProj} to the
inverse system of crystals of $\CO_{\fX_{[\nu]}/R,n}$-modules $(\CF_{[\nu],n})_{n\in\N}$
on $(\fX_{[\nu]}/R)_{\prism}$,  
$i_{[\nu]}\colon \fX_{[\nu]}\hookrightarrow \fY_{[\nu]}=\Spf(B_{[\nu]})$,
and $\ut_{[\nu]}\in B_{[\nu]}^{\Lambda_{[\nu]}}$ for each $\nu\in\N$, we see that the
resolution of inverse systems of $\CO_{\fX_{\star}/R,n}$-modules
$ (\CF_{\star,n})_{n\in \N}\to 
(\CL_{A_{\star},B_{\star}}(M_{\star,n}\otimes_{D_{\star}}
q\Omega^{\bullet}_{D_{\star}/R}))_{n\in\N}$
defined by \eqref{eq:CrystalqDRResolHC} induces a resolution
$\CF_{\star}\cong\varprojlim_n\CF_{\star,n}\to \varprojlim_n \CL_{A_{\star},B_{\star}}
(M_{\star,n}\otimes_{D_{\star}}q\Omega^{\bullet}_{D_{\star}/R})$.
By \eqref{eq:qdRCpxLinZarProj} and Proposition \ref{eq:LindRCpxZarProjFunct}, we obtain
a morphism of inverse systems of complexes of $R$-modules on 
$\fX_{\star,\Zar}^{\sim}$
$$(v_{D_{\star}*}(\CM_{\star,n}\otimes_{\CO_{\fD_{\star}}}
q\Omega^{\bullet}_{\fD_{\star}/R}))_{n\in\N}
\longrightarrow (u_{\fX_{\star}/R*}(\CL_{A_{\star},B_{\star}}
(M_{\star,n}\otimes_{D_{\star}}q\Omega^{\bullet}_{D_{\star}/R})))_{n\in\N}.$$
By Proposition \ref{prop:LinLocCohProjlim},
we obtain an isomorphism 
\begin{equation}\label{eq:LinqdRCpxProjHCProj}
Ru_{\fX_{\star}/R*}\CF_{\star}
\cong v_{D_{\star}*}(\varprojlim_n (\CM_{\star,n}\otimes_{\CO_{\fD_{\star}}}
q\Omega^{\bullet}_{\fD_{\star}/R}))\end{equation} 
since $Ru_{\fX_{\star}/R*}$ can be computed on each component
$\fX_{[\nu]}$ as mentioned after \eqref{eq:LinqdRCpxProjHC2}.
By the 
same argument as the proof of Theorem \ref{th:CrystalCohqHiggsSimpl} deriving
\eqref{eq:CrystalCohqHiggsSimpl} from \eqref{eq:CrysZarProjDolCpxHC},
we obtain \eqref{eq:CrystalCohqHiggsSimplProj} from \eqref{eq:LinqdRCpxProjHCProj}
by applying Lemma \ref{lem:SimplPrismEnvZarProj} below to the inverse system of 
quasi-coherent $\CO_{\fD_{[\nu],n}}$-modules
$(\CM_{[\nu],n}\otimes_{\CO_{\fD_{[\nu]}}}q\Omega^r_{\fD_{[\nu]}/R})_{n\in \N}$
for each $\nu\in \N$ and $r\in \N$.
\end{proof}

\begin{lemma}\label{lem:SimplPrismEnvZarProj}
 Let $\nu\in \N$, 
and let $(\CN_n)_{n\in \N}$ be an inverse system of 
quasi-coherent $\CO_{\fD_{[\nu],n}}$-modules $\CN_n$ on $\ofD_{[\nu],\Zar}$
such that $\CN_{n+1}\to \CN_n$ is an epimorphism
for every $n\in \N$. Then the morphism 
$g_{[\nu]*}(v_{D_{[\nu]}*}\varprojlim_n\CN_n)
\to Rg_{[\nu]*}(v_{D_{[\nu]}*}\varprojlim_n\CN_n)$
is an isomorphism.
\end{lemma}

\begin{proof}
We abbreviate $g_{[\nu]}$ and $v_{D_{[\nu]}}$ to 
$g$ and $v$, respectively, to simplify the notation. 
For $n\in \N$ and a quasi-coherent $\CO_{\fD_{[\nu],n}}$-module $\CL$,
we see $v_*\CL\xrightarrow{\cong}Rv_*\CL$
and $\Gamma(\fU,\CL)\xrightarrow{\cong}R\Gamma(\fU,\CL)$
for any open affine formal subscheme $\fU\subset\fD_{[\nu]}$
by induction on $n$ since the morphism $v\colon 
\ofD_{[\nu]}\to \fX_{[\nu]}$ is affine.
Let $\fV$ be an open affine formal subscheme of $\fX_{[\nu]}$,
and let $\fU$ be $v^{-1}(\fV)$, which is an affine open 
of $\ofD_{[\nu]}$. Then
we have 
$R\Gamma(\fV,v_*\CN_n)
=R\Gamma(\fV,Rv_*\CN_n)
=R\Gamma(\fU,\CN_n)
=\Gamma(\fU,\CN_n)
=\Gamma(\fV,v_*\CN_n)$,
and the homomorphism 
$\Gamma(\fV,v_{*}\CN_{n+1})
\to\Gamma(\fV,v_{*}\CN_n)$
is surjective since $\Gamma(\fV,v_{*}\CN_m)
=\Gamma(\fU,\CN_m)$ $(m=n,n+1)$
and $H^1(\fU,\Ker[\CN_{n+1}\to \CN_n])=0$.
Hence $R\varprojlim_n(v_{*}\CN_n)
\cong \varprojlim_n (v_{*}\CN_n)$
(\cite[Lemma IV.4.2.3]{AGT}). Similarly
we see 
$(g\circ v)_*\CL
\cong R(g\circ v)_*\CL$
for a quasi-coherent $\CO_{\fD_{[\nu]},n}$-module $\CL$,
and 
$R\varprojlim_n(g\circ v)_*\CN_n
\cong \varprojlim_n(g\circ v)_*\CN_n$
since the composition $g\circ v$ is
also affine as $\fX$ is assumed to be separated. 
We obtain the desired isomorphism
as 
$g_*(v_*\varprojlim_n\CN_n)
\cong \varprojlim_n (g\circ v)_*\CN_n
\cong R\varprojlim_n (g\circ v)_*\CN_n
\cong R\varprojlim_n R(g\circ v)_*\CN_n
\cong R\varprojlim_n Rg_{*}(Rv_{*}\CN_n)
\cong R\varprojlim_n Rg_{*}(v_{*}\CN_n)
\cong Rg_{*}R\varprojlim_n(v_{*}\CN_n)
\cong Rg_{*}\varprojlim_n(v_{*}\CN_n)
\cong Rg_{*} (v_{*}\varprojlim_n\CN_n).
$
\end{proof}

\begin{remark}\label{rmk:CrystZarProjqDolSimpFrob}
(1) We keep the notation and assumption in Theorem \ref{th:CrystalCohqHiggsSimpl}.
Let $\varphi_n^*\CF$ be $\CF\otimes_{\CO_{\fX/R,n},\varphi_n}\CO_{\fX/R,n}$
as Remark \ref{rmk:crystalFrobPBTensor} (1). We have 
$\theta_{\prism}^{-1}(\varphi_n^*\CF)\cong\varphi_n^*\CF_{\star}
=\CF_{\star}\otimes_{\CO_{\fX_{\star}/R,n},\varphi_n}\CO_{\fX_{\star}/R,n}$.
We write $q\Omega(\CM_{\star})$ and $q\Omega(\varphi_{\fD_{\star,n}}^*\CM_{\star})$
for the complex $\CM_{\star}\otimes_{\CO_{\fD_{\star}}}q\Omega_{\fD_{\star}/R}^{\bullet}$
and $\varphi_{\fD_{\star,n}}^*\CM_{\star}\otimes_{\CO_{\fD_{\star}}}
q\Omega_{\fD_{\star}/R}^{\bullet}$ associated to 
$\CF_{\star}$ and $\varphi_n^*\CF_{\star}$, respectively. 
Let $\CF_{\star}\to\CL\Omega(M_{\star})$
and $\varphi_n^*\CF_{\star}\to \CL\Omega(\varphi_{D_{\star,n}}^*M_{\star})$
denote the resolution \eqref{eq:CrystalqDRResolHC} of $\CF_{\star}$
and $\varphi_n^*\CF_{\star}$. Then we have the following commutative
diagrams.
\begin{equation}\label{eq:LinResolInvSysSimpFrob}
\xymatrix@C=40pt@R=20pt{
\CF_{\star}\ar[r]^(.45){\eqref{eq:CrystalqDRResolHC}}\ar[d]_{\sigma_{\CF_{\star},n}}
&\CL\Omega(M_{\star})\ar[d]^{\CL(\varphi_{D_{\star,n},\Omega}^{\bullet}(M_{\star}))}
\\
\varphi_n^*\CF_{\star}\ar[r]^(.39){\eqref{eq:CrystalqDRResolHC}}&
\CL\Omega(\varphi_{D_{\star,n}}^*M_{\star})
}
\xymatrix@C=50pt@R=20pt{
v_{D_{\star}*}q\Omega(\CM_{\star})\ar[r]^{\cong}_{\eqref{eq:LinqdRCpxProjHC}}
\ar[d]^{v_{D_{\star}*}(\varphi^{\bullet}_{\fD_{\star,n},\Omega}(\CM_{\star}))}&
u_{\fX_{\star}/R*}\CL\Omega(M_{\star})
\ar[d]^{u_{\fX_{\star}/R*}\CL(\varphi_{D_{\star,n},\Omega}^{\bullet}(M_{\star}))}\\
v_{D_{\star}*}q\Omega(\varphi_{\fD_{\star,n}}^*\CM_{\star})
\ar[r]^{\cong}_{\eqref{eq:LinqdRCpxProjHC}}&
u_{\fX_{\star}/R*}\CL\Omega(\varphi_{D_{\star,n}}^*M_{\star})
}
\end{equation}
The left vertical map in the left diagram is the natural
$\varphi_n$-semilinear map (Remark \ref{rmk:LinResolTransFrobComp} (1)),
and the vertical morphism from $\CL\Omega(M_{\star})$
(resp.~$q\Omega(\CM_{\star})$) is defined by 
Remarks \ref{rmk:LinqdRCpxFrob} (1) and \ref{rmk:LinPBFrobComp} (1)
(resp.~\ref{rmk:qDolSheafCpxFrob} (1) and \ref{rmk:qDolCpxShfFbFunct} (1));
the left (resp.~right diagram) is commutative by 
Remark \ref{rmk:LinResolTransFrobComp} (1) (resp.~\ref{rmk:LinZarProjFrobComp} (1)).
Hence, by going back to the construction of \eqref{eq:CrystalCohqHiggsSimpl}
via \eqref{eq:CrystalqDRResolHC}, \eqref{eq:LinqdRCpxProjHC}, and \eqref{eq:LinqdRCpxProjHC2}, 
we can verify that the following diagram commutes.
\begin{equation}
\xymatrix@C=50pt@R=20pt{
Ru_{\fX_{\star}/R*}\CF\ar[r]^(.32){\cong}_(.32){\eqref{eq:CrystalCohqHiggsSimpl}}
\ar[d]_{Ru_{\fX_{\star}/R*}\sigma_{\CF,n}}&
(g_{\star*}(v_{D_{\star}*}(\CM_{\star}\otimes_{\CO_{\fD_{\star}}}q\Omega_{\fD_{\star}/R}^{\bullet})))_s
\ar[d]^{(g_{\star*}(v_{D_{\star}*}(\varphi_{\fD_{\star,n},\Omega}^{\bullet}(\CM_{\star}))))_s}
\\
Ru_{\fX_{\star}/R*}\varphi_n^*\CF\ar[r]^(.31){\cong}_(.31){\eqref{eq:CrystalCohqHiggsSimpl}}&
(g_{\star*}(v_{D_{\star}*}(\varphi_{\fD_{\star,n}}^*\CM_{\star}\otimes_{\CO_{\fD_{\star}}}q\Omega_{\fD_{\star}/R}^{\bullet})))_s
}
\end{equation}

(2) We follow the notation and assumption in Theorem \ref{thm:CrystalCohqHiggsSimplProj} 
and its proof. Then the morphisms $\sigma_{\CF_{\star,n},n}$, 
$\varphi_{\fD_{\star,n},\Omega}(\CM_{\star,n})$, and
$\CL_{A_{\star},B_{\star}}(\varphi^{\bullet}_{D_{\star,n},\Omega}(M_{\star,n}))$
obtained by applying the constructions in (1) above to $\CF_n$,
define morphisms of inverse systems with respect to $n$ by 
\eqref{eq:qdRcpxFrobPBFunct}. Therefore, by the construction of 
\eqref{eq:CrystalCohqHiggsSimplProj}, we see that the commutative diagrams
\eqref{eq:LinResolInvSysSimpFrob} for $\CF_n$ $(n\in \N)$ imply that the following diagram commutes.
Here $\hvarphi^*\CF$ is defined as in Remark \ref{rmk:crystalFrobPBTensor} (1), i.e.,
the object corresponding to $(\varphi_n^*\CF_n)_{n\in\N}\in 
\Ob(\Crystal_{\prism}^{\ad}(\CO_{\fX/R,\bullet}))$ 
via \eqref{eq:CompleteCrysInvSys}.

\begin{equation}
\xymatrix@C=40pt@R=20pt{
Ru_{\fX/R*}\CF
\ar[r]^(.3){\cong}_(.3){\eqref{eq:CrystalCohqHiggsSimplProj}}
\ar[d]_{Ru_{\fX/R*}\varprojlim_n\sigma_{\CF_n,n}}&
(g_{\star*}(v_{D_{\star}*}(\varprojlim_n\CM_{\star,n}\otimes_{\CO_{\fD_{\star}}}
q\Omega^{\bullet}_{\fD_{\star}/R})))_s
\ar[d]^{(g_{\star*}v_{D_{\star*}}\varprojlim_n\varphi^{\bullet}_{\fD_{\star,n},\Omega}(\CM_{\star,n}))_s}
\\
Ru_{\fX/R*}\hvarphi^*\CF
\ar[r]^(.3){\cong}_(.3){\eqref{eq:CrystalCohqHiggsSimplProj}}&
(g_{\star*}(v_{D_{\star}*}(\varprojlim_n\varphi_{\fD_{\star,n}}^*\CM_{\star,n}\otimes_{\CO_{\fD_{\star}}}
q\Omega^{\bullet}_{\fD_{\star}/R})))_s
}
\end{equation}
\end{remark}
\begin{remark}\label{rmk:CrystZarProjqDolProd}
(1) 
Since the product morphism of $q$-Higgs complexes
\eqref{eq:qDolShfCpxProd} is functorial with respect to $\fX\to \fY$
and $(t_i)_{i\in\Lambda}$ only when a map between index sets
of coordinates is injective (Remark \ref{rmk:qDolCpxShfFbFunct} (2)),
the complexes $\CM_{\star}\otimes_{\CO_{\fD_{\star}}}q\Omegab_{\fD_{\star}/R}$
on $\fD_{\star,\Zar}$ associated to $\CF$'s do not have product morphisms as
\eqref{eq:qDolShfCpxProd} in general. However, thanks to Remark 
\ref{rmk:qdRModfProdFunct}, we have a product whose codomain is 
the $q$-Higgs complex with respect to the diagonal immersion 
$\fX_{\star}\to \fY_{\star}\times_{\Spf(R)}\fY_{\star}$
which is compatible with the isomorphism \eqref{eq:CrysZarProjDolCpxHC}
as follows.\par
By applying some constructions given in Remarks \ref{rmk:LinPBProdComp2} 
and \ref{rmk:qdRModfProdFunct} to
$i_{[\nu]}\colon \fX_{[\nu]}\to \fY_{[\nu]}$, $\ut_{[\nu]}$ $([\nu]\in \Ob \Delta)$
and the morphisms among them associated to morphisms of the category $\Delta$,
we obtain a simplicial smooth $pR+\pq R$-adic affine formal scheme 
$\fY_{\star}(1)=\fY_{\star}\times_{\Spf(R)}\fY_{\star}$ over $R$,
a closed immersion $i_{\star}(1)\colon \fX_{\star}\hookrightarrow \fY(1)_{\star}$ over $R$,
a cosimplicial framed smooth $q$-pairs 
$((B_{\star}(1),J_{B_{\star}(1)}),\Lambda_{\star}(1),\ut^{(1)}_{\star})$,
a cosimplicial $q$-prism $(D_{\star}(1),\pq D_{\star}(1))$, simplicial
formal schemes $\fD_{\star}(1)$ and $\ofD_{\star}(1)$, a morphism 
$v_{D_{\star}(1)}\colon \ofD_{\star}(1)\to \fX_{\star}$, and
morphisms $p_{l,\fD_{\star}}\colon \fD_{\star}(1)\to \fD_{\star}$
$(l=0,1)$
compatible with $v_{D_{\star}}$ and $v_{D_{\star}(1)}$. \par
Let $\CF_l$ $(l=0,1)$ be crystals of $\CO_{\fX/R,n}$-modules
on $(\fX/R)_{\prism}$, set $\CF=\CF_0\otimes_{\CO_{\fX/R,n}}\CF_1$,
and let $\CF_{l\star}$ and $\CF_{\star}$ denote their pullbacks
on $(\fX_{\star}/R)_{\prism}$ under $\theta_{\prism}$. 
Then, by applying the constructions of \eqref{eq:CrystalqDRResolHC} and 
\eqref{eq:LinqdRCpxProjHC} to $\CF_l$ and $(i_{\star}\colon \fX_{\star}\to \fY_{\star},\ut_{\star})$
(resp.~$\CF$ and $(i_{\star}(1)\colon \fX_{\star}\to \fY_{\star}(1),\ut_{\star}^{(1)})$),
we obtain complexes
$\CL_{A_{\star},B_{\star}}(M_{l\star}\otimes_{D_{\star}}q\Omegab_{D_{\star}/R})$
and $\CM_{l\star}\otimes_{\CO_{\fD_{\star}}}q\Omegab_{\fD_{\star}/R}$
(resp.~
$\CL_{A_{\star},B_{\star}(1)}(M(1)_{\star}\otimes_{D_{\star}(1)}q\Omegab_{D_{\star}(1)/R})$
and $\CM(1)_{\star}\otimes_{\CO_{\fD_{\star}(1)}}q\Omegab_{\fD_{\star}(1)/R}$),
which we abbreviate to $\CL\Omegab(M_{l\star})$ and
$q\Omegab(\CM_{l\star})$ (resp.~$\CL\Omegab(M(1)_{\star})$
and $q\Omegab(\CM(1)_{\star})$).  By considering \eqref{eq:LinqdRCpxMdfProd} 
and \eqref{eq:qDolMdfProd} associated to $(i_{[\nu]},\ut_{[\nu]})$ and $\CF_{l[\nu]}$
for each $\nu\in \N$, and using \eqref{eq:LinqDolMdfProdFunct} and \eqref{eq:qdRCpxMdfProdFunct}
for each morphism in the category $\Delta$, we obtain morphisms
\begin{align}
&\CL\Omegab(M_{0\star})\otimes_{\CO_{\fX_{\star}/R},n}\CL\Omegab(M_{1\star})
\longrightarrow \CL\Omegab(M(1)_{\star}),
\label{eq:SimpLinqDolbProd}\\
&v_{D_{\star}*}q\Omegab(\CM_{0\star})\otimes_Rv_{D_{\star}*}q\Omegab(\CM_{1\star})
\longrightarrow 
v_{D_{\star}(1)*}q\Omegab(\CM(1)_{\star}).
\label{eq:SimpDolbProd}
\end{align}
By Remark \ref{rmk:ResolMdfProdComp}, the 
composition of \eqref{eq:SimpLinqDolbProd}
with the morphism 
$\CF_{\star}=\CF_{0\star}\otimes_{\CO_{\fX_{\star}/R},n}\CF_{1\star}
\to\CL\Omegab(M_{0\star})\otimes_{\CO_{\fX_{\star}/R,n}}
\CL\Omegab(M_{1\star})$
induced by \eqref{eq:CrystalqDRResolHC} for 
$(i_{\star},\ut_{\star})$ and $\CF_l$ $(l=0,1)$
coincides with the resolution \eqref{eq:CrystalqDRResolHC}
of $\CF_{\star}$ with respect to $(i_{\star}(1),\ut_{\star}^{(1)})$. 
On the other hand, the morphisms \eqref{eq:SimpLinqDolbProd} 
and \eqref{eq:SimpDolbProd} are compatible with the
morphisms \eqref{eq:LinqdRCpxProjHC} for 
$((i_{\star},\ut_{\star}),\CF_{l\star})$ and
$((i_{\star}(1),\ut_{\star}^{(1)}),\CF_{\star})$ by Remark \ref{rmk:LinqDolqDolMdfProdComp}.
Therefore, by the construction of \eqref{eq:CrysZarProjDolCpxHC},
we see that the following diagram is commutative.
\begin{equation}
\xymatrix@R=20pt@C=50pt{
Ru_{\fX_{\star}/R*}\CF_{0\star}\otimes^L_RRu_{\fX_{\star}/R*}\CF_{1\star}
\ar[r]\ar[d]_{\cong}^{\eqref{eq:CrysZarProjDolCpxHC}}&
Ru_{\fX_{\star}/R*}\CF_{\star}
\ar[d]_{\cong}^{\eqref{eq:CrysZarProjDolCpxHC}}\\
v_{D_{\star}*}q\Omegab(\CM_{0\star})
\otimes_R^Lv_{D_{\star}*}q\Omegab(\CM_{1\star})
\ar[r]^(.6){\eqref{eq:SimpDolbProd}}&
v_{D_{\star}(1)*}q\Omegab(\CM(1)_{\star}).
}
\end{equation}

(2) We can prove compatibility of the isomorphism \eqref{eq:LinqdRCpxProjHCProj}
with products similarly to (1) as follows. We keep the notation introduced in the second
paragraph of (1). Let $\CF_l$ $(l=0,1)$ be complete crystals of $\CO_{\fX/R}$-modules
on $(\fX/R)_{\prism}$, and let $\uCF_l=(\CF_{l,n})_{n\in\N}$ $(l=0,1)$ be the 
adic inverse systems of crystals of  $\CO_{\fX/R,n}$-modules on $(\fX/R)_{\prism}$
corresponding to $\CF_l$ by the equivalence \eqref{eq:CompleteCrysInvSys}.
Put $\CF=\CF_1\hotimes_{\CO_{\fX/R}}\CF_2$ and 
$\uCF=\uCF_1\otimes_{\CO_{\fX/R,\bullet}}\uCF_2$
(Remark \ref{rmk:crystalFrobPBTensor}  (2)).
Let $\CF_{l\star}$, $\CF_{\star}$, $\uCF_{l\star}$, and $\uCF_{\star}$ denote
the pullbacks of $\CF_l$, $\CF$, $\uCF_l$, and $\uCF$, respectively,
on $(\fX_{\star}/R)_{\prism}$ by $\theta_{\prism}$. Then
we obtain inverse systems of complexes
$\CL\Omegab(\uM_{l\star})$, $\CL\Omegab(\uM(1)_{\star})$,
$q\Omegab(\uCM_{l\star})$, and $q\Omegab(\uCM(1)_{\star})$
by applying the construction in the third paragraph of (1) to
$((i_{\star},\ut_{\star}),\CF_{l,n\star})$ and $((i_{\star}(1),\ut_{\star}^{(1)}),\CF_{n\star})$
for each $n\in \N$. The product morphisms \eqref{eq:SimpLinqDolbProd}
and \eqref{eq:SimpDolbProd} for $\CF_{l,n\star}$ and $\CF_{n\star}$ $(n\in\N)$
form inverse systems. Therefore, by the same argument as 
Remark \ref{rmk:PrisCohqdRFrobProd} (2), 
we can verify that the diagram below is commutative by 
using Remark \ref{rmk:ResolMdfProdComp} and Remark \ref{rmk:LinqDolqDolMdfProdComp}
instead of Remark \ref{rmk:LinResolTransFrobComp} (2) and Remark \ref{rmk:LinZarProjFrobComp} (2), 
respectively,
and by applying Lemma \ref{lem:ToposDirectImVarProjProd} to the morphism of ringed topos
$u_{\fX_{\star}/R}\colon ((\fX_{\star}/R)_{\prism}^{\sim},\CO_{\fX_{\star}/R})
\to (\fX_{\star,\Zar}^{\sim},R)$.
We abbreviate $u_{\fX_{\star}/R}$, $v_{D_{\star}}$, $v_{D_{\star}(1)}$,
$\CO_{\fX_{\star}/R}$, and $\varprojlim_n$ to 
$u$, $v$, $v^{(1)}$, $\CO$, and $\projl$, respectively. 
\begin{equation*}
\xymatrix@R=20pt{
Ru_*\CF_{0\star}\otimes_R^L
Ru_*\CF_{1\star}
\ar[r]
\ar[d]^{\cong}_{\eqref{eq:LinqdRCpxProjHCProj}\otimes\eqref{eq:LinqdRCpxProjHCProj}}
&
Ru_*(\CF_{0\star}\otimes_{\CO}\CF_{1\star})
\ar[r]&
Ru_*\CF_{\star}
\ar[d]^{\cong}_{\eqref{eq:LinqdRCpxProjHCProj}}\\
\projl v_*q\Omegab(\uCM_{0\star})\otimes^L_R
\projl v_*q\Omegab(\uCM_{1\star})
\ar[r]&
\projl (v_*q\Omegab(\uCM_{0\star})\otimes_R
v_*q\Omegab(\uCM_{1,\star}))
\ar[r]^(.61){\eqref{eq:SimpDolbProd}}&
\projl v_*^{(1)}q\Omegab(\uCM(1)_{\star})
}
\end{equation*}
\end{remark}

Finally we discuss the functoriality of the isomorphisms \eqref{eq:CrysZarProjDolCpxHC}
and \eqref{eq:LinqdRCpxProjHCProj} 
with respect to $\fX$, $\fX_{\star}$, $i_{\star}$ and $\ut_{\star}$.
Let $R'$ be another $q$-prism, and let $\fX'$, $\fX'_{\star}$,
$i'_{\star}\colon \fX'_{\star}\to \fY'_{\star}$, and $\ut'_{\star}$
be another set of objects over $R'$ satisfying the same condition as
$\fX$, $\fX_{\star}$, $i_{\star}$, and $\ut_{\star}$ over $R$.
We denote every object obtained from $\fX'$, $\fX'_{\star}$, ...
similarly to the constructions for $\fX$, $\fX_{\star}$, ... considered above
by the same symbol with a prime such as $D'$, $\fD'$. Suppose that
we are given a morphism of $q$-prisms $f\colon R\to R'$,
a morphism of formal schemes $g\colon \fX'\to \fX$ over
$f$, a morphism of simplicial formal schemes
$g_{\star}\colon \fX'_{\star}\to \fX_{\star}$ over $g$,
a morphism of simplicial formal schemes
$h_{\star}\colon \fY'_{\star}\to \fY_{\star}$
satisfying $h_{\star}\circ i'_{\star}=i_{\star}\circ g_{\star}$,
and a map of cosimplicial sets $\psi_{\star}\colon\Lambda_{\star}\to \Lambda'_{\star}$
with which $h_{\star}$ and $f$ define a morphism of cosimplicial framed smooth
$q$-pairs $((B_{\star},J_{\star})/R,\ut_{\star})
\to((B'_{\star},J'_{\star})/R',\ut'_{\star})$. 

Then the morphisms $g_{\star}$ and $h_{\star}$ induce a morphism of 
cosimplicial $q$-prisms $h_{D_{\star}}\colon D_{\star}\to D'_{\star}$
compatible with $g^*_{\star}\colon A_{\star}\to A'_{\star}$
and $h^*_{\star}\colon B_{\star}\to B'_{\star}$,
and then the following diagram of ringed topos 
commutative up to canonical isomorphisms
by the functoriality \eqref{eq:PrismZarProjFunct} 
of $u_{\fX/R}$ and $v_D$, where 
$R_n=R/(pR+\pq R)^{n+1}$ and $R'_n=R'/(pR'+\pq R')^{n+1}$.
\begin{equation}
\xymatrix@C=40pt@R=15pt{
((\fX'_{\star}/R')^{\sim}_{\prism},\CO_{\fX'_{\star}/R',n})
\ar[r]^(.57){u_{\fX'_{\star}/R}}\ar[d]_{g_{\star\prism}}&
(\fX^{\prime\sim}_{\star,\Zar},R'_n)\ar[d]^{g_{\star}}&
(\ofD^{\prime\sim}_{\star,\Zar},R'_n)\ar[l]_{v_{D'_{\star}}}\ar[d]^{\fh_{D_{\star}}}\\
((\fX_{\star}/R)^{\sim}_{\prism},\CO_{\fX_{\star}/R,n})\ar[r]^(.57){u_{\fX_{\star}/R}}&
(\fX^{\sim}_{\star,\Zar},R_n)&
(\ofD^{\sim}_{\star,\Zar},R_n)\ar[l]_{v_{D_{\star}}}
}
\end{equation}
The left square is compatible with the similar square
associated to $g\colon \fX'\to\fX$ via the morphisms $\theta_{\prism}$,
$\theta$ and the corresponding ones $\theta'_{\prism}$,
$\theta'$ associated to $\fX'_{\star}\to \fX'$. 

Let $\CF\in \Ob\Crystal_{\prism}(\CO_{\fX/R,n})$, and put
$\CF_{\star}=\theta_{\prism}^{-1}(\CF)$,
$\CF'=g_{\prism}^{-1}\CF\in \Ob\Crystal_{\prism}(\CO_{\fX'/R',n})$,
and $\CF'_{\star}=\theta_{\prism}^{\prime -1}\CF'$.
We have $g_{\star\prism}^{-1}\CF_{\star}=\CF'_{\star}$. Let
$\CM_{\star}\otimes_{\CO_{\fD_{\star}}}q\Omega_{\fD_{\star}/R}^{\bullet}$
(resp.~$\CM'_{\star}\otimes_{\CO_{\fD'_{\star}}}q\Omega_{\fD'_{\star}/R'}^{\bullet}$)
be the complex of $R_n$-modules on $\ofD_{\star,\Zar}$
(resp.~$R'_n$-modules on $\ofD'_{\star,\Zar})$ associated to
$\CF_{\star}$ and $\CF'_{\star}$, respectively, as mentioned
before Theorem \ref{th:CrystalCohqHiggsSimpl}. 
Then, by the functoriality \eqref{eq:qDolShfCpxFunct}
of this construction, we obtain a morphism of complexes of
$R_n$-modules on $\ofD_{\star,\Zar}$
\begin{equation}\label{eq:HPDolCpxFunct}
\CM_{\star}\otimes_{\CO_{\fD_{\star}}}q\Omega_{\fD_{\star}/R}^{\bullet}
\longrightarrow
\fh_{D_{\star}*}(\CM'_{\star}\otimes_{\CO_{\fD'_{\star}}}q\Omega_{\fD'_{\star}/R'}^{\bullet}).
\end{equation}

\begin{remark} By Remark \ref{rmk:qDolCpxShfFbFunct} (1), we see that the
morphisms $\varphi_{\fD_{\star,n},\Omega}^{\bullet}(\CM_{\star})$
and $\varphi_{\fD'_{\star,n},\Omega}^{\bullet}(\CM'_{\star})$
considered in \eqref{eq:LinResolInvSysSimpFrob} are compatible 
with the morphisms \eqref{eq:HPDolCpxFunct} for $\CM_{\star}$ and $\varphi_{\fD_{\star,n}}^*\CM_{\star}$.
\end{remark}
\begin{proposition}\label{eq:FunctCrysProjQDRHP}
Under the notation and assumption as above, the following diagram is commutative,
where the bottom morphism is induced by \eqref{eq:HPDolCpxFunct}
via $v_{D_{\star}*}\fh_{D_{\star}*}
\cong g_{\star*}v_{D'_{\star}*}$.
\begin{equation*}
\xymatrix@R=15pt{
Ru_{\fX_{\star}/R*}\CF_{\star}\ar[r]&
Ru_{\fX_{\star}/R*}Rg_{\star\prism*}\CF'_{\star}\ar[r]^{\cong}&
Rg_{\star*}Ru_{\fX'_{\star}/R'*}\CF'_{\star}\\
\ar[u]^{\eqref{eq:CrysZarProjDolCpxHC}}_{\cong}
v_{D_{\star}*}(\CM_{\star}\otimes_{\CO_{\fD_{\star}}}q\Omega_{\fD_{\star}/R}^{\bullet})
\ar[rr]&&
Rg_{\star*}(v_{D'_{\star}*}(\CM'_{\star}\otimes_{\CO_{\fD'_{\star}}}q\Omega_{\fD'_{\star}/R'}^{\bullet}))
\ar[u]^{\eqref{eq:CrysZarProjDolCpxHC}}_{\cong}
}
\end{equation*}
\end{proposition}

\begin{proof}
We abbreviate $\CM_{\star}\otimes_{\CO_{\fD_{\star}}}q\Omega_{\fD_{\star}/R}^{\bullet}$
and $\CM'_{\star}\otimes_{\CO_{\fD'_{\star}}}q\Omega_{\fD'_{\star}/R'}^{\bullet}$
to $q\Omega(\CM_{\star})$ and $q\Omega(\CM'_{\star})$. 
We write $\CF_{\star}\to \CL\Omega(M_{\star})$
and $\CF'_{\star}\to\CL\Omega(M'_{\star})$
for the resolution \eqref{eq:CrystalqDRResolHC} associated to $\CF_{\star}$ and $\CF'_{\star}$, respectively.
We have a morphism 
\begin{equation}\label{eq:LinDRCpxPBHC}
\CL\Omega(M_{\star})\to g_{\star\prism*}\CL\Omega(M'_{\star})
\end{equation}
defined by \eqref{eq:LineDRcpxPB} 
for $(f,g_{[\nu]},h_{[\nu]},\psi_{[\nu]})$ and $\CF_{[\nu]}$ $(\nu\in \N)$.
Then we see that it makes the following diagrams commutative
by \eqref{eq:CrystaldRResolFunct}, Lemma \ref{lem:LinZarqDRProjFunct}, and 
\eqref{eq:LinedRZARProjFunct},
where $u_{\star}$, $u'_{\star}$, $v_{\star}$, and $v'_{\star}$ denote
$u_{\fX_{\star}/R}$, $u_{\fX'_{\star}/R'}$, $v_{D_{\star}}$,
and $v_{D'_{\star}}$, respectively.
\begin{equation*}
\xymatrix@R=15pt@C=20pt{
\CF_{\star}\ar[r]\ar[d]& \CL\Omega(M_{\star})\ar[d]^{\eqref{eq:LinDRCpxPBHC}}\\
g_{\star\prism*}\CF'_{\star}\ar[r]&
g_{\star\prism*}\CL\Omega(M'_{\star})
}
\quad
\xymatrix@R=15pt@C=40pt{
u_{\star*}\CL\Omega(M_{\star})
\ar[r]^(.45){u_{\star*}\eqref{eq:LinDRCpxPBHC}}&
u_{\star*}g_{\star\prism*}\CL\Omega(M'_{\star})
\ar[r]^{\cong}&
g_{\star*}u'_{\star*}\CL\Omega(M'_{\star})\\
v_{\star*}q\Omega(\CM_{\star})\ar[r]^(.45){v_{\star*}\eqref{eq:HPDolCpxFunct}}
\ar[u]^{\eqref{eq:LinqdRCpxProjHC}}_{\cong}&
v_{\star*}\fh_{D_{\star*}}q\Omega(\CM'_{\star})\ar[r]^{\cong}&
g_{\star*}v'_{\star*}q\Omega(\CM'_{\star})\ar[u]^{g_{\star*}\eqref{eq:LinqdRCpxProjHC}}_{\cong}
}
\end{equation*}
We obtain the desired claim by combining 
$Ru_{\star*}$ of the left diagram with the right one
via the morphism $u_{\star*}\to Ru_{\star*}$ applied to
\eqref{eq:LinDRCpxPBHC}.
 \end{proof}
 
 Let $\CF$ be a complete crystal of $\CO_{\fX/R}$-modules on 
 $(\fX/R)_{\prism}$ (Definition \ref{def:PrismaticSite} (2)), and let 
 $(\CF_n)_{n\in\N}$ be the adic inverse system of crystals of
 $\CO_{\fX/R,n}$-modules on $(\fX/R)_{\prism}$ corresponding to 
 $\CF$ via the equivalence \eqref{eq:CompleteCrysInvSys}.
 By applying the constructions functorial in $\CF$ given 
 before Proposition \ref{eq:FunctCrysProjQDRHP} to $\CF_n$ $(n\in \N)$,
 we obtain inverse systems $(\CF_{\star,n})_{n\in\N}$,
 $(\CF'_{\star,n})_{n\in\N}$, 
 $(\CM_{\star,n}\otimes_{\CO_{\fD_{\star}}}q\Omega^{\bullet}_{\fD_{\star}/R})_{n\in\N}$,
 and $(\CM'_{\star,n}\otimes_{\CO_{\fD'_{\star}}}q\Omega^{\bullet}_{\fD'_{\star}/R'})_{n\in\N}$,
 and a morphism 
 \begin{equation}\label{eq:HPDolCpxFunctInvSys}
 (\CM_{\star,n}\otimes_{\CO_{\fD_{\star}}}q\Omega^{\bullet}_{\fD_{\star}/R})_{n\in\N},
\longrightarrow(
\fh_{\fD_{\star}*}(\CM'_{\star,n}\otimes_{\CO_{\fD'_{\star}}}q\Omega^{\bullet}_{\fD'_{\star}/R'}))_{n\in\N}
\end{equation}
of inverse systems of complexes of $R_n$-modules on $\fX_{\star,\Zar}$.
Put $\CF_{\star}=\theta_{\prism}^{-1}(\CF)$,
$\CF'=g_{\prism}^{-1}\CF\in \Ob\hCrystal_{\prism}(\CO_{\fX'/R'})$,
and $\CF'_{\star}=\theta_{\prism}^{\prime -1}\CF'$.
We have $\CF_{\star}\cong \varprojlim_n \CF_{\star,n}$
and $\CF'_{\star}\cong\varprojlim_n \CF'_{\star,n}$.

\begin{remark} By Remark \ref{rmk:qDolCpxShfFbFunct} (1), we see that
the inverse systems of morphisms 
$(\varphi_{\fD_{\star,n},\Omega}^{\bullet}(\CM_{\star,n}))_{n\in\N}$
and $(\varphi_{\fD'_{\star,n},\Omega}^{\bullet}(\CM'_{\star,n}))_{n\in\N}$
considered in Remark \ref{rmk:CrystZarProjqDolSimpFrob} (2) are compatible with \eqref{eq:HPDolCpxFunctInvSys}
for $(\CM_{\star,n})_{n\in\N}$ and $(\varphi_{\fD_{\star,n}}^*\CM_{\star,n})_{n\in\N}$.
\end{remark}
\begin{proposition}\label{prop:PrismZarProjqdRFunctHC}
 Under the notation and assumption as above, the following diagram is commutative,
where  the bottom morphism is induced by \eqref{eq:HPDolCpxFunctInvSys}
via $v_{D_{\star}*}\fh_{D_{\star}*}
\cong g_{\star*}v_{D'_{\star}*}$.
 \begin{equation*}
\xymatrix@C=10pt@R=20pt{
Ru_{\fX_{\star}/R*}\CF_{\star}\ar[r]&
Ru_{\fX_{\star}/R*}Rg_{\star\prism*}\CF'_{\star}\ar[r]^{\cong}&
Rg_{\star*}Ru_{\fX'_{\star}/R'*}\CF'_{\star}\\
\ar[u]^{\eqref{eq:LinqdRCpxProjHCProj}}_{\cong}
v_{D_{\star}*}(\varprojlim_n\CM_{\star,n}\otimes_{\CO_{\fD_{\star}}}q\Omega_{\fD_{\star}/R}^{\bullet})
\ar[rr]&&
Rg_{\star*}(v_{D'_{\star}*}(\varprojlim_n\CM'_{\star,n}\otimes_{\CO_{\fD'_{\star}}}q\Omega_{\fD'_{\star}/R'}^{\bullet}))
\ar[u]^{\eqref{eq:LinqdRCpxProjHCProj}}_{\cong}
}
\end{equation*}
 \end{proposition}

\begin{proof}
The construction of the two commutative diagrams in the proof of Proposition \ref{eq:FunctCrysProjQDRHP}
is functorial in $\CF$. Therefore the two diagrams associated to $\CF_n$ for each 
$n\in \N$ form inverse systems. We take their inverse limits over $n$. 
Then one can show the claim in the same way as the proof of Proposition 
\ref{eq:FunctCrysProjQDRHP} reviewing the construction of \eqref{eq:LinqdRCpxProjHCProj}. 
Note that the direct image functor under a morphism of topos preserves inverse limits.
\end{proof}

\bibliographystyle{amsplain}
\bibliography{PrismCrysQHiggs}

\end{document}